\documentclass[10pt]{amsart}

\newif\iffinalrun

\usepackage{amsmath}
\usepackage{latexsym}
\usepackage{amsfonts}
\usepackage{amssymb}
\usepackage{graphicx}
\usepackage{mathrsfs}
\usepackage{graphpap}
\usepackage{color}
\usepackage[margin=1.2in]{geometry}

\theoremstyle{definition}

\numberwithin{equation}{subsection}

\usepackage[all]{xy}
\theoremstyle{theorem}
\newtheorem{theo}[equation]{Theorem}
\newtheorem*{theo*}{Theorem}
\newtheorem{defi}[equation]{Definition}
\newtheorem{exam}[equation]{Example}
\newtheorem{lemm}[equation]{Lemma}
\newtheorem{coro}[equation]{Corollary}
\newtheorem{prop}[equation]{Proposition}
\newtheorem{rema}[equation]{Remark}

\newtheorem{conj}[equation]{Conjecture}

\iffinalrun
  \newcommand{\need}[1]{}
  \newcommand{\mar}[1]{}
\else
  \newcommand{\need}[1]{{\tiny *** #1}}
  \newcommand{\mar}[1]{\marginpar{\raggedright\tiny #1}}
\fi

\DeclareMathOperator{\Gal}{Gal}
\DeclareMathOperator{\Mat}{Mat}
\DeclareMathOperator{\Hom}{Hom}
\DeclareMathOperator{\ord}{ord}

\newcommand{\gr}{\operatorname{gr}}
\newcommand{\FBrModdd}[1][r]{\text{$\F$-$\operatorname{BrMod}_{\mathrm {dd}}^{#1}$}}
\newcommand{\Fp}[1][]{\mathbf{F}_{p^{#1}}}
\newcommand{\FFLMod}{\text{$\F$-$\operatorname{FLMod}$}}

\newcommand{\OEModdd}[1][r]{\text{$\mathcal{O}_E$-$\mathrm{Mod}_{\mathrm {dd}}^{#1}$}}

\newcommand{\fp}{\mathbf{F}_p}

\newcommand{\N}{\mathbf{N}}
\newcommand{\Q}{\mathbf{Q}}
\newcommand{\R}{\mathbf{R}}
\newcommand{\qp}{\mathbf{Q}_p}
\newcommand{\Qp}{\mathbf{Q}_p}

\newcommand{\Bst}{\mathbf{B}_{\mathrm{st}}}
\newcommand{\Bdr}{\mathbf{B}_{\mathrm{dR}}}
\newcommand{\barS}{\overline{S}}
\newcommand{\Ddr}{\mathrm{D}_{\mathrm{dR}}}
\newcommand{\Dst}{\mathrm{D}_{\mathrm{st}}}
\newcommand{\Vst}{\mathrm{V}_{\mathrm{st}}}
\newcommand{\Tst}{\mathrm{T}_{\mathrm{st}}}
\newcommand{\Tcris}{\mathrm{T}_{\mathrm{cris}}}

\newcommand{\Z}{\mathbf{Z}}

\newcommand{\F}{\mathbf{F}}

\newcommand{\rhobar}{\overline{\rho}}
\newcommand{\rbar}{\overline{r}}
\newcommand{\cM}{\mathcal{M}}
\newcommand{\cN}{\mathcal{N}}

\newcommand{\Fil}{\mathrm{Fil}}
\newcommand{\Ur}[1]{\mathrm{U}_{#1}}
\newcommand{\GL}{\mathrm{GL}}
\newcommand{\into}{\hookrightarrow}

\newcommand{\cA}{\mathcal{A}}
\newcommand{\cG}{\mathcal{G}}
\newcommand{\cO}{\mathcal{O}}
\newcommand{\cP}{\mathcal{P}}
\newcommand{\bA}{\mathbf{A}}
\newcommand{\bT}{\mathbf{T}}
\newcommand{\bC}{\mathbf{C}}
\newcommand{\Qpbar}{\overline{\Q}_p}

\newcommand{\WD}{\mathrm{WD}}
\newcommand{\rec}{\mathrm{rec}}
\newcommand{\Zp}{\mathbf{Z}_p}
\newcommand{\onto}{\twoheadrightarrow}

\newcommand{\Art}{\mathrm{Art}}

\begin{document}
\title{On mod $p$ local-global compatibility for $\GL_n(\Q_p)$ in the ordinary case}

\author{Chol Park}
\address{School of Mathematics, Korea Institute for Advanced Study, 85 Hoegiro Dondaemun-gu, Seoul 02455, Republic of Korea}
\email{cpark@kias.re.kr}

\author{Zicheng Qian}
\address{D\'epartement de Math\'ematiques Batiment 425, Facult\'e des Sciences d'Orsay Universit\'e Paris-Sud, 91405 Orsay, France}
\email{zicheng.qian@u-psud.math.fr}

\subjclass[2010]{11F80, 11F33.}

\maketitle

\begin{abstract}
Let $p$ be a prime number, $n>2$ an integer, and $F$ a CM field in which $p$ splits completely. Assume that a continuous automorphic Galois representation $\rbar:\Gal(\overline{\Q}/F)\rightarrow\GL_n(\overline{\F}_p)$ is upper-triangular and satisfies certain genericity conditions at a place $w$ above~$p$, and that every subquotient of $\rbar|_{\Gal(\overline{\Q}_p/F_w)}$ of dimension $>2$ is Fontaine--Laffaille generic. In this paper, we show that the isomorphism class of $\rbar|_{\Gal(\overline{\Q}_p/F_w)}$ is determined by $\GL_n(F_w)$-action on a space of mod $p$ algebraic automorphic forms cut out by the maximal ideal of a Hecke algebra associated to $\rbar$, 
assuming a weight elimination result which is a theorem of Bao V. Le Hung in his forthcoming paper~\cite{LeH}. In particular, we show that the wildly ramified part of $\rbar|_{\Gal(\overline{\Q}_p/F_w)}$ is determined by the action of Jacobi sum operators (seen as elements of $\F_p[\GL_n(\F_p)]$) on this space.
\end{abstract}

\setcounter{tocdepth}{2}

\tableofcontents

\section{Introduction}
It is believed that one can attach a smooth $\overline{\F}_p$-representation of $\GL_n(K)$ (or a packet of such representations) to a continuous Galois representation $\Gal(\overline{\Q}_p/K)\rightarrow\GL_n(\overline{\F}_p)$ in a natural way, that is called mod $p$ Langlands program for $\GL_n(K)$, where $K$ is a finite extension of $\Q_p$. This conjecture is well-understood for $\GL_2(\Q_p)$ (\cite{BL94}, \cite{Berger}, \cite{Bre03I}, \cite{Bre03II}, \cite{Col}, \cite{Pas}, \cite{CDP}, \cite{Emerton}). Beyond the $\GL_2(\Q_p)$-case, for instance $\GL_n(\Q_p)$ for $n>2$ or even $\GL_2(\Q_{p^f})$ for an unramified extension $\Q_{p^f}$ of $\Q_p$ of degree $f>1$, the situation is still quite far from being understood. One of the main difficulties is that there is no classification of such smooth representations of $\GL_n(K)$ unless $K=\Q_p$ and $n=2$: in particular, we barely understand the supercuspidal representations. Some of the difficulties in classifying the supercuspidal representations are illustrated in \cite{BP}, \cite{Hu} and \cite{Schraen}.

Let $F$ be a CM field in which $p$ is unramified, and $\rbar:\Gal(\overline{\Q}/F)\rightarrow\GL_n(\overline{\F}_p)$ an automorphic Galois representation. Although there is no precise statement of mod $p$ Langlands correspondence for $\GL_n(K)$ unless $K=\Q_p$ and $n=2$, one can define smooth representations $\Pi(\rbar)$ of $\GL_n(F_w)$ in the spaces of mod $p$ automorphic forms on a definite unitary group cut out by the maximal ideal of a Hecke algebra associated to $\rbar$, where $w$ is a place of $F$ above $p$. A precise definition of $\Pi(\rbar)$ when $p$ splits completely in $F$, which is our context, will be given in Section~\ref{subsec: Intro, mod p local-global comp}. (See also Section~\ref{subsec: Main results on l-g comp}.) One wishes that $\Pi(\rbar)$ is a candidate on the automorphic side corresponding to $\rbar|_{\Gal(\overline{\Q}_p/F_w)}$ for a mod $p$ Langlands correspondence in the spirit of Emerton~\cite{Emerton}. However, we barely understand the structure of $\Pi(\rbar)$ as a representation of $\GL_n(F_w)$, though the ordinary part of $\Pi(\rbar)$ is described in~\cite{BH} when $p$ splits completely in $F$ and $\rbar|_{\Gal(\overline{\Q}_p/F_w)}$ is ordinary. In particular, it is not known whether $\Pi(\rbar)$ and $\rbar|_{\Gal(\overline{\Q}_p/F_w)}$ determine each other. But we have the following conjecture:

\begin{conj}\label{main conjecture}
The local Galois representation $\rbar|_{\Gal(\overline{\Q}_p/F_w)}$ is determined by $\Pi(\rbar)$.
\end{conj}

This conjecture is widely expected to be true by experts but not explicitly written down before. The case $\GL_2(\Q_{p^f})$ was treated by Breuil--Diamond~\cite{BD}. Herzig--Le--Morra~\cite{HLM} considered the case $\GL_3(\Q_p)$ when $\rbar|_{\Gal(\overline{\Q}_p/F_w)}$ is upper-triangular and Fontaine--Laffaille, while the case $\GL_3(\Q_p)$ when $\rbar|_{\Gal(\overline{\Q}_p/F_w)}$ is an extension of a two dimensional irreducible representation by a character was considered by Le--Morra--Park~\cite{LMP}. We are informed that John Enns from the University of Toronto has worked on this conjecture for the group $\mathrm{GL_3}(\Q_{p^f})$. All of the results above are under certain generic assumptions on the tamely ramified part of~$\rbar|_{\Gal(\overline{\Q}_p/F_w)}$.

From another point of view, to a smooth admissible $\overline{\F}_p$-representation $\Pi$ of $\GL_n(K)$ for a finite extension $K$ of $\Q_p$, Scholze~\cite{Sch15} attaches a smooth admissible $\overline{\F}_p$-representation $S(\Pi)$ of $D^{\times}$ for a division algebra $D$ over $K$ with center $K$ and invariant $\frac{1}{n}$, which also has a continuous action of $\Gal(\overline{\Q}_p/K)$, via the mod $p$ cohomology of the Lubin--Tate tower. Using this construction, it was possible for Scholze to prove Conjecture~\ref{main conjecture} in full generality for $\mathrm{GL}_2(K)$(cf. \cite{Sch15}, Theorem 1.5). On the other hand, the proof of Theorem~1.5 of~\cite{Sch15} does not tell us where the invariants that determine $S(\Pi)$ lie. We do not know if there is any relation between these two different methods.

The weight part of Serre's conjecture already gives part of the information of $\Pi(\rbar)$: the local Serre weights of $\rbar$ at $w$ determine the socle of $\Pi(\rbar)|_{\GL_n(\cO_{F_w})}$ at least up to possible multiplicities, where $\cO_{F_w}$ is the ring of integers of $F_w$. If $\rbar|_{\Gal(\overline{\Q}_p/F_w)}$ is semisimple, then it is believed that the Serre weights of $\rbar$ at $w$ determine $\rbar|_{\Gal(\overline{\Q}_p/F_w)}$ up to twisting by unramified characters, but this is no longer the case if it is not semisimple:  the Serre weights are not enough to determine the wildly ramified part of $\rbar|_{\Gal(\overline{\Q}_p/F_w)}$, so that we need to understand a deeper structure of $\Pi(\rbar)$ than just its $\mathrm{GL}_n(\cO_{F_w})$-socle.

In this paper, we show that Conjecture \ref{main conjecture} is true when $p$ splits completely in $F$ and $\rbar|_{\Gal(\overline{\Q}_p/F_w)}$ is upper-triangular and sufficiently generic in a precise sense. Moreover, we describe the invariants in $\Pi(\rbar)$ that determine the wildly ramified part of $\rbar|_{\Gal(\overline{\Q}_p/F_w)}$. The generic assumptions on $\rbar|_{\Gal(\overline{\Q}_p/F_w)}$ ensure that very few Serre weights of $\rbar$ at $w$ will occur, 
which we call the weight elimination conjecture, Conjecture~\ref{Intro: weight elimination}. The weight elimination results are significant for our method to prove Conjecture~\ref{main conjecture}. But thanks to Bao V. Le Hung, this weight elimination conjecture is known to be true in his forthcoming paper~\cite{LeH}.
We follow the basic strategy in~\cite{BD, HLM}: we define Fontaine--Laffaille parameters on the Galois side using Fontaine--Laffaille modules as well as automorphic parameters on the automorphic side using the actions of Jacobi sum operators, and then identify them via the classical local Langlands correspondence. However, there are many new difficulties that didn't occur in \cite{BD} or in \cite{HLM}. For instance, the classification of semi-linear algebraic objects of rank $n>3$ on the Galois side is much more complicated. Moreover, failing of the multiplicity one property of the Jordan--H\"older factors of mod $p$ reduction of Deligne--Lusztig representations of $\GL_n(\Z_p)$ for $n>3$ implies that new ideas are required to show crucial non-vanishing of the automorphic parameters. In the rest of the introduction, we explain our ideas and results in more detail.

\subsection{Local Galois side}\label{subsec: intro, local galois side}
Let $E$ be a (sufficiently large) finite extension of $\Q_p$ with ring of integers $\cO_E$, a uniformizer $\varpi_E$, and residue field $\F$, and let $I_{\Q_p}$ be the inertia subgroup of $\Gal(\overline{\Q}_p/\Q_p)$ and $\omega$ the fundamental character of niveau $1$. We also let $\rhobar_0:\Gal(\overline{\Q}_p/\Q_p)\rightarrow \GL_n(\F)$ be a continuous (Fontaine-Laffaille) ordinary generic Galois representation. Namely, there exists a basis $\underline{e}:=(e_{n-1},e_{n-2},\cdots,e_0)$ for $\rhobar_0$ such that with respect to $\underline{e}$ the matrix form of $\rhobar_0$ is written as follows:
\begin{equation}\label{Intro: ordinary representation}
\rhobar_0|_{I_{\Q_p}}\cong
\left(
  \begin{array}{ccccccc}
   \omega^{c_{n-1}+(n-1)} & \ast_{n-1} & \ast & \cdots &\ast & \ast \\
    0 & \omega^{c_{n-2}+(n-2)} & \ast_{n-2} & \cdots &\ast & \ast \\
    0 & 0  &\omega^{c_{n-3}+(n-3)} & \cdots &\ast & \ast \\
    \vdots & \vdots & \vdots & \ddots &\vdots&\vdots \\
    0 & 0 & 0 & \cdots & \omega^{c_1+1} & \ast_1 \\
    0 & 0 & 0 & \cdots& 0     &\omega^{c_0} \\
  \end{array}
\right)
\end{equation}
for some integers $c_i$ satisfying some genericity conditions (cf. Definition~\ref{definition: genericity condition}). We also assume that $\rhobar_0$ is maximally non-split, i.e., $\ast_i\not=0$ for all $i\in\{1,2,\cdots,n-1\}$.

Our goal on the Galois side is to show that the Frobenius eigenvalues of certain potentially crystalline lifts of $\rhobar_0$ determine the Fontaine--Laffaille parameters of $\rhobar_0$, which parameterizes the wildly ramified part of $\rhobar_0$.
When the unramified part and the tamely ramified part of $\rhobar_0$ are fixed, we define the Fontaine--Laffaille parameters via the Fontaine--Laffaille modules corresponding to~$\rhobar_0$ (cf. Definition~\ref{definiton: Fontaine--Laffaille parameters}). These parameters vary over the space of $\frac{(n-1)(n-2)}{2}$ copies of the projective line $\mathbb{P}^1(\F)$, and we write $\mathrm{FL}_n^{i_0,j_0}(\rhobar_0)\in\mathbb{P}^1(\F)$ for each pair of integers $(i_0,j_0)$ with $0\leq j_0<j_0+1<i_0\leq n-1$. 
For each such pair $(i_0,j_0)$, the Fontaine--Laffaille parameter $\mathrm{FL}_n^{i_0,j_0}(\rhobar_0)$ is determined by the subquotient $\rhobar_{i_0,j_0}$ of $\rhobar_0$ which is determined by the subset $(e_{i_0},e_{i_0-1},\cdots,e_{j_0})$ of $\underline{e}$ (cf. (\ref{subquotient of rhobar0})): in fact, we have the identity $\mathrm{FL}_n^{i_0,j_0}(\rhobar_0)=\mathrm{FL}_{i_0-j_0+1}^{i_0-j_0,0}(\rhobar_{i_0,j_0})$ (cf. Lemma~\ref{lemm: FL parameter with dual rep}).

Since potentially crystalline lifts of $\rhobar_0$ are not Fontaine--Laffaille in general, we are no longer able to use Fontaine--Laffaille theory to study such lifts of $\rhobar_0$; we use Breuil modules and strongly divisible modules for their lifts. It is obvious that any lift of $\rhobar_0$ determines the Fontaine--Laffaille parameters, but it is not obvious how one can explicitly visualize the information that determines $\rhobar_0$ in those lifts. Motivated by the automorphic side, we believe that for each pair $(i_0,j_0)$ as above the Fontaine--Laffaille parameter $\mathrm{FL}_n^{i_0,j_0}(\rhobar_0)$ is determined by a certain product of Frobenius eigenvalues of the potentially crystalline lifts of $\rhobar_0$ with Hodge--Tate weights $\{-(n-1),\cdots,-1,0\}$ and Galois type $\bigoplus_{i=0}^{n-1}\widetilde{\omega}^{k_i^{i_0,j_0}}$ where $\widetilde{\omega}$ is the Teichm\"uler lift of the fundamental character $\omega$ of niveau $1$ and
\begin{equation}\label{Intro, Galois types}
k_i^{i_0,j_0}\equiv
\left\{
\begin{array}{ll}
c_{i_0}+i_0-j_0-1 & \hbox{for $i=i_0$;}\\
c_{j_0}-(i_0-j_0-1) & \hbox{for $i=j_0$;}\\
c_i & \hbox{otherwise}
\end{array}
\right.
\end{equation}
modulo $(p-1)$. Here, $c_i$ are the integers determining the tamely ramified part of $\rhobar_0$ in~ (\ref{Intro: ordinary representation}) and our normalization of the Hodge--Tate weight of the cyclotomic character $\varepsilon$ is $-1$.

Our main result on the Galois side is the following:
\begin{theo}[Theorem~\ref{thm: main theorem Galois}]\label{Intro, thm, Galois main}
Fix $i_0,j_0\in\Z$ with $0\leq j_0<j_0+1<i_0\leq n-1$. Assume that $\rhobar_0$ is generic (cf. Definition~\ref{definition: genericity condition}) and that $\rhobar_{i_0,j_0}$ is Fontaine--Laffaille generic (cf. Definition~\ref{definition: Fontaine--Laffaille generic}), and let $(\lambda^{i_0,j_0}_{n-1},\lambda^{i_0,j_0}_{n-2},\cdots,\lambda^{i_0,j_0}_0)\in (\cO_E)^{n}$ be the Frobenius eigenvalues on the $(\widetilde{\omega}^{k^{i_0,j_0}_{n-1}}, \widetilde{\omega}^{k^{i_0,j_0}_{n-2}}, \cdots, \widetilde{\omega}^{k^{i_0,j_0}_0})$-isotypic components of $\Dst^{\Qp,n-1}(\rho_0)$ where $\rho_0$ is a potentially crystalline lift of $\rhobar_0$ with Hodge--Tate weights $\{-(n-1),-(n-2),\cdots,-1,0\}$ and Galois type $\bigoplus_{i=0}^{n-1}\widetilde{\omega}^{k^{i_0,j_0}_i}$.

Then the Fontaine--Laffaille parameter $\mathrm{FL}_n^{i_0,j_0}$ associated to $\rhobar_{0}$ is computed as follows:
$$\mathrm{FL}_n^{i_0,j_0}(\rhobar_{0})= \left[1:\overline{\left(\frac{p^{[(n-1)-\frac{i_0+j_0}{2}](i_0-j_0-1)}}{\prod_{i=j_0+1}^{i_0-1}\lambda^{i_0,j_0}_{i}}\right)} \right] \in \mathbb{P}^1(\F).$$
\end{theo}

Note that by $\overline{\bullet}\in\F$ in the theorem above we mean the image of $\bullet\in\cO_E$ under the natural surjection $\cO_E\twoheadrightarrow\F$.
We also note that $\rhobar_{i_0,j_0}$ being Fontaine--Laffaille generic implies $\mathrm{FL}_n^{i_0,j_0}(\rhobar_0)\neq 0, \infty$ for all $i_0,j_0$ as in Theorem~\ref{Intro, thm, Galois main}, but is a strictly stronger assumption if $i_0-j_0\geq 3$.

Let us briefly discuss our strategy for the proof of Theorem~\ref{Intro, thm, Galois main}. Recall that the Fontaine--Laffaille parameter $\mathrm{FL}_n^{i_0,j_0}(\rhobar_0)$ is defined in terms of the Fontaine--Laffaille module corresponding to $\rhobar_0$. Thus we need to describe $\mathrm{FL}_n^{i_0,j_0}(\rhobar_0)$ by the data of the Breuil modules of inertial type $\bigoplus_{i=0}^{n-1}\omega^{k_i^{i_0,j_0}}$ corresponding to $\rhobar_0$, and we do this via \'etale $\phi$-modules, which requires classification of such Breuil modules. If the filtration of the Breuil modules is of a certain shape, then a certain product of the Frobenius eigenvalues of the Breuil modules determines a Fontaine--Laffaille parameter (cf. Proposition~\ref{prop: Breuil modules for newest FL 3}). In order to get such a filtration, we need to assume that $\rhobar_{i_0,j_0}$ is Fontaine--Laffaille generic (cf. Definition~\ref{definition: Fontaine--Laffaille generic}). Then we determine the structure of the filtration of the strongly divisible modules lifting the Breuil modules by direct computation, which immediately gives enough properties of Frobenius eigenvalues of the potentially crystalline representations we consider. But this whole process is subtle for general $i_0,j_0$. To resolve this issue we prove that any potentially crystalline lift of $\rhobar_0$ with Hodge--Tate weights $\{-(n-1),-(n-2),\cdots,0\}$ and Galois type $\bigoplus_{i=0}^{n-1}\widetilde{\omega}^{k_i^{i_0,j_0}}$ has a potentially crystalline subquotient $\rho_{i_0,j_0}$ of Hodge--Tate weights $\{-i_0,\cdots,-j_0\}$ and of Galois type $\bigoplus_{i=j_0}^{i_0}\widetilde{\omega}^{k_i^{i_0,j_0}}$ lifting $\rhobar_{i_0,j_0}$. More precisely,
\begin{theo}[Corollary~\ref{coro: reducibility of certain lifts}]\label{Intro, thm, reducibility}
Every potentially crystalline lift $\rho_0$ of $\rhobar_0$ with Hodge--Tate weights $\{-(n-1),-(n-2),\cdots,0\}$ and Galois type $\bigoplus_{i=0}^{n-1}\widetilde{\omega}^{k_i^{i_0,j_0}}$ is a successive extension
$$\rho_0\cong
\begin{pmatrix}
\rho_{n-1,n-1} & \cdots & \ast & \ast & \ast & \cdots & \ast \\
 & \ddots &\vdots &\vdots& \vdots &\ddots & \vdots \\
 &  &\rho_{i_0+1,i_0+1}&\ast&\ast & \cdots &\ast\\
 &  & &\rho_{i_0,j_0}&\ast&\cdots &\ast\\
 &  & & &\rho_{j_0-1,j_0-1}&\cdots &\ast\\
 &  & & &  & \ddots & \vdots \\
 &  & & &  &        &\rho_{0,0}
\end{pmatrix}
$$
where
\begin{itemize}
\item for $n-1\geq i>i_0$ and $j_0>i\geq 0$, $\rho_{i,i}$ is a $1$-dimensional potentially crystalline lift of $\rhobar_{i,i}$ with Hodge--Tate weight $-i$ and Galois type $\widetilde{\omega}^{k_i^{i_0,j_0}}$;
\item $\rho_{i_0,j_0}$ is a $(i_0-j_0+1)$-dimensional potentially crystalline lift of $\rhobar_{i_0,j_0}$ with Hodge--Tate weights $\{-i_0,-i_0+1,\cdots,-j_0\}$ and Galois type $\bigoplus_{i=j_0}^{i_0}\widetilde{\omega}^{k_i^{i_0,j_0}}$.
\end{itemize}
\end{theo}

Note that we actually prove the niveau $f$ version of Theorem~\ref{Intro, thm, reducibility} since it adds only little more extra work (cf. Corollary~\ref{coro: reducibility of certain lifts}). 

The representation $\rho_{i_0,j_0}\otimes \varepsilon^{-j_0}$ is a $(i_0-j_0+1)$-dimensional potentially crystalline lift of $\rhobar_{i_0,j_0}$ with Hodge--Tate weights $\{-(i_0-j_0),-(i_0-j_0-1),\cdots,0\}$ and Galois type $\bigoplus_{i=j_0}^{i_0}\widetilde{\omega}^{k_i^{i_0,j_0}}$, so that, by Theorem~\ref{Intro, thm, reducibility}, Theorem~\ref{Intro, thm, Galois main} reduces to the case $(i_0,j_0)=(n-1,0)$: we prove Theorem~\ref{Intro, thm, Galois main} when $(i_0,j_0)=(n-1,0)$, and then use the fact $\mathrm{FL}_n^{i_0,j_0}(\rhobar_0)=\mathrm{FL}_{i_0-j_0+1}^{i_0-j_0,0}(\rhobar_{i_0,j_0})$ to get the result for general $i_0,j_0$.

The Weil--Deligne representation $\mathrm{WD}(\rho_0)$ associated to $\rho_0$ (as in Theorem~\ref{Intro, thm, Galois main}) contains those Frobenius eigenvalues of $\rho_0$. We then use the classical local Langlands correspondence for $\GL_n$ to transport the Frobenius eigenvaluses of $\rho_0$ (and so the Fontaine--Laffaille parameters of $\rhobar_0$ as well by Theorem~\ref{Intro, thm, Galois main}) to the automorphic side (cf. Corollary~\ref{coro: main thm}).

\subsection{Local automorphic side}\label{subsec: intro, local automorphic side}
We start by introducing the Jacobi sum operators in characteristic $p$. Let $T$ (resp. $B$) be the maximal torus (resp. the maximal Borel subgroup) consisting of diagonal matrices (resp. of upper-triangular matrices) of $\GL_n$. We let $X(T):=\mathrm{Hom}(T, \mathbf{G}_m)$ be the group of characters of $T$ and $\Phi^+$ be the set of positive roots with respect to $(B,T)$. We define $\epsilon_i\in X(T)$ as the projection of $T\cong\mathbf{G}_m^{n}$ onto the $i$-th factor. Then the elements $\{\epsilon_i\mid 1\leq i\leq n\}$ forms a $\Z$-basis of the free abelian group $X(T)$. We will use the notation $(d_1,d_2,\cdots,d_n)\in\Z^{n}$ for the element $\sum_{k=1}^nd_k\epsilon_k\in X(T)$. Note that the group of characters of the finite group $T(\F_p)\cong(\F_p^{\times})^{n}$ can be identified with $X(T)/(p-1)X(T)$, and therefore we sometimes abuse the notation $(d_1,d_2,\cdots,d_n)$ for its image in $X(T)/(p-1)X(T)$. We define $\Delta:=\{\alpha_k:=\epsilon_k-\epsilon_{k+1}\mid 1\leq k\leq n-1\}\subset\Phi^+$ as the set of positive simple roots.  Note that we write $s_k$ for the reflection of the simple root $\alpha_k$. For an element $w$ in the Weyl group $W$, we define $\Phi^+_w=\{\alpha\in\Phi^+\mid w(\alpha)\in-\Phi^+\}\subseteq\Phi^+$ and $U_w=\prod_{\alpha\in\Phi^+_w}U_{\alpha}$, where $U_{\alpha}$ is a subgroup of $U$ whose only non-zero off-diagonal entry corresponds to $\alpha$. Note in particular that $\Phi^+=\Phi^+_{w_0}$, where $w_0$ is the longest element in $W$. For $w\in W$ and for a tuple of integers $\underline{k}=(k_{\alpha})_{\alpha\in\Phi_w^+}\in\{0,1,\cdots,p-1\}^{|\Phi_w^+|}$, we define the Jacobi sum operator
\begin{equation}\label{Intro: Jacobi sum operator}
S_{\underline{k},w}:=\sum_{A\in U_w(\F_p)}\left(\prod_{\alpha\in\Phi^+_w}A_{\alpha}^{k_{\alpha}}\right)A\cdot w\in\F_p[\GL_n(\F_p)]
\end{equation}
where $A_{\alpha}$ is the entry of $A$ corresponding to $\alpha\in\Phi_w^+$. In Section~\ref{sec: local automorphic side}, we establish many technical results, both conceptual and computational, around these Jacobi sum operators. The use of these Jacobi sum operators can be traced back to at least \cite{CarterLusztig}, and are widely used for $\GL_2$ in \cite{BP} and \cite{Hu} for instance. 
But systematic computation with these operators seems to be limited to $\GL_2$ or $\GL_3$. In this paper, we need to do some specific but technical computation on some special Jacobi sum operators for $\GL_n(\F_p)$, which is enough for our application to Theorem~\ref{Intro, thm, main results} below.

By the discussion on the local Galois side, our target on the local automorphic side is to capture the Frobenius eigenvalues coming from the local Galois side. By the classical local Langlands correspondence, the Frobenius eigenvalues of $\rho_0$ are transported to the unramified part of $\chi$ in the tamely ramified principal series $\mathrm{Ind}^{\mathrm{GL}_n(\Q_p)}_{B(\Q_p)}\chi$ corresponding to the Weil--Deligne representation $\mathrm{WD}(\rho_0)$ attached to $\rho_0$ in Theorem~\ref{Intro, thm, Galois main}, and it is standard to use $U_p$-operators to capture the information in the unramified part of $\chi$.

The normalizer of the Iwahori subgroup $I$ in $\mathrm{GL}_n(\Q_p)$ is cyclic modulo $I$, and this cyclic quotient group is generated by an element $\Xi_n\in\mathrm{GL}_n(\Q_p)$ that is explicitly defined in (\ref{generator of normalizer}). One of our goals is to translate the eigenvalue of $U_p$-operators into the action of $\Xi_n$ on the space $(\mathrm{Ind}^{\mathrm{GL}_n(\Q_p)}_{B(\Q_p)}\chi)|_{\mathrm{GL}_n(\Z_p)}$. This is firstly done for $\mathrm{GL}_2(\Q_{p^f})$ in \cite{BD}, and then the method is generalized to $\mathrm{GL}_3(\Q_p)$ in the ordinary case by \cite{HLM}. 
Both \cite{BD} and \cite{HLM} need a pair of group algebra operators: for instance, group algebra operators $\widehat{S}, \widehat{S}^{\prime}\in\Q_p[\mathrm{GL}_3(\Q_p)]$ are defined in \cite{HLM} and the authors prove an intertwining identity of the form $\widehat{S}^{\prime}\cdot\Xi_3=c\widehat{S}$ on a certain $I(1)$-fixed subspace of $\mathrm{Ind}^{\mathrm{GL}_3(\Q_p)}_{B(\Q_p)}\chi$ with $\chi$ assumed to be tamely ramified, where $I(1)$ is the maximal pro-$p$ subgroup of $I$. Here, the constant $c\in\cO_E^{\times}$ captures the eigenvalues of $U_p$-operators. This is the first technical point on the local automorphic side, and we generalize the results in \cite{BD} and \cite{HLM} by the following theorem.

For an $n$-tuple of integers $(a_{n-1},a_{n-2},\cdots,a_0)\in\Z^n$, we write $\mathcal{S}_n$ and $\mathcal{S}^{\prime}_n$ for $S_{\underline{k}^1,w_0}$ with $\underline{k}^1=(k^1_{i,j})$ and $S_{\underline{k}^{1,\prime},w_0}$ with $\underline{k}^{1,\prime}=(k^{1,\prime}_{i,j})$ respectively, where $k^1_{i,i+1}=[a_0-a_{n-i}]_1+n-2$, $k^{1,\prime}_{i,i+1}=[a_{n-i-1}-a_{n-1}]_1+n-2$ for $1\leq i\leq n-1$, and $k^1_{i,j}=k^{1,\prime}_{i,j}=0$ otherwise. Here, $(i,j)$ is the entry corresponding to $\alpha$ if $\alpha=\epsilon_i-\epsilon_j\in\Phi^+$ and by $[x]_1$ for $x\in\Z$ we mean the integer in $[0,p-1)$ such that $x\equiv [x]_1$ modulo $(p-1)$. We define $\widehat{\mathcal{S}}_n\in\Z_p[\mathrm{GL}_n(\Z_p)]$ (resp. $\widehat{\mathcal{S}}_n^{\prime}\in \Z_p[\mathrm{GL}_n(\Z_p)]$) by taking the Teichm\"uller lifts of the coefficients and the entries of the matrices of $\mathcal{S}_n\in \F_p[\mathrm{GL}_n(\F_p)]$ (resp. of $\mathcal{S}_n^{\prime}\in \F_p[\mathrm{GL}_n(\F_p)]$).

We use the notation $\bullet$ for the composition of maps or group operators to distinguish from the notation $\circ$ for a $\mathcal{O}_E\text{-}$lattice inside a representation.

\begin{theo}[Theorem~\ref{theo: identity}]\label{Intro, thm, identity}
Assume that the $n$-tuple of integers $(a_{n-1},a_{n-2},\cdots,a_0)$ is $n$-generic in the lowest alcove (cf. Definition~\ref{defi: generic on tuples}), and let $$\Pi_p=\mathrm{Ind}^{\GL_n(\Q_p)}_{B(\Q_p)}(\chi_1\otimes\chi_2\otimes\chi_3\otimes ...\otimes\chi_{n-2}\otimes\chi_{n-1}\otimes\chi_0)$$ be a tamely ramified principal series representation with the smooth characters $\chi_k: \Q_p^{\times}\rightarrow E^{\times}$ satisfying $\chi_k|_{\Z_p^{\times}}=\widetilde{\omega}^{a_k}$ for $0\leq k\leq n-1$.

On the $1$-dimensional subspace $\Pi_p^{I(1), (a_1,a_2,...,a_{n-1},a_0)}$ we have the identity:
\begin{equation}\label{inter identity characteristic 0}
\widehat{\mathcal{S}}_n^{\prime}\bullet(\Xi_n)^{n-2}=p^{n-2}\kappa_n\left(\prod^{n-2}_{k=1}\chi_{k}(p)\right)\widehat{\mathcal{S}}_n
\end{equation}
for $\kappa_n\in\Z_p^{\times}$ satisfying $\kappa_n\equiv \varepsilon^{\ast}\cdot\mathcal{P}_n(a_{n-1},\cdots,a_0)$ mod $(\varpi_E)$ where
$$\varepsilon^{\ast}=\prod_{k=1}^{n-2}(-1)^{a_0-a_k}$$
and
$$\mathcal{P}_n(a_{n-1},\cdots,a_0)=\prod_{k=1}^{n-2}\prod_{j=0}^{n-3}\frac{a_k-a_{n-1}+j}{a_0-a_k+j}\in\Z_p^{\times}.$$

\end{theo}

In fact, there are many identities similar to the one in (\ref{inter identity characteristic 0}) for each operator $U^i_n$ for $1\leq i\leq n-1$ that is defined in (\ref{operator}), but it is clear from the proof of Theorem~\ref{Intro, thm, identity} in Section~\ref{subsec: Jacobi sums in char. 0} that we need to choose $U^{n-2}_n$ for the $U_p$-operator acting on $\Pi_p^{I(1), (a_1,a_2,...,a_{n-1},a_0)}$, motivated from the local Galois side via Theorem \ref{Intro, thm, Galois main}. The crucial point here is that the constant
$p^{n-2}\kappa_n\left(\prod^{n-2}_{k=1}\chi_k(p)\right)$, which is closely related to $\mathrm{FL}_n^{n-1,0}(\rhobar_0)$ via Theorem~\ref{Intro, thm, Galois main} and classical local Langlands correspondence, should lie in $\mathcal{O}_E^{\times}$ for each $\Pi_p$ appearing in our application of Theorem \ref{Intro, thm, identity} to Theorem \ref{Intro, thm, main results}.

The next step is to consider the mod $p$ reduction of the identity (\ref{inter identity characteristic 0}), which is effective to capture $p^{n-2}\prod^{n-2}_{k=1}\chi_{k}(p)$ modulo $(\varpi_E)$ only if $\widehat{\mathcal{S}}_n\widehat{v}\not\equiv 0$ modulo $(\varpi_E)$ for $\widehat{v}\in\Pi_p^{I(1), (a_1,a_2,...,a_{n-1},a_0)}$. It turns out that this non-vanishing property is very technical to prove for general $\GL_n(\Q_p)$.  Before we state our non-vanishing result, we fix a little more notation: let
\begin{equation*}\label{Intro: certain weights}
\left\{
  \begin{array}{ll}
   \mu^{\ast}:=(a_{n-1}-n+2,a_{n-2},\cdots,a_1,a_0+n-2);  & \hbox{} \\
   \mu_1:=(a_1,a_2,\cdots,a_{n-3},a_{n-2},a_{n-1},a_0);  & \hbox{} \\
   \mu'_1:=(a_{n-1},a_0,a_1,a_2,\cdots,a_{n-3},a_{n-2})  & \hbox{}
  \end{array}
\right.
\end{equation*}
be three characters of $T(\F_p)$, and write $\pi_1$ and $\pi'_1$ for two characteristic $p$ principal series induced by the characters $\mu_1$ and $\mu'_1$ respectively (cf. (\ref{certain weights})). Note that we can attach an irreducible representation $F(\lambda)$ of $\mathrm{GL}_n(\F_p)$ to each $\lambda\in X(T)/(p-1)X(T)$ satisfying some regular conditions (cf. the beginning of Section~\ref{sec: local automorphic side}). If we assume that $(a_{n-1},\cdots,a_0)\in\Z^n$ is $n$-generic in the lowest alcove, the characters $\mu^{\ast}$, $\mu_1$ and $\mu_1^{\prime}$ do satisfy the regular condition and thus we have three irreducible representations $F(\mu^{\ast})$, $F(\mu_1)$ and $F(\mu^{\prime}_1)$ of $\mathrm{GL}_n(\F_p)$.  There is a unique quotient $\mathcal{V}$ (resp. $\mathcal{V}^{\prime}$) (up to isomorphism) of $\pi_1$ (resp. of $\pi^{\prime}_1$) whose socle is isomorphic to $F(\mu^{\ast})$, since $F(\mu^{\ast})$ has multiplicity one in $\pi_1$ (resp. in $\pi_1^{\prime}$) by Theorem~\ref{conj: mult}. 



We are now ready to state the non-vanishing theorem.

\begin{theo}[Corollary~\ref{coro: isomorphism}]\label{Intro, thm, Auto main}
Assume that the $n$-tuple of integers $(a_{n-1},a_{n-2},\cdots,a_0)$ is $2n$-generic in the lowest alcove (cf. Definition~\ref{defi: generic on tuples}).

Then we have
\begin{equation*}
0\neq \mathcal{S}_n\left(\mathcal{V}^{ U(\F_p),\mu_1}\right)\subseteq \mathcal{V}\,\,\mbox{ and }\,\,
0\neq \mathcal{S}^{\prime}_n\left((\mathcal{V}^{\prime})^{U(\F_p),\mu^{\prime}_1}\right)\subseteq \mathcal{V}^{\prime}.
\end{equation*}
\end{theo}

The definition of $\mu_1,\mu_1^{\prime}$ and $\mu^{\ast}$ is motivated by our application of Theorem \ref{Intro, thm, Auto main} to Theorem \ref{Intro, thm, main results} and is closely related to the Galois types we choose in Theorem \ref{Intro, thm, Galois main}. We emphasize that, technically speaking, it is crucial that $F(\mu^{\ast})$ has multiplicity one in $\pi_1$ and $\pi^{\prime}_1$. The proof of Theorem~\ref{Intro, thm, Auto main} is technical and makes full use of the results in Sections~\ref{subsec: Jacobi sumes in char. p}, \ref{subsec: Some technical formula}, and \ref{subsec: Proof of non-vanishing}.

\subsection{Weight elimination and automorphy of a Serre weight}\label{subsec: Intro, weight elimination}
The weight part of Serre's conjecture is considered as a first step towards mod $p$ Langlands program, since it gives a description of the socles of $\Pi(\rbar)|_{\GL_n(\Z_p)}$ up to possible multiplicities. Substantial progress has been made for the groups $\GL_2(\cO_K)$, where $\cO_K$ is the ring of integers of a finite extension $K$ of $\Q_p$ (\cite{BDJ}, \cite{Gee}, \cite{GK}, \cite{GLS14}, \cite{GLS15}). For groups in higher semisimple rank, we also have a detailed description. (See \cite{EGH}, \cite{HLM}, \cite{LMP}, \cite{MP}, \cite{LLHLM} for $\GL_3$; \cite{herzig-duke}, \cite{GG}, \cite{BLGG}, \cite{LLL}, \cite{GHS} for general~$n$.)

Weight elimination results are significant for the proof of our main global application, Theorem~\ref{Intro, thm, main results}. For the purpose of this introduction, we quickly review some notation. Let $F^+$ be the maximal totally real subfield of a CM field $F$, and assume that $p$ splits completely in $F$. Fix a place $w$ of $F$ above $p$ and set $v:=w|_{F^+}$. We assume that $\rbar$ is automorphic: this means that there exist a totally definite unitary group $G_n$ defined over $F^+$ that is an outer form of $\GL_{n/F^+}$ and split at places above $p$, an integral model $\cG_n$ of $G_n$ such that $\cG_n\times \cO_{F^+_{v'}}$ is reductive if $v'$ is a finite place of $F^+$ that splits in $F$, a compact open subgroup $U=\cG_n(\cO_{F^+_v})\times U^v\subseteq \cG_n(\cO_{F^+_v})\times {G_n}(\mathbf{A}_{F^+}^{\infty,v})$ that is sufficiently small and unramified above $p$, a Serre weight $V=\bigotimes_{v'|p}V_{v'}$ that is an irreducible smooth $\overline{\F}_p$-representation of $\cG_n(\cO_{F^+,p})$, and a maximal ideal $\mathfrak{m}_{\rbar}$ associated to~$\rbar$ in the Hecke algebra acting on the space $S(U,V)$ of mod $p$ algebraic automorphic forms such that
\begin{equation}\label{Intro: definition of modular}
S(U,V)[\mathfrak{m}_{\rbar}]\neq0.
\end{equation}

We write $W(\rbar)$ for the set of Serre weights $V$ satisfying~(\ref{Intro: definition of modular}) for some $U$, and $W_w(\rbar)$ for the set of local Serre weights $V_v$, that is irreducible smooth representations of $\cG_n(\cO_{F^+_v})\cong \GL_n(\cO_{F_w})\cong \GL_n(\Z_p)$,  such that $V_v\otimes (\bigotimes_{v'\neq v}V_{v'})\in W(\rbar)$ for an irreducible smooth representation $\bigotimes_{v'\neq v}V_{v'}$ of $\prod_{v'\neq v}\cG_n(\cO_{F^+_{v'}})$. The local Serre weights $V_v$ have an explicit description as representations of $\GL_n(\F_p)$: there exists a $p$-restricted (i.e. $0\leq a_i-a_{i-1}\leq p-1$ for all $1\leq i\leq n-1$) weight $\underline{a}:=(a_{n-1},a_{n-2},\cdots,a_0)\in X(T)$ such that $F(\underline{a})\cong V_v$ where $F(\underline{a})$ is the irreducible socle of the dual Weyl module associated to $\underline{a}$ (cf. Section~\ref{subsec: Serre weigts and pot crys lifts} as well as the beginning of Section~\ref{sec: local automorphic side}).

Assume that $\rbar|_{\Gal(\overline{\Q}_p/F_w)}\cong\rhobar_0$, where $\rhobar_0$ is defined as in~(\ref{Intro: ordinary representation}).
We define certain characters $\mu^{\square}$ and $\mu^{\square, i_1,j_1}$ of $T(\F_p)$ and principal series $\pi^{i_1,j_1}$ of $\mathrm{GL}_n(\F_p)$ at the beginning of Section~\ref{subsec: weight elimination}. Our main conjecture for weight elimination is
\begin{conj}[Conjecture~\ref{conj: weight elimination}]\label{Intro: weight elimination}
Assume that $\rhobar_{i_0,j_0}$ is Fontaine--Laffaille generic and that $\mu^{\square,i_1,j_1}$ is $2n$-generic. Then we have an inclusion
\begin{equation}\label{elimination equation}
 W_w(\rbar)\cap\mathrm{JH}((\pi^{i_1,j_1})^{\vee})\subseteq \{F(\mu^{\square})^{\vee}, F(\mu^{\square,i_1,j_1})^{\vee}\}.
\end{equation}
\end{conj}

We emphasize that the condition $\rhobar_{i_0,j_0}$ is Fontaine--Laffaille generic is crucial in Conjecture~\ref{Intro: weight elimination}. For example, if $n=4$ and $(i_0,j_0)=(3,0)$ and we assume merely $\mathrm{FL}^{3,0}_4(\rhobar_0)\neq 0, \infty$ (which is strictly weaker than Fontaine--Laffaille generic), then we expect that an extra Serre weight can possibly appear in $W_w(\rbar)\cap\mathrm{JH}((\pi^{i_1,j_1})^{\vee})$.

The Conjecture~\ref{Intro: weight elimination} is motivated by the proof of Theorem~\ref{Intro, thm, Galois main} and the theory of \emph{shape} in \cite{LLHLM}. The special case $n=3$ of Conjecture~\ref{Intro: weight elimination} was firstly proven in \cite{HLM} and can also be deduced from the computations of Galois deformation rings in \cite{LLHLM}.

\begin{rema}\label{LeH}
In an earlier version of this paper, we prove Conjecture~\ref{Intro: weight elimination} for $n\leq 5$. But our method is rather elaborate to execute for general $n$. We are informed that Bao V. Le Hung can prove Conjecture~\ref{Intro: weight elimination} completely in his forthcoming paper~\cite{LeH}. Therefore, Conjecture~\ref{Intro: weight elimination} becomes a theorem based on the results in~\cite{LeH}.
\end{rema}

Finally, we also show the automorphy of the Serre weight $F(\mu^{\square})^{\vee}$. In other words,
\begin{equation}\label{Intro automorphy of Serre weight}
F(\mu^{\square})^{\vee}\in W_w(\rbar)\cap\mathrm{JH}((\pi^{i_1,j_1})^{\vee}).
\end{equation}
Showing the automorphy of a Serre weight, in general, is very subtle. But thanks to the work of \cite{BLGG} we are able to show the automorphy of $F(\mu^{\square})^{\vee}$ by checking the existence of certain potentially diagonalizable crystalline lifts of $\rhobar_0$ (cf. Proposition~\ref{prop: modularity}).

\subsection{Mod $p$ local-global compatibility}\label{subsec: Intro, mod p local-global comp}
We now give a sketch of our ideas towards our main results on mod $p$ local-global compatibility.  As discussed at the beginning of this introduction, we prove that $\Pi(\rbar)$ determines the ordinary representation $\rhobar_0$. We first recall the definition of $\Pi(\rbar)$.

Keep the notation of the previous sections, and write $b_i=-c_{n-1-i}$ for all $0\leq i\leq n-1$, with $c_i$ as in (\ref{Intro: ordinary representation}). 
We fix a place $w$ of $F$ above $p$ and write $v:= w|_{F^+}$, and we let $\rbar:G_{F}\rightarrow\GL_n(\F)$ be an irreducible automorphic representation, of a Serre weight $V\cong\bigotimes_{v'}V_{v'}$ (cf. Section~\ref{subsec: Intro, weight elimination}), with $\rbar|_{G_{F_w}}\cong\rhobar_0$.

Let $V':= \bigotimes_{v'\neq v}V_{v'}$ and set $S(U^v,V'):= \underset{\longrightarrow}\lim S(U^v\cdot U_v,V')$ where the direct limit runs over compact open subgroups $U_v\subseteq \cG_n(\cO_{F^+_v})$. This space $S(U^v,V')$ has a natural smooth action of $G_n(F_v^{+})\cong\GL_n(F_w)\cong\GL_n(\Q_p)$ by right translation as well as an action of a Hecke algebra that commutes with the action of $G_n(F_v^+)$. We define
$$\Pi(\rbar):= S(U^v,V')[\mathfrak{m}_{\rbar}]$$
where $\mathfrak{m}_{\rbar}$ is the maximal ideal of the Hecke algebra associated to $\rbar$. In the spirit of \cite{Emerton}, this is a candidate on the automorphic side for a mod $p$ Langlands correspondence corresponding to $\rhobar_0$. Note that the definition of $\Pi(\rbar)$ relies on $U^v$ and $V^{\prime}$ as well as choice of a Hecke algebra, but we suppress them in the notation.

Fix $n-1\geq i_0>j_0+1>j_0\geq 0$, and define $i_1$ and $j_1$ by the equation $i_1+i_0=j_1+j_0=n-1$. Let $P_{i_1,j_1}\supset B$ be the standard parabolic subgroup of $\GL_{n/\Z_p}$ corresponding to the subset $\{\alpha_k\mid n-j_1\leq k\leq n-1-i_1 \}$ of the set $\Delta$ of positive simple roots with respect to $(B,T)$. In particular, we have $P_{0,n-1}=\mathrm{GL}_n$.  We denote the unipotent radical of $P_{i_1,j_1}$ by $N_{i_1,j_1}$. We also denote the opposite parabolic subgroup and its unipotent radical by $P^-_{i_1,j_1}$ and $N^-_{i_1,j_1}$. We fix a standard choice of Levi subgroup $L_{i_1,j_1}:=P_{i_1,j_1}\cap P^-_{i_1,j_1}$.  We can embeds $\GL_{j_1-i_1+1}$ into $\GL_n$ with image denoted by $G_{i_1,j_1}$ such that $L_{i_1,j_1}=G_{i_1,j_1}T$.

Recall that $\mathcal{S}_n$ and $\mathcal{S}_n^{\prime}$ are completely determined by fixing the data $n$ and $(a_{n-1},\cdots,a_0)$. We define $\mathcal{S}_{i_1,j_1}\in\F_p[\GL_{j_1-i_1+1}(\F_p)]$ (resp. $\mathcal{S}_{i_1,j_1}^{\prime}\in\F_p[\GL_{j_1-i_1+1}(\F_p)]$) by replacing $n$ and $(a_{n-1},\cdots,a_1,a_0)$ by $j_1-i_1+1$ and $(b_{j_1}+j_1-i_1-1,b_{j_1-1},\cdots,b_{i_1+1},b_{i_1}-j_1+i_1+1)$, and then we define $\mathcal{S}^{i_1,j_1}$ (resp. $\mathcal{S}^{i_1,j_1,\prime}$) to be the image of $\mathcal{S}_{i_1,j_1}$ (resp. $\mathcal{S}_{i_1,j_1}^{\prime}$)  in $\F_p[\GL_n(\F_p)]$) via the embedding $\GL_{j_1-i_1+1}\cong G_{i_1,j_1}\hookrightarrow\GL_n$.

We now state the main results in this paper.
\begin{theo}[Theorem~\ref{theo: lgc}]\label{Intro, thm, main results}
Fix a pair of integers $(i_0,j_0)$ satisfying $0\leq j_0<j_0+1<i_0\leq n-1$, and let $\rbar:G_{F}\rightarrow\GL_n(\F)$ be an irreducible automorphic representation with $\rbar|_{G_{F_w}}\cong\rhobar_0$. Assume that
\begin{itemize}
\item $\mu^{\square,i_1,j_1}$ is $2n$-generic;
\item $\rhobar_{i_0,j_0}$ is Fontaine--Laffaille generic.
\end{itemize}
Assume further that
\begin{equation}\label{Intro, weight elimination in main results}
\{F(\mu^{\square})^{\vee}\}\subseteq W_w(\rbar)\cap\mathrm{JH}((\pi^{i_1,j_1})^{\vee})\subseteq \{F(\mu^{\square})^{\vee}, F(\mu^{\square,i_1,j_1})^{\vee}\}.
\end{equation}

Then there exists a primitive vector (c.f. Definition \ref{definition: primitive}) in $\Pi(\rbar)^{I(1),\mu^{i_1,j_1}}$.
Moreover, for each primitive vector $v^{i_1,j_1}\in \Pi(\rbar)^{I(1),\mu^{i_1,j_1}}$, there exists a connected (cf. Definition~\ref{definition: connected}) vector $v^{i_1,j_1,\prime}\in \Pi(\rbar)^{I(1),\mu^{i_1,j_1,\prime}}$ to $v^{i_1,j_1}$ such that  $\mathcal{S}^{i_1,j_1,\prime}v^{i_1,j_1,\prime}\neq0$
and
\begin{equation*}
\mathcal{S}^{i_1,j_1,\prime}v^{i_1,j_1,\prime}=\varepsilon^{i_1,j_1}\mathcal{P}_{i_1,j_1}(b_{n-1},\cdots,b_0) \cdot \mathrm{FL}^{i_0,j_0}_n(\rbar|_{G_{F_w}})\cdot \mathcal{S}^{i_1,j_1}v^{i_1,j_1}
\end{equation*}
where
$$\varepsilon^{i_1,j_1}=\prod_{k=i_1+1}^{j_1-1}(-1)^{b_{i_1}-b_k-j_1+i_1+1}$$
and
$$\mathcal{P}_{i_1,j_1}(b_{n-1},\cdots,b_0)=\prod_{k=i_1+1}^{j_1-1}\prod_{j=1}^{j_1-i_1-1}\frac{b_k-b_{j_1}-j}{b_{i_1}-b_k-j}\in\Z_p^{\times}.$$
\end{theo}

Note that the conditions in (\ref{Intro, weight elimination in main results}) can be removed under some standard Taylor--Wiles conditions (cf. Remark~\ref{LeH} and (\ref{Intro automorphy of Serre weight})).

Theorem~\ref{Intro, thm, main results} relies on the choice of a principal series type (the niveau $1$ Galois type $\bigoplus_{i=0}^{n-1}\widetilde{\omega}^{k^{i_0,j_0}_i}$). But this choice is somehow the unique one that could possibly make our strategy of the proof of Theorem \ref{Intro, thm, main results} work.

Our proof of Theorem~\ref{Intro, thm, main results} is a bit different from the one of \cite{HLM}: there are at least two new inputs. Firstly, in \cite{HLM}, they require the freeness of a certain module over a Hecke algebra, proved by patching argument, to kill a certain shadow weight (which corresponds to the weight $F(\mu^{\square, i_1,j_1})$ in our context if we fix $(i_0,j_0)$). In our proof, we use purely modular representation theoretic arguments. Secondly, we cannot apply Theorem~\ref{Intro, thm, identity} and Theorem~\ref{Intro, thm, Auto main} directly to our local global-compatibility for general $(i_1,j_1)$. We need intermediate steps, for example Proposition \ref{non-vanishing1}, to use the results of Theorem~\ref{Intro, thm, Auto main}.

We quickly review the main strategy of the proof of Theorem~\ref{Intro, thm, main results}. The idea of the proof is essentially the combination of the proof of the $(i_0,j_0)=(n-1,0)$-case and the fact the general $(i_0,j_0)$-case comes from parabolic induction, whose accurate meaning will be clear in the following. We let $\widetilde{V}^{\prime}$ be a lift of $V^{\prime}$ defined in (\ref{prime to v}), assuming that each local factor of $V^{\prime}$ is in the lowest alcove. Then we consider
$$\widetilde{\Pi}(\rbar):=S(U^v,\widetilde{V}')_{\mathfrak{m}_{\rbar}}.$$
Note that  $\widetilde{\Pi}(\rbar)\otimes_{\cO_E}E$ is a smooth $E$-representation of $\mathrm{GL}_n(\Q_p)$ which also depends on $U^v$ and $\widetilde{V}^{\prime}$, but we omit them from the notation.

We consider the natural surjection onto the coinvariant space
$$\mathrm{Pr}: \widetilde{\Pi}(\rbar)\otimes_{\cO_E}E\twoheadrightarrow(\widetilde{\Pi}(\rbar)\otimes_{\cO_E}E)_{N^-_{i_1,j_1}(\Q_p)}.$$
Now we fix a pair of vectors $v^{i_1,j_1}\in\Pi(\rbar)^{I(1),\mu^{i_1,j_1}}$ and $v^{i_1,j_1,\prime}\in\Pi(\rbar)^{I(1),\mu^{i_1,j_1,\prime}}$
that have lifts
$\widehat{v}^{i_1,j_1}\in\widetilde{\Pi}(\rbar)^{I(1),\mu^{i_1,j_1}}$ and $\widehat{v}^{i_1,j_1,\prime}\in\widetilde{\Pi}(\rbar)^{I(1),\mu^{i_1,j_1,\prime}}$, respectively, such that
$$\langle\mathrm{GL}_n(\Z_p)\widehat{v}^{i_1,j_1}\rangle_E=\langle\mathrm{GL}_n(\Z_p)\widehat{v}^{i_1,j_1,\prime}\rangle_E$$
and
$$\mathrm{Pr}(\widehat{v}^{i_1,j_1,\prime})=\Xi_{i_1,j_1}\cdot \mathrm{Pr}(\widehat{v}^{i_1,j_1})$$
where $\langle\mathrm{GL}_n(\Z_p)\ast\rangle_E$ is the $E$-subrepresentation generated by $\ast$ in $\widetilde{\Pi}(\rbar)\otimes_{\cO_E}E$ as a representation of $\GL_n(\Z_p)$ and $\Xi_{i_1,j_1}\in \GL_{j_1-i_1+1}(\Q_p)\hookrightarrow L_{i_1,j_1}(\Q_p)$ is defined in (\ref{the operator Xi_i_1,j_1}). Note in particular that $\Xi_{i_1,j_1}$ lies in the normalizor of the standard Iwahori subgroup of $\GL_{j_1-i_1+1}(\Q_p)$ in $\GL_{j_1-i_1+1}(\Q_p)$.

We further define the following subspaces:
\begin{align*}
\Pi^{i_1,j_1}&:=\langle\mathrm{GL}_n(\Q_p)\widehat{v}^{i_1,j_1}\rangle_E= \langle\mathrm{GL}_n(\Q_p)\widehat{v}^{i_1,j_1,\prime}\rangle_E\subseteq\widetilde{\Pi}(\rbar)\otimes_{\cO_E}E;\\
\widetilde{\pi}^{i_1,j_1}&:=\langle\mathrm{GL}_n(\Z_p)\widehat{v}^{i_1,j_1}\rangle_E =\langle\mathrm{GL}_n(\Z_p)\widehat{v}^{i_1,j_1,\prime}\rangle_E\subseteq\Pi^{i_1,j_1};\\
(\widetilde{\pi}^{i_1,j_1})^{\circ}&:=\widetilde{\pi}^{i_1,j_1}\cap\widetilde{\Pi}(\rbar);\\
\overline{(\widetilde{\pi}^{i_1,j_1})^{\circ}}&:=(\widetilde{\pi}^{i_1,j_1})^{\circ}\otimes_{\cO_E}\F.
\end{align*}
We need to assume that $v^{i_1,j_1}$ is primitive (c.f. Definition \ref{definition: primitive}), which is a technical condition ensuring that we can pick the lift $\widehat{v}^{i_1,j_1}$ such that $\Pi^{i_1,j_1}$ is an irreducible smooth representation of $\mathrm{GL}_n(\Q_p)$. We show that primitive vectors always exist (and thus this technical assumption is harmless). Note that $\Pi^{i_1,j_1}$ is semisimple with finite length without this assumption.

The strategy of the proof of Theorem~\ref{Intro, thm, main results} is summarized in the following two diagrams:
$$
\xymatrix{
\overline{(\widetilde{\pi}^{i_1,j_1})^{\circ}}&\ar@{->>}[l](\widetilde{\pi}^{i_1,j_1})^{\circ}\ar@{^{(}->}[r]&\widetilde{\pi}^{i_1,j_1}\ar@{^{(}->}[r]&\Pi^{i_1,j_1}\ar@{->>}[r]&\mathrm{Pr}(\Pi^{i_1,j_1})}
$$
and
$$
\xymatrix{
&v^{i_1,j_1}\ar@{->}[dl]_{S^{i_1,j_1}}&\widehat{v}^{i_1,j_1}\ar@{->}[r]\ar@{->}[l]&\mathrm{Pr}(\widehat{v}^{i_1,j_1}) \ar@{->}[dd]^{\Xi_{i_1,j_1}}\\
  \F[S^{i_1,j_1}v^{i_1,j_1}]=\F[S^{i_1,j_1,\prime}v^{i_1,j_1,\prime}]&&&\\
&v^{i_1,j_1,\prime}\ar@{->}[ul]^{S^{i_1,j_1,\prime}}&\widehat{v}^{i_1,j_1,\prime}\ar@{->}[r]\ar@{->}[l]&\mathrm{Pr}(\widehat{v}^{i_1,j_1,\prime}). }
$$
In the second diagram, by $\F[v]$ for a non-zero vector $v$ in a $\F$-vector space we mean the $\F$-line generated by $v$. Theorem~\ref{Intro, thm, main results} says that $S^{i_1,j_1}v^{i_1,j_1}$ and $S^{i_1,j_1,\prime}v^{i_1,j_1,\prime}$ are non-zero and differ by a scalar in $\F^{\times}$ that captures the Fontaine--Laffaille parameter $\mathrm{FL}_n^{i_0,j_0}(\rhobar_0)$.

We note that the two leftmost diagonal arrows in the second picture above are where we apply Theorem \ref{Intro, thm, Auto main} together with Proposition \ref{general nonvanishing}. The rightmost vertical arrow in the second picture above is where we apply Theorem \ref{inter identity characteristic 0} inside the smooth Jacquet module $\mathrm{Pr}(\Pi^{i_1,j_1})$ seen as a representation of $\GL_{j_1-i_1+1}(\Q_p)\hookrightarrow L_{i_1,j_1}(\Q_p)$.

One of the crucial points of the proof is that we deduce from Morita theory (recalled in Section~\ref{subsec: Morita theory}) that there exists an $\mathcal{O}_E$-representation $(\widetilde{\pi}^{i_1,j_1, L_{i_1,j_1}})^{\circ}$ of $ L_{i_1,j_1}(\F_p)$ such that
$$(\widetilde{\pi}^{i_1,j_1})^{\circ}=\mathrm{Ind}^{\mathrm{GL}_n(\F_p)}_{P^-_{i_1,j_1}(\F_p)}(\widetilde{\pi}^{i_1,j_1, L_{i_1,j_1}})^{\circ}.$$
Namely, the lattice $(\widetilde{\pi}^{i_1,j_1})^{\circ}$ in the principal series type $\widetilde{\pi}^{i_1,j_1}$ comes from the parabolic induction from $ L_{i_1,j_1}(\F_p)$. This fact essentially follows from the assumption (\ref{Intro, weight elimination in main results}) together with some elementary arguments from Morita theory in Section~\ref{subsec: Morita theory}.

\begin{coro}\label{Intro, coro}
Keep the notation of Theorem~\ref{Intro, thm, main results} and assume that each assumption in Theorem~\ref{Intro, thm, main results} holds for all $(i_0,j_0)$ such that $0\leq j_0<j_0+1<i_0\leq n-1$.

Then the Galois representation $\rhobar_0$ is determined by $\Pi(\rbar)$ in the sense of Remark \ref{main remark}.
\end{coro}

Remark \ref{main remark} roughly says the following: although we give an explicit strategy to recover the invariants $\mathrm{FL}_n^{i_0,j_0}(\rhobar_0)$ inside the representation $\Pi(\rbar)$, this strategy does not say that $\mathrm{FL}_n^{i_0,j_0}(\rhobar_0)$ can be recovered from an explicit formula which depends only the structure of $\Pi(\rbar)$ as a smooth admissible representation of $\mathrm{GL}_n(\Q_p)$. In fact, $\Pi(\rbar)$ has many natural restrictions coming from its definition. For example, $\Pi(\rbar)$ admits natural lifts in characteristic zero (smooth or Banach) that satisfy various conditions. Our strategy to recover $\mathrm{FL}_n^{i_0,j_0}(\rhobar_0)$ relies on the existence of these restrictions on $\Pi(\rbar)$. Assuming these restrictions on $\Pi(\rbar)$, our construction of $v^{i_1,j_1,\prime}$ from $v^{i_1,j_1}$ is canonical and is independent of various choices of lifts into characteristic zero as discussed in Remark \ref{main remark}. If $(i_0,j_0)=(n-1,0)$, then we simply have $v^{0,n-1,\prime}=\Xi_{0,n-1}v^{0,n-1}$ and $\mathrm{FL}_n^{n-1,0}(\rhobar_0)$ can actually be recovered from $\Pi(\rbar)$ through an explicit formula regardless of the restrictions on $\Pi(\rbar)$ mentioned above.

Finally, we note that if $\mu^{\square}$ is $3n$-generic, then $\mu^{\square,i_1,j_1}$ is $2n$-generic for each $(i_1,j_1)$ such that $0\leq i_1<i_1+1<j_1\leq n-1$.

\subsection{Notation}
Much of the notation introduced in this section will also be (or have already been) introduced in the text, but we try to collect together various definitions here for ease of reading.

We let $E$ be a (sufficiently large) extension of $\Q_p$ with ring of integers $\cO_E$, a uniformizer $\varpi_E$, and residue field $\F$. We will use these rings $E$, $\cO_E$, and $\F$ for the coefficients of our representations. We also let $K$ be a finite extension of $\Q_p$ with ring of integers $\cO_K$, a uniformizer $\varpi$, and residue field $k$. Let $W(k)$ be the ring of Witt vectors over $k$ and write $K_0$ for $W(k)[\frac{1}{p}]$. ($K_0$ is the maximal absolutely unramified subextension of $K$.) In this paper, we are interested only in the fields $K$ that are tamely ramified extension of $\Q_p$, in which case we let $e:= [K:K_0]=p^f-1$ where $f=[k:\F_p]$.

For a field $F$, we write $G_F$ for $\Gal(\overline{F}/F)$ where $\overline{F}$ is a separable closure of $F$. For instance, we are mainly interested in $G_{\Q_p}$ as well as $G_{K_0}$ in this paper. The choice of a uniformizer $\varpi\in K$ provides us with a map:
$$
\widetilde{\omega}_{\varpi}:G_{\Q_p} {\longrightarrow} W(k) \,:\,
g\longmapsto\frac{g(\varpi)}{\varpi}
$$
whose reduction mod $(\varpi)$ will be denoted as $\omega_{\varpi}$. This map factors through $\Gal(K/\Q_p)$ and $\widetilde{\omega}_{\varpi}|_{G_{K_0}}$ becomes a homomorphism. Note that the choice of the embedding $\sigma_0: k\into \F$ provides us with a fundamental character of niveau $f$, namely $\omega_f:= \sigma_0\circ \omega_{\varpi}\vert_{\Gal(K/K_0)}$, and we fix the embedding in this paper.

For $a\in k$, we write $\widetilde{a}$ for its Teichm\"uler lift in $W(k)$. We also use the notation $\lceil a\rceil$ for~$\widetilde{a}$, in particular, in Section~\ref{subsec: Jacobi sums in char. 0}. When the notation for an element $\bullet$ in $k$ is quite long, we prefer $\lceil\bullet\rceil$ to $\widetilde{\bullet}$. For instance, if $a,b,c,d\in k$ then we write $$\lceil (a-b)(a-c)(a-d)(b-c)(b-d)\rceil\,\,\mbox{ for }\,\,\widetilde{(a-b)(a-c)(a-d)(b-c)(b-d)}.$$
Note that $\widetilde{\omega}_{\varpi}$ is the Teichm\"uler lift of ${\omega}_{\varpi}$.

We normalize the Hodge--Tate weight of the cyclotomic character $\varepsilon$ to be $-1$. Our normalization on class field theory sends the geometric Frobenius to the uniformizers. If $a\in\F^{\times}$ or $a\in \cO_E^{\times}$ then we write $\mathrm{U_{a}}$ for the unramified character sending the geometric Frobenius to $a$. We may regard a character of $G_{\Q_p}$ as a character of $\Q_p^{\times}$ via our normalization of class field theory.

As usual, we write $S$ for the $p$-adic completion of $W(k)[u,\frac{u^{ie}}{i!}]_{i\in\mathbf{N}}$, and let $S_{\cO_E}:= S\otimes_{\Z_p}\cO_E$ and $S_E:= S_{\cO_E}\otimes_{\Z_p}\Q_p$.  We also let $\barS_{\F}:= S_{\cO_E}/(\varpi_E, \Fil^p S_{\cO_E})\cong (k\otimes_{\F_p}\F)[u]/u^{ep}$. Choose a uniformizer $\varpi$ of $K$ and let $E(u)\in W(k)[u]$ be the monic minimal polynomial of $\varpi$. The group $\Gal(K/K_0)$ acts on $S$ via the character $\widetilde{\omega}_{\varpi}$, and we write $(S_{\cO_E})_{\widetilde{\omega}_{\varpi}^m}$ for the $\widetilde{\omega}_{\varpi}^m$-isotipycal component of $S$ for $m\in\Z$. We define $(\barS_{\F})_{\omega_{\varpi}^m}$ in a similar fashion. If $\cO_E$ or $\F$ are clear, we often omit them, i.e., we write $S_{\widetilde{\omega}_{\varpi}^m}$ and $\barS_{\omega_{\varpi}^m}$ for $(S_{\cO_E})_{\widetilde{\omega}_{\varpi}^m}$ and $(\barS_{\F})_{\omega_{\varpi}^m}$ respectively. In particular, 
$\barS_0:=\barS_{\omega_{\varpi}^{0}}\cong(k\otimes_{\F_p}\F)[u^e]/u^{ep}$ and $$S_0:= S_{\widetilde{\omega}_{\varpi}^0}=\left\{\sum_{i=0}^{\infty}a_i\frac{E(u)^i}{i!}\mid a_i\in W(k)\otimes_{\Z_p}\cO_E\,\, \mbox{and}\,\, a_i\rightarrow0\,\,\,p\mbox{-adically}\right\}.$$

The association $a\otimes b\mapsto (\sigma(a)b)_{\sigma}$ gives rise to an isomorphism $k\otimes_{\F_p}\F\cong\coprod_{\sigma:k\hookrightarrow\F}\F$, and we write $e_{\sigma}$ for the idempotent element in $k\otimes_{\F_p}\F$ that corresponds to the idempotent element in $\coprod_{\sigma:k\hookrightarrow\F}\F$ whose only non-zero entry is $1$ at the position of $\sigma$.

To lighten the notation, we often write $G$ for $\GL_{n/\Z_p}$. (By $G_n$, we mean an outer form of $\GL_n$ defined in Section~\ref{subsec: auto forms}.)
We let $B$ be the Borel subgroup of $G$ consisting of upper-triangular matrices of $G$, $U$ the unipotent subgroup of $B$, and $T$ the torus of diagonal matrices of $\GL_n$. We also write $B^{-}$ and $U^{-}$ for the opposite Borel of $B$ and the unipotent subgroup of $B^{-}$, respectively. Let $\Phi^+$ denote the set of positive roots with respect to $(B,T)$, and $\Delta=\{\alpha_k\}_{1\leq k\leq n-1}$ the subset of positive simple roots. We also let $W$ be the Weyl group of $\GL_n$, which is often considered as a subgroup of $\GL_n$, and let $s_k$ be the simple reflection corresponding to $\alpha_k$. We write $w_0$ for the longest Weyl element in $W$, and we hope that the reader is not confused with places $w$ or $w'$ of $F$.

We often write $K$ for $\GL_n(\Z_p)$ for brevity. (Note that we use $K$ for a tamely ramified extension of $\Q_p$ as well, and we hope that it does not confuse the reader.) We will often use the following three open compact subgroups of $\GL_n(\Z_p)$: if we let $\mathrm{red}:\GL_n(\Z_p)\twoheadrightarrow\GL_n(\F_p)$ be the natural mod $p$ reduction map, then
\begin{equation*}
K(1):=\mathrm{Ker}(\mathrm{red})\,\subset\, I(1):=\mathrm{red}^{-1}( U(\F_p)) \,\subset\, I:=\mathrm{red}^{-1}( B(\F_p))\, \subset\, K.
\end{equation*}
If $M$ is a free $\F$-module with a smooth action of $K$, then $I$ acts on the pro $p$ Iwahori fixed subspace $M^{I(1)}$ via $I/I(1)\cong T(\F_p)$. We write $M^{I(1),\mu}$ for the eigenspace with respect to a character $\mu: T(\F_p)\rightarrow\F_p^{\times}$. $M^{I(1)}$ decomposes as $$M^{I(1)}\cong \bigoplus M^{I(1),\mu}$$ as $ T(\F_p)$-representations, where the direct sum runs over the characters $\mu$ of $ T(\F_p)$. In the obvious similar fashion, we define $M^{I(1),\mu}$ when $M$ is a free $\cO_E$-module or a free $E$-module.

By $[m]_f$ for a rational number $m\in \Z[\frac{1}{p}]\subset\Q$ we mean the unique integer in $[0,e)$ congruent to $m$ mod~$(e)$ via the natural surjection $\Z[\frac{1}{p}]\twoheadrightarrow\Z/e\Z$. By $\lfloor y \rfloor$ for $y\in\R$ we mean the floor function of $y$, i.e., the biggest integer less than or equal to $y$. For a set $A$, we write $|A|$ for the cardinality of $A$. If $V$ is a finite-dimensional $\F$-representation of a group $H$, then we write $\mathrm{soc}_{H}V$ and $\mathrm{cosoc}_{H}V$ for the socle of $V$ and the cosocle of $V$, respectively. If $v$ is a non-zero vector in a free module over $\F$ (resp. over $\cO_E$, resp. over $E$), then we write $\F[v]$ (resp. $\cO_E[v]$, resp. $E[v]$) for the $\F$-line (resp. the $\cO_E$-line, resp. the $E$-line) generated by~$v$.

We write $\overline{x}$ for the image of $x\in\cO_E$ under the natural surjection $\cO_E\twoheadrightarrow\F$. We also have a natural surjection $\mathbb{P}^1(\cO_E)\twoheadrightarrow\mathbb{P}^1(\F)$ defined by letting $\overline{[x:y]}\in\mathbb{P}^1(\F)$ be the image of $[x:y]\in\mathbb{P}^1(\cO_E)$ where
\begin{equation*}
\overline{[x:y]}=
\left\{
  \begin{array}{ll}
    {[1:\overline{(\frac{y}{x})}]} & \hbox{if $\frac{y}{x}\in\cO_E$;} \\
    {[\overline{(\frac{x}{y})}:1]} & \hbox{if $\frac{x}{y}\in\cO_E$.}
  \end{array}
\right.
\end{equation*}
We often write $\frac{y}{x}$ for $[x:y]\in\mathbb{P}^1(\F)$ if $x\not=0$.

\subsection*{Acknowledgements}
The authors express their deepest gratitude to Christophe Breuil for his encouragements, his careful reading of an earlier draft of this paper, and his numerous helpful comments and suggestions. The authors also sincerely thank Florian Herzig for his constant interest, his many helpful comments and suggestions, and pointing out some mistakes in an earlier draft of this paper. The authors thank Yiwen Ding, Yongquan Hu, and Stefano Morra for plenty of discussions. The second named author thanks James E. Humphreys for kind guidance through messages on references for modular representation theory. The second named author thanks Jian Qian for giving some numerical evidence through computer programming. The first named author thanks Seunghwan Chang for his helpful comments and suggestions.

\section{Integral $p$-adic Hodge theory: preliminary}
In this section, we do a quick review of some (integral) $p$-adic Hodge theory which will be needed later. We note that all of the results in this section are already known or easy generalization of known results. We closely follow \cite{EGH} as well as \cite{HLM} in this section.

\subsection{Filtered $(\phi,N)$-modules with descent data}\label{subsec: Potentially semi-stable representations}
In this section, we review potentially semi-stable representations and their corresponding linear algebra objects, admissible filtered $(\phi,N)$-modules with descent data.

Let $K$ be a finite extension of $\Q_p$, and $K_0$ the maximal unramified subfield of $K$, so that $K_0=W(k)\otimes_{\Z_p}\Q_p$ where $k$ is the residue field of $K$. We fix the uniformizer $p\in\Q_p$, so that we fix an embedding $\Bst\hookrightarrow\Bdr$.  We also let $K'$ be a subextension of $K$ with $K/K'$ Galois, and write $\phi\in\Gal(K_0/\Q_p)$ for the arithmetic Frobenius.

A $p$-adic Galois representation $\rho:G_{K'}\rightarrow \GL_n(E)$ is potentially semi-stable if there is a finite extension $L$ of $K'$ such that $\rho|_{G_L}$ is semi-stable, i.e., $\mathrm{rank}_{L_0\otimes E}\Dst^{K'}(V)=\dim_E V$, where $V$ is an underlying vector space of $\rho$ and $\Dst^{K'}(V):= (\Bst\otimes_{\Q_p} V)^{G_L}$. We often write $\Dst^{K'}(\rho)$ for $\Dst^{K'}(V)$. If $K$ is the Galois closure of $L$ over $K'$, then $\rho|_{G_{K}}$ is semi-stable, provided that $\rho|_{G_L}$ is semi-stable.

\begin{defi}
A filtered $(\phi,N,K/K',E)$-module of rank $n$ is a free $K_0\otimes E$-module $D$ of rank $n$ together with
\begin{itemize}
\item a $\phi\otimes 1$-automorphism $\phi$ on $D$;
\item a nilpotent $K_0\otimes E$-linear endomorphism $N$ on D;
\item a decreasing filtration $\{\Fil^iD_k\}_{i\in\Z}$ on $D_K=K\otimes_{K_0}D$ consisting of $K\otimes_{\Q_p}E$-submodules of $D_K$, which is exhaustive and separated;
\item a $K_0$-semilinear, $E$-linear action of $\Gal(K/K')$ which commutes with $\phi$ and $N$ and preserves the filtration on $D_K$.
\end{itemize}
\end{defi}
We say that $D$ is \emph{(weakly) admissible} if the underlying filtered $(\phi,N,K/K,E)$-module is weakly admissible in the sense of \cite{Fon94}.  The action of $\Gal(K/K')$ on $D$ is often called \emph{descent data} action. If $V$ is potentially semi-stable, then $\Dst^{K'}(V)$ is a typical example of an admissible filtered $(\phi,N,K/K',E)$-module of rank $n$.
\begin{theo}[\cite{CF}, Theorem~4.3]\label{theo: equivalence of semi-stable reps}
There is an equivalence of categories between the category of weakly admissible filtered $(\phi,N,K/K',E)$-modules of rank $n$ and the category of $n$-dimensional potentially semi-stable $E$-representations of $G_{K'}$ that become semi-stable upon restriction to $G_K$.
\end{theo}

Note that Theorem~\ref{theo: equivalence of semi-stable reps} is proved in \cite{CF} in the case $K=K'$, and that \cite{Savitt} gives a generalization to the statement with non-trivial descent data in a formal nature.

If $V$ is potentially semi-stable, then so is its dual $V^{\vee}$. We define $\Dst^{\ast,K'}(V):= \Dst^{K'}(V^{\vee})$. Then $\Dst^{\ast,K'}$ gives an anti-equivalence of categories from the category of $n$-dimensional potentially semi-stable $E$-representations of $G_{K'}$ that become semi-stable upon restriction to $G_K$ to the category of weakly admissible filtered $(\phi,N,K/K',E)$-modules of rank $n$, with quasi-inverse
$$\Vst^{\ast,K'}(D):=\Hom_{\phi,N}(D,\Bst)\cap\Hom_{\Fil}(D_K,\Bdr).$$
It will often be convenient to use covariant functors. We define an equivalence of categories: for each $r\in\Z$
$$\Vst^{K',r}(D):= \Vst^{\ast,K'}(D)^{\vee}\otimes \varepsilon^r.$$ The functor $\Dst^{K',r}$ defined by $\Dst^{K',r}(V):= \Dst^{K'}(V\otimes \varepsilon^{-r})$ is a quasi-inverse of $\Vst^{K',r}$.

For a given potentially semi-stable representation $\rho:G_{K'}\rightarrow \GL_n(E)$, one can attach a Weil--Deligne representation $\WD(\rho)$ to $\rho$, as in \cite{CDT}, Appendix B.1. We refer to $\WD(\rho)|_{I_{\Qp}}$ as to the \emph{Galois type} associated to $\rho$.  Note that $\WD(\rho)$ is defined via the filtered $(\phi,N,K/K',N)$-module $\Dst^{K'}(\rho)$ and that $\WD(\rho)|_{I_{K'}}\cong\WD(\rho\otimes\varepsilon^{r})|_{I_{K'}}$ for all $r\in\Z$.

Finally, we say that a potentially semi-stable representation $\rho$ is \emph{potentially crystalline} if the monodromy operator $N$ on $\Dst^{K'}(\rho)$ is trivial.

\subsection{Strongly divisible modules with descent data}\label{subsec: strongly divisible modules}
In this section, we review strongly divisible modules that correspond to Galois stable lattices in potentially semi-stable representations. We keep the notation of Section~\ref{subsec: Potentially semi-stable representations}

From now on, we assume that $K/K'$ is a tamely ramified Galois extension with ramification index $e(K/K')$. We fix a uniformizer $\varpi\in K$ with $\varpi^{e(K/K')}\in K'$. Let $e$ be the absolute ramification index of $K$ and $E(u)\in W(k)[u]$ the minimal polynomial of $\varpi$ over~$K_0$.

Let $S$ be the $p$-adic completion of $W(k)[u,\frac{u^{ie}}{i!}]_{i\in\N}$. The ring $S$ has additional structures:
\begin{itemize}
\item a continuous, $\phi$-semilinear map $\phi:S\rightarrow S$ with $\phi(u)=u^p$ and $\phi(\frac{u^{ie}}{i!})=\frac{u^{pie}}{i!}$;
\item a continuous, $W(k)$-linear derivation of $S$ with $N(u)=-u$ and $N(\frac{u^{ie}}{i!})=-ie\frac{u^{ie}}{i!}$;
\item a decreasing filtration $\{\Fil^iS\}_{i\in\Z_{\geq 0}}$ of $S$ given by letting $\Fil^iS$ be the $p$-adic completion of the ideal $\sum_{j\geq i}\frac{E(u)^j}{j!}S$;
\item a group action of $\Gal(K/K')$ on $S$ defined for each $g\in \Gal(K/K')$ by the continuous ring isomorphism $\widehat{g}:S\rightarrow S$ with $\widehat{g}(w_i \frac{u^i}{\lfloor i/e\rfloor !})=g(w_i)h_g^i\frac{u^i}{\lfloor i/e\rfloor !}$ for $w_i\in W(k)$, where $h_g\in W(k)$ satisfies $g(\varpi)=h_g\varpi$.
\end{itemize}
Note that $\phi$ and $N$ satisfies $N\phi=p\phi N$ and that $\hat{g}(E(u))=E(u)$ for all $g\in\Gal(K/K')$ since we assume $\varpi^{e(K/K')}\in K'$. We write $\phi_i$ for $\frac{1}{p^i}\phi$ on $\Fil^iS$. For $i \leq p-1$ we have $\phi(\Fil^i S)\subseteq p^i S$.

Let $S_{\cO_E}:= S\otimes_{\Z_p}\cO_E$ and $S_E:= S_{\cO_E}\otimes_{\Z_p}\Q_p$. We extend the definitions of $\phi$, $N$, $\Fil^iS$, and the action of $\Gal(K/K')$ to $S_{\cO_E}$ (resp. to $S_{E}$) $\cO_E$-linearly (resp. $E$-linearly).
\begin{defi}
Fix a positive integer $r<p-1$. A strongly divisible $\cO_E$-module with descent data of weight $r$ is a free $S_{\cO_E}$-module $\widehat{\cM}$ of finite rank together with
\begin{itemize}
\item a $S_{\cO_E}$-submodule $\Fil^r\widehat{\cM}$;
\item additive maps $\phi,N:\widehat{\cM}\rightarrow\widehat{\cM}$;
\item $S_{\cO_E}$-semilinear bijections $\widehat{g}:\widehat{\cM}\rightarrow\widehat{\cM}$ for each $g\in\Gal(K/K')$
\end{itemize}
such that
\begin{itemize}
\item $\Fil^{r}S_{\cO_E}\cdot\widehat{\cM}\subseteq\Fil^{r}\widehat{\cM}$;
\item $\Fil^{r}\widehat{\cM}\cap I\widehat{\cM}=I\Fil^{r}\widehat{\cM}$ for all ideals $I$ in $\cO_E$;
\item $\phi(sx)=\phi(s)\phi(x)$ for all $s\in S_{\cO_E}$ and for all $x\in\widehat{\cM}$;
\item $\phi(\Fil^{r}\widehat{\cM})$ is contained in $p^{r}\widehat{\cM}$ and generates it over $S_{\cO_E}$;
\item $N(sx)=N(s)x+sN(x)$ for all $s\in S_{\cO_E}$ and for all $x\in\widehat{\cM}$;
\item $N\phi=p\phi N$;
\item $E(u) N(\Fil^{r}\widehat{\cM})\subseteq\Fil^{r}\widehat{\cM}$;
\item for all $g\in\Gal(K/K')$ $\widehat{g}$ commutes with $\phi$ and $N$, and preserves $\Fil^r\widehat{\cM}$;
\item $\widehat{g}_1\circ\widehat{g}_2=\widehat{g_1\cdot g_2}$ for all $g_1,g_2\in\Gal(K/K')$.
\end{itemize}
\end{defi}

We write $\OEModdd[r]$ for the category of strongly divisible $\cO_E$-modules with descent data of weight $r$. It is easy to see that the map $\phi_r= \frac{1}{p^r}\phi:\Fil^{r}\widehat{\cM}\rightarrow \widehat{\cM}$ satisfies $cN\phi_r(x)=\phi_r(E(u)N(x))$ for all $x\in\Fil^r\widehat{\cM}$ where $c:= \frac{\phi(E(u))}{p}\in S^{\times}$.

For a strongly divisible $\cO_E$-module $\widehat{\cM}$ with descent data of weight $r$, we define a $G_{K'}$-module $\Tst^{\ast,K'}(\widehat{\cM})$ as follows (cf. \cite{EGH}, Section~3.1.): $$\Tst^{\ast,K'}(\widehat{\cM}):=\Hom_{\Fil^r,\phi,N}(\widehat{\cM},\widehat{\mathbf{A}}_{\mathrm{st}}).$$

\begin{prop}[\cite{EGH}, Proposition~3.1.4]
The functor $\Tst^{\ast,K'}$ provides an anti-equivalence of categories from the category $\OEModdd[r]$ to the category of $G_{K'}$-stable $\cO_E$-lattices in finite-dimensional $E$-representations of $G_{K'}$ which become semi-stable over $K$ with Hodge--Tate weights lying in $[-r,0]$, when $0<r<p-1$.
\end{prop}

Note that the case $K=K'$ and $E=\Q_p$ is proved by Liu~\cite{Liu}.

In this paper, we will be mainly interested in covariant functors $\Tst^{K',r}$ from the category $\OEModdd[r]$ to the category $\mathrm{Rep}_{\cO_E}^{K-\mathrm{st},[-r,0]}G_{K'}$ of $G_{K'}$-stable $\cO_E$-lattices in finite-dimensional $E$-representations of $G_{K'}$ which become semi-stable over $K$ with Hodge--Tate weights lying in $[-r,0]$ defined by
$$\Tst^{K',r}(\widehat{\cM}):= \Tst^{\ast,K'}(\widehat{\cM})^{\vee}\otimes \varepsilon^r.$$

Let $\widehat{\cM}$ in $\OEModdd[r]$, and define a free $S_E$-module $\mathcal{D}:= \widehat{\cM}\otimes _{\Z_p}\Q_p$. We extend $\phi$ and $N$ on $\mathcal{D}$, and define a filtration on $\mathcal{D}$ as follows: $\Fil^r\mathcal{D}=\Fil^r\widehat{\cM}[\frac{1}{p}]$ and
\begin{equation}\label{equation filtration of filtered module/S}
\Fil^i\mathcal{D}:=
\left\{
\begin{array}{ll}
\mathcal{D} & \hbox{if $i\leq 0$;}\\
\{x\in\mathcal{D}\mid E(u)^{r-i}x\in\Fil^r\mathcal{D}\}& \hbox{if $0\leq i\leq r$;}\\
\sum_{j=0}^{i-1}(\Fil^{i-j}S_{\Q_p})(\Fil^j\mathcal{D})& \hbox{if $i>r$, inductively.}
\end{array}
\right.
\end{equation}
We let $D:= \mathcal{D}\otimes_{S_{\Q_p},s_0}K_0$ and $D_K:= \mathcal{D}\otimes_{S_{\Q_p},s_{\varpi}}K$, where $s_0:S_{\Q_p}\rightarrow K_0$ and $s_{\varpi}:S_{\Q_p}\rightarrow K$ are defined by $u\mapsto 0$ and  $u\mapsto\varpi$ respectively, which induce $\phi$ and $N$ on $D$ and the filtration on $D_K$ by taking $s_{\varpi}(\Fil^i\mathcal{D})$. The $K_0$-vector space $D$ also inherits an $E$-linear action and a semi-linear action of $\Gal(K/K')$. Then it turns out that $D$ is a weakly admissible filtered $(\phi,N,K/K',E)$-module with $\Fil^{r+1}D=0$. Moreover, there is a compatibility (cf. \cite{EGH}, Proof of Proposition~3.1.4.): if $D$ corresponds to $\mathcal{D}=\widehat{\cM}[\frac{1}{p}]$, then
$$\Tst^{K',r}(\widehat{\cM})[\frac{1}{p}]\cong\Vst^{K',r}(D).$$

\subsection{Breuil modules with descent data}
In this section, we review Breuil modules with descent data. We keep the notation of Section~\ref{subsec: strongly divisible modules}, and assume further that $K'\subseteq K_0$.

We let $\barS:= S/(\varpi_E,\Fil^pS)\cong (k\otimes_{\F_p}\F)[u]/u^{ep}$. It is easy to check that $\barS$ inherits $\phi$, $N$, the filtration of $S$, and the action of $\Gal(K/K')$.
\begin{defi}
Fix a positive integer $r<p-1$. A \emph{Breuil modules with descent data of weight $r$} is a free $\barS$-module $\cM$ of finite rank together with
\begin{itemize}
\item a $\barS$-submodule $\Fil^r\cM$ of $\cM$;
\item maps $\phi_{r}:\Fil^r\cM\rightarrow\cM$ and $N:\cM\rightarrow\cM$;
\item additive bijections $\widehat{g}:\cM\rightarrow\cM$ for all $g\in\Gal(K/K')$
\end{itemize}
such that
\begin{itemize}
\item $\Fil^r\cM$ contains $u^{er}\cM$;
\item $\phi_{r}$ is $\F$-linear and $\phi$-semilinear (where $\phi:k[u]/u^{ep}\rightarrow k[u]/u^{ep}$ is the $p$-th power map) with image generating $\cM$ as $\barS$-module;
\item $N$ is $k\otimes_{\F_p}\F$-linear and satisfies
     \begin{itemize}
     \item $N(ux)=uN(x)-ux$ for all $x\in\cM$,
     \item $u^{e}N(\Fil^r\cM)\subseteq \Fil^r\cM$, and
     \item $\phi_{r}(u^{e}N(x))=cN(\phi_{r}(x))$ for all $x\in\Fil^r\cM$, where $c\in (k[u]/u^{ep})^{\times}$ is the image of $\frac{1}{p}\phi(E(u))$ under the natural map $S\rightarrow k[u]/u^{ep}$.
     \end{itemize}
\item $\widehat{g}$ preserves $\Fil^r\cM$ and commutes with the $\phi_r$ and $N$, and the action satisfies $\widehat{g}_1\circ \widehat{g}_2=\widehat{g_1\cdot g_2}$ for all $g_1,g_2\in\Gal(K/K')$. Furthermore, if $a\in k\otimes_{\F_p}\F$ and $m\in\cM$ then $\widehat{g}(au^im)=g(a)((\frac{g(\varpi)}{\varpi})^i\otimes 1)u^i\widehat{g}(m)$.
\end{itemize}
\end{defi}

We write $\FBrModdd$ for the category of Breuil modules with descent data of weight $r$. For $\cM\in\FBrModdd$, we define a $G_{K'}$-module as follows (cf. \cite{EGH}, Section~3.2): $$\Tst^{\ast}(\cM):=\Hom_{\mathrm{BrMod}}(\cM,\widehat{\mathbf{A}}).$$ This gives an exact faithful contravariant functor from the category $\FBrModdd$ to the category $\mathrm{Rep}_{\F}G_{K'}$ of finite dimensional $\F$-representations of $G_{K'}$. We also define a covariant functor as follows: for each $r\in\Z$
$$\Tst^{r}(\cM):= \Tst^{\ast}(\cM)^{\vee}\otimes \omega^r,$$
in which we will be more interested in this paper.

If $\widehat{M}$ is a strongly divisible module with descent data, then $$\cM:= \widehat{\cM}/(\varpi_E,\Fil^p S)$$ is naturally an object in $\FBrModdd$ ($\Fil^r\cM$ is the image of $\Fil^r\widehat{\cM}$ in $\cM$, the map $\phi_r$ is induced by $\frac{1}{p^r}\phi|_{\Fil^r\widehat{\cM}}$, and $N$ and $\widehat{g}$ are those coming from $\widehat{\cM}$). Moreover, there is a compatibility: if $\widehat{\cM}\in\OEModdd[r]$ and we let $\cM=\widehat{\cM}/(\varpi_E,\Fil^p S)$ then
$$\Tst^{K',r}(\widehat{\cM})\otimes_{\cO_E}\F\cong\Tst^r(\cM).$$
(See \cite{EGH}, Lemma 3.2.2 for detail.)

There is a notion of duality of Breuil modules, which will be convenient for our computation of Breuil modules as we will see later.
\begin{defi}\label{defi: dual Breuil modules}
Let $\cM\in \FBrModdd$. We define $\cM^{\ast}$ as follows:
\begin{itemize}
\item $\cM^{\ast}:= \Hom_{k[u]/u^{ep}-\mathrm{Mod}}(\cM,k[u]/u^{ep})$;
\item $\Fil^r\cM^{\ast}:= \{f\in\cM^{\ast}\mid f(\Fil^r\cM)\subseteq u^{er}k[u]/u^{ep}\}$;
\item $\phi_r(f)$ is defined by $\phi_r(f)(\phi_r(x))=\phi_r(f(x))$ for all $x\in\Fil^r\cM$ and $f\in\Fil^r\cM^{\ast}$, where $\phi_r:u^{er}k[u]/u^{ep}\rightarrow k[u]/u^{ep}$ is the unique semilinear map sending $u^{er}$ to $c^r$;
\item $N(f):= N\circ f-f\circ N$, where $N:k[u]/u^{ep}\rightarrow k[u]/u^{ep}$ is the unique $k$-linear derivation such that $N(u)=-u$;
\item $(\widehat{g}f)(x)=g(f(\widehat{g}^{-1}x))$ for all $x\in\cM$ and $g\in\Gal(K/K')$, where $\Gal(K/K')$ acts on $k[u]/u^{ep}$ by $g(au^i)=g(a)(\frac{g(\varpi)}{\varpi})^iu^i$ for $a\in k$.
\end{itemize}
\end{defi}

If $\cM$ is an object of $\FBrModdd$ then so is $\cM^{\ast}$. Moreover, we have $\cM\cong\cM^{\ast\ast}$ and $$\Tst^{\ast}(\cM^{\ast})\cong\Tst^{r}(\cM).$$
(cf. \cite{Caruso11}), Section 2.1.)

Finally, we review the notion of Breuil submodules developed mainly by \cite{Caruso11}. See also \cite{HLM}, Section 2.3.

\begin{defi}
Let $\cM$ be an object of $\FBrModdd$. A \emph{Breuil submodule} of $\cM$ is an $\barS$-submodule $\cN$ of $\cM$ if $\cN$ satisfies
\begin{itemize}
\item $\cN$ is a $k[u]/u^{ep}$-direct summand of $\cM$;
\item $N(\cN)\subseteq\cN$ and $\widehat{g}(\cN)\subseteq\cN$ for all $g\in\Gal(K/K')$;
\item $\phi_r(\cN\cap\Fil^r\cM)\subseteq\cN$.
\end{itemize}
\end{defi}

If $\cN$ is a Breuil submodule of $\cM$, then $\cN$ and $\cM/\cN$ are also objects of $\FBrModdd$. We now state a crucial result we will use later.
\begin{prop}[\cite{HLM}, Proposition~2.3.5]\label{prop: one to one between Breuil and rep}
Let $\cM$ be an object in $\FBrModdd$.

Then there is a natural inclusion preserving bijection
$$\Theta:\{\mbox{Breuil submodules in }\cM\}\rightarrow\{G_{K'}\mbox{-subrepresentations of }\Tst^r(\cM) \}$$
sending $\cN\subseteq\cM$ to the image of $\Tst^r(\cN)\hookrightarrow\Tst^r(\cM)$. Moreover, if $\cM_2\subseteq\cM_1$ are Breuil submodules of $\cM$, then $\Theta(\cM_1)/\Theta(\cM_2)\cong\Tst^r(\cM_1/\cM_2)$.
\end{prop}

We will also need classification of Breuil modules of rank $1$ as follows. We denote the Breuil modules in the following lemma by $\cM(a,s,\lambda)$.
\begin{lemm}[\cite{MP}, Lemma~3.1]\label{Lemma: classification of rank-one Breuil modules}
Let $k:=\F_{p^f}$, $e:=p^f-1$, $\varpi:= \sqrt[e]{-p}$, and $K'=\Q_p$. We also let $\cM$ be a rank-one object in $\FBrModdd[r]$.

Then there exists a generator $m\in\cM$ such that:
\begin{enumerate}
    \item $\cM=\barS_{\F}\cdot m$;
	\item $\Fil^r\cM=u^{s(p-1)}\cM$ where $0\leq s\leq \frac{re}{p-1}$;
    \item $\varphi_r(u^{s(p-1)}m)=\lambda m$ for some $\lambda\in(\F_{p^f}\otimes_{\F_p} \F)^{\times}$;
	\item $\widehat{g}(m)=(\omega_f(g)^a\otimes 1)m$ for all $g\in\Gal(K/K_0)$ where $a$ is an integer such that $a+ps\equiv 0$ mod~$(\frac{e}{p-1})$;
	\item $N(m)=0$.
\end{enumerate}
Moreover, one has
$$
\Tst^{r}(\cM)\vert_{I_{\Qp}}=\omega_f^{a+ps}.
$$
\end{lemm}

The following lemma will be used to determine if the Breuil modules violate the maximal non-splitness.
\begin{lemm}[\cite{MP}, Lemma~3.2]\label{lemma: splitting lemma niveau f}
Let $k:=\F_{p^f}$, $e:=p^f-1$, $\varpi:= \sqrt[e]{-p}$, and $K'=\Q_p$. We also let $\cM_{x}:=\cM(k_x,s_x,\lambda_x)$ and $\cM_{y}:=\cM(k_y,s_y,\lambda_y)$ be rank-one objects in $\FBrModdd[r]$. Assume that the integers $k_x,k_y,s_x,s_y\in\Z$ satisfy
\begin{equation}\label{first inequality in the splitting lemma}
p(s_{y}-s_{x})+[k_{y}-k_{x}]_f>0.
\end{equation}
Assume further that $f<p$ and let
$$
0\rightarrow \cM_{x}\rightarrow\cM\rightarrow \cM_{y}\rightarrow 0
$$
be an extension in $\FBrModdd[r]$, with $\Tst^{\ast}(\cM)$ being Fontaine--Laffaille.

If the exact sequence of $\barS_{\F}$-modules
\begin{equation}\label{exact sequence fil}
0\rightarrow \Fil^r\cM_{x}\rightarrow\Fil^r\cM\rightarrow \Fil^r\cM_{y}\rightarrow 0
\end{equation}
splits, then the $G_{\Qp}$-representation $\Tst^{\ast}(\cM)$ splits as a direct sum of two characters.

In particular, provided that $pk_{y}\not\equiv k_{x}$ modulo $e$ and that $s_y(p-1)<re$ if $f>1$, the representation $\Tst^{\ast}(\cM)$ splits as a direct sum of two characters if the element $j_0\in\Z$ uniquely defined by
\begin{equation}\label{second inequality in the splitting lemma}
j_0e+[p^{-1}k_{y}-k_{x}]_f< s_{x}(p-1)\leq (j_0+1)e+[p^{-1}k_{y}-k_{x}]_f
\end{equation}
satisfies
\begin{equation}\label{third inequality in the splitting lemma}
(r+j_0)e+[p^{-1}k_{y}-k_{x}]_f<(s_{x}+s_{y})(p-1).
\end{equation}
\end{lemm}

\subsection{Linear algebra with descent data}
In this section, we introduce the notion of framed basis for a Breuil module $\cM$ and framed system of generators for $\Fil^r\cM$. Throughout this section, we assume that $K_0=K'$ and continue to assume that $K$ is a tamely ramified Galois extension of $K'$. We also fix a positive integer $r<p-1$.

\begin{defi}
Let $n\in\N$ and let $(k_{n-1}, k_{n-2},\dots,k_0)\in \Z^n$ be an $n$-tuple. A rank $n$ Breuil module $\cM\in \FBrModdd$ \emph{is of (inertial) type} $\omega_{\varpi}^{k_{n-1}}\oplus\dots\oplus\omega_{\varpi}^{k_0}$ if $\cM$ has an $\barS$-basis $(e_{n-1},\cdots,e_0)$ such that $\widehat{g}e_i=(\omega_{\varpi}^{k_i}(g)\otimes 1) e_i$ for all $i$ and all $g\in \Gal(K/K_0)$. We call such a basis a \emph{framed basis of $\cM$}.

We also say that $\underline{f}:=(f_{n-1},f_{n-2},\dots,f_0)$ is a \emph{framed system of generators of $\Fil^r\cM$} if $\underline{f}$ is a system of $\barS$-generators for $\Fil^r\cM$ and $\widehat{g}f_i=(\omega_{\varpi}^{p^{-1}k_i}(g)\otimes 1) f_i$ for all $i$ and all $g\in \Gal(K/K_0)$.
\end{defi}

The existence of a framed basis and a framed system of generators for a given Breuil module $\cM\in\FBrModdd$ is proved in \cite{HLM}, Section~2.2.2.

Let $\cM\in\FBrModdd$ be of inertial type $\bigoplus_{i=0}^{n-1}\omega_{\varpi}^{k_i}$, and let $\underline{e}:= (e_{n-1},\dots,e_{0})$ be a framed basis for $\cM$ and $\underline{f}:=(f_{n-1},\dots,f_{0})$
be a framed system of generators for $\Fil^r\cM$. The \emph{matrix of the filtration}, with respect to $\underline{e},\underline{f}$, is the matrix
$\Mat_{\underline{e},\underline{f}}(\Fil^r\cM)\in \mathrm{M}_{n}(\barS)$ such that
$$
\underline{f}=\underline{e}\cdot \Mat_{\underline{e},\underline{f}}(\Fil^r\cM).
$$
Similarly, we define the \emph{matrix of the Frobenius} with respect to $\underline{e}$, $\underline{f}$ as the matrix $\mathrm{Mat}_{\underline{e},\underline{f}}(\varphi_r)\in \mathrm{GL}_{n}(\barS)$ characterized by
$$
(\phi_r(f_{n-1}),\cdots,\phi_r(f_0))=\underline{e}\cdot \mathrm{Mat}_{\underline{e},\underline{f}}(\varphi_r).
$$

As we require $\underline{e},\,\underline{f}$ to be compatible with the framing, the entries in the matrix of the filtration satisfy the important additional properties:
$$
\Mat_{\underline{e},\underline{f}}(\Fil^r\cM)_{i,j}\in \barS_{\omega_{\varpi}^{p^{f-1}k_j-k_i}}.
$$
More precisely, $\Mat_{\underline{e},\underline{f}}(\Fil^r\cM)_{i,j}=u^{[p^{f-1}k_j-k_i]_f}s_{i,j}$, where $s_{i,j}\in \barS_{\omega_{\varpi}^{0}}=k\otimes_{\fp}\F[u^e]/(u^{ep})$.

We can therefore introduce the subspace $\mathrm{M}_{n}^{\Box}(\barS)$ of matrices with framed type $\overline{\tau}=\bigoplus_{i=0}^{n-1}\omega_{\varpi}^{k_i}$ as
$$
\mathrm{M}_{n}^{\Box}(\barS):=
\left\{V\in \mathrm{M}_{n}(\barS),\,\,V_{i,j}\in \barS_{\omega_{\varpi}^{k_j-k_i}}\,\, \text{for\,all}\,\,0\leq i,j\leq n-1 \right\}.
$$
We also define
$$
\mathrm{GL}_{n}^{\Box}(\barS):= \mathrm{GL}_{n}(\barS)\cap \mathrm{M}_{n}^{\Box}(\barS).
$$
Similarly, we define
$$
\mathrm{M}_{n}^{\Box,\prime}(\barS):=
\left\{V\in \mathrm{M}_{n}(\barS),\,\,V_{i,j}\in \barS_{\omega_{\varpi}^{p^{f-1}k_j-k_i}}\,\, \text{for\,all}\,\,0\leq i,j\leq n-1 \right\}
$$
and
$$
\GL_{n}^{\Box,\prime}(\barS):=
\left\{V\in \GL_{n}(\barS),\,\,V_{i,j}\in \barS_{\omega_{\varpi}^{p^{f-1}(k_j-k_i)}}\,\, \text{for\,all}\,\,0\leq i,j\leq n-1 \right\}.
$$
As $\varphi_r(f_i)$ is a $\omega_f^{k_i}$-eigenvector for the action of $\Gal(K/K_0)$ we deduce that
\begin{eqnarray*}
\Mat_{\underline{e},\underline{f}}(\Fil^r\cM)\in \mathrm{M}_{n}^{\Box,\prime}(\barS)\,\,\mbox{ and }\,\,
\Mat_{\underline{e},\underline{f}}(\varphi_r)\in \mathrm{GL}_{n}^{\Box}(\barS).
\end{eqnarray*}

We use similar terminologies for strongly divisible modules $\widehat{\cM}\in\OEModdd[r]$.
\begin{defi}
Let $n\in\N$ and let $(k_{n-1}, k_{n-2},\dots,k_0)\in \Z^n$ be an $n$-tuple. A rank $n$ strongly divisible module $\widehat{\cM}\in \OEModdd[r]$ \emph{is of (inertial) type} $\widetilde{\omega}_{\varpi}^{k_{n-1}}\oplus\dots\oplus\widetilde{\omega}_{\varpi}^{k_0}$ if $\widehat{\cM}$ has an $S_{\cO_E}$-basis $\underline{\widehat{e}}:= (\widehat{e}_{n-1},\cdots,\widehat{e}_0)$ such that $\widehat{g}\widehat{e}_i=(\widetilde{\omega}_{\varpi}^{k_i}(g)\otimes 1) \widehat{e}_i$ for all $i$ and all $g\in \Gal(K/K_0)$. We call such a basis a \emph{framed basis for $\widehat{\cM}$}.

We also say that $\underline{\widehat{f}}:=(\widehat{f}_{n-1},\widehat{f}_{n-2},\dots,\widehat{f}_0)$ is a \emph{framed system of generators for $\Fil^r\widehat{\cM}$} if $\underline{\widehat{f}}$ is a system of $S$-generators for $\Fil^r\widehat{\cM}/\Fil^rS\cdot\widehat{\cM}$ and $\widehat{g}\widehat{f}_i=(\widetilde{\omega}_{\varpi}^{p^{-1}k_i}(g)\otimes 1) \widehat{f}_i$ for all $i$ and all $g\in \Gal(K/K_0)$.
\end{defi}

One can readily check the existence of a framed basis for $\widehat{\cM}$ and a framed system of generators for $\Fil^r\widehat{\cM}$, by Nakayama Lemma. We also define
$$\Mat_{\underline{\widehat{e}},\underline{\widehat{f}}}(\Fil^r\widehat{\cM})\,\,\mbox{  and  }\,\,\Mat_{\underline{\widehat{e}},\underline{\widehat{f}}}(\phi_r)$$
each of whose entries satisfies
$$
\Mat_{\underline{\widehat{e}},\underline{\widehat{f}}}(\Fil^r\widehat{\cM})_{i,j}\in S_{\widetilde{\omega}_{\varpi}^{p^{f-1}k_j-k_i}}\,\,\mbox{  and  }\,\,\Mat_{\underline{\widehat{e}},\underline{\widehat{f}}}(\phi_r)_{i,j}\in S_{\widetilde{\omega}_{\varpi}^{k_j-k_i}},
$$
in the similar fashion to Breuil modules. In particular,
\begin{eqnarray*}
\Mat_{\underline{\widehat{e}},\underline{\widehat{f}}}(\Fil^r\widehat{\cM})\in \mathrm{M}_{n}^{\Box,\prime}(S)\,\,\mbox{ and }\,\,
\Mat_{\underline{\widehat{e}},\underline{\widehat{f}}}(\varphi_r)\in \mathrm{GL}_{n}^{\Box}(S)
\end{eqnarray*}
where $\mathrm{M}_n^{\square,\prime}(S)$ and $\GL_n^{\square}(S)$ are defined in the similar way to Breuil modules. We also define $\GL_n^{\square,\prime}(S)$ in the similar way to Breuil modules again.

The inertial types on a Breuil module $\cM$ and on a strongly divisible modules are closely related to the Weil--Deligne representation associated to a potentially crystalline lift of $\Tst^{r}(\cM)$.

\begin{prop}[\cite{LMP}, Proposition~2.12]\label{prop: intertial type of a Breuil module}
Let $\widehat{\cM}$ be an object in $\OEModdd$ and let $\cM:= \widehat{\cM}\otimes_S S/(\varpi_E,\Fil^p S)$ be the Breuil module corresponding to the mod $p$ reduction of~$\widehat{\cM}$.

If $\Tst^{K_0,r}(\widehat{\cM})[\frac{1}{p}]$ has Galois type $\bigoplus_{i=0}^{n-1}\widetilde{\omega}_{f}^{k_i}$ for some integers $k_i$, then $\widehat{\cM}$ (resp. $\cM$) is of inertial type $\bigoplus_{i=0}^{n-1}\widetilde{\omega}_{\varpi}^{k_i}$ (resp. $\bigoplus_{i=0}^{n-1}\omega_{\varpi}^{k_i}$).
\end{prop}

Finally, we need a technical result on change of basis of Breuil modules with descent data.
\begin{lemm}[\cite{HLM}, Lemma~2.2.8]\label{lemm: base change matrix}
Let $\cM\in \FBrModdd$ be of type $\bigoplus_{i=0}^{n-1}\omega_{\varpi}^{k_i}$, and let $\underline{e}$, $\underline{f}$ be a framed basis for $\cM$ and a framed system of generators for $\Fil^r\cM$ respectively. Write $V:=\Mat_{\underline{e},\underline{f}}(\Fil^r\cM)\in\mathrm{M}_n^{\square,\prime}(\barS)$ and $A:= \mathrm{Mat}_{\underline{e},\underline{f}}(\varphi_r)\in\GL_n^{\square}(\barS)$, and assume that there are invertible matrices $R\in \GL_n^{\square}(\barS)$ and $C\in \GL_n^{\square,\prime}(\barS)$ such that
\begin{equation*}
R\cdot V\cdot C \equiv V'\,\,\mathrm{mod}\,(u^{e(r+1)}),
\end{equation*}
for some $V'\in\mathrm{M}_{n}^{\square,\prime}(\barS)$.

Then $\underline{e}':= \underline{e}\cdot R^{-1}$ forms another framed basis for $\cM$ and $\underline{f}':=\underline{e}'\cdot V'$ forms another framed system of generators for $\Fil^r\cM$  such that
$$\Mat_{\underline{e}',\underline{f}'}(\Fil^r\cM)=V'\in\mathrm{M}_n^{\square,\prime}(\barS)\quad \mbox{and}\quad\Mat_{\underline{e}',\underline{f}'}(\phi_r)=R\cdot A\cdot \phi(C)\in \GL_n^{\square}(\barS).$$
In particular, if $R^{-1}=A$ then $\Mat_{\underline{e}',\underline{f}'}(\phi_r)=\phi(C)$.
\end{lemm}

The statement of Lemma~\ref{lemm: base change matrix} is slightly more general than \cite{HLM}, Lemma~2.2.8, but exactly the same argument works.

\subsection{Fontaine--Laffaille modules}
In this section, we briefly recall the theory of Fontaine--Laffaille modules over $\F$, and we continue to assume that $K_0=K'$ and that $K$ is a tamely ramified Galois extension of $K'$.

\begin{defi}
A \emph{Fontaine--Laffaille module} over $k\otimes_{\Fp}\F$ is the datum $(M,\Fil^{\bullet}M,\phi_{\bullet})$~of
\begin{itemize}
	\item a free $k\otimes_{\Fp}\F$-module $M$ of finite rank;
	\item a decreasing, exhaustive and separated filtration $\{\Fil^{j}M\}_{j\in\Z}$ on $M$ by $k\otimes_{\Fp}\F$-submodules;
	\item a $\phi$-semilinear isomorphism $\phi_{\bullet}: \mathrm{gr}^{\bullet}M\rightarrow M$, where $\mathrm{gr}^{\bullet}M:=\bigoplus_{j\in\Z}\frac{\Fil^jM}{\Fil^{j+1}M}$.
\end{itemize}
\end{defi}

We write $\FFLMod_k$ for the category of Fontaine--Laffaille modules over $k\otimes_{\Fp}\F$, which is abelian. If the field $k$ is clear from the context, we simply write $\FFLMod$ to lighten the notation.

Given a Fontaine--Laffaille module $M$, the set of its Hodge--Tate weights in the direction of $\sigma\in\Gal(k/\fp)$ is defined as
$\mathrm{HT}_{\sigma}:=\left\{i\in\N\mid e_{\sigma}\Fil^iM\neq e_{\sigma}\Fil^{i+1}M\right\}$.  In the remainder of this paper we will be focused on Fontaine--Laffaille modules with \emph{parallel} Hodge--Tate weights, i.e. we will assume that for all $i\in\N$, the submodules $\Fil^iM$ are free over $k\otimes_{\fp}\F$.

\begin{defi}
Let $M$ be a Fontaine--Laffaille module with parallel Hodge--Tate weights. A $k\otimes_{\fp}\F$ basis $\underline{f}=(f_0,f_1,\dots,f_{n-1})$ on $M$ is \emph{compatible with the filtration} if for all $i\in\Z_{\geq0}$ there exists $j_i\in\Z_{\geq 0}$ such that
$\Fil^iM=\sum_{j=j_i}^{n}k\otimes_{\fp}\F\cdot f_j$. In particular, the principal symbols $(\mathrm{gr}(f_0),\dots,\mathrm{gr}(f_{n-1}))$ provide a $k\otimes_{\fp}\F$ basis for $\mathrm{gr}^{\bullet}M$.
\end{defi}
Note that if the graded pieces of the Hodge filtration have rank at most one then any two compatible basis on $M$ are related by a lower-triangular matrix in $\GL_n(k\otimes_{\fp}\F)$.  Given a Fontaine--Laffaille module and a compatible basis $\underline{f}$, it is convenient to describe the Frobenius action via a matrix $\mathrm{Mat}_{\underline{f}}(\phi_{\bullet})\in\GL_n(k\otimes_{\fp}\F)$, defined in the obvious way using the principal symbols $(\mathrm{gr}(f_0),\dots,\mathrm{gr}(f_{n-1}))$ as a basis on $\mathrm{gr}^{\bullet}M$.

It is customary to write $\FFLMod^{[0,p-2]}$ to denote the full subcategory of $\FFLMod$ formed by those modules $M$ verifying $\Fil^0M=M$ and $\Fil^{p-1}M=0$ (it is again an abelian category). We have the following description of mod $p$ Galois representations of $G_{K_0}$ via Fontaine--Laffaille modules:

\begin{prop}[\cite{FL}, Theorem~6.1]
There is an exact fully faithful contravariant functor
$$
\mathrm{T}^{\ast}_{\mathrm{cris}, K_0}:\,\,\FFLMod_k^{[0,p-2]}\rightarrow\mathrm{Rep}_{\F}(G_{K_0})
$$
which is moreover compatible with the restriction over unramified extensions: if $L_0/K_0$ is unramified with residue field $l/k$ and if $M$ is an object in $\FFLMod_k^{[0,p-2]}$, then $l\otimes_{k} M$ is naturally regarded as an object in $\FFLMod_{l}^{[0,p-2]}$ and
$$
\mathrm{T}^{\ast}_{\mathrm{cris}, L_0}(l\otimes_{k} M)\cong \mathrm{T}^{\ast}_{\mathrm{cris}, K_0}(M)\vert_{G_{L_0}}.
$$
\end{prop}

We will often write $\Tcris^{\ast}$ for $\mathrm{T}^{\ast}_{\mathrm{cris}, K_0}$ if the base field $K_0$ is clear from the context.

\begin{defi}
We say that $\rhobar\in\mathrm{Rep}_{\F}G_{K_0}$ is \emph{Fontaine--Laffaille} if $\Tcris^{\ast}(M)\cong\rhobar$ for some $M\in\FFLMod^{[0,p-2]}$.
\end{defi}

\subsection{\'Etale $\phi$-modules}
In this section, we review the theory of \'etale $\phi$-modules, first introduced by Fontaine \cite{Fon90}, and its connection with Breuil modules and Fontaine--Laffaille modules. Throughout this section, we continue to assume that $K_0=K'$ and that $K$ is a tamely ramified Galois extension of $K'$.

Let $p_0:=-p$, and let $\underline{p}$ be identified with a sequence $(p_n)_n\in\left(\overline{\Q}_p\right)^{\N}$ verifying $p_{n}^p=p_{n-1}$ for all $n$. We also fix $\varpi:= \sqrt[e]{-p}\in K$, and let $\varpi_0=\varpi$. We fix a sequence $(\varpi_n)_n\in\left(\overline{\Q}_p\right)^{\N}$  such that $\varpi_n^e=p_n$ and $\varpi_n^p=\varpi_{n-1}$ for all $n\in \N$, and which is compatible with the norm maps $K(\varpi_{n+1})\rightarrow K(\varpi_n)$ (cf. \cite{Bre}, Appendix A). By letting $K_{\infty}:= \cup_{n\in\N}K(\varpi_n)$ and $(K_{0})_{\infty}:= \cup_{n\in\N}K_0(p_n)$, we have a canonical isomorphism $\Gal(K_{\infty}/(K_{0})_{\infty})\overset{\sim}\longrightarrow \Gal(K/K_0)$ and we will identify $\omega_\varpi$ as a character of $\Gal(K_{\infty}/(K_{0})_{\infty})$. The field of norms $k((\underline{\varpi}))$ associated to $(K, \varpi)$ is then endowed with a residual action of $\Gal(K_{\infty}/(K_0)_{\infty})$, which is completely determined by $\widehat{g}(\underline{\varpi})=\omega_{\varpi}(g) \underline{\varpi}$.

We define the category $\left(\phi,\F\otimes_{\Fp}k((\underline{p}))\right)\mbox{-}\mathfrak{Mod}$ of \'etale $(\phi,\F\otimes_{\Fp}k((\underline{p})))$-modules as the category of free $\F\otimes_{\Fp}k((\underline{p}))$-modules of finite rank $\mathfrak{M}$ endowed with a semilinear map $\phi:\mathfrak{M}\rightarrow \mathfrak{M}$ with respect to the Frobenius on $k((\underline{p}))$ and inducing an isomorphism $\phi^{\ast}\mathfrak{M}\rightarrow \mathfrak{M}$ (with obvious morphisms between objects). We also define the category $(\phi,\F\otimes_{\Fp}k((\underline{\varpi})))\mbox{-}\mathfrak{Mod}_{\mathrm{dd}}$ of \'etale $(\phi,\F\otimes_{\Fp}k((\underline{\varpi})))$-modules with descent data: an object $\mathfrak{M}$ is defined as for the category  $(\phi,\F\otimes_{\Fp}k((\underline{p})))\mbox{-}\mathfrak{Mod}$ but we moreover require that $\mathfrak{M}$ is endowed with a semilinear action of $\Gal(K_{\infty}/(K_0)_{\infty})$ (semilinear with respect to the residual action on $\F\otimes_{\Fp}k((\underline{\varpi}))$ where $\F$ is endowed with the trivial $\Gal(K_{\infty}/(K_0)_{\infty})$-action) commuting with $\phi$.

By work of Fontaine \cite{Fon90}, there are anti-equivalences
$$
\left(\phi,\F\otimes_{\Fp}k((\underline{p}))\right)\mbox{-}\mathfrak{Mod}\stackrel{\sim} {\longrightarrow}\mathrm{Rep}_{\F}(G_{(K_0)_{\infty}})
$$
and
$$
\left(\phi,\F\otimes_{\Fp}k((\underline{\varpi}))\right)\mbox{-}\mathfrak{Mod}_{\mathrm{dd}} \stackrel{\sim}\longrightarrow \mathrm{Rep}_{\F}(G_{(K_0)_{\infty}})
$$
given by
$$\mathfrak{M}\longmapsto \mathrm{Hom}\left(\mathfrak{M},k((\underline{p}))^{\mathrm{sep}}\right)$$
and
$$\mathfrak{M}\mapsto\mathrm{Hom}\left(\mathfrak{M},k((\underline{\varpi}))^{\mathrm{sep}}\right)$$
respectively. See also \cite{HLM}, Appendix A.2.


The following proposition summarizes the relation between the various categories and functors we introduced above.
\begin{prop}[\cite{HLM}, Proposition 2.2.9]\label{prop: relations categories}
There exist faithful functors
$$M_{k((\underline{\varpi}))}: \FBrModdd[r]\rightarrow \left(\phi,\F\otimes_{\F_p}k((\underline{\varpi}))\right)\mbox{-}\mathfrak{Mod}_{\mathrm{dd}}$$ and
$$\mathcal{F}:\FFLMod^{[0,p-2]}\rightarrow \left(\phi,\F\otimes_{\Fp}k((\underline{p}))\right)\mbox{-}\mathfrak{Mod}$$
fitting in the following commutative diagram:
\begin{equation*}
\xymatrix{
\FBrModdd[r]\ar@{->}[dd]_{\Tst^*}\ar@{->}[rrr]^{M_{k((\underline{\varpi}))}}&&&
\left(\phi,\F\otimes_{\Fp}k((\underline{\varpi}))\right)\mbox{-}\mathfrak{Mod}_{\mathrm{dd}} \ar@{->}[ddl]_{\Hom(-,k((\underline{\varpi}))^{\mathrm{sep}})}\\&&&\\
\mathrm{Rep}_{\F}(G_{K_{0}})\ar@{->}[rr]^{\mathrm{Res}}&&\mathrm{Rep}_{\F}(G_{(K_{0})_{\infty}})\\&&&\\
\FFLMod^{[0,p-2]}\ar@{->}[uu]^{\mathrm{T}^*_{\mathrm{cris}}}\ar@{->}[rrr]_{\mathcal{F}} &&& \left(\phi,\F\otimes_{\Fp}k((\underline{p}))\right)\mbox{-}\mathfrak{Mod} \ar@{->}[uuuu]_{-\otimes_{k((\underline{p}))}k((\underline{\varpi}))} \ar@{->}[uul]^{\Hom(-,k((\underline{p}))^{\mathrm{sep}})} }
\end{equation*}
where the descent data is relative to $K_0$ and the functor $\mathrm{Res}\circ \Tcris^{\ast}$
is fully faithful.
\end{prop}
Note that the functors $M_{k((\underline{\varpi}))}$ and $\mathcal{F}$ are defined in~\cite{BD}. (See also \cite{HLM}, Appendix A). The following is an immediate consequence of Proposition~\ref{prop: relations categories}, which is also stated in \cite{LMP}, Corollary~2.14.

\begin{coro}\label{corollary comparison1}
Let $0\leq r\leq p-2$, and let $\cM$ (resp. $M$) be an object in $\FBrModdd$ (resp. in $\FFLMod^{[0,p-2]}$).  Assume that $\Tst^*(\cM)$ is Fontaine--Laffaille.
If
$$
M_{k((\underline{\varpi}))}(\cM)\cong {\mathcal{F}}(M)\otimes_{k((\underline{p}))}k((\underline{\varpi}))
$$
then one has an isomorphism of $G_{K_0}$-representations
$$
\Tst^*(\cM)\cong \Tcris^*(M).
$$
\end{coro}

The following two lemmas are very crucial in this paper, as we will see later, which describe the functors $M_{k((\underline{\varpi}))}$ and $\mathcal{F}$ respectively.
\begin{lemm}[\cite{HLM}, Lemma 2.2.6]\label{lemma: Breuil to etale}
Let $\cM$ be a Breuil module of inertial type $\bigoplus_{i=0}^{n-1}\omega_{\varpi}^{k_i}$ with a framed basis $\underline{e}$ for $\cM$ and a framed system of generators $\underline{f}$ for $\Fil^r\cM$, and write $\cM^{\ast}$ for its dual as defined in Definition~\ref{defi: dual Breuil modules}. Let $V=\Mat_{\underline{e},\underline{f}}(\Fil^r\cM)\in \mathrm{M}_{n}^{\square,'}(\barS)$ and $A= \mathrm{Mat}_{\underline{e},\underline{f}}(\phi_r)\in
\mathrm{GL}_{n}^{\square}(\barS)$.

Then there exists a basis $\mathfrak{e}$ for $M_{k((\underline{\varpi}))}(\cM^{\ast})$ with $\widehat{g}\cdot \mathfrak{e}_i=(\omega_{\varpi}^{-p^{-1}k_i}(g)\otimes 1)\mathfrak{e}_i$ for all $i\in\{0,1,\cdots,n-1\}$ and $g\in\Gal(K/K_0)$, such that the Frobenius $\phi$ on $M_{k((\underline{\varpi}))}(\cM^{\ast})$ is described~by
$$
\mathrm{Mat}_{\underline{\mathfrak{e}}}(\phi)=\widehat{V}^t\left(\widehat{A}^{-1}\right)^{t}\in
\mathrm{M}_{n}(\F\otimes_{\Fp}k[[\underline{\varpi}]])
$$
where $\widehat{V}$, $\widehat{A}$ are lifts of $V,\,A$ in $\mathrm{M}_{n}(\F\otimes_{\Fp}k[[\underline{\varpi}]])$ via the reduction morphism $\F\otimes_{\Fp}k[[\underline{\varpi}]]\onto \barS$ induced by $\underline{\varpi}\mapsto u$ and $\mathrm{Mat}_{\underline{\mathfrak{e}}}(\phi)_{i,j}\in
\left(\F\otimes_{\Fp}k[[\underline{\varpi}]]\right)_{\omega_{\varpi}^{p^{-1}k_i-k_j}}$.
\end{lemm}

\begin{lemm}[\cite{HLM}, Lemma 2.2.7]\label{lemma: Fontaine to etale}
Let $M\in\FFLMod^{[0,p-2]}$ be a rank $n$ Fontaine--Laffaille module with parallel Hodge--Tate weights $0\leq m_{0}\leq \dots\leq m_{n-1}\leq p-2$ (counted with multiplicity). Let $\underline{e}=(e_{0},\dots,e_{n-1})$ be a $k\otimes_{\fp}\F$ basis for $M$, compatible with the Hodge filtration $\mathrm{Fil}^{\bullet}M$ and let $F\in \mathrm{M}_{n}(k\otimes_{\fp}\F)$ be the associated matrix of the Frobenius $\phi_{\bullet}:\gr^{\bullet}M\rightarrow M$.

Then there exists a basis $\underline{\mathfrak{e}}$ for $\mathfrak{M}:= \mathcal{F}(M)$ such that the Frobenius $\phi$ on $\mathfrak{M}$ is described~by
$$
\mathrm{Mat}_{\underline{\mathfrak{e}}}(\phi)=\mathrm{Diag}\left(\underline{p}^{m_{0}},\cdots,\underline{p}^{m_{n-1}}\right)\cdot F\in \mathrm{M}_{n}(\F\otimes_{\Fp}k[[\underline{p}]]).
$$
\end{lemm}

\section{Local Galois side}\label{sec: local Galois side}
In this section, we study ordinary Galois representations and their potentially crystalline lifts. In particular, we prove that the Frobenius eigenvalues of certain potentially crystalline lifts preserve the information of the wildly ramified part of ordinary representations.

Throughout this section, we let $f$ be a positive integer, $K'=\Q_p$, $e=p^f-1$, and $K=\Q_{p^f}(\sqrt[e]{-p})$. We also fix $\varpi:=\sqrt[e]{-p}$, and let $\barS=(\F_{p^f}\otimes_{\F_p}\F)[u]/u^{e p}$ and $\barS_0:=\barS_{\omega_f^0}=(\F_{p^f}\otimes_{\F_p}\F)[u^e]/u^{e p}\subseteq\barS$. Recall that by $[m]_f$ for a rational number $m\in \Z[\frac{1}{p}]$ we mean the unique integer in $[0,e)$ congruent to $m$ mod~$(e)$.


We say that a representation $\rhobar_0:G_{\Q_p}\rightarrow\GL_{n}(\F)$ is \emph{ordinary} if it is isomorphic to a representation whose image is contained in the Borel subgroup of upper-triangular matrices. Namely, an ordinary representation has a basis $\underline{e}:=(e_{n-1},e_{n-2},\cdots,e_0)$ that gives rise to a matrix form as follows:
\begin{equation}\label{ordinary representation}
\rhobar_0\cong
\left(
  \begin{array}{cccccc}
   \Ur{\mu_{n-1}}\omega^{c_{n-1}+(n-1)} & \ast_{n-1} & \cdots &\ast & \ast \\
    0 & \Ur{\mu_{n-2}}\omega^{c_{n-2}+(n-2)} & \cdots &\ast & \ast \\
    \vdots & \vdots & \ddots &\vdots&\vdots \\
    0 & 0 & \cdots & \Ur{\mu_1}\omega^{c_1+1} & \ast_1 \\
    0 & 0 & \cdots& 0     &\Ur{\mu_0}\omega^{c_0} \\
  \end{array}
\right)
\end{equation}
Here, $\Ur{\mu}$ is the unramified character sending the geometric Frobenius to $\mu\in\F^{\times}$ and $c_i$ are integers. By $\rhobar_0$, we always mean an $n$-dimensional ordinary representation that is written as in (\ref{ordinary representation}). For $n-1\geq i\geq j\geq 0$, we write
\begin{equation}\label{subquotient of rhobar0}
\rhobar_{i,j}
\end{equation}
for the $(i-j+1)$-dimensional subquotient of $\rhobar_0$ determined by the subset $(e_i,e_{i-1},\cdots,e_j)$ of the basis $\underline{e}$. For instance, $\rhobar_{i,i}=\Ur{\mu_{i}}\omega^{c_i+i}$ and $\rhobar_{n-1,0}=\rhobar_0$.

An ordinary representation $G_{\Q_p}\rightarrow\GL_{n}(\F)$ is \emph{maximally non-split} if its socle filtration has length $n$. For instance, $\rhobar_0$ in (\ref{ordinary representation}) is maximally non-split if and only if $\ast_i\not=0$ for all $i=1,2,\cdots,n-1$. In this paper, we are interested in ordinary maximally non-split representations satisfying a certain genericity condition.
\begin{defi}\label{definition: genericity condition}
We say that $\rhobar_0$ is \emph{generic} if $$c_{i+1}-c_i> n-1 \mbox{ for all }i\in\{0,1,\cdots,n-2\}\mbox{ and }c_{n-1}-c_0< (p-1)-(n-1).$$
We say that $\rhobar_0$ is \emph{strongly generic} if $\rhobar_0$ is generic and $$c_{n-1}-c_0< (p-1)-(3n-5).$$
\end{defi}
Note that this strongly generic condition implies $p>n^2+2(n-3)$.

We describe a rough shape of the Breuil modules with descent data from $K$ to $K'=\Q_p$ corresponding to $\rhobar_0$. Let $r$ be a positive integer with $p-1>r\geq n-1$, and let $\cM\in\FBrModdd[r]$ be a Breuil module of inertial type $\bigoplus_{i=0}^{n-1} \omega_f^{k_i}$ such that $\Tst^{r}(\cM)\cong\rhobar_{0}$, for some $k_i\in\Z$. By Proposition~\ref{prop: one to one between Breuil and rep}, we note that $\cM$ is a successive extension of $\cM_i$, where $\cM_i:=\cM(k_i,r_i,\nu_i)$ (cf. Lemma~\ref{Lemma: classification of rank-one Breuil modules}) is a rank one Breuil module of inertial type $\omega_f^{k_i}$ such that
\begin{equation}\label{equation in type elimination}
\omega_f^{k_i+pr_i}\cong\Tst^{r}(\cM_i)|_{I_{\Q_p}}\cong\omega^{c_i+i}
\end{equation}
for each $i\in\{0,1,\cdots,n-1\}$.  More precisely, there exist a framed basis $\underline{e}=(e_{n-1},e_{n-2},\cdots,e_0)$ for $\cM$ and a framed system of generators $\underline{f}=(f_{n-1},f_{n-2},\cdots,f_0)$ for $\Fil^{r}\cM$ such that
\begin{equation}\label{filtration of Breuil Module: niveau f}
\Mat_{\underline{e},\underline{f}}(\Fil^{r}\cM)=\begin{pmatrix}
u^{r_{n-1} (p-1)}&u^{[p^{-1}k_{n-2}-k_{n-1}]_{f}} v_{n-1,n-2}&\cdots&u^{[p^{-1}k_0-k_{n-1}]_{f}} v_{n-1,0}\\
0&u^{r_{n-2} (p-1)}&\cdots&u^{[p^{-1}k_0-k_{n-2}]_{f}} v_{n-2,0}\\
\vdots&\vdots&\ddots&\vdots\\
0&0&\cdots&u^{r_0 (p-1)}
\end{pmatrix},
\end{equation}

\begin{equation}\label{Frobenius of Breuil Module: niveau f}
\Mat_{\underline{e},\underline{f}}(\phi_{r})=
\begin{pmatrix}
\nu_{n-1} &u^{[k_{n-2}-k_{n-1}]_{f}} w_{n-1,n-2}&\cdots&u^{[k_0-k_{n-1}]_{f}} w_{n-1,0}\\
0 & \nu_{n-2} & \cdots &u^{[k_0-k_{n-2}]_{f}} w_{n-2,0}\\
\vdots & \vdots & \ddots& \vdots \\
0&0&\cdots&\nu_0
\end{pmatrix},
\end{equation}
and
\begin{equation}\label{Monodoromy of Breuil Module: niveau f}
\Mat_{\underline{e}}(N)=
\begin{pmatrix}
0 &u^{[k_{n-2}-k_{n-1}]_{f}} \gamma_{n-1,n-2}&\cdots & u^{[k_1-k_{n-1}]_{f}} \gamma_{n-1,1} &u^{[k_0-k_{n-1}]_{f}} \gamma_{n-1,0}\\
0 & 0 & \cdots &u^{[k_1-k_{n-2}]_{f}} \gamma_{n-2,1} &u^{[k_0-k_{n-2}]_{f}} \gamma_{n-2,0}\\
\vdots & \vdots & \ddots& \vdots& \vdots \\
0&0&\cdots& 0& u^{[k_0-k_{1}]_{f}} \gamma_{1,0} \\
0&0&\cdots& 0 &0
\end{pmatrix}
\end{equation}
for some $\nu_i\in(\F_{p^f}\otimes_{\F_p}\F)^{\times}$ and for some $v_{i,j},w_{i,j},\gamma_{i,j}\in\barS_0$.

Fix $0\leq j\leq i\leq n-1$. We define the Breuil submodule
\begin{equation}\label{subquotient of general BM}
\cM_{i,j}
\end{equation} that is a subquotient of $\cM$ determined by the basis $(e_i,e_{i-1},\cdots,e_j)$. For instance, $\cM_{i,i}\cong \cM_i$ for all $0\leq i\leq n-1$. We note that $\Tst^r(\cM_{i,j})\cong \rhobar_{i,j}$ by Proposition~\ref{prop: one to one between Breuil and rep}.

We will keep these notation and assumptions for $\cM$ throughout this paper.

\subsection{Elimination of Galois types}\label{subsec: elimination of Galois types}
In this section, we find out the possible Galois types of niveau $1$ for potentially semi-stable lifts of $\rhobar_0$ with Hodge--Tate weights $\{-(n-1),-(n-2),\cdots,0\}$.

We start this section with the following elementary lemma.
\begin{lemm} \label{lemma determinant}
Let $\rho:G_{\qp}\rightarrow\GL_n(E)$ be a potentially semi-stable representation with Hodge--Tate weights $\{-(n-1),...,-2,-1,0\}$ and of Galois type $\bigoplus_{i=0}^{n-1}\widetilde{\omega}_{f}^{k_i}$.

Then $$\det(\rho)|_{I_{\qp}}=\varepsilon^{\frac{n(n-1)}{2}}\cdot\widetilde{\omega}_{f}^{\sum_{i=0}^{n-1}k_i},$$
where $\varepsilon$ is the cyclotomic character.
\end{lemm}

\begin{proof}
$\det(\rho)$ is a potentially crystalline character of $G_{\Q_p}$ with Hodge--Tate weight $-(\sum_{i=0}^{n-1}i)$ and of Galois type $\widetilde{\omega}_{f}^{\sum_{i=0}^{n-1}k_i}$, i.e., $\det(\rho)\cdot \widetilde{\omega}_{f}^{-\sum_{i=0}^{n-1}k_i}$ is a crystalline character with Hodge--Tate weight $-(\sum_{i=0}^{n-1}i)=-\frac{n(n-1)}{2}$ so that $\det(\rho)|_{I_{\Q_p}}\cdot \widetilde{\omega}_{f}^{-\sum_{i=0}^{n-1}k_i}\cong\varepsilon^{\frac{n(n-1)}{2}}$.
\end{proof}


We will only consider the Breuil modules $\cM$ corresponding to the mod $p$ reduction of the strongly divisible modules that corresponds to the Galois stable lattices in potentially semi-stable lifts of $\rhobar_0$ with Hodge--Tate weights $\{-(n-1),-(n-2),\cdots,-1,0\}$, so that we may assume that $r=n-1$, i.e., $\cM\in\FBrModdd[n-1]$.

\begin{lemm}\label{lemm: breuil modules niveau 1}
Let $f=1$. Assume that $\rhobar_0$ is generic, and that $\cM\in\FBrModdd[n-1]$ corresponds to the mod $p$ reduction of a strongly divisible module $\widehat{\cM}$ such that $\Tst^{n-1}(\cM)\cong\rhobar_0$ and $\Tst^{\Q_p,n-1}(\widehat{\cM})$ is a Galois stable lattice in a potentially semi-stable lift of $\rhobar_0$ with Hodge--Tate weights $\{-(n-1),-(n-2),\cdots,0\}$ and Galois type $\bigoplus_{i=0}^{n-1} \widetilde{\omega}^{k_i}$ for some integers $k_i$.

Then there exists a framed basis $\underline{e}$ for $\cM$ and a framed system of generators $\underline{f}$ for $\Fil^{n-1}\cM$ such that $\Mat_{\underline{e},\underline{f}}(\Fil^{n-1}\cM)$, $\Mat_{\underline{e},\underline{f}}(\phi_{n-1})$, and $\Mat_{\underline{e}}(N)$ are as in (\ref{filtration of Breuil Module: niveau f}), (\ref{Frobenius of Breuil Module: niveau f}), and (\ref{Monodoromy of Breuil Module: niveau f}) respectively. Moreover, the $(k_i, r_i)$ satisfy the following properties:
\begin{enumerate}
\item $k_i\equiv c_i+i-r_i$ mod $(e)$ for all $i\in\{0,1,\cdots,n-1\}$;
\item $0\leq r_i\leq n-1$ for all $i\in\{0,1,\cdots,n-1\}$;
\item $\sum_{i=0}^{n-1}r_i=\frac{(n-1)n}{2}$.
\end{enumerate}
\end{lemm}

\begin{proof}
Note that the inertial type of $\cM$ is $\bigoplus_{i=0}^{n-1} \omega^{k_i}$ by Proposition~\ref{prop: intertial type of a Breuil module}. The first part of the Lemma is obvious from the discussion at the beginning of Section~\ref{sec: local Galois side}.

We now prove the second part of the Lemma.  We may assume that the rank-one Breuil modules $\cM_i$ are of weight $n-1$, so that $0\leq r_i\leq n-1$ for $i=\{0,1,...,n-1\}$ by Lemma~\ref{Lemma: classification of rank-one Breuil modules}. By the equation~(\ref{equation in type elimination}), we have $k_i\equiv c_i+i-r_i$ mod $(e)$, as $e=p-1$. By looking at the determinant of $\rhobar_0$ we deduce the conditions $$\omega^{\frac{n(n-1)}{2}+k_{n-1}+k_{n-2}+\cdots+k_0}=\det\Tst^{n-1}(\cM)|_{I_{\qp}}=\det\rhobar_0|_{I_{\qp}}= \omega^{c_{n-1}+c_{n-2}+\cdots+c_0+\frac{n(n-1)}{2}}$$ from Lemma~\ref{lemma determinant}, and hence we have $r_{n-1}+r_{n-2}+\cdots+r_0=\frac{n(n-1)}{2}$ (as $p> n^2+2(n-3)$ due to the genericity of $\rhobar_0$).
\end{proof}

One can further eliminate Galois types of niveau $1$ if $\rhobar_0$ is maximally non-split.
\begin{prop}\label{prop: breuil modules niveau 1, maximally non-split}
Keep the assumptions and notation of Lemma~\ref{lemm: breuil modules niveau 1}. If the tuple $(k_i,r_i)$ further satisfy one of the following conditions
\begin{itemize}
\item $r_i=n-1$ for some $i\in\{0,1,2,\cdots,n-2\}$;
\item $r_i=0$ for some  $i\in\{1,2,3,\cdots,n-1\}$,
\end{itemize}
then $\rhobar_0$ is not maximally non-split.
\end{prop}

\begin{proof}
The main ingredient is Lemma~\ref{lemma: splitting lemma niveau f}. We fix $i\in\{0,1,2,\cdots,n-2\}$ and identify $x={i+1}$ and $y={i}$ and all the other following. From the results in Lemma~\ref{lemm: breuil modules niveau 1}, it is easy to compute that $[k_i-k_{i+1}]_1=e-(c_{i+1}-c_i+1)+(r_{i+1}-r_i)$. By the genericity conditions in Definition~\ref{definition: genericity condition} and by part (ii) of Lemma~\ref{lemm: breuil modules niveau 1}, we see that $0<[k_i-k_{i+1}]_1<e$ so that if $r_i\geq r_{i+1}$ then the equation~(\ref{first inequality in the splitting lemma}) in Lemma~\ref{lemma: splitting lemma niveau f} holds.

If $r_{i+1}e\leq [k_i-k_{i+1}]_1$ and $r_i\geq r_{i+1}$, then $\ast_{i+1}=0$ by Lemma~\ref{lemma: splitting lemma niveau f}. Since $0<[k_i-k_{i+1}]_1<e$, we have $r_{i+1}e\leq [k_i-k_{i+1}]_1$ if and only if $r_{i+1}=0$, in which case $\rhobar_0$ is not maximally non-split.

We now apply the second part of Lemma~\ref{lemma: splitting lemma niveau f}. It is easy to check that $j_0=r_{i+1}-1$. One can again readily check that the equation~(\ref{third inequality in the splitting lemma}) is equivalent to $r_i=n-1$, in which case $\ast_{i+1}=0$ so that $\rhobar_0$ is not maximally non-split.
\end{proof}

Note that all of the Galois types that will appear later in this section will satisfy the conditions in Lemma~\ref{lemm: breuil modules niveau 1}, and Proposition~\ref{prop: breuil modules niveau 1, maximally non-split} as well if we further assume that $\rhobar_0$ is maximally non-split.

\subsection{Fontaine--Laffaille parameters}
In this section, we parameterize the wildly ramified part of generic and maximally non-split ordinary representations using Fontaine--Laffaille theory.

We start this section by recalling that if $\rhobar_0$ is generic then $\rhobar_0\otimes \omega^{-c_0}$ is Fontaine--Laffaille (cf. \cite{GG}, Lemma 3.1.5), so that there is a Fontaine--Laffaille module $M$ with Hodge--Tate weights $\{0, c_1-c_0+1,\cdots, c_{n-1}-c_0+(n-1)\}$ such that $\Tcris^{\ast}(M)\cong\rhobar_0\otimes\omega^{-c_0}$ (if we assume that $\rhobar_0$ is generic).
\begin{lemm}\label{lemm: Fontaine--Laffaille module}
Assume that $\rhobar_0$ is generic, and let $M\in \FFLMod^{[0,p-2]}_{\F_p}$ be a Fontaine--Laffaille module such that $\Tcris^{\ast}(M)\cong \rhobar_0\otimes\omega^{-c_0}$.

Then there exists a basis $\underline{e}=(e_0,e_1,\cdots,e_{n-1})$ for $M$ such that
\begin{equation*}
\Fil^{j}M=
\left\{
  \begin{array}{ll}
   M  & \hbox{ if $j\leq 0$;} \\
   \F(e_{i},\cdots,e_{n-1})  & \hbox{ if $c_{i-1}-c_0+i-1<j\leq c_{i}-c_0+i$;} \\
   0  & \hbox{ if $c_{n-1}-c_0+n-1<j$.}
  \end{array}
\right.
\end{equation*}
and
\begin{equation}\label{matrix of Frobenius of FL module}
\Mat_{\underline{e}}(\phi_{\bullet})=
\begin{pmatrix}
\mu_0^{-1}&\alpha_{0,1}&\alpha_{0,2}&\cdots &\alpha_{0,n-2}&\alpha_{0,n-1}\\
0&\mu_1^{-1}&\alpha_{1,2}&\cdots&\alpha_{1,n-2}&\alpha_{1,n-1}\\
0&0&\mu_{2}^{-1}&\cdots&\alpha_{2,n-2}&\alpha_{2,n-1}\\
\vdots&\vdots&\vdots&\ddots&\vdots&\vdots\\
0&0&0&\cdots&\mu_{n-2}^{-1}&\alpha_{n-2,n-1}\\
0&0&0&\cdots&0&\mu_{n-1}^{-1}
\end{pmatrix}
\end{equation}
where $\alpha_{i,j}\in\F$.
\end{lemm}
Note that the basis $\underline{e}$ on $M$ in Lemma~\ref{lemm: Fontaine--Laffaille module} is compatible with the filtration.

\begin{proof}
This is an immediate generalization of \cite{HLM}, Lemma~2.1.7.
\end{proof}

For $i\geq j$, the subset $(e_j,\cdots,e_i)$ of $\underline{e}$ determines a subquotient $M_{i,j}$ of the Fontaine--Laffaille module $M$, which is also a Fontaine--Laffaille module with the filtration induced from $\Fil^sM$ in the obvious way and with Frobenius described as follows:
\begin{equation*}
A_{i,j}:=
\begin{pmatrix}
\mu_j^{-1}&\alpha_{j,j+1}&\cdots&\alpha_{j,i-1} &\alpha_{j,i}\\
0&\mu_{j+1}^{-1}&\cdots&\alpha_{j+1,i-1}&\alpha_{j+1,i}\\
\vdots&\vdots&\ddots&\vdots&\vdots\\
0&0&\cdots&\mu_{i-1}^{-1} &\alpha_{i-1,i}\\
0&0&\cdots&0 &\mu_{i}^{-1}
\end{pmatrix}.
\end{equation*}
Note that $\Tcris^{\ast}(M_{i,j})\otimes\omega^{c_0}\cong \rhobar_{i,j}$. We let $A_{i,j}'$ be the $(i-j)\times(i-j)$-submatrix of $A_{i,j}$ obtained by deleting the left-most column and the lowest row of $A_{i,j}$.

\begin{lemm}\label{lemm: invariant elements}
Keep the assumptions and notation of Lemma~\ref{lemm: Fontaine--Laffaille module}, and let $0\leq j<j+1<i\leq n-1$. Assume further that $\rhobar_0$ is maximally non-split.

If $\det A_{i,j}'\neq (-1)^{i-j+1}\mu_{j+1}^{-1}\cdots\mu_{i-1}^{-1}\alpha_{j,i}$, then $[\alpha_{j,i}:\det A_{i,j}']\in\mathbb{P}^1(\F)$ does not depend on the choice of basis $\underline{e}$ compatible with the filtration.
\end{lemm}

\begin{proof}
This is an immediate generalization of \cite{HLM}, Lemma~2.1.9.
\end{proof}

\begin{defi}\label{definiton: Fontaine--Laffaille parameters}
Keep the assumptions and notation of Lemma~\ref{lemm: invariant elements}, and assume further
that $\rhobar_0$ satisfies
\begin{equation}\label{condition 1 on FL parameter}
\det A_{i,j}'\neq (-1)^{i-j+1}\mu_{j+1}^{-1}\cdots\mu_{i-1}^{-1}\alpha_{j,i}
\end{equation}
for all $i,j\in\Z$ with $0\leq j<j+1<i\leq n-1$.

The Fontaine--Laffaille parameter associated to $\rhobar_0$ is defined as
$$\mathrm{FL}_{n}(\rhobar_0):= \left(\mathrm{FL}_{n}^{i,j}(\rhobar_0)\right)_{i,j}\in[\mathbb{P}^1(\F)]^{\frac{(n-2)(n-1)}{2}}$$
where
$$
\mathrm{FL}_{n}^{i,j}(\rhobar_0):= \left[\alpha_{j,i}:(-1)^{i-j+1}\cdot\det A_{i,j}'\right]\in\mathbb{P}^1(\F)
$$
for all $i,j\in\Z$ such that $0\leq j<j+1<i\leq n-1$.
\end{defi}
We often write $\frac{y}{x}$ for $[x:y]\in \mathbb{P}^1(\F)$ if $x\neq 0$. The conditions in (\ref{condition 1 on FL parameter}) for $i,j$ guarantee the well-definedness of $\mathrm{FL}_n^{i,j}(\rhobar_0)$ in $\mathbb{P}^1(\F)$. We also point out that $\mathrm{FL}_n^{i,j}(\rhobar_0)\neq(-1)^{i-j}\mu_{j+1}^{-1}\cdots\mu_{i-1}^{-1}$ in $\mathbb{P}^1(\F)$.

One can define the inverses of the elements in $\mathbb{P}^1(\F)$ in a natural way: for $[x_1:x_2]\in \mathbb{P}^1(\F)$, $[x_1:x_2]^{-1}:=[x_2:x_1]\in \mathbb{P}^1(\F)$.

\begin{lemm}\label{lemm: FL parameter with dual rep}
Assume that $\rhobar_0$ is generic. Then
\begin{enumerate}
\item $\rhobar_0^{\vee}$ is generic;
\item if $\rhobar_0$ is strongly generic, then so is $\rhobar_0^{\vee}$;
\item if $\rhobar_0$ is maximally non-split, then so is $\rhobar_0^{\vee}$;
\item if $\rhobar_0$ is maximally non-split, then the conditions in~(\ref{condition 1 on FL parameter}) are stable under $\rhobar_0\mapsto\rhobar_0^{\vee}$.
\end{enumerate}
Assume further that $\rhobar_0$ is maximally non-split and satisfies the conditions in~(\ref{condition 1 on FL parameter}).
\begin{enumerate}\setcounter{enumi}{4}
\item for all $i,j\in\Z$ with $0\leq j<j+1<i\leq n-1$, $\mathrm{FL}_n^{i,j}(\rhobar_0)=\mathrm{FL}_n^{i,j}(\rhobar_0\otimes\omega^b)$ for any $b\in\Z$;
\item for all $i,j\in\Z$ with $0\leq j<j+1<i\leq n-1$, $\mathrm{FL}_n^{i,j}(\rhobar_0)=\mathrm{FL}_{i-j+1}^{i-j,0}(\rhobar_{i,j})$;
\item for all $i,j\in\Z$ with $0\leq j<j+1<i\leq n-1$, $\mathrm{FL}_n^{i,j}(\rhobar_0)^{-1}=\mathrm{FL}_n^{n-1-j,n-1-i}(\rhobar_0^{\vee})$.
\end{enumerate}
\end{lemm}

\begin{proof}
(i), (ii) and (iii) are easy to check. We leave them for the reader.

The only effect on Fontaine--Laffaille module by twisting $\omega^b$ is shifting the jumps of the filtration. Thus (v) and (vi) are obvious.

For (iv) and (vii), one can check that the Frobenius of the Fontaine--Laffaille module associated to $\rhobar_0^{\vee}$ is described by
$$\left(
    \begin{array}{ccccc}
      0 & 0 &\cdots& 0 & 1 \\
      0 & 0 &\cdots& 1 & 0 \\
      \vdots & \vdots &\ddots& \vdots & \vdots \\
      0 & 1 &\cdots& 0 & 0 \\
      1 & 0 &\cdots& 0 & 0 \\
    \end{array}
  \right)
\cdot [\Mat_{\underline{e}}(\phi_{\bullet})^t]^{-1} \cdot
\left(
    \begin{array}{ccccc}
      0 & 0 &\cdots& 0 & 1 \\
      0 & 0 &\cdots& 1 & 0 \\
      \vdots & \vdots &\ddots& \vdots & \vdots \\
      0 & 1 &\cdots& 0 & 0 \\
      1 & 0 &\cdots& 0 & 0 \\
    \end{array}
  \right)
$$
where $\Mat_{\underline{e}}(\phi_{\bullet})$ is as in (\ref{matrix of Frobenius of FL module}). Now one can check them by direct computation.
\end{proof}

We end this section by defining certain numerical conditions on Fontaine--Laffaille parameters. We consider the matrix $(1,n)w_0\Mat_{\underline{e}}(\phi_{\bullet})^t$, where $\Mat_{\underline{e}}(\phi_{\bullet})$ is the upper-triangular matrix in (\ref{matrix of Frobenius of FL module}). Here, $w_0$ is the longest element of the Weyl group $W$ associated to $T$ and $(1,n)$ is a permutation in $W$. Note that the anti-diagonal matrix displayed in the proof of Lemma~\ref{lemm: FL parameter with dual rep} is $w_0$ seen as an element in $\GL_n(\F)$. For $1\leq i\leq n-1$ we let $B_i$ be the square matrix of size $i$ that is the left-bottom corner of $(1,n)w_0\Mat_{\underline{e}}(\phi_{\bullet})^t$.

\begin{defi}\label{definition: Fontaine--Laffaille generic}
Keep the notation and assumptions of Definition~\ref{definiton: Fontaine--Laffaille parameters}. We say that $\rhobar_0$ is \emph{Fontaine--Laffaille generic} if moreover $\det B_i\not=0$ for all $1\leq i\leq n-1$ and $\rhobar_0$ is strongly generic.
\end{defi}
We emphasize that by an ordinary representation $\rhobar_0$ being Fontaine--Laffaille generic, we \emph{always mean} that $\rhobar_0$ satisfies the maximally non-splitness and the conditions in (\ref{condition 1 on FL parameter}) as well as $\det B_i\not=0$ for all $1\leq i\leq n-1$ and the strongly generic assumption (cf. Definition~\ref{definition: genericity condition}).

Although the Frobenius matrix of a Fontaine--Laffaille module depends on the choice of basis, it is easy to see that the non-vanishing of the determinants above is independent of the choice of basis compatible with the filtration. Note that the conditions in Definition~\ref{definition: Fontaine--Laffaille generic} are necessary and sufficient conditions for $$(1,n)w_0\Mat_{\underline{e}}(\phi_{\bullet})^t\in B(\F)w_0B(\F) $$ in the Bruhat decomposition, which will significantly reduce the size of the paper (cf. Remark~\ref{remark on why Fontaine--Laffaille generic}). We also note that
\begin{itemize}
\item $\det B_1\not=0$ if and only if $\mathrm{FL}^{n-1,0}_n(\rhobar_0)\not=\infty$;
\item $\det B_{n-1}\not=0$ if and only if $\mathrm{FL}^{n-1,0}_n(\rhobar_0)\not=0$.
\end{itemize}
Finally, we point out that the locus of Fontaine--Laffaille generic ordinary Galois representations $\rhobar_0$ forms a (Zariski) open subset in $[\mathbb{P}^1(\F)]^{\frac{(n-1)(n-2)}{2}}$.

\begin{rema}\label{remark on why Fontaine--Laffaille generic}
Definition~\ref{definition: Fontaine--Laffaille generic} comes from the fact that the list of Serre weights of $\rhobar_0$ is then minimal in the sense of Conjecture~\ref{conj: weight elimination}. It is very crucial in the proof of Theorem~\ref{theo: lgc} as it is more difficult to track the Fontaine--Laffaille parameters on the automorphic side if we have too many Serre weights. Moreover, these conditions simplify our proof for Theorem~\ref{thm: main theorem Galois}.
\end{rema}

\subsection{Breuil modules of certain inertial types of niveau $1$}\label{subsec: Breuil modules of certain types}
In this section, we classify the Breuil modules with certain inertial types, corresponding to the ordinary Galois representations $\rhobar_0$ as in (\ref{ordinary representation}), and we also study their corresponding Fontaine--Laffaille parameters.

Throughout this section, we always assume that $\rhobar_0$ is strongly generic. Since we are only interested in inertial types of niveau $1$, we let $f=1$, $e=p-1$, and $\varpi=\sqrt[e]{-p}$. We define the following integers for $0\leq i\leq n-1$:
\begin{equation}\label{Galois types of newest}
r^{(0)}_{i}:=
\left\{
  \begin{array}{ll}
    1 & \hbox{if $i=n-1$;} \\
    i & \hbox{if $0<i<n-1$;} \\
    n-2 & \hbox{if $i=0$.}
 \end{array}
\right.
\end{equation}
We also set $$k^{(0)}_i:= c_i+i-r^{(0)}_i$$ for all $i\in\{0,1,\cdots,n-1\}$.

We first classify the Breuil modules of inertial types described as above.
\begin{lemm}\label{lemm: Breuil modules for newest FL, classification}
Assume that $\rhobar_0$ is strongly generic and that $\cM\in\FBrModdd[n-1]$ corresponds to the mod $p$ reduction of a strongly divisible modules $\widehat{\cM}$ such that $\Tst^{\Q_p,n-1}(\widehat{\cM})$ is a Galois stable lattice in a potentially semi-stable lift of $\rhobar_0$ with Hodge--Tate weights $\{-(n-1),-(n-2),\cdots,0\}$ and Galois type $\bigoplus_{i=0}^{n-1} \widetilde{\omega}^{k_i^{(0)}}$.

Then $\cM\in \FBrModdd[n-1]$ can be described as follows: there exist a framed basis $\underline{e}$ for $\cM$ and a framed system of generators $\underline{f}$ for $\Fil^{n-1}\cM$ such that
\begin{equation*}
\Mat_{\underline{e},\underline{f}}(\Fil^{n-1}\cM)=
\begin{pmatrix}
u^{r^{(0)}_{n-1}e}& \beta_{n-1,n-2}u^{r^{(0)}_{n-1}e-k^{(0)}_{n-1,n-2}}& \cdots & \beta_{n-1,0}u^{r^{(0)}_{n-1}e-k^{(0)}_{n-1,0}} \\
0&u^{r^{(0)}_{n-2}e} & \cdots& \beta_{n-2,0}u^{r^{(0)}_{n-2}e-k^{(0)}_{n-2,0}}\\
\vdots& \vdots & \ddots&\vdots \\
0&0& \cdots &u^{r^{(0)}_0 e}
\end{pmatrix}
\end{equation*}
and
\begin{equation*}
\Mat_{\underline{e},\underline{f}}(\phi_{n-1})=\mathrm{Diag}\left(\nu_{n-1},\,\nu_{n-2},\,\cdots,\,\nu_0\right)
\end{equation*}
where $k_{i,j}^{(0)}:=k^{(0)}_i-k^{(0)}_j$, $\nu_i\in\F^{\times}$ and $\beta_{i,j}\in\F$. Moreover,
$$\Mat_{\underline{e}}(N)=\left(\gamma_{i,j}\cdot u^{[k^{(0)}_j-k^{(0)}_i]_1}\right)$$
where $\gamma_{i,j}=0$ if $i\leq j$ and $\gamma_{i,j}\in u^{e[k^{(0)}_j-k^{(0)}_i]_1}\barS_0$ if $i>j$.
\end{lemm}
Note that $\underline{e}$ and $\underline{f}$ in Lemma~\ref{lemm: Breuil modules for newest FL, classification} are not necessarily the same as the ones in Lemma~\ref{lemm: breuil modules niveau 1}.

\begin{proof}
We keep the notation in (\ref{filtration of Breuil Module: niveau f}), (\ref{Frobenius of Breuil Module: niveau f}), and (\ref{Monodoromy of Breuil Module: niveau f}). That is, there exist a framed basis $\underline{e}$ for $\cM$ and a framed system of generators $\underline{f}$ for $\Fil^{n-1}\cM$ such that
$\Mat_{\underline{e},\underline{f}}(\Fil^{n-1}\cM)$, $\Mat_{\underline{e},\underline{f}}(\phi_{n-1})$, $\Mat_{\underline{e}}(N)$ are given as in (\ref{filtration of Breuil Module: niveau f}), (\ref{Frobenius of Breuil Module: niveau f}), and (\ref{Monodoromy of Breuil Module: niveau f}) respectively.  Since $k_i\equiv k_i^{(0)}$ mod $(p-1)$, we have $r_i=r^{(0)}_i$ for all $i\in\{0,1,\cdots,n-1\}$ by Lemma~\ref{lemm: breuil modules niveau 1}), following the notation of Lemma~\ref{lemm: breuil modules niveau 1}.

We start to prove the following claim: if $n-1\geq i>j\geq 0$ then
\begin{equation}\label{equation e-ki+kj}
e-(k_{i}^{(0)}-k_j^{(0)})\geq n.
\end{equation}
Indeed, by the strongly generic assumption, Definition~\ref{definition: genericity condition}
\begin{align*}
  e-(k_{i}^{(0)}-k_j^{(0)}) & =(p-1)-(c_i+i-r^{(0)}_i)+(c_j+j-r^{(0)}_j) \\
   & =(p-1)-(c_i-c_j)-(i-j)+(r_i^{(0)}-r_j^{(0)}) \\
   & \geq (p-1)-(c_{n-1}-c_0)-(n-1-0)+(1-(n-2))\\
   & \geq 3n-4-2n+4=n.
\end{align*}
Note that this claim will be often used during the proof later.

We now diagonalize $\Mat_{\underline{e},\underline{f}}(\phi_{n-1})$ with some restriction on the powers of the entries of $\Mat_{\underline{e},\underline{f}}(\Fil^{n-1}\cM)$. Let $V_0 =\Mat_{\underline{e},\underline{f}}(\Fil^{n-1}\cM)\in\mathrm{M}_n^{\square}(\barS)$ and $A_0 = \Mat_{\underline{e},\underline{f}}(\phi_{n-1})\in\GL_n^{\square}(\barS)$. We also let $V_1\in\mathrm{M}_n^{\square}(\barS)$ be the matrix obtained
from $V_0$ by replacing $v_{i,j}$ by $v'_{i,j}\in\barS_0$, and $B_1\in\GL_n^{\square}(\barS)$ the matrix obtained from $A_0$ by replacing $w_{i,j}$ by $w'_{i,j}\in\barS_0$. It is straightforward to check that $A_0\cdot V_1= V_0\cdot B_1$ if and only if for all $i>j$
\begin{multline}\label{equation AV=VB niveau 1}
\nu_i v'_{i,j}u^{[k^{(0)}_j-k^{(0)}_i]_1}+ \sum_{s=j+1}^{i-1}w_{is}v'_{s,j}u^{[k^{(0)}_s-k^{(0)}_i]_1+[k^{(0)}_j-k^{(0)}_s]_1}+ w_{i,j}u^{r^{(0)}_je+[k^{(0)}_j-k^{(0)}_i]_1}\\
=w'_{i,j}u^{r^{(0)}_i e+[k^{(0)}_j-k^{(0)}_i]_1}+ \sum_{s=j+1}^{i-1}v_{i,s}w'_{s,j}u^{[k^{(0)}_s-k^{(0)}_i]_1+[k^{(0)}_j-k^{(0)}_s]_1}+\nu_j v_{i,j}u^{[k^{(0)}_j-k^{(0)}_i]_1}.
\end{multline}
Note that the power of $u$ in each term of (\ref{equation AV=VB niveau 1}) is congruent to $[k^{(0)}_j-k^{(0)}_i]_1$ modulo $(e)$. It is immediate that for all $i>j$ there exist $v'_{i,j}\in\barS_0$ and $w'_{i,j}\in\barS_0$ satisfying the equation (\ref{equation AV=VB niveau 1}) with the following additional properties: for all $i>j$
\begin{equation}\label{equation degree of entry of filtration, niveau 1}
\deg v'_{i,j}< r^{(0)}_ie.
\end{equation}
Letting $\underline{e}':= \underline{e}A_0$, we have
$$\Mat_{\underline{e}',\underline{f}'}(\Fil^{n-1}\cM)=V_1\mbox{  and  }\Mat_{\underline{e}',\underline{f}'}(\phi_{n-1})=\phi(B_1)$$
where $\underline{f}'=\underline{e}'V_1$, by Lemma~\ref{lemm: base change matrix}. Note that $\phi(B_1)$ is congruent to a diagonal matrix modulo $(u^{ne})$ by (\ref{equation e-ki+kj}).  We repeat this process one more time. We may assume that $w_{i,j}\in u^{ne}\barS_0$, i.e., that $A_0\equiv B_1$ modulo $(u^{ne})$ where $B_1$ is assumed to be a diagonal matrix. It is obvious that there exists an upper-triangular matrix $V_1=(v'_{i,j}u^{[p^{-1}k^{(0)}_j-k^{(0)}_i]_2})$ whose entries have bounded degrees as in (\ref{equation degree of entry of filtration, niveau 1}), satisfying the equation $A_0V_1\equiv V_0B_1$ modulo $(u^{ne})$. By Lemma~\ref{lemm: base change matrix}, we get $\Mat_{\underline{e}',\underline{f}'}(\phi_{n-1})$ is diagonal.  Hence, we may assume that $\Mat_{\underline{e},\underline{f}}(\phi_{n-1})$ is diagonal and that $\deg v_{i,j}$ in $\Mat_{\underline{e},\underline{f}}(\Fil^{n-1}\cM)$ is bounded as in (\ref{equation degree of entry of filtration, niveau 1}), and we do so. Moreover, this change of basis do not change the shape of $\Mat_{\underline{e}}(N)$, so that we also assume that $\Mat_{\underline{e}}(N)$ is still as in (\ref{Monodoromy of Breuil Module: niveau f}).

We now prove that
\begin{equation}\label{equation degree of entry of filtration, niveau 1-2}
v_{i,j}u^{[k_j^{(0)}-k_i^{(0)}]_1}=\beta_{i,j}u^{r^{(0)}_{i}e-(k_i^{(0)}-k_j^{(0)})}
\end{equation}
for all $n-1\geq i>j\geq 0$, where $\beta_{i,j}\in\F$. Note that this is immediate for $i=n-1$ and $i=1$, since $r^{(0)}_i=1$ if $i=n-1$ or $i=1$. To prove (\ref{equation degree of entry of filtration, niveau 1-2}), we induct on $i$. The case $i=1$ is done as above. Fix $p_0\in\{2,3,\cdots,n-2\}$, and assume that (\ref{equation degree of entry of filtration, niveau 1-2}) holds for all $i\in\{1,2,\cdots,p_0-1\}$ and for all $j<i$. We consider the subquotient $\cM_{p_0,0}$ of $\cM$ defined in (\ref{subquotient of general BM}). By abuse of notation, we write $\underline{e}=(e_{p_0},\cdots, e_0)$ for the induced framed basis for $\cM_{p_0,0}$ and $\underline{f}=(f_{p_0},\cdots,f_0)$ for the induced framed system of generators for $\Fil^{n-1}\cM_{p_0,0}$.

We claim that for $p_0\geq j\geq 0$ $$u^eN(f_{j})\in \barS_0u^ef_j+\sum_{t=j+1}^{p_0}\barS_0u^{[k^{(0)}_{j}-k^{(0)}_{t}]_1}f_t.$$
Consider $N(f_{j})=N(f_{j}-u^{r^{(0)}_je}e_{j})+N(u^{r^{(0)}_je}e_{j})$. It is easy to check that $N(f_{j}-u^{r^{(0)}_je}e_{j})$ and $N(u^{r^{(0)}_je}e_{j})+r^{(0)}_je f_{j}$ are $\barS$-linear combinations of $e_{n-1},\cdots,e_{j+1}$, and they are, in fact, $\barS_0$-linear combinations of $u^{[k^{(0)}_j-k^{(0)}_{n-1}]}e_{n-1},\cdots,u^{[k^{(0)}_j-k^{(0)}_{j+1}]}e_{j+1}$ since they are $\omega^{k^{(0)}_{j}}$-invariant. Since $u^eN(f_{j})\in\Fil^{n-1}\cM\supset u^{(n-1)e}\cM$ and $u^eN(f_{j})+r^{(0)}_je u^ef_{j}=[N(f_{j}-u^{r^{(0)}_je}e_{j})]+[N(u^{r^{(0)}_je}e_{j})+r^{(0)}_je f_{j}]$, we conclude that $$u^eN(f_{j})+r^{(0)}_je u^ef_{j}\in  \sum_{t=j+1}^{p_0}\barS_0u^{[k^{(0)}_{j}-k^{(0)}_{t}]_1} f_t,$$ which completes the claim.

Let
$\Mat_{\underline{e},\underline{f}}(N|_{\cM_{p_0,0}})=\left(\gamma_{i,j}\cdot u^{[k^{(0)}_j-k^{(0)}_i]_1}\right)$
where $\gamma_{i,j}=0$ if $i\leq j$ and $\gamma_{i,j}\in\barS_0$ if $i>j$. We also claim that $$\gamma_{i,j}\in u^{e[k^{(0)}_j-k^{(0)}_i]_1}\barS_0$$ for $p_0\geq i>j\geq 0$, which can be readily checked from the equation $cN\phi_{n-1}(f_j)=\phi_{n-1}(u^eN(f_j))$. (Note that $c=1\in\barS$ as $E(u)=u^e+p$.) Indeed, we have $$cN\phi_{n-1}(f_j)=N(\nu_je_j)=\nu_j\sum_{i=j+1}^{p_0}\gamma_{i,j}u^{[k^{(0)}_j-k^{(0)}_i]_1}e_i.$$ On the other hand, since $\Mat_{\underline{e},\underline{f}}(\phi_{n-1}|_{\cM_{p_0,0}})$ is diagonal, the previous claim immediately implies that $$\phi_r(u^eN(f_j))\in \sum_{t=j+1}^{p_0}\barS_0 u^{p[k^{(0)}_{j}-k^{(0)}_{t}]_1}e_{t}.$$ Hence, we conclude the claim.

We now finish the proof of (\ref{equation degree of entry of filtration, niveau 1-2}) by inducting on $p_0-j$ as well. Write $v_{i,j}=\sum_{t=0}^{r^{(0)}_i-1}x_{i,j}^{(t)}u^{te}$ for $x_{i,j}^{(t)}\in\F$. We need to prove $x_{p_0,j}^{(t)}=0$ for $t\in\{0,1,\cdots,r^{(0)}_{p_0}-2\}$. Assume first $j=p_0-1$, and we compute $N(f_j)$ as follows:
\begin{multline*}
N(f_{p_0-1})=-\sum_{t=0}^{r^{(0)}_{p_0}-1}x_{p_0,p_0-1}^{(t)}[e(t+1)-(k^{(0)}_{p_0}-k^{(0)}_{p_0-1})]u^{e(t+1)-(k^{(0)}_{p_0}-k^{(0)}_{p_0-1})}e_{p_0}\\
+\gamma_{p_0,p_0-1}u^{(r^{(0)}_{p_0-1}+1)e-(k^{(0)}_{p_0}-k^{(0)}_{p_0-1})}e_{p_0}- r^{(0)}_{p_0-1}eu^{r^{(0)}_{p_0-1}e}e_{p_0-1},
\end{multline*}
which immediately implies
\begin{multline*}
N(f_{p_0-1})\equiv\sum_{t=0}^{r^{(0)}_{p_0}-1}x_{p_0,p_0-1}^{(t)}[er^{(0)}_{p_0-1}-e(t+1)+(k^{(0)}_{p_0}-k^{(0)}_{p_0-1})]u^{e(t+1)-(k^{(0)}_{p_0}-k^{(0)}_{p_0-1})}e_{p_0} \\ +\gamma_{p_0,p_0-1}u^{(r^{(0)}_{p_0-1}+1)e-(k^{(0)}_{p_0}-k^{(0)}_{p_0-1})}e_{p_0}
\end{multline*}
modulo $\Fil^{n-1}\cM_{p_0,0}$. Since $\gamma_{p_0,p_0-1}\in u^{e[e-(k^{(0)}_{p_0}-k^{(0)}_{p_0-1})]}\barS_0$ and $e-(k^{(0)}_{p_0}-k^{(0)}_{p_0-1})\geq n$ by (\ref{equation e-ki+kj}), we get
\begin{equation*}
N(f_{p_0-1})\equiv\sum_{t=0}^{r^{(0)}_{p_0}-1}x_{p_0,p_0-1}^{(t)}[er^{(0)}_{p_0-1}-e(t+1)+(k^{(0)}_{p_0}-k^{(0)}_{p_0-1})]u^{e(t+1)-(k^{(0)}_{p_0}-k^{(0)}_{p_0-1})}e_{p_0}
\end{equation*}
modulo $\Fil^{n-1}\cM_{p_0,0}$, so that
\begin{equation*}
u^{e}N(f_{p_0-1})\equiv\sum_{t=0}^{r^{(0)}_{p_0}-2}x_{p_0,p_0-1}^{(t)}[er^{(0)}_{p_0-1}-e(t+1)+(k^{(0)}_{p_0}-k^{(0)}_{p_0-1})]u^{e(t+2)-(k^{(0)}_{p_0}-k^{(0)}_{p_0-1})}e_{p_0}
\end{equation*}
modulo $\Fil^{n-1}\cM_{p_0,0}$.

It is easy to check that
\begin{equation}\label{equation congruence by genericity condition}
er^{(0)}_{p_0-1}-e(t+1)+(k^{(0)}_{p_0}-k^{(0)}_{p_0-1})\not \equiv 0
\end{equation} modulo $(p)$ for all $0\leq t\leq r^{(0)}_{p_0}-2$. Indeed, $er^{(0)}_{p_0-1}-e(t+1)+(k^{(0)}_{p_0}-k^{(0)}_{p_0-1})\equiv -r^{(0)}_{p_0-1}+(t+1)+(k^{(0)}_{p_0}-k^{(0)}_{p_0-1})= (t+1)+(c_{p_0}-c_{p_0-1}+1)-r^{(0)}_{p_0}$ modulo $(p)$. Since $0\leq t\leq r^{(0)}_{p_0}-2$,
$$0<(c_{p_0}-c_{p_0-1}+2)-r^{(0)}_{p_0}\leq (t+1)+(c_{p_0}-c_{p_0-1}+1)-r^{(0)}_{p_0}\leq (c_{p_0}-c_{p_0-1}-1)<p$$
by the strongly generic conditions, Definition~\ref{definition: genericity condition}. Hence, we conclude that $x_{p_0,p_0-1}^{(t)}=0$ for all $0\leq t\leq r^{(0)}_{p_0}-2$ since $u^{e}N(f_{p_0-1})\in\Fil^{n-1}\cM_{p_0,0}$. This completes the proof of (\ref{equation degree of entry of filtration, niveau 1-2}) for $j=p_0-1$.

Assume that (\ref{equation degree of entry of filtration, niveau 1-2}) holds for $i=p_0$ and $j\in\{p_0-1,p_0-2,\cdots,s+1\}$. We compute $N(f_s)$ for $p_0-1>s\geq 0$ as follows: using the induction hypothesis on $i\in\{1,2,\cdots,p_0-1\}$
\begin{multline*}
N(f_s)=-\sum_{t=0}^{r^{(0)}_{p_0}-1}x_{p_0,s}^{(t)}[e(t+1)-(k^{(0)}_{p_0}-k^{(0)}_{s})]u^{e(t+1)-(k^{(0)}_{p_0}-k^{(0)}_s)}e_{p_0}\\
  +\sum_{i=s+1}^{p_0-1}\beta_{i,s}u^{r^{(0)}_i e-(k^{(0)}_i-k^{(0)}_s)}\left(\sum_{s=i+1}^{p_0}\gamma_{s,i}u^{e-(k^{(0)}_s-k^{(0)}_i)}e_s-[r^{(0)}_i e-(k^{(0)}_i-k^{(0)}_s)]e_i\right)\\
  +u^{r^{(0)}_s e}\sum_{i=s+1}^{p_0}\gamma_{i,s}u^{e-(k^{(0)}_i-k^{(0)}_s)}e_i-r^{(0)}_s e u^{r^{(0)}_s e}e_s.
\end{multline*}
Since $\gamma_{i,j}\in u^{e[e-(k^{(0)}_i-k^{(0)}_j)]}\barS_0$, we have
\begin{multline*}
N(f_s)\equiv -\sum_{t=0}^{r^{(0)}_{p_0}-1}x_{p_0,s}^{(t)}[e(t+1)-(k^{(0)}_{p_0}-k^{(0)}_{s})]u^{e(t+1)-(k^{(0)}_{p_0}-k^{(0)}_s)}e_{p_0}\\
  -\sum_{i=s+1}^{p_0-1}\beta_{i,s}[r^{(0)}_i e-(k^{(0)}_i-k^{(0)}_s)]u^{r^{(0)}_i e-(k^{(0)}_i-k^{(0)}_s)}e_i-r^{(0)}_s e u^{r^{(0)}_s e}e_s
\end{multline*}
modulo $\Fil^{n-1}\cM_{p_0,0}$, which immediately implies
\begin{multline*}
N(f_s)\equiv \sum_{t=0}^{r^{(0)}_{p_0}-1}x_{p_0,s}^{(t)}[r^{(0)}_se-e(t+1)+(k^{(0)}_{p_0}-k^{(0)}_{s})]u^{e(t+1)-(k^{(0)}_{p_0}-k^{(0)}_s)}e_{p_0}\\
  +\sum_{i=s+1}^{p_0-1}\beta_{i,s}[r^{(0)}_s e-r^{(0)}_i e+(k^{(0)}_i-k^{(0)}_s)]u^{r^{(0)}_i e-(k^{(0)}_i-k^{(0)}_s)}e_i
\end{multline*}
modulo $\Fil^{n-1}\cM_{p_0,0}$. Now, from the induction hypothesis on $j\in\{p_0-1,p_0-2,\cdots,s+1\}$,
$$u^e\sum_{i=s+1}^{p_0-1}\beta_{i,s}[r^{(0)}_s e-r^{(0)}_i e+(k^{(0)}_i-k^{(0)}_s)]u^{r^{(0)}_i e-(k^{(0)}_i-k^{(0)}_s)}e_i\in\Fil^{n-1}\cM_{p_0,0}$$
and so we have
\begin{equation*}
u^eN(f_s)\equiv \sum_{t=0}^{r^{(0)}_{p_0}-2}x_{p_0,s}^{(t)}[r^{(0)}_se-e(t+1)+(k^{(0)}_{p_0}-k^{(0)}_{s})]u^{e(t+2)-(k^{(0)}_{p_0}-k^{(0)}_s)}e_{p_0}
\end{equation*}
modulo $\Fil^{n-1}\cM_{p_0,0}$. By the same argument as (\ref{equation congruence by genericity condition}), one can readily check that $r^{(0)}_se-e(t+1)+(k^{(0)}_{p_0}-k^{(0)}_{s})\not\equiv 0$ modulo $(p)$ for all $0\leq t\leq r^{(0)}_{p_0}-2$. Hence, we conclude that
$x_{p_0,s}^{(t)}=0$ for all $0\leq t\leq r^{(0)}_{p_0}-2$ as $u^eN(f_s)\in\Fil^{n-1}\cM_{p_0,0}$, which completes the proof.
\end{proof}

\begin{prop}\label{prop: Breuil modules for newest FL}
Keep the assumptions and notation of Lemma~\ref{lemm: Breuil modules for newest FL, classification}. Assume further that $\rhobar_0$ is maximally non-split and satisfies the conditions in (\ref{condition 1 on FL parameter}).

Then $\beta_{i,i-1}\in\F^{\times}$ for $i\in\{1,2,\cdots,n-1\}$ and we have the following identities: for $0\leq j<j+1<i\leq n-1$
$$\mathrm{FL}_n^{i,j}(\rhobar_0)=\left[\beta_{i,j}\nu_{j+1}\cdots\nu_{i-1}:(-1)^{i-j+1}\det A'_{i,j}\right]\in\mathbb{P}^1(\F)$$ where
$$A'_{i,j}=
\begin{pmatrix}
\beta_{j+1,j} & \beta_{j+2,j}&\beta_{j+3,j}&\cdots&\beta_{i-1,j} &\beta_{i,j}\\
1&\beta_{j+2,j+1}&\beta_{j+3,j+1}&\cdots&\beta_{i-1,j+1} & \beta_{i,j+1}\\
0&1&\beta_{j+3,j+2}&\cdots&\beta_{i-1,j+2} & \beta_{i,j+2}\\
\vdots&\vdots&\vdots&\ddots&\vdots&\vdots\\
0 & 0&0 &\cdots &\beta_{i-1,i-2}& \beta_{i,i-2}\\
0 & 0&0 &\cdots &1& \beta_{i,i-1}
\end{pmatrix}.$$
\end{prop}

\begin{proof}
We may assume $c_0=0$ by Lemma~\ref{lemm: FL parameter with dual rep}.  We let $V:= \Mat_{\underline{e},\underline{f}}(\Fil^{n-1}\cM)$ and $A:= \Mat_{\underline{e},\underline{f}}(\phi_{n-1})$ be as in the statement of Lemma~\ref{lemm: Breuil modules for newest FL, classification}. By Lemma~\ref{lemma: Breuil to etale}, the $\phi$-module over $\F\otimes_{\fp}\fp((\underline{\varpi}))$ defined by $\mathfrak{M}:= M_{\fp((\underline{\varpi}))}(\cM^{\ast})$ is described as follows:
$$\Mat_{\underline{\mathfrak{e}}}(\phi)=(U_{i,j})$$
where
$$
U_{i,j}=
\left\{
\begin{array}{ll}
\nu_j^{-1}\cdot \underline{\varpi}^{r^{(0)}_j e} & \hbox{if $i=j$};\\
0 & \hbox{if $i > j$};\\
\nu_j^{-1}\cdot\beta_{j,i}\cdot \underline{\varpi}^{r^{(0)}_je-(k^{(0)}_j-k^{(0)}_i)} & \hbox{if $i<j$}
\end{array}
\right.$$
in a framed basis $\underline{\mathfrak{e}}=(\mathfrak{e}_{n-1},\mathfrak{e}_{n-2},\cdots, \mathfrak{e}_{0})$ with dual type $\omega^{-k_{n-1}^{(0)}}\oplus\omega^{-k_{n-2}^{(0)}}\cdots\oplus\omega^{-k_0^{(0)}}$.


By considering the change of basis $\underline{\mathfrak{e}}'=(\underline{\varpi}^{k^{(0)}_{n-1}}\mathfrak{e}_{n-1},\underline{\varpi}^{k^{(0)}_{n-2}}\mathfrak{e}_{n-2}, \cdots, \underline{\varpi}^{k^{(0)}_0}\mathfrak{e}_{0})$, $\Mat_{\underline{\mathfrak{e}}'}(\phi)$ is described as follows:
$$\Mat_{\underline{\mathfrak{e}}'}(\phi)=(V_{i,j})$$
where
$$
V_{i,j}=
\left\{
\begin{array}{ll}
\nu_j^{-1}\cdot \underline{\varpi}^{e(k^{(0)}_j+r^{(0)}_j)} & \hbox{if $i=j$};\\
0 & \hbox{if $i > j$};\\
\nu_j^{-1}\cdot\beta_{j,i}\cdot \underline{\varpi}^{e(k^{(0)}_j+r^{(0)}_j)} & \hbox{if $i<j$.}
\end{array}
\right.$$

Since $k^{(0)}_i=c_i+i-r^{(0)}_i$ for each $n-1\geq i\geq 0$, we easily see that the $\phi$-module $\mathfrak{M}_0$ is the base change via $\F\otimes_{\fp}\F_{p}((\underline{p}))\rightarrow \F\otimes_{\fp}\F_{p}((\underline{\varpi}))$ of the $\phi$-module $\mathfrak{M}_0$ over $\F\otimes_{\fp}\F_{p}((\underline{p}))$ described by
\begin{equation*}
\Mat_{\underline{\mathfrak{e}}''}(\phi)=
\begin{pmatrix}
\nu_{n-1}^{-1}\underline{p}^{c_{n-1}+(n-1)}&0& \cdots&0 \\
\nu_{n-1}^{-1}\beta_{n-1,n-2}\underline{p}^{c_{n-1}+(n-1)}&\nu_{n-2}^{-1}\underline{p}^{c_{n-2}+(n-2)} & \cdots&0\\
\vdots&\vdots&\ddots&\vdots\\
\nu_{n-1}^{-1}\beta_{n-1,0}\underline{p}^{c_{n-1}+(n-1)}& \nu_{n-2}^{-1}\beta_{n-2,0}\underline{p}^{c_{n-2}+(n-2)}& \cdots&\nu_0^{-1}\underline{p}^{c_0}
\end{pmatrix}
\end{equation*}
in an appropriate basis $\underline{\mathfrak{e}}''=(\mathfrak{e}''_{n-1},\mathfrak{e}''_{n-2},\cdots, \mathfrak{e}''_{0})$, which can be rewritten as
\begin{equation*}
\Mat_{\underline{\mathfrak{e}}''}(\phi)=
\underbrace{
\begin{pmatrix}
\nu_{n-1}^{-1}&0& \cdots&0 \\
\nu_{n-1}^{-1}\beta_{n-1,n-2}&\nu_{n-2}^{-1} & \cdots&0\\
\vdots&\vdots&\ddots&\vdots\\
\nu_{n-1}^{-1}\beta_{n-1,0} & \nu_{n-2}^{-1}\beta_{n-2,0} & \cdots& \nu_0^{-1}
\end{pmatrix}}_{=: H'}
\cdot\mathrm{Diag}\left(\underline{p}^{c_{n-1}+n-1},\cdots,\underline{p}^{c_{1}+1},\underline{p}^{c_{0}}\right).
\end{equation*}

By considering the change of basis $\underline{\mathfrak{e}}'''= \underline{\mathfrak{e}}''\cdot H'$ and then reversing the order of the basis~$\underline{\mathfrak{e}}'''$,
the Frobenius $\phi$ of $\mathfrak{M}_0$ with respect to this new basis is described as follows:
\begin{equation}\label{matrix of Frobenius of Breuil-FL}
\Mat(\phi)=\mathrm{Diag}\left(\underline{p}^{c_0},\underline{p}^{c_1+1},\cdots, \underline{p}^{c_{n-1}+(n-1)}\right)
\underbrace{
\begin{pmatrix}
\nu_0^{-1}&\nu_1^{-1}\beta_{1,0}& \cdots &\nu_{n-1}^{-1}\beta_{n-1,0} \\
0&\nu_1^{-1} & \cdots&\nu_{n-1}^{-1}\beta_{n-1,1}\\
\vdots& \vdots & \ddots&\vdots \\
0&0& 0&\nu_{n-1}^{-1}
\end{pmatrix}}_{=: H}
\end{equation}
with respect to the new basis described as above.

The last displayed upper-triangular matrix $H$ is the Frobenius of the Fontaine--Laffaille module $M$ such that $\Tcris^{\ast}(M)\cong\rhobar_0\cong\Tst^{r}(\cM)$, by Lemma~\ref{lemma: Fontaine to etale}.  Hence, we get the desired results (cf. Definition~\ref{definiton: Fontaine--Laffaille parameters}).
\end{proof}

\begin{rema}\label{remark on F.L. generic on Breuil modules}
We emphasize that the matrix $H$ is the Frobenius of the Fontaine--Laffaille module $M$, with respect to a basis $(e_0,e_1,\cdots, e_{n-1})$ compatible with the filtration, such that $\Tcris^{\ast}(M)\cong\rhobar_0\cong\Tst^{r}(\cM)$, so that we can now apply the conditions in (\ref{condition 1 on FL parameter}) as well as Definition~\ref{definition: Fontaine--Laffaille generic} to the Breuil modules in Lemma~\ref{lemm: Breuil modules for newest FL, classification}. Moreover, $H$ can be written as
$$
H=\underbrace{
\begin{pmatrix}
1& \beta_{1,0}& \cdots & \beta_{n-1,0} \\
0&1 & \cdots& \beta_{n-1,1}\\
\vdots& \vdots & \ddots&\vdots \\
0&0& 0& 1
\end{pmatrix}}_{=: H'}
\cdot\, \mathrm{Diag}\left(\nu_0^{-1},\nu_1^{-1},\cdots,\nu_{n-1}^{-1}\right),
$$
so that we have $(1,n)w_0H^t\in B(\F)w_0 B(\F)$ if and only if $(1,n)w_0 (H')^t\in B(\F)w_0 B(\F)$. Hence, $\rhobar_0$ being Fontaine--Laffaille generic is a matter only of the entries of the filtration of the Breuil modules if the Breuil modules are written as in Lemma~\ref{lemm: Breuil modules for newest FL, classification}.
\end{rema}

\subsection{Fontaine--Laffaille parameters vs Frobenius eigenvalues}\label{subsec: Fontaine vs Frobenius}
In this section, we study further the Breuil modules of Lemma~\ref{lemm: Breuil modules for newest FL, classification}. We show that if the filtration is of a certain shape then a certain product of Frobenius eigenvalues (of the Breuil modules) corresponds to the newest Fontaine--Laffaille parameter, $\mathrm{FL}_n^{n-1,0}(\rhobar_0)$. To get such a shape of the filtration, we assume further that $\rhobar_0$ is Fontaine--Laffaille generic.

\begin{lemm}\label{lemm: Breuil modules for newest FL}
Keep the assumptions and notation of Lemma~\ref{lemm: Breuil modules for newest FL, classification}. Assume further that $\rhobar_0$ is Fontaine--Laffaille generic (c.f. Definition~\ref{definition: Fontaine--Laffaille generic}).

Then $\cM\in \FBrModdd[n-1]$ can be described as follows: there exist a framed basis $\underline{e}$ for $\cM$ and a framed system of generators $\underline{f}$ for $\Fil^{n-1}\cM$ such that
\begin{equation*}
\Mat_{\underline{e},\underline{f}}(\phi_{n-1})=\mathrm{Diag}\left(\mu_{n-1},\,\mu_{n-2},\,\cdots,\,\mu_0\right)
\end{equation*}
and
\begin{equation*}
\Mat_{\underline{e},\underline{f}}(\Fil^{n-1}\cM)=\left(U_{i,j}\right)
\end{equation*}
where
\begin{equation}\label{filtration of Breuil Module under FL generic}
U_{i,j}=
\left\{
  \begin{array}{ll}
    u^{r^{(0)}_{n-1}e-(k^{(0)}_{n-1}-k^{(0)}_0)} & \hbox{if $i=n-1$ and $j=0$;} \\
    u^{r^{(0)}_i e}& \hbox{if $0<i=j<n-1$;} \\
    x_{i,j}\cdot u^{r^{(0)}_i e-(k^{(0)}_i-k^{(0)}_j)} & \hbox{if $n-1>i>j$;}\\
    u^{r^{(0)}_{0}e+(k^{(0)}_{n-1}-k^{(0)}_0)} & \hbox{if $i=0$ and $j=n-1$;} \\
    x_{0,j}\cdot u^{r^{(0)}_0 e+(k^{(0)}_j-k^{(0)}_0)} & \hbox{if $i=0\leq j< n-1$;}\\
    0 & \hbox{otherwise}.
 \end{array}
\right.
\end{equation}
Here, $\mu_i\in\F^{\times}$ and $x_{i,j}\in\F$.

Moreover, we have the following identity:
$$\mathrm{FL}_n^{n-1,0}(\rhobar_0)=\prod_{i=1}^{n-2}\mu_i^{-1}.$$
\end{lemm}

Due to the size of the matrix, we decide to describe the matrix $\Mat_{\underline{e},\underline{f}}(\Fil^{n-1}\cM)$ as~(\ref{filtration of Breuil Module under FL generic}). But for the reader we visualize the matrix $\Mat_{\underline{e},\underline{f}}(\Fil^{n-1}\cM)$ below, although it is less accurate:
$$\begin{pmatrix}
0 & 0  & \cdots & 0 & u^{r^{(0)}_{n-1}e-k^{(0)}_{n-1,0}}\\
0 & u^{r^{(0)}_{n-2}}&\cdots&x_{n-2,1} u^{r^{(0)}_{n-2} e-k^{(0)}_{n-2,1}} & x_{n-2,0} u^{r^{(0)}_{n-2} e-k^{(0)}_{n-2,0}}\\
\vdots&\vdots&\ddots&\vdots&\vdots\\
0 & 0 &\cdots &u^{r^{(0)}_{1}}& x_{1,0} u^{r^{(0)}_{1} e-k^{(0)}_{1,0}}\\
u^{r^{(0)}_{0}e+k^{(0)}_{n-1,0}} & x_{0,n-2} u^{r^{(0)}_0 e+k^{(0)}_{n-2,0}} &\cdots &x_{0,1} u^{r^{(0)}_0 e+k^{(0)}_{1,0}} & x_{0,0} u^{r^{(0)}_0 e}
\end{pmatrix}$$
where $k^{(0)}_{i,j}:=k^{(0)}_i-k^{(0)}_j$.
\begin{proof}
Let $\underline{e}_0$ be a framed basis for $\cM$ and $\underline{f}_0$ a framed system of generators for $\Fil^{n-1}\cM$ such that $V_{0}:=\Mat_{\underline{e}_0,\underline{f}_0}(\Fil^{n-1}\cM)$ and $A_{0}:=\Mat_{\underline{e}_0,\underline{f}_0}(\phi_{n-1})$ are given as in Lemma~\ref{lemm: Breuil modules for newest FL, classification}. So, in particular, $V_0$ is upper-triangular and $A_0$ is diagonal.

By Proposition~\ref{prop: Breuil modules for newest FL}, the upper-triangular matrix $H$ in~(\ref{matrix of Frobenius of Breuil-FL}) is the Frobenius of the Fontaine--Laffaille module corresponding to $\rhobar_0$, as in Definition~\ref{definiton: Fontaine--Laffaille parameters}. Since we assume that $\rhobar_0$ is Fontaine--Laffaille generic, we have $(1,n)w_0 H^t\in B(\F)w_0B(\F)$ as discussed right after Definition~\ref{definiton: Fontaine--Laffaille parameters}, so that we have $w_0 H^tw_0\in (1,n)B(\F)w_0B(\F)w_0$. Equivalently, $w_0 (H')^tw_0\in (1,n)B(\F)w_0B(\F)w_0$ by Remark~\ref{remark on F.L. generic on Breuil modules}, where $H'$ is defined in Remark~\ref{remark on F.L. generic on Breuil modules}. Hence, comparing $V_0$ with $w_0 (H')^tw_0$, there exists a lower-triangular matrix $C\in\mathrm{GL}_{n}^{\square}(\barS)$ such that $$V_0\cdot C = V_1:= \left(U_{i,j}\right)_{0\leq i,j\leq n-1}$$ where $U_{i,j}$ is described as in~(\ref{filtration of Breuil Module under FL generic}), since any matrix in $w_0 B(\F) w_0$ is lower-triangular. From the identity $V_0\cdot C=V_1$,  we have $V_1=\Mat_{\underline{e}_1,\underline{f}_1}(\Fil^{n-1}\cM)$ and $A_1:=\Mat_{\underline{e}_1,\underline{f}_1}(\phi_{n-1})=A_0\cdot \phi(C)$ by Lemma~\ref{lemm: base change matrix}, where $\underline{e}_1:=\underline{e}_0$ and $\underline{f}_1:=\underline{e}_1V_1$. If $i<j$, then $[k^{(0)}_j-k^{(0)}_i]_1=k^{(0)}_j-k^{(0)}_i\geq n$ as $\rhobar_0$ is strongly generic, so that $A_1$ is congruent to a diagonal matrix $B'_2\in\mathrm{GL}_{n}(\F)$ modulo $(u^{ne})$ as $C=(c_{i,j}\cdot u^{[k^{(0)}_j-k^{(0)}_i]_1})$ is a lower-triangular and $A_0$ is diagonal.

Let $V_2$ be the matrix obtained from $V_1$ by replacing $x_{i,j}$ in (\ref{filtration of Breuil Module under FL generic}) by $y_{i,j}$, and $B_2=(b_{i,j})$ is the diagonal matrix defined by taking $b_{i,i}=b'_{i,i}$ if $1\leq i\leq n-2$ and $b_{i,i}=b'_{n-1-i,n-1-i}$ otherwise, where $B'_2=(b'_{i,j})$. Then it is obvious that there exist $y_{i,j}\in\F$ such that
$$A_1\cdot V_2\equiv V_1\cdot B_2 $$ modulo $(u^{ne})$. Letting $\underline{e}_2:= \underline{e}_1\cdot A_1$, we have $V_2=\Mat_{\underline{e}_2,\underline{f}_2}(\Fil^{n-1}\cM)$ and $\Mat_{\underline{e}_2,\underline{f}_2}(\phi_{n-1})=\phi(B_2)$ by Lemma~\ref{lemm: base change matrix}. Note that $A_2:= \Mat_{\underline{e}_2,\underline{f}_2}(\phi_{n-1})$ is diagonal. Hence, there exist a framed basis for $\cM$ and a framed system of generators for $\Fil^{n-1}\cM$ such that
$\Mat_{\underline{e},\underline{f}}(\phi_{n-1})$ and $\Mat_{\underline{e},\underline{f}}(\Fil^{n-1}\cM)$ are described as in the statement.

We now prove the second part of the lemma. It is harmless to assume $c_0=0$ by Lemma~\ref{lemm: FL parameter with dual rep}. Let $V:= \Mat_{\underline{e},\underline{f}}(\Fil^{n-1}\cM)$ and $A:= \Mat_{\underline{e},\underline{f}}(\phi_{n-1})$ be as in the first part of the lemma. By Lemma~\ref{lemma: Breuil to etale}, the $\phi$-module over $\F\otimes_{\fp}\fp((\underline{\varpi}))$ defined by $\mathfrak{M}:= M_{\fp((\underline{\varpi}))}(\cM^{\ast})$ is described as follows: there exists a basis $\underline{\mathfrak{e}}=(\mathfrak{e}_{n-1},\mathfrak{e}_{n-2},\cdots, \mathfrak{e}_{0})$, compatible with decent data, such that $\Mat_{\underline{\mathfrak{e}}}(\phi)=(\widehat{A}^{-1}\widehat{V})^t$ where $\widehat{V}^t$ and $(\widehat{A}^{-1})^t$ are computed as follows:
$$
\widehat{V}^t=
\begin{pmatrix}
0 & 0  & \cdots & 0 & \underline{\varpi}^{r^{(0)}_{0}e+k^{(0)}_{n-1,0}}\\
0 & \underline{\varpi}^{r^{(0)}_{n-2}}&\cdots&0 & x_{0,n-2} \underline{\varpi}^{r^{(0)}_0 e+k^{(0)}_{n-2,0}}\\
\vdots&\vdots&\ddots&\vdots&\vdots\\
0 & x_{n-2,1} \underline{\varpi}^{r^{(0)}_{n-2} e-k^{(0)}_{n-2,1}} &\cdots &\underline{\varpi}^{r^{(0)}_{1}}& x_{0,1} \underline{\varpi}^{r^{(0)}_0 e+k^{(0)}_{1,0}} \\
\underline{\varpi}^{r^{(0)}_{n-1}e-k^{(0)}_{n-1,0}} & x_{n-2,0} \underline{\varpi}^{r^{(0)}_{n-2} e-k^{(0)}_{n-2,0}} &\cdots & x_{1,0} \underline{\varpi}^{r^{(0)}_{1} e-k^{(0)}_{1,0}} & x_{0,0} \underline{\varpi}^{r^{(0)}_0 e}
\end{pmatrix}$$
and $$\widehat{A}^{-1}=\mathrm{Diag}\left(\mu_{n-1}^{-1},\,\mu_{n-2}^{-1},\,\cdots,\,\mu_0^{-1}\right).$$

By considering the change of basis $\underline{\mathfrak{e}}'=(\underline{\varpi}^{k^{(0)}_{n-1}}\mathfrak{e}_{n-1},\underline{\varpi}^{k^{(0)}_{n-2}}\mathfrak{e}_{n-2}, \cdots, \underline{\varpi}^{k^{(0)}_1}\mathfrak{e}_{1}, \underline{\varpi}^{k^{(0)}_0}\mathfrak{e}_{0})$, we have
$$\Mat_{\underline{\mathfrak{e}}'}(\phi)=(\widehat{V}^t)'\cdot\mathrm{Diag}\left(\mu_{n-1}^{-1},\,\mu_{n-2}^{-1},\,\cdots,\,\mu_0^{-1}\right)$$ where
$$
(\widehat{V}^t)'=\begin{pmatrix}
0 & 0  & \cdots & 0 & \underline{\varpi}^{e(k^{(0)}_{0}+r^{(0)}_{0})}\\
0 & \underline{\varpi}^{e(k^{(0)}_{n-2}+r^{(0)}_{n-2})}&\cdots&0 & x_{0,n-2} \underline{\varpi}^{e(k^{(0)}_{0}+r^{(0)}_{0})}\\
\vdots&\vdots&\ddots&\vdots&\vdots\\
0 & x_{n-2,1} \underline{\varpi}^{e(k^{(0)}_{n-2}+r^{(0)}_{n-2})} &\cdots &\underline{\varpi}^{e(k^{(0)}_{1}+r^{(0)}_{1})}& x_{0,1} \underline{\varpi}^{e(k^{(0)}_{0}+r^{(0)}_{0})} \\
\underline{\varpi}^{e(k^{(0)}_{n-1}+r^{(0)}_{n-1})} & x_{n-2,0} \underline{\varpi}^{e(k^{(0)}_{n-2}+r^{(0)}_{n-2})} &\cdots & x_{1,0} \underline{\varpi}^{e(k^{(0)}_{1}+r^{(0)}_{1})} & x_{0,0} \underline{\varpi}^{e(k^{(0)}_{0}+r^{(0)}_{0})}
\end{pmatrix}.
$$

Since $k^{(0)}_j+r^{(0)}_j=c_j+j$ for all $j$, it is immediate that the $\phi$-module $\mathfrak{M}$ over $\F\otimes_{\fp}\F_{p}((\underline{\varpi}))$ is the base change via $\F\otimes_{\fp}\F_{p}((\underline{p}))\rightarrow \F\otimes_{\fp}\F_{p}((\underline{\varpi}))$ of the $\phi$-module $\mathfrak{M}_0$ over $\F\otimes_{\fp}\F_{p}((\underline{p}))$ described by
$$\Mat_{\underline{\mathfrak{e}}''}(\phi)=F''\cdot \mathrm{Diag}\left(\underline{p}^{c_{n-1}+n-1},\underline{p}^{c_{n-2}+n-2},\cdots,\underline{p}^{c_0}\right)$$ where
$$
F''=
\begin{pmatrix}
0 & 0 & 0  & \cdots & 0 & \mu_{0}^{-1}\\
0 & \mu_{n-2}^{-1}& 0 & \cdots & 0 & \mu_0^{-1}x_{0,n-2}\\
0 & \mu_{n-2}^{-1}x_{n-2,n-3}   & \mu_{n-3}^{-1}  &\cdots & 0 & \mu_0^{-1}x_{0,n-3}\\
\vdots & \vdots &\vdots & \ddots &\vdots &\vdots \\
0 & \mu_{n-2}^{-1}x_{n-2,1} & \mu_{n-3}^{-1}x_{n-3,1} &\cdots & \mu_{1}^{-1} & \mu_0^{-1}x_{0,1} \\
\mu_{n-1}^{-1} & \mu_{n-2}^{-1}x_{n-2,0} &  \mu_{n-3}^{-1}x_{n-3,0} &\cdots & \mu_{1}^{-1}x_{1,0} & \mu_{0}^{-1}x_{0,0}
\end{pmatrix},
$$
in an appropriate basis $\underline{\mathfrak{e}}''$.

Now, consider the change of basis $\underline{\mathfrak{e}}'''=\underline{\mathfrak{e}}''\cdot F''$ and then reverse the order of the basis~$\mathfrak{e}'''$. Then the matrix of the Frobenius $\phi$ for $\mathfrak{M}_0$ with respect to this new basis is given by
$$\mathrm{Diag}\left(\underline{p}^{c_0},\underline{p}^{c_1+1},\cdots,\underline{p}^{c_{n-1}+n-1}\right) \cdot F $$ where
$$
F=
\begin{pmatrix}
\mu_{0}^{-1}x_{0,0}& \mu_{1}^{-1}x_{1,0} & \mu_2^{-1}x_{2,0} & \cdots &\mu_{n-2}^{-1}x_{n-2,0}&\mu_{n-1}^{-1}\\
\mu_{0}^{-1}x_{0,1} & \mu_{1}^{-1}& \mu_2^{-1}x_{2,1}& \cdots & \mu_{n-2}^{-1}x_{n-2,1}&0\\
\mu_{0}^{-1}x_{0,2} & 0            &\mu_2^{-1}  &\cdots & \mu_{n-2}^{-1}x_{n-2,2}&0\\
\vdots & \vdots &\vdots & \ddots &\vdots &\vdots \\
\mu_{0}^{-1}x_{0,n-2} & 0 & 0 &\cdots & \mu_{n-2}^{-1} & 0 \\
\mu_0^{-1} &0& 0 &\cdots & 0&0
\end{pmatrix}.
$$

By Lemma~\ref{lemma: Fontaine to etale}, there exists a Fontaine--Laffaille module $M$ such that $\mathcal{F}(M)=\mathfrak{M}_0$ with Hodge--Tate weights $(c_0, c_1+1, \cdots, c_{n-1}+n-1)$ and $\Mat_{\underline{e}}(\phi_{\bullet})=F$ for some basis $\underline{e}$ of $M$ compatible with the Hodge filtration on $M$. On the other hand, since $\Tcris^{\ast}(M)\cong\rhobar_0$, there exists a basis $\underline{e}'$ of $M$ compatible with the Hodge filtration on $M$ such that
$$\Mat_{\underline{e}'}(\phi_{\bullet})=
\underbrace{\begin{pmatrix}
w_0& w_{0,1}& \cdots & w_{0,n-2}& w_{0,n-1}\\
0& w_1&\cdots & w_{1,n-2}& w_{1,n-1}\\
\vdots & \vdots &\ddots & \vdots & \vdots \\
0&0& \cdots & w_{n-2} & w_{n-2,n-1} \\
0& 0 & \cdots & 0& w_{n-1}
\end{pmatrix}}_{=: G}$$
where $w_{i,j}\in\F$ and $w_i\in\F^{\times}$ by Lemma~\ref{lemm: Fontaine--Laffaille module}. Since both $\underline{e}$ and $\underline{e}'$ are compatible with the Hodge filtration on $M$, there exists a unipotent lower-triangular $n\times n$-matrix $U$ such that
$$U\cdot F=G.$$
Note that we have $w_{0,n-1}=\mu_{n-1}^{-1}$ by direct computation.

Let $U'$ be the $(n-1)\times (n-1)$-matrix obtained from $U$ by deleting the right-most column and the lowest row, and $F'$ (resp. $G'$) the $(n-1)\times (n-1)$-matrix obtained from $F$ (resp. $G$) by deleting the left-most column and the lowest row. Then they still satisfy $G'=U'\cdot F'$ as $U$ is a lower-triangular unipotent matrix, so that
$$\mathrm{FL}_n^{n-1,0}(\rhobar_0)=\left[w_{0,n-1}:(-1)^n\det G'\right]= \left[\mu_{n-1}^{-1}:(-1)^{n}\det F'\right]=\left[1: \prod_{i=1}^{n-2}\mu_i^{-1}\right],$$ which completes the proof.
\end{proof}

\begin{prop}\label{prop: Breuil modules for newest FL 3}
Keep the assumptions and notation of Lemma~\ref{lemm: Breuil modules for newest FL}.

Then $\cM\in \FBrModdd[n-1]$ can be described as follows: there exist a framed basis $\underline{e}$ for $\cM$ and a framed system of generators $\underline{f}$ for $\Fil^{n-1}\cM$ such that
$$\Mat_{\underline{e},\underline{f}}(\Fil^{n-1}\cM)=
\begin{pmatrix}
0 & 0 & 0 &\cdots& 0 & u^{e-(k^{(0)}_{n-1}-k^{(0)}_0)}\\
0 & u^{(n-2)e}& 0 & \cdots& 0 & 0 \\
0 & 0 & u^{(n-3)e}&\cdots& 0  & 0 \\
\vdots&\vdots&\vdots&\ddots&\vdots&\vdots\\
0 & 0&0 &\cdots & u^e & 0\\
u^{(n-2)e+(k^{(0)}_{n-1}-k^{(0)}_0)} & 0&0 &\cdots & 0 & 0
\end{pmatrix}.
$$

Moreover, if we let
$$
\Mat_{\underline{e},\underline{f}}(\phi_{n-1})=\left(\alpha_{i,j}u^{[k^{(0)}_j-k^{(0)}_i]_1}\right)
$$
for $\alpha_{i,i}\in \barS_0^{\times}$ and $\alpha_{i,j}\in \barS_0$ if $i\neq j$ then we have the following identity:
$$\mathrm{FL}_n^{n-1,0}(\rhobar_0)=\prod_{i=1}^{n-2}(\alpha_{i,i}^{(0)})^{-1}=\prod_{i=1}^{n-2}\mu_i^{-1}$$
where $\alpha_{i,j}^{(0)}\in\F$ is determined by $\alpha_{i,j}^{(0)}\equiv \alpha_{i,j}$ modulo $(u^e)$.
\end{prop}

Note that $\Mat_{\underline{e},\underline{f}}(\phi_{n-1})$ always belong to $\GL_n^{\square}(\barS)$ as $\underline{e}$ and $\underline{f}$ are framed.

\begin{proof}
We let $\underline{e}_0$ (resp. $\underline{e}_1$) be a framed basis for $\cM$ and $\underline{f}_0$ (resp. $\underline{f}_1$) be a framed system of generators for $\Fil^{n-1}\cM$ such that $\Mat_{\underline{e}_0,\underline{f}_0}(\Fil^{n-1}\cM)$ and $\Mat_{\underline{e}_0,\underline{f}_0}(\phi_{n-1})$ (resp. $\Mat_{\underline{e}_1,\underline{f}_1}(\Fil^{n-1}\cM)$ and $\Mat_{\underline{e}_1,\underline{f}_1}(\phi_{n-1})$) are given as in the statement of Lemma~\ref{lemm: Breuil modules for newest FL} (resp. in the statement of Proposition~\ref{prop: Breuil modules for newest FL 3}).
We also let $V_0=\Mat_{\underline{e}_0,\underline{f}_0}(\Fil^{n-1}\cM)$ and $A_0=\Mat_{\underline{e}_0,\underline{f}_0}(\phi_{n-1})$ as well as $V_1=\Mat_{\underline{e}_1,\underline{f}_1}(\Fil^{n-1})\cM$ and $A_1=\Mat_{\underline{e}_1,\underline{f}_1}(\phi_{n-1})$.

It is obvious that there exist $R=(r_{i,j}u^{[k^{(0)}_j-k^{(0)}_i]_1})$ and $C=(c_{i,j}u^{[k^{(0)}_j-k^{(0)}_i]_1})$ in $\GL_n^{\square}(\barS)$ such that
\begin{equation*}
R\cdot V_0\cdot C= V_1\mbox{ and }\underline{e}_1=\underline{e}_0 R^{-1}
\end{equation*}
for $r_{i,j}$ and $c_{i,j}$ in $\barS_0$. From the first equation above, we immediately get the identities:
$$r_{n-1,n-1}^{(0)}\cdot c_{0,0}^{(0)}=1=r_{0,0}^{(0)}\cdot c_{n-1,n-1}^{(0)}\,\,\mbox{and}\,\,r_{i,i}^{(0)}\cdot c_{i,i}^{(0)}=1$$ for $0<i<n-1$, where $r_{i,j}^{(0)}\in\F$ (resp. $c_{i,j}^{(0)}\in\F$) is determined by $r_{i,j}^{(0)}\equiv r_{i,j}$ modulo $(u^e)$ (resp. $c_{i,j}^{(0)}\equiv c_{i,j}$ modulo $(u^e)$). By Lemma~\ref{lemm: base change matrix}, we see that $A_1=R\cdot A_0\cdot \phi(C)$.

Hence, if we let $A_1=\left(\alpha_{i,j}u^{[k^{(0)}_j-k^{(0)}_i]_1}\right)$ then
$$r_{i,i}^{(0)}\cdot \mu_i\cdot c_{i,i}^{(0)}=\alpha_{i,i}^{(0)}$$ for each $0< i< n-1$ since $R$ and $C$ are diagonal modulo $(u)$, so that we have $$\prod_{i=1}^{n-2}\mu_i=\prod_{i=1}^{n-2}\alpha_{i,i}^{(0)}$$ which completes its proof.
\end{proof}

Note that the matrix in the statement of Proposition~\ref{prop: Breuil modules for newest FL 3} gives rise to the elementary divisors of $\cM/\Fil^{n-1}\cM$.

\subsection{Filtration of strongly divisible modules}\label{subsec: filtration of strongly divisible modules}
In this section, we describe the filtration of the strongly divisible modules lifting the Breuil modules described in Proposition~\ref{prop: Breuil modules for newest FL 3}. Throughout this section, we keep the notation $r_i^{(0)}$ as in (\ref{Galois types of newest}) as well as~$k^{(0)}_i$.

We start to recall the following lemma, which is easy to prove but very useful.
\begin{lemm}\label{lemm: dimension of filtration of sdm}
Let $0<f\leq n$ be an integer, and let $\widehat{\cM}\in\OEModdd[n-1]$ be a strongly divisible module corresponding to a lattice in a potentially semi-stable representation $\rho:G_{\Qp}\rightarrow \GL_n(E)$ with Hodge--Tate weights $\{-(n-1), -(n-2),\cdots,0\}$ and Galois type of niveau $f$ such that $\Tst^{\Qp,n-1}(\widehat{\cM})\otimes_{\cO_E}\F\cong\rhobar_0$.

If we let $$X^{(i)}:= \left(\frac{\Fil^{n-1}\widehat{\cM}\cap \Fil^iS\cdot\widehat{\cM}}{\Fil^{n-1}S\cdot\widehat{\cM}}\right)\otimes_{\cO_E}E$$ for $i\in\{0,1,\cdots,n-1\}$, then for any character $\xi:\mathrm{Gal}(K/K_0)\rightarrow K^{\times}$ we have that the $\xi$-isotypical component $X^{(i)}_{\xi}$ of $X^{(i)}$ is a free $K_0\otimes E$-module of finite rank
$$\mathrm{rank}_{K_0\otimes_{\Q_p}E} X^{(i)}_{\xi}=\frac{n(n-1)}{2}-\frac{i(i+1)}{2}.$$
Moreover, multiplication by $u\in S$ induces an isomorphism $X^{(0)}_{\xi}\overset{\sim}\longrightarrow X^{(0)}_{\xi \widetilde{\omega}}$.
\end{lemm}

\begin{proof}
Since $\rho$ has Hodge--Tate weights $\{-(n-1),-(n-2),\cdots,0\}$, by the analogue with $E$-coefficients of \cite{Bre97}, Proposition A.4, we deduce that
$$\Fil^{n-1}\mathcal{D}=\Fil^{n-1}S_E\widehat{f}_{n-1}\oplus\Fil^{n-2}S_E\widehat{f}_{n-2}\oplus\cdots\oplus\Fil^1S_E\widehat{f}_{1}\oplus S_E\widehat{f}_{0}$$
for some $S_E$-basis $\widehat{f}_0,\cdots,\widehat{f}_{n-1}$ of $\mathcal{D}$, where $\mathcal{D}:=\widehat{\cM}[\frac{1}{p}]\cong S_E\otimes_{E}\Dst^{\Q_p,n-1}(V)$, so that we also have
$$\Fil^{n-1}\mathcal{D}\cap\Fil^{i}S_E\mathcal{D}=\Fil^{n-1}S_E\widehat{f}_{n-1}\oplus\Fil^{n-2}S_E\widehat{f}_{n-2}\oplus\cdots\oplus \Fil^iS_E\widehat{f}_{i}\oplus\cdots\oplus \Fil^iS_E\widehat{f}_{0}.$$
Since $\rho\cong\Tst^{\Qp,n-1}(\widehat{\cM})\otimes_{\cO_E}E$ is a $G_{\Q_p}$-representation, $\Fil^i(K\otimes_{K_0}\Dst^{\Q_p,n-1}(\rho))\cong K\otimes_{\Q_p}\Fil^i\Ddr(\rho\otimes\varepsilon^{1-n})$, so that $X^{(i)}\cong \frac{\Fil^{n-1}\mathcal{D}\cap\Fil^iS_E\mathcal{D}}{\Fil^{n-1}S_E\mathcal{D}}$ is a free $K_0\otimes_{\Q_p} E$-module. Since $\frac{S_E}{\Fil^{n-1}S_E}\cong\bigoplus_{i=0}^{n-2}\bigoplus_{j=0}^{e-1}(K_0\otimes_{\Q_p}E)u^jE(u)^i$, we have
$\mathrm{rank}_{K_0\otimes_{\Q_p}E} X^{(i)}=\left[\frac{n(n-1)}{2}-\frac{i(i+1)}{2}\right]e$.
We note that $\Gal(K/K_0)$ acts semisimply and that multiplication by $u$ gives rise to an $K_0\otimes_{\Q_p}E$-linear isomorphism on $S_E/\Fil^pS_E$ which cyclically permutes the isotypical components, which completes the proof.
\end{proof}

Note that Lemma~\ref{lemm: dimension of filtration of sdm} immediately implies that
\begin{equation}\label{equation dimension filtration sdm}
\mathrm{rank}_{K_0\otimes_{\Q_p} E} X^{(i)}_{\xi}-\mathrm{rank}_{K_0\otimes_{\Q_p} E} X^{(i+1)}_{\xi}=i+1.
\end{equation}
We will use this fact frequently to prove the main result, Proposition~\ref{prop: shape of filtration of sdm}, in this subsection.

To describe the filtration of strongly divisible modules, we need to analyze the $\Fil^{n-1}\cM$ of the Breuil modules $\cM$ we consider.
\begin{lemm}\label{lemm: elementary divisors}
Keep the notation and assumptions of Lemma~\ref{lemm: Breuil modules for newest FL, classification}.

\begin{enumerate}
\item If $u^a$ is an elementary divisor of $\cM/\Fil^{n-1}\cM$ then $$e-(k^{(0)}_{n-1}-k^{(0)}_0)\leq a\leq (n-2)e+(k^{(0)}_{n-1}-k^{(0)}_0).$$ Moreover, $\mathrm{FL}_n^{n-1,0}(\rhobar_0)\not=\infty$ (resp. $\mathrm{FL}_n^{n-1,0}(\rhobar_0)\not=0$) if and only if $u^{e-(k^{(0)}_{n-1}-k^{(0)}_0)}$ (resp. $u^{(n-2)e+(k^{(0)}_{n-1}-k^{(0)}_0)}$) is an elementary divisor of $\cM/\Fil^{n-1}\cM$.
\item If we further assume that $\rhobar_0$ is Fontaine--Laffaille generic, then $$\{u^{(n-2)e+(k^{(0)}_{n-1}-k^{(0)}_0)},\,u^{(n-2)e},\,u^{(n-3)e},\,\cdots,\,u^e,\,u^{e-(k^{(0)}_{n-1}-k^{(0)}_0)}\}$$ are the elementary divisors of $\cM/\Fil^{n-1}\cM$.
\end{enumerate}
\end{lemm}

\begin{proof}
The first part of (i) is obvious since one can obtain the Smith normal form of $\Mat_{\underline{e},\underline{f}}\Fil^{n-1}\cM$ by elementary row and column operations. By Proposition~\ref{prop: Breuil modules for newest FL}, we know that $\mathrm{FL}_n^{n-1,0}(\rhobar_0)\not=\infty$ if and only if $\beta_{n-1,0}\not=0$. Since $u^{e-(k^{(0)}_{n-1}-k^{(0)}_0)}$ has the minimal degree among the entries of $\Mat_{\underline{e},\underline{f}}\Fil^{n-1}\cM$, we conclude the equivalence statement for $\mathrm{FL}_n^{n-1,0}(\rhobar_0)\not=\infty$ holds. The last part of (i) is immediate from the other equivalence statement, $\mathrm{FL}_n^{n-1,0}(\rhobar_0)\not=\infty$ if and only if $\beta_{n-1,0}\not=0$, by considering $\cM^{*}$ and using Lemma~\ref{lemm: FL parameter with dual rep},~(vi).

Part (ii) is obvious from Proposition~\ref{prop: Breuil modules for newest FL 3}.
\end{proof}

\begin{prop}\label{prop: shape of filtration of sdm}
Assume that $\rhobar_0$ is Fontaine--Laffaille generic and keep the notation $r_{i}^{(0)}$ as in (\ref{Galois types of newest}) as well as $k_i^{(0)}$.  Let $\widehat{\cM}\in\OEModdd[n-1]$ be a strongly divisible module corresponding to a lattice in a potentially semi-stable representation $\rho:G_{\Qp}\rightarrow \GL_n(E)$ with Galois type $\bigoplus_{i=0}^{n-1}\widetilde{\omega}^{k^{(0)}_i}$ and Hodge--Tate weights $\{-(n-1), -(n-2),\cdots,0\}$ such that $\Tst^{\Qp,n-1}(\widehat{\cM})\otimes_{\cO_E}\F\cong\rhobar_0$.

Then there exists a framed basis $(\widehat{e}_{n-1},\widehat{e}_{n-2},\cdots, \widehat{e}_0)$ for $\widehat{\cM}$
and a framed system of generators $(\widehat{f}_{n-1},\widehat{f}_{n-2},\cdots, \widehat{f}_0)$ for $\Fil^{n-1}\widehat{\cM}$ modulo $\Fil^{n-1}S\cdot\widehat{\cM}$ such that $\Mat_{\underline{\widehat{e}},\underline{\widehat{f}}}\Fil^{n-1}\widehat{\cM}$ is described as follows:
$$
\begin{pmatrix}
-\frac{p^{n-1}}{\alpha} & 0 & 0 &\cdots& 0 & u^{e-(k^{(0)}_{n-1}-k^{(0)}_0)}\\
0 & E(u)^{n-2}& 0 & \cdots& 0 & 0 \\
0 & 0 & E(u)^{n-3}&\cdots& 0  & 0 \\
\vdots&\vdots&\vdots&\ddots&\vdots&\vdots\\
0 & 0&0 &\cdots & E(u) & 0\\
u^{k^{(0)}_{n-1}-k^{(0)}_{0}}\sum_{i=0}^{n-2}p^{n-2-i}E(u)^i  & 0&0 &\cdots & 0 & \alpha
\end{pmatrix}
$$
where $\alpha\in\cO_E$ with $0<v_p(\alpha)<n-1$.
\end{prop}

\begin{proof}
Note that we write the elements of $\widehat{\cM}$ in terms of coordinates with respect to a framed basis $\widehat{\underline{e}}:=(\widehat{e}_{n-1},\widehat{e}_{n-2},\cdots,\widehat{e}_{0})$. We let $\cM:= \widehat{\cM}\otimes_S\barS$, which is a Breuil module of weight $n-1$ and of type $\bigoplus_{i=0}^{n-1}\omega^{k^{(0)}_i}$ by Proposition~\ref{prop: intertial type of a Breuil module}. Note also that $\cM$ can be described as in Proposition~\ref{prop: Breuil modules for newest FL 3}, and we assume that $\cM$ has such a framed basis for $\cM$ and such a framed system of generators for $\Fil^{n-1}\cM$.

Let $$\widehat{f}_{0}=\left(\begin{array}{c} u^{e-(k^{(0)}_{n-1}-k^{(0)}_0)} \sum_{k=0}^{n-2}x_{n-1,k}E(u)^k \\u^{e-(k^{(0)}_{n-2}-k^{(0)}_0)} \sum_{k=0}^{n-2}x_{n-2,k}E(u)^k  \\ \vdots \\u^{e-(k^{(0)}_1-k^{(0)}_0)} \sum_{k=0}^{n-2}x_{1,k}E(u)^k  \\ \sum_{k=0}^{n-2}x_{0,k}E(u)^k   \\ \end{array}\right)\in \left(\frac{\Fil^{n-1}\widehat{M}}{\Fil^{n-1}S\widehat{M}}\right)_{\widetilde{\omega}^{k^{(0)}_0}},$$ where $x_{i,j}\in\cO_E$. The vector $\widehat{f}_{0}$ can be written as follows:
$$\widehat{f}_{0}=u^{e-(k^{(0)}_{n-1}-k^{(0)}_0)}\underbrace{\left(\begin{array}{c} \sum_{k=0}^{n-2}x_{n-1,k}E(u)^k \\u^{(k^{(0)}_{n-1}-k^{(0)}_{n-2})} \sum_{k=0}^{n-2}x_{n-2,k}E(u)^k  \\ \vdots \\u^{(k^{(0)}_{n-1}-k^{(0)}_1)} \sum_{k=0}^{n-2}x_{1,k}E(u)^k  \\ u^{(k^{(0)}_{n-1}-k^{(0)}_0)}\sum_{k=1}^{n-2}x_{0,k}[E(u)^{k}-p^{k}]/u^e   \\ \end{array}\right)}_{=: \widehat{e}'_{n-1}}+
\left(\begin{array}{c} 0 \\ 0  \\ \vdots \\ 0 \\ x_{0,0}+\sum_{k=1}^{n-2}x_{0,k}p^{k}   \\ \end{array}\right).$$

By (ii) of Lemma~\ref{lemm: elementary divisors}, we know that $u^{e-(k^{(0)}_{n-1}-k^{(0)}_0)}$ is an elementary divisor of $\cM/\Fil^{n-1}\cM$ and all other elementary divisors have bigger powers, so that we may assume $v_p(x_{n-1,0})=0$. Since $\Fil^{n-1}\cM\subseteq u^{e-(k^{(0)}_{n-1}-k^{(0)}_0)}\cM$, we must have $v_p(x_{0,0})>0$. So $\widehat{\underline{e}}_1:= (\widehat{e}'_{n-1},\widehat{e}_{n-2},\cdots,\widehat{e}_0)$ is a framed basis for $\widehat{\cM}$ by Nakayama lemma and we have the following coordinates of $\widehat{f}_{0}$ with respect to $\widehat{\underline{e}}_1$:
$$\widehat{f}_{0}=\left(\begin{array}{c} u^{e-(k^{(0)}_{n-1}-k^{(0)}_0)} \\ 0 \\ \vdots \\ 0  \\ \alpha   \\ \end{array}\right)\in \left(\frac{\Fil^{n-1}\widehat{M}}{\Fil^{n-1}S\widehat{M}}\right)_{\widetilde{\omega}^{k^{(0)}_0}}$$ for $\alpha\in \cO_E$ with $v_p(\alpha)>0$.

Since $u^{k^{(0)}_1-k^{(0)}_0}\widehat{f}_0\in \left(\frac{\Fil^{n-1}\widehat{M}}{\Fil^{n-1}S\cdot\widehat{M}}\right)_{\widetilde{\omega}^{k^{(0)}_1}}$, there exists $\widehat{f}_1$ such that
$$\widehat{f}_{1}=\left(\begin{array}{c} 0 \\u^{e-(k^{(0)}_{n-2}-k^{(0)}_1)} \sum_{k=0}^{n-2}y_{n-2,k}E(u)^k  \\ \vdots \\ \sum_{k=0}^{n-2}y_{1,k}E(u)^k  \\ u^{k^{(0)}_1-k^{(0)}_0}\sum_{k=0}^{n-2}y_{0,k}E(u)^k   \\ \end{array}\right)\in \left(\frac{\Fil^{n-1}\widehat{M}}{\Fil^{n-1}S\widehat{M}}\right)_{\widetilde{\omega}^{k^{(0)}_1}},$$ where $y_{i,j}\in\cO_E$. By Lemma~\ref{lemm: dimension of filtration of sdm}, we have $y_{i,0}=0$ for all $i$: otherwise, both $u^{k^{(0)}_1-k^{(0)}_0}\widehat{f}_0$ and $\widehat{f}_1$ belong to $X_{\widetilde{\omega}^{k^{(0)}_1}}^{(0)}-X_{\widetilde{\omega}^{k^{(0)}_1}}^{(1)}$ which violates (\ref{equation dimension filtration sdm}). Since $u^e$ is an elementary divisor of $\cM/\Fil^{n-1}\cM$ by (ii) of Lemma~\ref{lemm: elementary divisors}, we may also assume $y_{1,1}=1$. Hence, by the obvious change of basis we get $\widehat{f}_1$ as follows:
$$\widehat{f}_{1}=E(u)\left(\begin{array}{c} 0 \\ \vdots \\ 0 \\ 1  \\ 0   \\ \end{array}\right)\in \left(\frac{\Fil^{n-1}\widehat{M}}{\Fil^{n-1}S\widehat{M}}\right)_{\widetilde{\omega}^{k^{(0)}_1}}.
$$
By the same arguments, we get $\widehat{f}_i\in \left(\frac{\Fil^{n-1}\widehat{M}}{\Fil^{n-1}S\widehat{M}}\right)_{\widetilde{\omega}^{k^{(0)}_i}}$ for $i=1,2,\cdots,n-2$ as in the statement.

Note that the elements in the set
\begin{multline*}
\{u^{k^{(0)}_{n-1}-k^{(0)}_0}\widehat{f}_0,E(u)u^{k^{(0)}_{n-1}-k^{(0)}_0}\widehat{f}_0,\cdots,E(u)^{n-2}u^{k^{(0)}_{n-1}-k^{(0)}_0}\widehat{f}_0\}\\
\cup \{u^{k^{(0)}_{n-1}-k^{(0)}_1}\widehat{f}_1,E(u)u^{k^{(0)}_{n-1}-k^{(0)}_1}\widehat{f}_1,\cdots,E(u)^{n-3}u^{k^{(0)}_{n-1}-k^{(0)}_1}\widehat{f}_1\}\\
\cup \cdots
\cup \{u^{k^{(0)}_{n-1}-k^{(0)}_{n-2}}\widehat{f}_{n-2}\}
\end{multline*}
are linearly independent in $X^{(0)}_{\widetilde{\omega}^{k^{(0)}_{n-1}}}$ over $E$, so that the set forms a basis for $X^{(0)}_{\widetilde{\omega}^{k^{(0)}_{n-1}}}$ by Lemma~\ref{lemm: dimension of filtration of sdm}. Hence, $\widehat{f}_{n-1}$ is a linear combination of those elements over $E$. We have
$$u^{k^{(0)}_{n-1}-k^{(0)}_{0}}\left(\sum_{i=0}^{n-2}p^{n-2-i}E(u)^i\right)\widehat{f}_0= \left(\begin{array}{c} -p^{n-1} \\ 0 \\  \vdots\\ 0 \\ \alpha u^{k^{(0)}_{n-1}-k^{(0)}_{0}}\sum_{i=0}^{n-2}p^{n-2-i}E(u)^i \\ \end{array}\right).$$
Hence, we may let $$\widehat{f}_{n-1}:=\frac{1}{\alpha}u^{k^{(0)}_{n-1}-k^{(0)}_{0}}\left(\sum_{i=0}^{n-2}p^{n-2-i}E(u)^i\right)\widehat{f}_0\in \left(\frac{\Fil^{n-1}\widehat{M}}{\Fil^{n-1}S\widehat{M}}\right)_{\widetilde{\omega}^{k^{(0)}_{n-1}}}$$ since $u^{(n-2)e+(k^{(0)}_{n-1}-k^{(0)}_0)}$ is an elementary divisor for $\cM/\Fil^{n-1}\cM$ by (ii) of Lemma~\ref{lemm: elementary divisors}. Moreover, $v_p\left(\frac{p^{n-1}}{\alpha}\right)>0$ since $\Fil^{n-1}\cM\subseteq u^{e-(k^{(0)}_{n-1}-k^{(0)}_0)}\cM\subseteq u\cM$ by Proposition~\ref{prop: Breuil modules for newest FL 3}.

It is obvious that the $\widehat{f}_i$ mod $(\varpi_E,\Fil^pS)$ generate $\cM/\Fil^{n-1}\cM$ for $\cM$ written as in Proposition~\ref{prop: Breuil modules for newest FL}. By Nakayama Lemma, we conclude that the $\widehat{f}_i$ generate $\widehat{\cM}/\Fil^{n-1}\widehat{\cM}$, which completes the proof.
\end{proof}

\begin{coro}\label{coro: valuation of Frobenius eigenvalues}
Keep the notation and assumptions of Proposition~\ref{prop: shape of filtration of sdm}, and let $$(\lambda_{n-1},\lambda_{n-2},\cdots,\lambda_0)\in (\cO_E)^{n}$$ be the Frobenius eigenvalues on the $(\widetilde{\omega}^{k^{(0)}_{n-1}},\widetilde{\omega}^{k^{(0)}_{n-2}},\cdots,\widetilde{\omega}^{k^{(0)}_0})$-isotypic component of $\Dst^{\Qp,n-1}(\rho)$. Then
\begin{equation*}
v_p(\lambda_i)=
\left\{
  \begin{array}{ll}
   v_p(\alpha)  & \hbox{ if $i=n-1$} \\
   (n-1)-i  & \hbox{ if $n-1>i>0$} \\
   (n-1)-v_p(\alpha)  & \hbox{ if $i=0$.}
  \end{array}
\right.
\end{equation*}
\end{coro}

\begin{proof}
The proof goes parallel to the proof of \cite{HLM}, Corollary~2.4.11. 
\end{proof}

\subsection{Reducibility of certain lifts}\label{subsec: reducibility of certain lifts}
In this section, we let $1\leq f\leq n$ and $e=p^f-1$, and we prove that every potentially semi-stable lift of $\rhobar_0$ with Hodge--Tate weights $\{-(n-1),-(n-2),\cdots,0\}$ and certain prescribed Galois types $\bigoplus_{i=0}^{n-1}\widetilde{\omega}_f^{k_i}$ is reducible. We emphasize that we only assume that $\rhobar_0$ is generic (cf. Definition~\ref{definition: genericity condition}) for the results in this section.

\begin{prop}\label{prop: reduciblility of certain lifts}
Assume that $\rhobar_0$ is generic, and let $(k_{n-1},k_{n-2},\cdots,k_0)$ be an $n$-tuple of integers. Assume further that $k_0\equiv (p^{f-1}+p^{f-2}+\cdots+p+1)c_0$ modulo $(e)$ and that $k_i$ are pairwise distinct modulo $(e)$.

Then every potentially semi-stable lift of $\rhobar_0$ with Hodge--Tate weights $\{-(n-1),-(n-2),\cdots,0\}$ and Galois types $\bigoplus_{i=0}^{n-1}\widetilde{\omega}_f^{k_i}$ is an extension of a $1$-dimensional potentially semi-stable lift of $\rhobar_{0,0}$ with Hodge--Tate weight $0$ and Galois type $\widetilde{\omega}_f^{k_0}$ by an $(n-1)$-dimensional potentially semi-stable lift of $\rhobar_{n-1,1}$ with Hodge--Tate weights $\{-(n-1),-(n-2),\cdots,1\}$ and Galois types $\bigoplus_{i=1}^{n-1}\widetilde{\omega}_f^{k_i}$.
\end{prop}

Note that if $f=1$ then the assumption that $\rhobar_0$ is generic implies that $k_i$ are pairwise distinct modulo $(e)$ by Lemma~\ref{lemm: breuil modules niveau 1}. In fact, we believe that this is true for any $1\leq f\leq n$, but this requires extra works as we did in Lemma~\ref{lemm: breuil modules niveau 1}. Since we will need the results in this section only when $f=1$, we will add the assumption that $k_i$ are pairwise distinct modulo~$(e)$ in the proposition.

\begin{proof}
Let $\widehat{\cM}\in\OEModdd[n-1]$ be a strongly divisible module corresponding to a Galois stable lattice in a potentially semi-stable representation $\rho:G_{\Qp}\rightarrow \GL_n(E)$ with Galois type $\bigoplus_{i=0}^{n-1}\widetilde{\omega}_f^{k_i}$ and Hodge--Tate weights $\{-(n-1), -(n-2),\cdots,0\}$ such that $\Tst^{\Qp,n-1}(\widehat{\cM})\otimes_{\cO_E}\F\cong\rhobar_0$. We also let $\cM$ be the Breuil module corresponding to the mod $p$ reduction of $\widehat{\cM}$. $\widehat{\cM}$ (resp. $\cM$) is of inertial type $\bigoplus_{i=0}^{n-1}\widetilde{\omega}_f^{k_i}$ (resp. $\bigoplus_{i=0}^{n-1}\omega_f^{k_i}$) by Proposition~\ref{prop: intertial type of a Breuil module}.

We let $\underline{f}=(f_{n-1},f_{n-2},\cdots,f_0)$ (resp. $\underline{\widehat{f}}=(\widehat{f}_{n-1},\widehat{f}_{n-2},\cdots,\widehat{f}_0)$) be a framed system of generators for $\Fil^{n-1}\cM$ (resp. for $\Fil^{n-1}\widehat{\cM}$). We also let $\underline{e}=(e_{n-1},e_{n-2},\cdots,e_0)$ (resp. $\underline{\widehat{e}}=(\widehat{e}_{n-1},\widehat{e}_{n-2},\cdots,\widehat{e}_0)$) be a framed basis for $\cM$ (resp. for $\widehat{\cM}$). If $x=a_{n-1}e_{n-1}+\cdots+a_0e_0\in\cM$, we will write $[x]_{e_i}$ for $a_i$ for $i\in\{0,1,\cdots,n-1\}$. We define $[x]_{\widehat{e}_i}$ for $x\in\widehat{\cM}$ in the obvious similar way. We may assume that $\Mat_{\underline{e},\underline{f}}(\Fil^{n-1}\cM)$, $\Mat_{\underline{e},\underline{f}}(\phi_{n-1})$, and $\Mat_{\underline{e}}(N)$ are written as in (\ref{filtration of Breuil Module: niveau f}), (\ref{Frobenius of Breuil Module: niveau f}), and (\ref{Monodoromy of Breuil Module: niveau f}) respectively, and we do so.

By the equation~(\ref{equation in type elimination}), we deduce $r_0\equiv 0$ modulo $(e)$ from our assumption on $k_0$. Recall that $p>n^2+2(n-3)$ by the generic condition. Since $0\leq r_0\leq (n-1)(p^{f}-1)/(p-1)$ by (ii) of Lemma~\ref{Lemma: classification of rank-one Breuil modules}, we conclude that $r_0=0$. Thus, we may let $f_0$ satisfy that $[f_0]_{e_i}=0$ if $0<i\leq n-1$ and $[f_0]_{e_0}=1$, so that we can also let
$$\widehat{f}_{0}=\left(\begin{array}{c} 0 \\ \vdots \\ 0  \\ 1    \end{array}\right).$$
Hence, we can also assume that $[\widehat{f}_j]_{\widehat{e}_0}=0$ for $0<j\leq n-1$. We let $V_0=\Mat_{\underline{\widehat{e}},\underline{\widehat{f}}}(\Fil^{n-1}\widehat{\cM})\in\mathrm{M}_n^{\square,\prime}(S_{\cO_E})$ and $A_0=\Mat_{\underline{\widehat{e}},\underline{\widehat{f}}}(\phi_{n-1})\in\GL_n^{\square}(S_{\cO_E})$.

We construct a sequence of framed bases $\{\underline{\widehat{e}}^{(m)}\}$ for $\widehat{\cM}$ by change of basis, satisfying that $$\Mat_{\underline{\widehat{e}}^{(m)},\underline{\widehat{f}}^{(m)}}(\Fil^{n-1}\widehat{\cM})\in \mathrm{M}_n^{\square,\prime}(S_{\cO_E})\,\,\mbox{ and }\,\,\Mat_{\underline{\widehat{e}}^{(m)},\underline{\widehat{f}}^{(m)}}(\phi_{n-1})\in\GL_n^{\square}(S_{\cO_E})$$ converge to certain desired forms as $m$ goes to $\infty$.  We let $V^{(m)}\in \mathrm{M}_n^{\square,\prime}(S_{\cO_E})$ and $A^{(m)}\in\GL_n^{\square}(S_{\cO_E})$ for a non-negative integer $m$.
We may write $$(x_{n-1}^{(m+1)}u^{[k_{n-1}-k_0]_f},x_{n-2}^{(m+1)}u^{[k_{n-2}-k_0]_f},\cdots, x_{m+1}^{(1)}u^{[k_{m+1}-k_0]_f},x_{0}^{(m+1)})$$ for the last row of $(A^{(m)})^{-1}$, where $x_{0}^{(m+1)}\in (S_{\cO_E}^{\times})_{0}$ and $x_{j}^{(m+1)}\in (S_{\cO_E})_0$ for $0<j\leq n-1$. We define an $n\times n$-matrix $R^{(m+1)}$ as follows:
$$R^{(m+1)}=
\begin{pmatrix}
1 & 0 &\cdots& 0 & 0 \\
0 & 1 &\cdots& 0 & 0 \\
\vdots&\vdots&\ddots&\vdots&\vdots\\
 0&0 &\cdots & 1 & 0 \\
 \frac{x_{n-1}^{(m+1)}}{x_0^{(m+1)}}u^{[k_{n-1}-k_0]_f}&\frac{x_{n-2}^{(m+1)}}{x_0^{(m+1)}}u^{[k_{n-2}-k_0]_f} &\cdots &\frac{x_{1}^{(m+1)}}{x_0^{(m+1)}}u^{[k_{1}-k_0]_f}& 1
\end{pmatrix}.
$$
We also define
$$C^{(m+1)}=
\begin{pmatrix}
1 & 0 &\cdots& 0 & 0 \\
0 & 1 &\cdots& 0 & 0 \\
\vdots&\vdots&\ddots&\vdots&\vdots\\
 0&0 &\cdots & 1 & 0 \\
 y_{n-1}^{(m+1)} u^{[p^{-1}(k_{n-1}-k_0)]_f}& y_{n-2}^{(m+1)} u^{[p^{-1}(k_{n-2}-k_0)]_f} &\cdots &y_{1}^{(m+1)}u^{[p^{-1}(k_{1}-k_0)]_f}& 1
\end{pmatrix}
$$ by the equation $$R^{(m+1)}\cdot V^{(m)}\cdot C^{(m+1)}=V^{(m)}$$
where $y_j^{(m+1)}\in (S_{\cO_E})_0$ for $0<j\leq n-1$. Note that the existence of such a matrix $C^{(m+1)}$ is obvious, since $p^{-1}k_0\equiv k_0$ modulo $(e)$ by our assumption on $k_0$ immediately implies
$[p^{-1}(k_j-k_0)]_f\leq [k_s-k_0]_f+[p^{-1}k_j-k_s]_f$. We also note that $R^{(m+1)}\in\GL_n^{\square}(S_{\cO_E})$ and $C^{(m+1)}\in\GL_n^{\square,\prime}(S_{\cO_E})$.

Let $V^{(m+1)}=V^{(m)}$ for all $m\geq 0$. Assume that $V^{(m)}=\Mat_{\underline{\widehat{e}}^{(m)},\underline{\widehat{f}}^{(m)}}(\Fil^{n-1}\widehat{\cM})$ and $A^{(m)}=\Mat_{\underline{\widehat{e}}^{(m)},\underline{\widehat{f}}^{(m)}}(\phi_{n-1})$, with respect to a framed basis $\underline{\widehat{e}}^{(m)}$ and a framed system of generators $\underline{\widehat{f}}^{(m)}$. If we let $\underline{\widehat{e}}^{(m+1)}=\underline{\widehat{e}}^{(m)}\cdot (R^{(m+1)})^{-1}$, then
$$
  \begin{array}{ll}
    \phi_{n-1}(\underline{\widehat{e}}^{(m+1)}V^{(m+1)}) & = \phi_{n-1}(\underline{\widehat{e}}^{(m)}(R^{(m+1)})^{-1}V^{(m+1)}) \\
     & = \phi_{n-1}(\underline{\widehat{e}}^{(m)}V^{(m)}C^{(m+1)})\\
     & = \underline{\widehat{e}}^{(m)}A^{(m)}\phi(C^{(m+1)})\\
     & = \underline{\widehat{e}}^{(m+1)}R^{(m+1)}\cdot A^{(m)}\cdot \phi(C^{(m+1)}).
  \end{array}
$$
Hence, we get $$V^{(m+1)}=\Mat_{\underline{\widehat{e}}^{(m+1)},\underline{\widehat{f}}^{(m+1)}}(\Fil^{n-1}\widehat{\cM})\,\,\mbox{ and }\,\, R^{(m+1)}\cdot A^{(m)}\cdot \phi(C^{(m+1)})=\Mat_{\underline{\widehat{e}}^{(m+1)},\underline{\widehat{f}}^{(m+1)}}(\phi_{n-1}),$$ where $\underline{\widehat{f}}^{(m+1)}:=\underline{\widehat{e}}^{(m+1)}V^{(m+1)}$.

We compute the matrix product $A^{(m+1)}:= R^{(m+1)}\cdot A^{(m)}\cdot \phi(C^{(m+1)})$ as it follows. If we let $A^{(m)}=\left(\alpha_{i,j}^{(m)}u^{[k_j-k_i]_f}\right)_{0\leq i,j\leq n-1}$ for $\alpha_{i,j}^{(m)}\in (S_{\cO_E})_0$ if $i\not=j$ and $\alpha_{i,i}^{(m)}\in (S_{\cO_E}^{\times})_0$, then
\begin{equation}\label{equation A m+1}
A^{(m+1)}=\left(\alpha_{i,j}^{(m+1)}u^{[k_j-k_i]_f}\right)_{0\leq i,j\leq n-1}\in \GL_n^{\square}(S_{\cO_E})
\end{equation}
where $\alpha_{i,j}^{(m+1)}u^{[k_j-k_i]_f}$ is described as follows:
\begin{equation*}
\left\{
  \begin{array}{ll}
    \alpha_{i,j}^{(m)}u^{[k_j-k_i]_f}+\alpha_{i,0}^{(m)}u^{[k_0-k_i]_f} \phi(y_j^{(m+1)})u^{p[p^{-1}(k_j-k_0)]_f} & \hbox{if $i>0$ and $j>0$;} \\
    \alpha_{i,0}^{(m)}u^{[k_0-k_i]_f} & \hbox{if $i>0$ and $j=0$;}\\
    \frac{1}{x^{(m+1)}_0}\phi(y_j^{(m+1)})u^{p[p^{-1}(k_j-k_0)]_f} & \hbox{if $i=0$ and $j>0$;}\\
    \frac{1}{x^{(m+1)}_0} & \hbox{if $i=0$ and $j=0$.}\\
  \end{array}
\right.
\end{equation*}

Let $V^{(0)}=V_0$ and $A^{(0)}=A_0$. We apply the algorithm above to $V^{(0)}$ and $A^{(0)}$.  By the algorithm above, we have two matrices $V^{(m)}$ and $A^{(m)}$ for each $m\geq 0$. We claim that
\begin{equation*}
\left\{
  \begin{array}{ll}
    \alpha_{i,j}^{(m+1)}-\alpha_{i,j}^{(m)}\in u^{(1+p+\cdots+p^{m})e}S_{\cO_E} & \hbox{if $i>0$ and $j>0$;} \\
    \alpha_{i,j}^{(m+1)}=\alpha_{i,j}^{(m)} & \hbox{if $i>0$ and $j=0$;}\\
    \alpha_{i,j}^{(m+1)}\in u^{(1+p+\cdots+p^{m})e}S_{\cO_E} & \hbox{if $i=0$ and $j>0$;}\\
    \alpha_{i,j}^{(m+1)}-\alpha_{i,j}^{(m)}\in u^{(1+p+\cdots+p^{m-1})e}S_{\cO_E} & \hbox{if $i=0$ and $j=0$.}\\
  \end{array}
\right.
\end{equation*}
It is obvious that the case $i>0$ and $j=0$ from the computation~(\ref{equation A m+1}). For the case $i=0$ and $j>0$ we induct on $m$. Note that $p[p^{-1}(k_j-k_0)]_f-[k_j-k_0]_f=p([p^{-1}k_j]_f-k_0)-(k_j-k_0)\geq e$ if $j>0$. From the computation~(\ref{equation A m+1}) again, it is obvious that it is true for $m=0$. Assume that it holds for $m$. This implies that $x_j^{(m+1)}\in u^{(1+p+\cdots+p^{m-1})e}S_{\cO_E}$ for $0<j\leq n-1$ and so $y_j^{(m+1)}\in u^{(1+p+\cdots+p^{m-1})e}S_{\cO_E}$. Since $\phi(y_j^{(m+1)})u^{p[p^{-1}(k_j-k_0)]_f-[k_j-k_0]_f})\in u^{(1+p+\cdots+p^{m})e}S_{\cO_E}$, by the computation~(\ref{equation A m+1}) we conclude that the case $i=0$ and $j>0$ holds. The case $i>0$ and $j>0$ follows easily from the case $i=0$ and $j>0$, since $[p^{-1}(k_j-k_0)]_f+[k_0-k_i]_f-[k_j-k_i]_f=p([p^{-1}k_j]_f-k_0)+e+k_0-k_i-[k_j-k_i]_f\geq p[p^{-1}k_j]_f-k_j-(p-1)k_0\geq e$. Finally, we check the case $i=0$ and $j=0$. We also induct on $m$ for this case. It is obvious that it holds for $m=0$. Note that $R^{(m+1)}\equiv I_{n}$ modulo $u^{(1+p+\cdots+p^{m-1})e}S_{\cO_E}$. Since $A^{(m+1)}= R^{(m+1)}\cdot A^{(m)}\cdot \phi(C^{(m+1)})$, we conclude that the case $i=0$ and $j=0$ holds.

The previous claim says the limit of $A^{(m)}$ exists (entrywise), say $A^{(\infty)}$. By definition, we have $V^{(\infty)}=V^{(m)}$ for all $m\geq0$. In other words, there exist a framed basis $\underline{\widehat{e}}^{(\infty)}$ for $\widehat{\cM}$ and a framed system of generators $\underline{\widehat{f}}^{(\infty)}$ for $\Fil^{n-1}\widehat{\cM}$ such that $$\Mat_{\underline{\widehat{e}}^{(\infty)},\underline{\widehat{f}}^{(\infty)}}(\Fil^{n-1}\widehat{\cM})=V^{(\infty)} \in\mathrm{M}_n^{\square,\prime}(S_{\cO_E})$$ and $$\Mat_{\underline{\widehat{e}}^{(\infty)},\underline{\widehat{f}}^{(\infty)}}(\phi_{n-1})=A^{(\infty)}\in\GL_n^{\square}(S_{\cO_E}).$$
Note that $(V^{(\infty)})_{i,j}=0$ if either $i=0$ and $j>0$ or $i>0$ and $j=0$, and that $(A^{(\infty)})_{i,j}=0$ if $i=0$ and $j>0$.

Since $\underline{\widehat{e}}^{(\infty)}$ is a framed basis for $\widehat{\cM}$, we may write $$\Mat_{\underline{\widehat{e}}^{(\infty)}}(N)=\left(\gamma_{i,j}u^{[k_j-k_i]_f}\right)_{0\leq i,j\leq n-1} \in\mathrm{M}_n^{\square}(S_{\cO_E}) $$ for the matrix of the monodromy operator of $\widehat{\cM}$ where $\gamma_{i,j}\in (S_{\cO_E})_0$, and let $$A^{(\infty)}=\left(\alpha_{i,j}^{(\infty)}u^{[k_j-k_i]_f}\right)_{0\leq i,j\leq n-1}\in\GL_n^{\square}(S_{\cO_E}).$$ We claim that $\gamma_{0,j}=0$ for $n-1\geq j>0$. 
Recall that $\alpha_{0,j}^{(\infty)}=0$ for $j>0$, and write $\underline{\widehat{f}}^{(\infty)}=(\widehat{f}_{n-1}^{(\infty)}, \widehat{f}_{n-2}^{(\infty)},\cdots,\widehat{f}_{0}^{(\infty)})$ and $\underline{\widehat{e}}^{(\infty)}=(\widehat{e}_{n-1}^{(\infty)}, \widehat{e}_{n-2}^{(\infty)},\cdots,\widehat{e}_{0}^{(\infty)})$. We also write $$\widehat{f}_{j}^{(\infty)}=\sum_{i=1}^{n-1}\beta_{i,j}^{(\infty)}u^{[p^{-1}k_j-k_i]}\widehat{e}_i^{(\infty)}$$ where $\beta_{i,j}^{(\infty)}\in(S_{\cO_E})_{0}$, for each $0<j\leq n-1$. From the equation $$[cN\phi_{n-1}(\widehat{f}_{j}^{(\infty)})]_{\widehat{e}_{0}^{(\infty)}}= [\phi_{n-1}(E(u)N(\widehat{f}_{j}^{(\infty)}))]_{\widehat{e}_{0}^{(\infty)}}$$ for $n-1\geq j>0$,
we have the identity
\begin{equation}\label{equation identity, reducibility}
\sum_{i=1}^{n-1}\alpha^{(\infty)}_{i,j}u^{[k_j-k_i]_f+[k_i-k_0]_f}\gamma_{0,i}= p\sum_{i=1}^{n-1}\beta_{i,j}^{(\infty)}u^{p[p^{-1}k_j-k_i]_f+p[k_i-k_0]_f}\phi(\gamma_{0,i})\alpha^{(\infty)}_{0,0}
\end{equation}
for each $n-1\geq j>0$. Choose an integer $s$ such that $\ord_u(\gamma_{0,s}u^{[k_s-k_0]_f})\leq \ord_u(\gamma_{0,i}u^{[k_i-k_0]_f})$ for all $n-1\geq i>0$, and consider the identity~(\ref{equation identity, reducibility}) for $j=s$. Then the identity $(\ref{equation identity, reducibility})$ induces
$$\alpha^{(\infty)}_{s,s}u^{[k_s-k_0]_f}\gamma_{0,s}\equiv 0$$
modulo $(u^{\ord_u(\gamma_{0,s})+[k_s-k_0]_f+1})$. Note that $\alpha^{(\infty)}_{s,s}\in S_{\cO_E}^{\times}$, so that we get $\gamma_{0,s}=0$. Recursively, we conclude that $\gamma_{0,j}=0$ for all $0<j\leq n-1$.

Finally, it is now easy to check that $(\widehat{e}_{n-1}^{(\infty)}, \widehat{e}_{n-2}^{(\infty)},\cdots,\widehat{e}_{1}^{(\infty)})$ determines a strongly divisible modules of rank $n-1$, that is a submodule of $\widehat{\cM}$. This completes the proof.
\end{proof}

\begin{coro}\label{coro: reducibility of certain lifts}
Fix a pair of integers $(i_0,j_0)$ with $0\leq j_0\leq i_0\leq n-1$. Assume that $\rhobar_0$ is generic, and let $(k_{n-1},k_{n-2},\cdots,k_0)$ be an $n$-tuple of integers. Assume further that $$k_i=(p^{f-1}+p^{f-2}+\cdots +p+1)c_i$$ for $i> i_0$ and for $i<j_0$ and that the $k_i$ are pairwise distinct modulo $(e)$.

Then every potentially semi-stable lift $\rho$ of $\rhobar_0$ with Hodge--Tate weights $\{-(n-1),-(n-2),\cdots,0\}$ and Galois types $\bigoplus_{i=0}^{n-1}\widetilde{\omega}_f^{k_i}$ is a successive extension
$$\rho\cong
\begin{pmatrix}
\rho_{n-1,n-1} & \cdots & \ast & \ast & \ast & \cdots & \ast \\
 & \ddots &\vdots &\vdots& \vdots &\ddots & \vdots \\
 &  &\rho_{i_0+1,i_0+1}&\ast&\ast & \cdots &\ast\\
 &  & &\rho_{i_0,j_0}&\ast&\cdots &\ast\\
 &  & & &\rho_{j_0-1,j_0-1}&\cdots &\ast\\
 &  & & &  & \ddots & \vdots \\
 &  & & &  &        &\rho_{0,0}
\end{pmatrix}
$$
where
\begin{itemize}
\item $\rho_{i,i}$ is a $1$-dimensional potentially semi-stable lift of $\rhobar_{i,i}$ with Hodge--Tate weights $-i$ and Galois type $\widetilde{\omega}^{k_i^{i_0,j_0}}$ for $n-1\geq i>i_0$ and for $j_0>i\geq 0$;
\item $\rho_{i_0,j_0}$ is a $(i_0-j_0+1)$-dimensional potentially semi-stable lift of $\rhobar_{i_0,j_0}$ with Hodge--Tate weights $\{-i_0,-i_0+1,\cdots,-j_0\}$ and Galois types $\bigoplus_{i=j_0}^{i_0}\widetilde{\omega}^{k_i^{i_0,j_0}}$.
\end{itemize}
\end{coro}

\begin{proof}
Proposition~\ref{prop: reduciblility of certain lifts} implies this corollary recursively. Let $\cM\in\FBrModdd[n-1]$ be a Breuil module corresponding to the mod $p$ reduction of a strongly divisible module $\widehat{\cM}\in\OEModdd[n-1]$ corresponding to a Galois stable lattice in a potentially semi-stable representation $\rho:G_{\Qp}\rightarrow \GL_n(E)$ with Galois type $\bigoplus_{i=0}^{n-1}\widetilde{\omega}_f^{k_i}$ and Hodge--Tate weights $\{-(n-1), -(n-2),\cdots,0\}$ such that $\Tst^{\Qp,n-1}(\widehat{\cM})\otimes_{\cO_E}\F\cong\rhobar_0$. Both $\widehat{\cM}$ (resp. $\cM$) is of inertial type $\bigoplus_{i=0}^{n-1}\widetilde{\omega}_f^{k_i}$ (resp. $\bigoplus_{i=0}^{n-1}\omega_f^{k_i}$) by Proposition~\ref{prop: intertial type of a Breuil module}. We may assume that $\Mat_{\underline{e},\underline{f}}(\Fil^{n-1}\cM)$, $\Mat_{\underline{e},\underline{f}}(\phi_{n-1})$, and $\Mat_{\underline{e}}(N)$ are written as in (\ref{filtration of Breuil Module: niveau f}), (\ref{Frobenius of Breuil Module: niveau f}), and (\ref{Monodoromy of Breuil Module: niveau f}) respectively, and we do so.

By the equation~(\ref{equation in type elimination}), it is easy to see that $r_i= (p^{f-1}+p^{f-2}+\cdots +p+1)i$ for $i> i_0$ and for $i<j_0$, by our assumption on $k_i$. By Proposition~\ref{prop: reduciblility of certain lifts}, there exists an $(n-1)$-dimensional subrepresentation $\rho'_{n-1,1}$ of $\rho$ whose quotient is $\rho_{0,0}$ which is a potentially semi-stable lift of $\rhobar_{0,0}$ with Hodge--Tate weight $0$ and Galois type $\widetilde{\omega}_f^{k_0}$. Now consider $\rho'_{n-1,1}\otimes \varepsilon^{-1}$. Apply Proposition~\ref{prop: reduciblility of certain lifts} to $\rho'_{n-1,1}\otimes \varepsilon^{-1}$. Recursively, one can readily check that $\rho$ has subquotients $\rho_{i,i}$ for $0\leq i\leq j_0-1$. Considering $\rho^{\vee}\otimes \varepsilon^{n-1}$, one can also readily check that $\rho$ has subquotients $\rho_{i,i}$ lifting $\rhobar_{i,i}$ for $n-1\geq i\geq i_0+1$.
\end{proof}

The results in Corollary~\ref{coro: reducibility of certain lifts} reduce many of our computations for the main results on the Galois side.

\subsection{Main results on the Galois side}\label{subsec: main result on the Galois side}
In this section, we state and prove the main local results on the Galois side, that connects the Fontaine--Laffaille parameters of $\rhobar_0$ with the Frobenius eigenvalues of certain potentially semi-stable lifts of $\rhobar_0$. Throughout this section, we assume that $\rhobar_0$ is Fontaine--Laffaille generic. We also fix $f=1$ and $e=p-1$.

Fix $i_0,j_0\in\Z$ with $0\leq j_0<j_0+1<i_0\leq n-1$, and define the $n$-tuple of integers $$(r^{i_0,j_0}_{n-1},\,\,r^{i_0,j_0}_{n-2},\,\,\cdots,\,\,r^{i_0,j_0}_0)$$ as follows:
\begin{equation}\label{Special Galois type in general}
r^{i_0,j_0}_i:=
\left\{
  \begin{array}{ll}
    i & \hbox{if $i_0\not=i\not=j_0$;} \\
    j_0+1 & \hbox{if $i=i_0$;}\\
    i_0-1 & \hbox{if $i=j_0$.}\\
  \end{array}
\right.
\end{equation}
We note that if we replace $n$ by $i_0-j_0+1$ in the definition of $r^{(0)}_i$ in (\ref{Galois types of newest}) then we have the identities:
\begin{equation}\label{Special Galois type reducing to (n-1,0)}
r^{i_0,j_0}_{j_0+i}=j_0+r^{(0)}_i
\end{equation}
for all $0\leq i\leq i_0-j_0$. In particular, $r^{n-1,0}_i=r^{(0)}_i$ for all $0\leq i\leq n-1$.

From the equation $k^{i_0,j_0}_{i}\equiv c_i+i-r^{i_0,j_0}_{i}$ mod $(e)$ (cf. Lemma~\ref{lemm: breuil modules niveau 1}, (i)), this tuple immediately determines an $n$-tuple $(k^{i_0,j_0}_{n-1},\,\,k^{i_0,j_0}_{n-2},\,\,\cdots,\,\,k^{i_0,j_0}_0)$ of integers mod $(e)$, which will determine the Galois types of our representations. We set $$k^{i_0,j_0}_{i}:= c_i+i-r^{i_0,j_0}_{i}$$ for all $i\in\{0,1,\cdots,n-1\}$.

The following is the main result on the Galois side.
\begin{theo} \label{thm: main theorem Galois}
Let $i_0,j_0$ be integers with $0\leq j_0<j_0+1<i_0\leq n-1$. Assume that $\rhobar_0$ is generic and that $\rhobar_{i_0,j_0}$ is Fontaine--Laffaille generic. Let $(\lambda^{i_0,j_0}_{n-1},\lambda^{i_0,j_0}_{n-2},\cdots,\lambda^{i_0,j_0}_0)\in (\cO_E)^{n}$ be the Frobenius eigenvalues on the $(\widetilde{\omega}^{k^{i_0,j_0}_{n-1}}, \widetilde{\omega}^{k^{i_0,j_0}_{n-2}}, \cdots, \widetilde{\omega}^{k^{i_0,j_0}_0})$-isotypic components of $\Dst^{\Qp,n-1}(\rho_0)$ where $\rho_0$ is a potentially semi-stable lift of $\rhobar_0$ with Hodge--Tate weights $\{-(n-1),-(n-2),\cdots,-1,0\}$ and Galois types $\bigoplus_{i=0}^{n-1}\widetilde{\omega}^{k^{i_0,j_0}_i}$.

Then the Fontaine--Laffaille parameter $\mathrm{FL}_n^{i_0,j_0}$ associated to $\rhobar_{0}$ is computed as follows:
$$\mathrm{FL}_n^{i_0,j_0}(\rhobar_{0})= \overline{\left(\frac{p^{[(n-1)-\frac{i_0+j_0}{2}](i_0-j_0-1)}}{\prod_{i=j_0+1}^{i_0-1}\lambda^{i_0,j_0}_{i}} \right)}\in \mathbb{P}^1(\F).$$
\end{theo}

We first prove Theorem~\ref{thm: main theorem Galois} for the case $(i_0,j_0)=(n-1,0)$ in the following proposition, which is the key first step to prove Theorem~\ref{thm: main theorem Galois}.

\begin{prop} \label{prop: main theorem Galois}
Keep the assumptions and notation of Theorem~\ref{thm: main theorem Galois}, and assume further $(i_0,j_0)=(n-1,0)$. Then Theorem~\ref{thm: main theorem Galois} holds.
\end{prop}

Recall that $(k_{n-1}^{n-1,0},\cdots,k_0^{n-1,0})$ in Proposition~\ref{prop: main theorem Galois} is the same as $(k^{(0)}_{n-1},\cdots,k^{(0)}_0)$ in (\ref{Galois types of newest}). To lighten the notation, we let $k_i=k_i^{n-1,0}$ and $\lambda_i=\lambda_i^{n-1,0}$ during the proof of Proposition~\ref{prop: main theorem Galois}. We heavily use the results in Sections~\ref{subsec: Breuil modules of certain types}, \ref{subsec: Fontaine vs Frobenius} and \ref{subsec: filtration of strongly divisible modules} to prove this proposition.

\begin{proof}
Let $\widehat{\cM}\in\OEModdd[n-1]$ be a strongly divisible module corresponding to a Galois stable lattice in a potentially semi-stable representation $\rho_0:G_{\Qp}\rightarrow \GL_n(E)$ with Galois type $\bigoplus_{i=0}^{n-1}\widetilde{\omega}^{k_i}$ and Hodge--Tate weights $\{-(n-1), -(n-2),\cdots,0\}$ such that $\Tst^{\Qp,n-1}(\widehat{\cM})\otimes_{\cO_E}\F\cong\rhobar_0$. We also let $\cM$ be the Breuil module corresponding to the mod $p$ reduction of $\widehat{\cM}$. $\widehat{\cM}$ (resp. $\cM$) is of inertial type $\bigoplus_{i=0}^{n-1}\widetilde{\omega}^{k_i}$ (resp. $\bigoplus_{i=0}^{n-1}\omega^{k_i}$) by Proposition~\ref{prop: intertial type of a Breuil module}.

We let $\underline{\widehat{f}}=(\widehat{f}_{n-1},\widehat{f}_{n-2},\cdots,\widehat{f}_1,\widehat{f}_0)$ be a framed system of generators for $\Fil^{n-1}\widehat{\cM}$, and $\underline{\widehat{e}}=(\widehat{e}_{n-1},\widehat{e}_{n-2},\cdots,\widehat{e}_1, \widehat{e}_0)$ be a framed basis for $\widehat{\cM}$. We may assume that $\Mat_{\underline{\widehat{e}},\underline{\widehat{f}}}(\Fil^{n-1}\widehat{\cM})$ is described as in Propostion~\ref{prop: shape of filtration of sdm}, and we do so.

Define $\alpha_i\in\F^{\times}$ by the condition $\phi_{n-1}(\widehat{f}_i)\equiv\widetilde{\alpha}_i\widehat{e}_i$ modulo $(\varpi_E,\,u)$ for all $i\in\{0,1,\cdots,n-1\}$. There exists a framed basis $\underline{e}=(e_{n-1},e_{n-2},\cdots,e_0)$ for $\cM$ and a framed system of generators $\underline{f}=(f_{n-1},f_{n-2},\cdots,f_0)$ for $\Fil^{n-1}\cM$ such that
$\Mat_{\underline{e},\underline{f}}(\Fil^{n-1}\cM)$ is given as in Proposition~\ref{prop: Breuil modules for newest FL 3} and
$$\Mat_{\underline{e},\underline{f}}(\phi_{n-1})=\left(\alpha_{i,j} u^{[k_j-k_i]_1}\right)\in\GL_{n}^{\square}(\barS)
$$
for some $\alpha_{i,j}\in \barS_0$ with $\alpha_{i,i}\equiv \alpha_i$ mod $(u^e)$.

Recall that $\widehat{f}_i=E(u)^i\widehat{e}_i$ for $i\in\{1,2,\cdots,n-2\}$ by Proposition~\ref{prop: shape of filtration of sdm}.
Write $\phi_{n-1}(\widehat{f}_j)=\sum_{i=0}^{n-1}\widehat{\alpha}_{i,j} u^{[k_j-k_i]_1}\widehat{e}_i$ for some $\widehat{\alpha}_{i,j}\in S_0$. Then we get $$s_0(\widehat{\alpha}_{i,i})\equiv \frac{p^{i}\lambda_i}{p^{n-1}}\pmod{\varpi_E}$$
for $i\in\{1,2,\cdots,n-2\}$ since $\phi_{n-1}=\frac{1}{p^{n-1}}\phi$ for the Frobenius $\phi$ on $\Dst^{\Q_p,n-1}(\rho_0)$, so that we have $$\prod_{i=1}^{n-2}\widetilde{\alpha}_i\equiv \prod_{i=1}^{n-2}\frac{\lambda_i}{p^{n-1-i}}\pmod{\varpi_E}.$$ (Note that $\frac{\lambda_i}{p^{n-1-i}}\in\cO_E^{\times}$ by Corollary~\ref{coro: valuation of Frobenius eigenvalues}.) This completes the proof, by applying the results in Proposition~\ref{prop: Breuil modules for newest FL 3}.
\end{proof}

We now prove Theorem~\ref{thm: main theorem Galois}, which is the main result on the Galois side.
\begin{proof}[Proof of Theorem~\ref{thm: main theorem Galois}]
Recall from the identities in (\ref{Special Galois type reducing to (n-1,0)}) that
$$(r^{i_0,j_0}_{i_0}, r^{i_0,j_0}_{i_0-1},\cdots, r^{i_0,j_0}_{j_0})=j_0+(1,n'-2,n'-3,\cdots,1,n'-2)$$ where $n':=i_0-j_0+1$.  Recall also that $\rho_0$ has a subquotent $\rho_{i_0,j_0}$ that is a potentially semi-stable lift of $\rhobar_{i_0,j_0}$ with Hodge--Tate weights $\{-i_0,-(i_0-1),...,-j_0\}$ and of Galois type $\bigoplus_{i=j_0}^{i_0}k^{i_0,j_0}_i$, by Corollary~\ref{coro: reducibility of certain lifts}.

It is immediate that $\rho'_{i_0,j_0}:= \rho_{i_0,j_0} \varepsilon^{-j_0} \widetilde{\omega}^{j_0}$ is another potentially semi-stable lift of $\rhobar_{i_0,j_0}$ with Hodge--Tate weights $\{-(i_0-j_0),-(i_0-j_0-1),...,0\}$ and of Galois type $\bigoplus_{i=j_0}^{i_0}\widetilde{\omega}^{k^{i_0,j_0}_i+j_0}$. We let $(\eta_{i_0},\eta_{i_0-1},\cdots,\eta_{j_0})\in (\cO_E)^{i_0-j_0+1}$ (resp. $(\delta_{i_0},\delta_{i_0-1},\cdots,\delta_{j_0})\in (\cO_E)^{i_0-j_0+1}$) be the Frobenius eigenvalues on the $(\widetilde{\omega}^{k^{i_0,j_0}_{i_0}},\widetilde{\omega}^{k^{i_0,j_0}_{i_0-1}},\cdots, \widetilde{\omega}^{k^{i_0,j_0}_{j_0}})$-isotypic component of $\Dst^{\Qp,i_0-j_0}(\rho_{i_0,j_0})$ (resp. on the $(\widetilde{\omega}^{k^{i_0,j_0}_{i_0}+j_0},\widetilde{\omega}^{k^{i_0,j_0}_{i_0-1}+j_0},\cdots, \widetilde{\omega}^{k^{i_0,j_0}_{j_0}+j_0})$-isotypic component of $\Dst^{\Qp,i_0-j_0}(\rho'_{i_0,j_0})$). Then we have
$$p^{-j_0}\delta_i=\eta_i$$
for all $i\in\{j_0,j_0+1,\cdots,i_0\}$ and, by Proposition~\ref{prop: main theorem Galois},
$$
\mathrm{FL}_{i_0-j_0+1}^{i_0-j_0,0}(\rhobar_{i_0,j_0})= \overline{\left[\left(\prod_{i=j_0+1}^{i_0-1}\delta_i\right) : p^{\frac{(i_0-j_0)(i_0-j_0-1)}{2}}\right]}\in \mathbb{P}^1(\F).
$$

But we also have that $$p^{n-1-(i_0-j_0)}\eta_{i}=\lambda^{i_0,j_0}_i$$ for all $i\in\{j_0,j_0+1,\cdots,i_0\}$ by Corollary~\ref{coro: reducibility of certain lifts}. Hence, we have $\delta_i=p^{-(n-1-i_0)}\lambda_{i}^{i_0,j_0}$ for all $i\in\{j_0,j_0+1,\cdots,i_0\}$ and we conclude that
$$
\mathrm{FL}_n^{i_0,j_0}(\rhobar_0)=\mathrm{FL}_{i_0-j_0+1}^{i_0-j_0,0}(\rhobar_{i_0,j_0})= \overline{\left[\left(\prod_{i=j_0+1}^{i_0-1}\lambda^{i_0,j_0}_{i}\right) : p^{[(n-1)-\frac{i_0+j_0}{2}](i_0-j_0-1)}\right]}\in \mathbb{P}^1(\F).
$$
(Note that $\mathrm{FL}_n^{i_0,j_0}(\rhobar_0)=\mathrm{FL}_{i_0-j_0+1}^{i_0-j_0,0}(\rhobar_{i_0,j_0})$ by Lemma~\ref{lemm: FL parameter with dual rep}.)
\end{proof}

In the following corollary, we prove that the Weil--Deligne representation $\WD(\rho_0)$ associated to $\rho_0$ still contains Fontaine--Laffaille parameters. As we will see later, we will transport this information to the automorphic side via local Langlands correspondence.

\begin{coro}\label{coro: main thm}
Keep the assumptions and notation of Theorem~\ref{thm: main theorem Galois}.

Then $\rho_0$ is, in fact, potentially crystalline and
$$\WD(\rho_0)^{\mathrm{F-ss}}=\WD(\rho_0)\cong\bigoplus_{i=0}^{n-1}\Omega_{i}$$ where $\Omega_{i}:\Q_p^{\times}\rightarrow E^{\times}$ is defined by $\Omega_{i}:= \mathrm{U}_{\lambda_i^{i_0,j_0}/p^{n-1}}\cdot \widetilde{\omega}^{k_i^{i_0,j_0}}$ for all $i\in\{0,1,\cdots,n-1\}$. Moreover, 
$$\mathrm{FL}^{i_0,j_0}_n(\rhobar_0)=\overline{\left(\frac{\prod_{i=j_0+1}^{i_0-1} \Omega^{-1}_{i}(p)}{p^{\frac{(i_0+j_0)(i_0-j_0-1)}{2}}}\right)}\in \mathbb{P}^1(\F).$$
\end{coro}

\begin{proof}
Notice that $\phi$ is diagonal on $D:=\Dst^{\Q_p}(\rho_0)$ with respect to a framed basis $\underline{e}:= (e_{n-1},e_{n-2}\cdots,e_{0})$ (which satisfies $ge_i=\widetilde{\omega}^{k_i^{i_0,j_0}}(g)e_i$ for all $i$ and for all $g\in\Gal(K/\Q_p)$) since $\widetilde{\omega}^{k_i^{i_0,j_0}}$ are all distinct. Hence, we have $\WD(\rho_0)=\WD(\rho_0)^{\mathrm{F-ss}}$. Similarly, since the descent data action on $D$ commutes with the monodromy operator $N$, it is immediate that $N=0$.

From the definition of $\WD(\rho_0)$ (cf. \cite{CDT}), the action of $W_{\Q_p}$ on $D$ can be described as follows: let $\alpha(g)\in\Z$ be determined by $\bar{g}=\phi^{\alpha(g)}$, where $\phi$ is the arithmetic Frobenius in $G_{\F_p}$ and $\bar{g}$ is the image under the surjection $W_{\Q_p}\twoheadrightarrow \Gal(K/\Q_p)$. Then $$\WD(\rho_0)(g)\cdot e_i= \left(\frac{\lambda_i^{i_0,j_0}}{p^{n-1}}\right)^{-\alpha(g)}\cdot\widetilde{\omega}^{k_i^{i_0,j_0}}(g)\cdot e_i$$ for all $i\in\{0,1,\cdots,n-1\}$. (Recall that $\Dst^{\Q_p,n-1}(\rho_0)=\Dst^{\Q_p}(\rho_0\otimes\varepsilon^{-(n-1)})$, so that the $\frac{\lambda_i^{i_0,j_0}}{p^{n-1}}$ are the Frobenius eigenvalues of the Frobenius on $D$.) Write $\Omega_{i}$ for the eigen-character with respect to $e_i$.

Recall that we identify the geometric Frobenius with the uniformizers in $\Q_p^{\times}$ (by our normalization of class field theory), so that $\Omega_{i}(p)=\frac{p^{n-1}}{\lambda_i^{i_0,j_0}}$ which completes the proof by applying Theorem~\ref{thm: main theorem Galois}.
\end{proof}

\section{Local automorphic side}\label{sec: local automorphic side}
In this section, we establish several results concerning representation theory of $\GL_n$, that will be applied to the proof of our main results on mod $p$ local-global compatibility, Theorem~\ref{theo: lgc}. The main results in this section are the non-vanishing result, Corollary~\ref{coro: isomorphism}, as well as the intertwining identity in characteristic $0$, Theorem~\ref{theo: identity}.

We start this section by fixing some notation. Let $G:=\GL_{n/\Z_p}$ and $T$ be the maximal split torus consisting of diagonal matrices. We fix a Borel subgroup $B\subseteq G$ consisting of upper-triangular matrices, and let $U\subseteq B$ be the maximal unipotent subgroup. Let $\Phi^+$ denote the set of positive roots with respect to $(B,T)$, and $\Delta=\{\alpha_k\}_{1\leq k\leq n-1}$ the subset of positive simple roots. Let $X(T)$ and $X^{\vee}(T)$ denote the abelian group of characters and cocharacters respectively. We often say a weight for an element in $X(T)$, and write $X(T)_+$ for the set of dominant weights. The set $\Phi^+$ induces a partial order on $X(T)$: for $\lambda, \mu\in X(T)$ we say that $\lambda\leq \mu$ if $\mu-\lambda\in\sum_{\alpha\in\Phi^+}\Z_{\geq 0}\alpha$.

We use the $n$-tuple of integers $\lambda=(d_1,d_2,\cdots,d_n)$ to denote the character associated to the weight $\sum_{k=1}^nd_k\epsilon_k$ of $T$ where for each $1\leq i\leq n$ $\epsilon_i$ is a weight of $T$ defined by
\begin{equation*}
\mathrm{diag}(x_1,x_2,\cdots,x_n)\overset{\epsilon_i}\mapsto x_i.
\end{equation*}
We will often use the following weight
$$\eta:=(n-1,n-2,\cdots,1,0).$$

We let $\overline{G},\, \overline{B},\, \cdots$ be the base change to $\F_p$ of $G,\, B,\, \cdots$ respectively. The Weyl group of $G$, denoted by $W$, has a standard lifting inside $G$ as the group of permutation matrix, likewise with $\overline{G}$. We denote the longest Weyl element by $w_0$ which is lifted to the antidiagonal permutation matrix in $G$ or $\overline{G}$. We use the notation $s_i$ for the simple reflection corresponding to $\alpha_i=\epsilon_i-\epsilon_{i+1}$ for $1\leq i\leq n-1$. We define the length $\ell(w)$ of $w\in W$ to be its minimal length of decomposition into product of $s_i$ for $1\leq i\leq n-1$. Given $A\in U(\F_p)$, we use $A_{\alpha}$ or equivalently $A_{i,j}$ to denote the $(i,j)$-entry of $A$, where $\alpha=\epsilon_i-\epsilon_j$ is the positive root corresponding to the pair $(i,j)$ with $1\leq i<j\leq n$.

Given a representation $\pi$ of $ G(\F_p)$, we use the notation $\pi^{\mu}$ for the $ T(\F_p)$-eigenspace with the eigencharacter $\mu$. Given an algebraic representation $V$ of $G$ or $\overline{G}$, we use the notation $V_{\lambda}$ for the weight space of $V$ associated to the weight $\lambda$. For any representation $V$ of $\overline{G}$ or $ G(\F_p)$ with coefficient in $\F_p$, we define
$$V_{\F}:=V \otimes_{\F_p}\F$$
to be the extension of coefficient of $V$ from $\F_p$ to $\F$. Similarly, we write $V_{\overline{\F}_p}$ for $V \otimes_{\F_p}\overline{\F}_p$.

It is easy to observe that we can identify the character group of $ T(\F_p)$ with $X(T)/(p-1)X(T)$. The natural action of the Weyl group $W$ on $T$ and thus on $ T(\F_p)$ induces an action of $W$ on the character group $X(T)$ and $X(T)/(p-1)X(T)$. We carefully distinguish the notation between them. We use the notation $w\lambda$ (resp. $\mu^w$) for the action of $W$ on $X(T)$ (resp. $X(T)/(p-1)X(T)$) satisfying
$$w\lambda(x)=\lambda(w^{-1}xw)\mbox{ for all }x\in T$$
and
$$\mu^w(x)=\mu(w^{-1}xw)\mbox{ for all }x\in T(\F_p).$$
As a result, without further comments, the notation $w\lambda$ is a weight but $\mu^w$ is just a character of $ T(\F_p)$. There is another dot action of $W$ on $X(T)$ defined by
$$w\cdot \lambda=w(\lambda+\eta)-\eta\mbox{ for all }\lambda\in X(T)\mbox{ and } w\in W.$$

The affine Weyl group $\widetilde{W}$ of $G$ is defined as the semi-direct product of $W$ and $X(T)$ with respect to the natural action of $W$ on $X(T)$. We denote the image of $\lambda\in X(T)$ in $\widetilde{W}$ by $t_{\lambda}$. Then the dot action of $W$ on $X(T)$ extends to the dot action of $\widetilde{W}$ on $X(T)$ through the following formula
$$\widetilde{w}\cdot \lambda=w\cdot (\lambda+p\mu)$$
if $\widetilde{w}=wt_{\mu}$. We use the notation $\lambda\uparrow\mu$ for $\lambda,\mu\in X(T)$ if $\lambda\leq \mu$ and $\lambda\in \widetilde{W}\cdot\mu$.
We define a specific element of $\widetilde{W}$ by
$$\widetilde{w}_h:=w_0t_{-\eta}$$
following Section 4 of \cite{LLL}.

We usually write $K$ for $\GL_n(\Z_p)$ for brevity. We will also often use the following three open compact subgroups of $\GL_n(\Z_p)$: if we let $\mathrm{red}:\GL_n(\Z_p)\twoheadrightarrow\GL_n(\F_p)$ be the natural mod $p$ reduction map, then
\begin{equation*}
K(1):=\mathrm{Ker}(\mathrm{red})\,\subset\, I(1):=\mathrm{red}^{-1}( U(\F_p)) \,\subset\, I:=\mathrm{red}^{-1}( B(\F_p))\, \subset\, K.
\end{equation*}

For each $\alpha\in\Phi^+$, there exists a subgroup $U_{\alpha}$ of $G$ such that $xu_{\alpha}(t)x^{-1}=u_{\alpha}(\alpha(x)t)$ where $x\in T$ and $u_{\alpha}:\mathbb{G}_a\rightarrow U_{\alpha}$ is an isomorphism sending $1$ to $1$ in the entry corresponding to~$\alpha$.
For each $\alpha\in\Phi^+$, we have the following equalities by \cite{Jantzen2003} II.1.19 (5) and (6):
\begin{equation}\label{Taylor expansion}
u_{\alpha}(t)=\sum_{m\geq 0}t^m(X^{\rm{alg}}_{\alpha,m}).
\end{equation}
where $X^{\rm{alg}}_{\alpha,m}$ is an element in the algebra of distributions on $G$ supported at the origin $1\in G$. The equation (\ref{Taylor expansion}) is actually just the Taylor expansion with respect to $t$ of the operation $u_{\alpha}(t)$ at the origin $1$. We use the same notation $X^{\rm{alg}}_{\alpha,m}$ if $G$ is replaced by $\overline{G}$.


We define the set of $p$-restricted weights as
$$X_1(T):=\{\lambda\in X(T) \mid 0\leq \langle \lambda, \alpha^{\vee}\rangle\leq p-1\mbox{ for all }\alpha\in\Delta \}$$
and the set of central weights as
$$X_0(T):=\{\lambda\in X(T) \mid \langle \lambda, \alpha^{\vee}\rangle=0\mbox{ for all }\alpha\in\Delta \}.$$
We also define the set of $p$-regular weights as
$$X^{\rm{reg}}_1(T):=\{\lambda\in X(T) \mid 1\leq \langle \lambda, \alpha^{\vee}\rangle\leq p-2\mbox{ for all }\alpha\in\Delta \},$$
and in particular we have $X^{\rm{reg}}_1(T)\subsetneq X_1(T)$. We say that $\lambda\in X(T)$ \emph{lies in the lowest $p$-restricted alcove} if
\begin{equation}\label{lowest restricted alcove}
0<\langle\lambda+\eta,\alpha^{\vee}\rangle<p\mbox{ for all }\alpha\in\Phi^+.
\end{equation}
We define a subset $\widetilde{W}^+$ of $\widetilde{W}$ as
$$\widetilde{W}^+:=\{\widetilde{w}\in \widetilde{W} \mid \widetilde{w}\cdot \lambda\in X_+(T)\mbox{ for each }\lambda \mbox{ in the lowest }p\mbox{-restricted alcove}\}.$$
We define another subset $\widetilde{W}^{\rm{res}}$ of $\widetilde{W}$ as
\begin{equation}\label{restricted subset}
\widetilde{W}^{\rm{res}}:=\{\widetilde{w}\in\widetilde{W}\mid \widetilde{w}\cdot\lambda\in X_1(T)\mbox{ for each }\lambda \mbox{ in the lowest }p\mbox{-restricted alcove}\}.
\end{equation}
In particular, we have the inclusion
$$\widetilde{W}^{\rm{res}}\subseteq \widetilde{W}^+.$$

For any weight $\lambda\in X(T)$, we let
\begin{equation*}
H^0(\lambda):=\left(\mathrm{Ind}^{\overline{G}}_{\overline{B}}w_0\lambda\right)^{\rm{alg}}_{/\F_p}
\end{equation*}
be the associated dual Weyl module. Note by \cite{Jantzen2003}, Proposition II.2.6 that $H^0(\lambda)\neq 0$ if and only if $\lambda\in X(T)_+$. Assume that $\lambda\in X(T)_+$, we write $F(\lambda):=\mathrm{soc}_{\overline{G}}(H^0(\lambda))$ for its irreducible socle as an algebraic representation (cf. \cite{Jantzen2003} part II, section 2).  When $\lambda$ is running through $X_1(T)$, the $F(\lambda)$ exhaust all the irreducible representations of $ G(\F_p)$. On the other hand, two weights $\lambda, \lambda^{\prime}\in X_1(T)$ satisfies
$$F(\lambda)\cong F(\lambda^{\prime})$$
as $ G(\F_p)$-representation if and only if
$$\lambda-\lambda^{\prime}\in (p-1)X_0(T).$$

If $\lambda\in X^{\rm{reg}}_1(T)$, then the structure of $F(\lambda)$ as a $ G(\F_p)$-representation depends only on the image of $\lambda$ in $X(T)/(p-1)X(T)$, namely as a character of $ T(\F_p)$. Conversely, given a character $\mu$ of $ T(\F_p)$ which lies in the image of
$$X^{\rm{reg}}_1(T)\rightarrow X(T)/(p-1)X(T),$$
its inverse image in $X^{\rm{reg}}_1(T)$ is uniquely determined up to a translation of $(p-1)X_0(T)$. In this case, we say that $\mu$ is $p$-regular.  Whenever it is necessary for us to lift an element in $X(T)/(p-1)X(T)$ back into $X_1(T)$ (or maybe $X^{\rm{reg}}_1(T)$), we will clarify the choice of the lift.

Consider the standard Bruhat decomposition
\begin{equation*}
G=\bigsqcup_{w\in W} BwB=\bigsqcup_{w\in W}U_wwB=\bigsqcup_{w\in W}BwU_{w^{-1}}.
\end{equation*}
where $U_w$ is the unique subgroup of $U$ satisfying $BwB=U_w wB$ as schemes over $\Z_p$. The group $U_w$ has a unique form of $\prod_{\alpha\in\Phi^+_w}U_{\alpha}$ for the subset $\Phi^+_w$ of $\Phi^+$ defined by $\Phi^+_w=\{\alpha\in\Phi^+,w(\alpha)\in-\Phi^+\}$. (If $w=1$, we understand $\prod_{\alpha\in\Phi^+_w}U_{\alpha}$ to be the trivial group $1$.) We also have the following Bruhat decomposition:
\begin{equation}\label{Bruhat}
 G(\F_p)=\bigsqcup_{w\in W}  B(\F_p)w B(\F_p)=\bigsqcup_{w\in W} U_w(\F_p)w B(\F_p)=\bigsqcup_{w\in W} B(\F_p)w U_{w^{-1}}(\F_p).
\end{equation}


Given any integer $x$, recall that we use the notation $[x]_1$ to denote the only integer satisfying $0\leq [x]_1\leq p-2$ and $[x]_1\equiv x$ mod $(p-1)$. Given two non-negative integers $m$ and $k$ with $m\geq k$, we use the notation $c_{m,k}$ for the binomial number $\frac{m!}{(m-k)!k!}$. We use the notation $\bullet$ for composition of maps and, in particular, composition of elements in the group algebra $\F_p[G(\F_p)]$.

\subsection{Jacobi sums in characteristic $p$}\label{subsec: Jacobi sumes in char. p}
In this section we establish several fundamental properties of Jacobi sum operators on mod $p$ principal series representations.

\begin{defi}\label{defi: generic on tuples}
A weight $\lambda\in X(T)$ is called \emph{$k$-generic} for $k\in\Z_{>0}$ if for each $\alpha\in\Phi^+$ there exists $m_{\alpha}\in\Z$ such that
$$m_{\alpha}p+k<\langle\lambda,\alpha^{\vee}\rangle<(m_{\alpha}+1)p-k.$$
In particular, the $n$-tuple of integers $(a_{n-1},\cdots,a_1,a_0)$ is called \emph{$k$-generic in the lowest alcove} if
\begin{equation*}
a_i-a_{i-1}>k\quad \forall\; 1\leq i\leq n-1\text{  and  } a_{n-1}-a_0<p-k.
\end{equation*}
\end{defi}

Note that $(a_{n-1},\cdots,a_0)-\eta$ lies the lowest restricted alcove in the sense of (\ref{lowest restricted alcove}) if $(a_{n-1},\cdots,a_0)$ is $k$-generic in the lowest alcove for some $k>0$. Note also that the existence of a $n$-tuple of integers satisfying $k$-generic in the lowest alcove implies $p>n(k+1)-1$.

We use the notation $\pi$ for a general principal series representation:
\begin{equation*}
\pi:=\mathrm{Ind}^{ G(\F_p)}_{ B(\F_p)}\mu_{\pi}=\{f: G(\F_p)\rightarrow\F_p\mid f(bg)=\mu_{\pi}(b)f(g)\quad\forall (b,g)\in B(\F_p)\times \GL_n(\F_p)\}
\end{equation*}
where $\mu_{\pi}$ is a mod $p$ character of $ T(\F_p)$. The action of $\GL_n(\F_p)$ on $\pi$ is given by $(g\cdot f)(g^{\prime})=f(g^{\prime}g)$. We will assume throughout this article that $\mu_{\pi}$ is $p$-regular. By definition we have
$$\mathrm{cosoc}_{ G(\F_p)}(\pi)=F(\mu_{\pi})\mbox{ and }\mathrm{soc}_{ G(\F_p)}(\pi)=F(\mu_{\pi}^{w_0}).$$
By Bruhat decomposition we can deduce that
$$\mathrm{dim}_{\F_p}\pi^{ U(\F_p),\mu_{\pi}^w}=1$$
for each $w\in W$. We denote by $v_{\pi}$ a non-zero fixed vector in $\pi^{ U(\F_p),\mu_{\pi}}$.

Given an element $w\in W$, we let $\mu_{\pi^{\prime}}:=\mu_{\pi}^{w}$ and consider the principal series
$$\pi^{\prime}:=\mathrm{Ind}^{ G(\F_p)}_{ B(\F_p)} \mu_{\pi^{\prime}}.$$
As $\mathrm{dim}_{\F_p}(\pi^{\prime})^{ U(\F_p),\mu_{\pi}}=1$, by Frobenius reciprocity we have a unique non-zero morphism up to scalar
\begin{equation}\label{intertwining morphism}
\overline{\mathcal{T}}^{\pi}_{w}: \pi\rightarrow\pi^{\prime}.
\end{equation}
Given an element $w^{\prime}\in W$, we also let $\mu_{\pi^{\prime\prime}}:=\mu_{\pi^{\prime}}^{w^{\prime}}=\mu_{\pi}^{w^{\prime}w}$ and consider the principal series
$$\pi^{\prime\prime}:=\mathrm{Ind}^{ G(\F_p)}_{ B(\F_p)}\mu_{\pi^{\prime\prime}},$$
and we also have a unique non-zero morphism up to scalar
$$\overline{\mathcal{T}}^{\pi^{\prime}}_{w^{\prime}}: \pi^{\prime}\rightarrow\pi^{\prime\prime}.$$
In particular, we have
\begin{equation}\label{image of morphism}
\overline{\mathcal{T}}^{\pi}_{w}\left(\pi^{ U(\F_p),\mu_{\pi}}\right)=(\pi^{\prime})^{ U(\F_p),\mu_{\pi}}.
\end{equation}
We also have by Theorem 4.4 and Proposition 5.6 of \cite{CarterLusztig} that
\begin{equation}\label{composition of intertwine}
\overline{\mathcal{T}}^{\pi^{\prime}}_{w^{\prime}}\bullet \overline{\mathcal{T}}^{\pi}_{w}=\left\{
\begin{array}{lll}
c\overline{\mathcal{T}}^{\pi}_{w^{\prime}w}&\mbox{ if }&\ell(w)+\ell(w^{\prime})=\ell(w^{\prime}w)\\
0&\mbox{ if }&\ell(w)+\ell(w^{\prime})>\ell(w^{\prime}w)\\
\end{array}\right.
\end{equation}
for some $c\in\F_p^{\times}$.

Given integers $0\leq k_{\alpha}\leq p-1$ indexed by $\alpha\in\Phi^+_w$ for a certain $1\neq w\in W$, we define the Jacobi sum operators
\begin{equation}\label{Jacobi sum operator}
S_{\underline{k},w}:=\sum_{A\in  U_w(\F_p)}\left(\prod_{\alpha\in\Phi^+_w}A_{\alpha}^{k_{\alpha}}\right)A\cdot w\in\F_p[ G(\F_p)]
\end{equation}
where $\underline{k}:=(k_{\alpha})_{\alpha\in\Phi^+_w}$. These Jacobi sum operators play a main role on the local automorphic side as a crucial computation tool. These operators already appeared in \cite{CarterLusztig} for example.

For each $\alpha\in\Phi^+$ and each integer $m$ satisfying $0\leq m\leq p-2$, we define the operator
\begin{equation}\label{common used operator}
X_{\alpha,m}:=\sum_{t\in\F_p}t^{p-1-m}u_{\alpha}(t)\in\F_p[ U(\F_p)]\subseteq\F_p[ G(\F_p)].
\end{equation}
The transition matrix between $\{u_{\alpha}(t)\mid t\in\F_p^{\times}\}$ and $\{X_{\alpha,m}\mid 0\leq m\leq p-2\}$ is a Vandermonde matrix
$$\left(t^k\right)_{t\in\F^{\times}_p, 1\leq k\leq p-1}$$
which has a non-zero determinant. Hence, we also have a converse formula
\begin{equation}\label{inversion formula}
u_{\alpha}(t)=-\sum_{m=0}^{p-2}t^mX_{\alpha,m} \mbox{ for all }t\in\F_p.
\end{equation}
By the equation (\ref{Taylor expansion}), we note that we have the equality
\begin{equation}\label{relation between operators}
X_{\alpha,m}=-\sum_{k\geq 0}X^{\rm{alg}}_{\alpha,m+(p-1)k}.
\end{equation}
We also define
\begin{equation}\label{groupoperators}
\mathcal{X}_{m_1,\dots,m_{n-1}}:=X_{\alpha_1,m_1}\circ\cdots\circ X_{\alpha_{n-1},m_{n-1}}\in\F_p[ U(\F_p)]\subseteq\F_p[ G(\F_p)]
\end{equation}
for each tuple of integers $(m_1,\cdots,m_{n-1})$ satisfying $0\leq m_i\leq p-2$ for each $1\leq i\leq n-1$.

\begin{lemm}\label{main lemma}
Fix $w\in W$ and $\alpha_0=(i_0,j_0)\in\Phi^+_w$. Given a tuple of integers $\underline{k}=(k_{i,j})\in\{0,1,\cdots,p-1\}^{|\Phi^+_w|}$ satisfying
\begin{equation}\label{vanishing of entry}
k_{i_0,j}=0\,\mbox{ for all }\,(i_0,j)\in\Phi^+_w\,\mbox{ with }\,j\geq j_0+1,
\end{equation}
we have
\begin{equation*}
X_{\alpha_0,m}\bullet S_{\underline{k},w}=\left\{
\begin{array}{lll}
(-1)^{m+1}c_{k_{\alpha_0},m}S_{\underline{k}^{\prime},w}&\mbox{if}&m\leq k_{\alpha_0}\\
0&\mbox{if}& m>k_{\alpha_0}
\end{array}\right.
\end{equation*}
where $\underline{k}^{\prime}=(k^{\prime}_{\alpha})_{\alpha\in\Phi_w}$ satisfies
\begin{equation*}
k^{\prime}_{\alpha}=
\left\{\begin{array}{ll}
k_{\alpha_0}-m&\hbox{if $\alpha=\alpha_0$;}\\
k_{\alpha}& \hbox{otherwise.}\\
\end{array}\right.
\end{equation*}
\end{lemm}

\begin{proof}
We prove this lemma by direct computation.
\begin{multline}\label{first computation}
X_{\alpha,m}\bullet S_{\underline{k},w}=\sum_{t\in\F_p}t^{p-1-m}\left(\sum_{A\in U_w(\F_p)}\left(\prod_{\alpha\in\Phi^+_w}A_{\alpha}^{k_{\alpha}}\right)u_{\alpha_0}(t)Aw\right)\\
=\sum_{t\in\F_p}t^{p-1-m}\left(\sum_{A\in U_w(\F_p)}\left(\prod_{\alpha\in\Phi^+_w,\alpha\neq\alpha_0}A_{\alpha}^{k_{\alpha}}\right)(A_{\alpha_0}-t)^{k_{\alpha_0}}Aw\right)\\
=\sum_{A\in U_w(\F_p)}\left(\prod_{\alpha\in\Phi^+_w,\alpha\neq\alpha_0}A_{\alpha}^{k_{\alpha}}\right)\left(\sum_{t\in\F_p}t^{p-1-m}(A_{\alpha_0}-t)^{k_{\alpha_0}}\right)Aw\\
\end{multline}
where the second equality follows from the change of variable $A\leftrightarrow u_{\alpha_0}(t)A$ and the assumption~(\ref{vanishing of entry}).

Notice that
\begin{align*}
\sum_{t\in\F_p}t^{p-1-m}(A_{\alpha_0}-t)^{k_{\alpha_0}}&=\sum_{t\in\F_p}t^{p-1-m}\left(\sum_{j=0}^{k_{\alpha_0}}(-1)^jc_{k_{\alpha_0},j}A_{k_{\alpha_0}}^{k_{\alpha_0}-j}t^j\right)\\
&=\sum_{j=0}^{k_{\alpha_0}}(-1)^jc_{k_{\alpha_0},j}A_{k_{\alpha_0}}^{k_{\alpha_0}-j}\left(\sum_{t\in\F_p}t^{p-1-m+j}\right),\\
\end{align*}
which can be easily seen to be
\begin{equation}\label{simple fact}
\left\{\begin{array}{lll}
(-1)^{m+1}c_{k_{\alpha_0},m}A_{k_{\alpha_0}}^{k_{\alpha_0}-m}&\mbox{if}&m\leq k_{\alpha_0}\\
0&\mbox{if}& m>k_{\alpha_0}.\\
\end{array}\right.
\end{equation}
The last computation~(\ref{simple fact}) follows from the fact that
\begin{equation*}
\sum_{t\in\F_p}t^{\ell}=\left\{
\begin{array}{lll}
0&\mbox{if}&1\leq \ell\leq p-2\\
-1&\mbox{if}&\ell=p-1\\
\end{array}\right.
\end{equation*}
Applying (\ref{simple fact}) back to (\ref{first computation}) gives us the result.
\end{proof}

\begin{lemm}\label{invariant vector}
Fix $w\in W$ and $\alpha_0=(i_0,j_0)\in\Phi^+_w$. Given a tuple of integers $\underline{k}=(k_{i,j})\in\{0,1,\cdots,p-1\}^{|\Phi^+_w|}$ satisfying
\begin{equation*}
k_{i_0,j}=0\mbox{ for all }(i_0,j)\in\Phi^+_w\mbox{ with }j\geq j_0,
\end{equation*}
we have
\begin{equation*}
u_{\alpha_0}(t)\bullet S_{\underline{k},w}=S_{\underline{k},w}.
\end{equation*}
\end{lemm}

\begin{proof}
By Lemma \ref{main lemma} we deduce that
$$X_{\alpha_0,m}\bullet S_{\underline{k},w}=\left\{
\begin{array}{lll}
-S_{\underline{k},w}&\mbox{if}&m=0\\
0&\mbox{if}&1\leq m\leq p-2\\
\end{array}\right.
$$
Therefore we conclude this lemma from (\ref{inversion formula}).
\end{proof}

\begin{lemm}\label{simple formula}
Let $m_i$ be integers in $\{0,1,\cdots,p-2\}$ for all $1\leq i\leq n-1$, and $\underline{k}=(k_{i,j})\in\{0,\cdots,p-1\}^{|\Phi^+_{w_0}|}$ with $k_{i,j}=0$ for all $1\leq i<i+1<j\leq n$.

If $m_i\leq k_{i,i+1}$ for all $1\leq i\leq n-1$, then
$$
\mathcal{X}_{m_1,\dots,m_{n-1}}\bullet S_{\underline{k},w_0}=\prod_{i=1}^{n-1}\left((-1)^{m_i+1}c_{k_{i,i+1},m_i}\right)S_{\underline{k}^{\prime},w_0} \in\F_p[ G(\F_p)]
$$
where $\underline{k}^{\prime}=(k^{\prime}_{i,j})$ satisfies
\begin{equation*}
k^{\prime}_{i,j}=
\left\{\begin{array}{lll}
k_{i,j}-m_i&\hbox{if $j=i+1$; }\\
0& \hbox{otherwise.}\\
\end{array}\right.
\end{equation*}
Otherwise, $\mathcal{X}_{m_1,\dots,m_{n-1}}\circ S_{\underline{k},w_0}=0$.
\end{lemm}

\begin{proof}
This Lemma follows directly from Lemma \ref{main lemma} and the definition in (\ref{groupoperators}). In fact, we only need to apply Lemma \ref{main lemma} to the operators $X_{\alpha_i,m_i}$ for $i=n-1,\cdots,1$ inductively.
\end{proof}

By the definition of principal series representations, we have the decomposition
\begin{equation}\label{consequence of Bruhat}
\pi=\oplus_{w\in W}\pi_w
\end{equation}
where $\pi_w\subset\pi|_{ B(\F_p)}$ consists of the functions supported on the Bruhat cell $ B(\F_p)w^{-1} B(\F_p)= B(\F_p)w^{-1} U_w(\F_p)$.

\begin{prop}\label{prop: basis}
Fix a non-zero vector $v_{\pi}\in\pi^{ U(\F_p),\mu_{\pi}}$. For each $ w\in W$ with $w\neq 1$, the set $$\left\{S_{\underline{k},w} v_{\pi}\mid \underline{k}=(k_{\alpha})_{\alpha\in\Phi^+_w}\in\{0,1,\cdots,p-1\}^{|\Phi^+_w|} \right\}$$ forms a $ T(\F_p)$-eigenbasis of $\pi_w$.
\end{prop}

\begin{proof}
We have a decomposition $\pi_w=\oplus_{A\in U_w(\F_p)}\pi_{w,A}$ where $\pi_{w,A}$ is the subspace of $\pi_w$ consisting of functions supported on $ B(\F_p)w^{-1}A^{-1}$. It is easy to observe by the definition of parabolic induction that $\mathrm{dim}_{\F_p}\pi_{w,A}=1$ and $\pi_{w,A}$ is generated by $Awv_{\pi}$.

We claim that the set of Jacobi sums with the Weyl element $w$, after being applied to $v_{\pi}$, differs from the set $\{Awv_{\pi},A\in U_w(\F_p)\}$ by an invertible matrix. More precisely, for a fixed $w\in W$, the set of vectors
\begin{equation*}
\{S_{\underline{k},w} v_{\pi}\mid S=((k_{\alpha})_{\alpha\in\Phi^+_w}, w), 0\leq k_{\alpha}\leq p-1\quad \forall \alpha\in\Phi^+_w\}
\end{equation*}
can be linearly represented by the set of vectors $\{Awv_{\pi},A\in U_w(\F_p)\}$ through the matrix $\left(m_{\underline{k},A}\right)$ where
\begin{equation*}
\underline{k}=(k_{\alpha})_{\alpha\in\Phi^+_w}\in \{0,1,\cdots,p-1\}^{|\Phi^+_w|},\qquad A\in  U_w(\F_p)
\end{equation*}
and $m_{\underline{k},A}:=\prod_{\alpha\in\Phi^+_w}A_{\alpha}^{k_{\alpha}}$. Note that this matrix is the $|\Phi^+_w|$-times tensor of the Vandermonde matrix
\begin{equation*}
\left(\lambda^k\right)_{\lambda\in\F_p, 0\leq k\leq p-1},
\end{equation*}
and therefore has a non-zero determinant. As a result, the matrix $\left(m_{\underline{k},A}\right)$ is invertible and $\{S_{\underline{k},w} v_{\pi}\mid 0\leq k_{\alpha}\leq p-1\quad \forall \alpha\in\Phi^+_w\}$ forms a basis of $\pi_{w}$.

The fact that this basis is a $ T(\F_p)$-eigenbasis is immediate by the following calculation:
\begin{align*}
    x\bullet S_{\underline{k},w} v_{\pi} & =\,\, x\bullet\left(\sum_{A\in  U_w(\F_p)}\left(\prod_{\alpha\in\Phi^+_w}A_{\alpha}^{k_{\alpha}}\right)A~w\right) v_{\pi} \\
     & =\,\, \left(\sum_{A\in  U_w(\F_p)}\left(\prod_{(i,j)\in\Phi^+_w}A_{i,j}^{k_{i,j}}\right)xAx^{-1}~ w\right)~ \left(w^{-1}xw\right)~ v_{\pi} \\
     & =\,\, \left(\sum_{B=xAx^{-1}\in  U_w(\F_p)}\left(\prod_{(i,j)\in\Phi^+_w}(B_{i,j}x_jx_i^{-1})^{k_{i,j}}\right)B~ w\right)~ \left(w^{-1}xw\right)~ v_{\pi} \\
     & =\,\, \mu_{\pi}(w^{-1}xw)\left(\prod_{(i,j)\in\Phi^+_w}(x_jx_i^{-1})^{k_{i,j}}\right)\left(\sum_{A\in U_w(\F_p)}\prod_{\alpha\in\Phi^+_w}A_{\alpha}^{k_{\alpha}}A~ w\right) v_{\pi}\\
     &=\,\, (\mu^{w}_{\pi}\lambda)(x)S_{\underline{k},w} v_{\pi},
  \end{align*}
where $x:=\mathrm{diag}(x_1,x_2,\cdots,x_n)$, $\lambda(x)=\prod_{1\leq i<j\leq n}(x_jx_i^{-1})^{k_{i,j}}$, and $B_{i,j}=A_{i,j}x_ix_j^{-1}$ for $1\leq i<j\leq n$.
\end{proof}

We can further describe the action of $ T(\F_p)$ on $S_{\underline{k},w}v_{\pi}$. By $\lfloor y \rfloor$ for $y\in\R$ we mean the floor function of $y$, i.e., the biggest integer less than or equal to $y$.
\begin{lemm}\label{lemm: eigen}
Let $\mu_{\pi}=(d_1,d_2,\cdots,d_{n-1},d_n)$. If we write $(\ell_1,\ell_2\cdots,\ell_{n-1},\ell_n)$ for the $ T(\F_p)$-eigencharacter of $S_{\underline{k},w}v_{\pi}$, then we have
\begin{equation*}
\ell_r \equiv d_{w^{-1}(r)}+\sum_{1\leq i<r} k_{i,r}-\sum_{r<j\leq n}k_{r,j}\pmod{p-1}
\end{equation*}
for all $1\leq r\leq n$, where $k_{i,j}=k_{\alpha}$ if $\alpha\in\Phi^+_w$ and $(i,j)$ corresponds to $\alpha$, and $k_{i,j}=0$ otherwise.

In particular,
\begin{enumerate}
\item if $k_{\alpha}=0$ for any $\alpha\in\Phi^+_w\setminus\Delta$, then for all $1\leq r\leq n$ $$\ell_r\equiv d_{w^{-1}(r)}+(1-\lfloor{1}/{r}\rfloor)k_{r-1,r}-(1-\lfloor{1}/{(n+1-r)}\rfloor)k_{r,r+1} \pmod{p-1};$$
\item if $w=w_0$ and $k_{i,j}=0$ for any $2\leq i<j\leq n$, then
$$
\ell_r \equiv \left\{
\begin{array}{ll}
d_{n}-\sum_{j=2}^nk_{1,j} \pmod{p-1} & \hbox{if $r=1$;}\\
 d_{n+1-r}+k_{1,r} \pmod{p-1} & \hbox{if $2\leq r\leq n$.}
\end{array}
\right.
$$
\end{enumerate}
\end{lemm}

\begin{proof}
The first part of the Lemma is a direct calculation as shown at the end of the proof of Proposition \ref{prop: basis}. The second part
follows trivially from the first part.
\end{proof}

Given any $w\in W$, we write $S_{\underline{0},w}$ for $S_{\underline{k},w}$ with $k_{\alpha}=0$ for all $\alpha\in\Phi^+_w$.
\begin{lemm}\label{Uinvariant}
$\F_p[S_{\underline{0},w} v_{\pi}]=\pi^{ U(\F_p),\mu_{\pi}^w}$.
\end{lemm}
\begin{proof}
Pick an arbitrary positive root $\alpha$. If $\alpha\in\Phi^+_w$, then we have (since $u_{\alpha}(t)\in U_w(\F_p)$)
\begin{equation*}
u_{\alpha}(t)\left(\sum_{A\in  U_w(\F_p)}A\right)=\left(\sum_{A\in  U_w(\F_p)}A\right)
\end{equation*}
and therefore $u_{\alpha}(t)S_{\underline{0},w}v_{\pi}=S_{\underline{0},w}v_{\pi}$ for any $t\in\F_p$.  On the other hand, if $\alpha\notin\Phi^+_w$, then
$$
u_{\alpha}(t)\left(\sum_{A\in  U_w(\F_p)}A\right)=\left(\sum_{A\in  U_w(\F_p)}A\right)u^{\prime}_{\alpha}(t)$$
and
$$u^{\prime}_{\alpha}(t)wv_{\pi}=wu^{\prime\prime}_{\alpha}(t)v_{\pi}=wv_{\pi}$$
where $u^{\prime}_{\alpha}(t)\in\prod_{\alpha\notin\Phi^+_w}\overline{U}_{\alpha}(\F_p)$ and $u^{\prime\prime}_{\alpha}(t)\in U(\F_p)$ are elements depending on~$x$, ~$w$ and~$\alpha$. Hence, $u_{\alpha}(t)S_{\underline{0},w}v_{\pi}=S_{\underline{0},w}v_{\pi}$ for any $t\in\F_p$ and any $\alpha\in\Phi^+$. So we conclude that $S_{\underline{0},w}v_{\pi}$ is $ U(\F_p)$-invariant as $\{u_{\alpha}(t)\}_{\alpha\in\Phi^+,t\in\F_p}$ generate $ U(\F_p)$.

Finally, we check that $x\cdot S_{\underline{0},w}v_{\pi}=\mu^w_{\pi}(x)S_{\underline{0},w}v_{\pi}$ for $x\in T(\F_p)$. But this is immediate from the following two easy computations:
$$
x\bullet\left(\sum_{A\in  U_w(\F_p)}A\right)=\left(\sum_{A\in  U_w(\F_p)}A\right)\bullet x\in\F_p[G(\F_p)]
$$
and
$$xwv_{\pi}=w\left(w^{-1}xw\right)v_{\pi}=w\mu_{\pi}(w^{-1}xw)v_{\pi}=\mu^w_{\pi}(x)wv_{\pi}.$$ This completes the proof.
\end{proof}

Note that Proposition~\ref{prop: basis}, Lemma~\ref{lemm: eigen}, and Lemma~\ref{Uinvariant} are very elementary and have essentially appeared in \cite{CarterLusztig}. In this article, we formulate them and give quick proofs of them for the convenience.

\begin{defi}
Given $\alpha, \alpha^{\prime}\in\Phi^+$, we say that $\alpha$ is \emph{strongly smaller than} $\alpha^{\prime}$ with the notation
$$\alpha\,\widetilde{\prec}\,\alpha^{\prime}$$
if there exist $1\leq i\leq j\leq k\leq n-1$ such that
$$\alpha=\sum_{r=i}^j\alpha_r\mbox{ and }\alpha^{\prime}=\sum_{r=i}^k\alpha_r.$$
We call a subset $\Phi^{\prime}$ of $\Phi^+$ \emph{good} if it satisfies the following:
\begin{enumerate}
\item if $\alpha, \alpha^{\prime}\in\Phi^\prime$, then $\alpha+\alpha^{\prime}\in\Phi^\prime$;
\item if $\alpha\in\Phi^\prime$ and $\alpha\,\widetilde{\prec}\,\alpha^{\prime}$, then $\alpha^{\prime}\in\Phi^{\prime}$.
\end{enumerate}
\end{defi}
We associate a subgroup of $U$ to $\Phi^{\prime}$ by
\begin{equation}\label{associated group}
U_{\Phi^{\prime}}:=\langle U_{\alpha}\mid \alpha\in\Phi^{\prime}\rangle
\end{equation}
and denote its reduction mod $p$ by $\overline{U}_{\Phi^{\prime}}$.  We define $U_1$ to be the subgroup scheme of $U$ generated by $U_{\alpha_r}$ for $2\leq r\leq n-1$, and denote its reduction mod $p$ by $\overline{U}_1$.

\begin{exam}\label{good subgroup}
The following are two examples of good subsets of $\Phi^+$, that will be important for us:
$$\left\{\sum_{r=i}^j\alpha_r\mid 1\leq i<j\leq n-1\right\}\,\mbox{ and }\,\left\{\sum_{r=i}^j\alpha_r\mid 2\leq i\leq j\leq n-1\right\}.$$
The subgroups of $U$ associated with the two good subsets via (\ref{associated group}) are $[U,U]$ and $U_1$ respectively.
\end{exam}

We recall that we have defined $\pi_w\subsetneq\pi$ in (\ref{consequence of Bruhat}) for each $w\in W$.
\begin{prop}
Let $\Phi^{\prime}\subseteq\Phi^+$ be good. Pick an element $ w\in W$ with $w\neq 1$. The following set of vectors
\begin{equation}\label{invariant basis}
\left\{S_{\underline{k},w}v_{\pi}\mid \underline{k}=(k_{\alpha})_{\alpha\in\Phi^+_w}\in\{0,1,\cdots,p-1\}^{|\Phi^+_w|}\mbox{ with } k_{\alpha}=0\,\,\forall\alpha\in\Phi^\prime\cap\Phi^+_w\right\}
\end{equation}
forms a basis of the subspace $\pi_w^{ U_{\Phi^{\prime}}(\F_p)}$ of $\pi_w$.
\end{prop}
\begin{proof}
By Proposition \ref{prop: basis}, the Jacobi sums with the Weyl element $w$, after being applied to $v_{\pi}$, form a $ T(\F_p)$-eigenbasis of $\pi_w$, and so we can firstly write any $ U_{\Phi^{\prime}}(\F_p)$-invariant vector $v$ in $\pi_w$ as a unique linear combination of Jacobi sums with the Weyl element $w$, namely
\begin{equation*}
v=\sum_{\underline{k}\in\{0,\cdots,p-1\}^{|\Phi^+_w|}} C_{\underline{k},w} S_{\underline{k},w}v_{\pi}\mbox{ for some }C_{\underline{k},w}\in\F_p.
\end{equation*}

Assume that $C_{\underline{k},w}\neq 0$ for certain tuple of integers $\underline{k}=(k_{\alpha})_{\alpha\in\Phi^+_w}$ such that $k_{\alpha}>0$ for some $\alpha\in\Phi^\prime\cap\Phi^+_w$. We choose $\alpha_0$ such that it is maximal with respect to the partial order $\widetilde{\prec}$ on $\Phi^+$ for the property
\begin{equation}\label{property}
C_{\underline{k},w}\neq 0, \qquad k_{\alpha_0}>0, \qquad\mbox{and}\qquad \alpha_0\in\Phi^\prime\cap\Phi^+_w.
\end{equation}
We may write $v$ as follows:
\begin{equation}\label{two parts}
v=\sum_{\substack{\underline{k}\in\{0,\cdots,p-1\}^{|\Phi^+_w|}\\ k_{\alpha_0}=0}} C_{\underline{k},w} S_{\underline{k},w}v_{\pi}+\sum_{\substack{\underline{k}\in\{0,\cdots,p-1\}^{|\Phi^+_w|}\\k_{\alpha_0}>0}} C_{\underline{k},w} S_{\underline{k},w}v_{\pi}.
\end{equation}
By the maximal assumption on $\alpha_0$ we know that if $C_{\underline{k},w}\neq 0$ and $\alpha_0\,\widetilde{\prec}\,\alpha$, then $k_{\alpha}=0$.  As a result, we deduce from Lemma \ref{invariant vector} that
\begin{equation}\label{equality}
u_{\alpha_0}(t)\sum_{\substack{\underline{k}\in\{0,\cdots,p-1\}^{|\Phi^+_w|}\\ k_{\alpha_0}=0}} C_{\underline{k},w} S_{\underline{k},w}v_{\pi}=\sum_{\substack{\underline{k}\in\{0,\cdots,p-1\}^{|\Phi^+_w|}\\ k_{\alpha_0}=0}} C_{\underline{k},w} S_{\underline{k},w}v_{\pi}
\end{equation}
for all $t\in\F_p$.

We define
$$\Phi^{\alpha_0,+}_w:=\{\alpha\in\Phi^+_w\mid \alpha_0\,\widetilde{\prec}\,\alpha\}\,\, \mbox{ and } \,\,\Phi^{\alpha_0,-}_w:=\Phi^+_w\setminus\Phi^{\alpha_0,+}_w,$$
and we use the notation
$$\underline{\ell}:=(\ell_{\alpha})_{\alpha\in\Phi^{\alpha_0,-}_w}\in\{0,\cdots,p-1\}^{|\Phi^{\alpha_0,-}_w|}$$
for a tuple of integers indexed by $\Phi^{\alpha_0,-}_w$. Given a tuple $\underline{\ell}$, we can define
$$\Lambda(\underline{\ell},\alpha_0):=\left\{\underline{k}\in\{0,\cdots,p-1\}^{|\Phi^+_w|}\,\left|\,
             \begin{array}{ll}
              \cdot\,\, k_{\alpha}=0 & \hbox{if $\alpha\in\Phi^{\alpha_0,+}_w\setminus\{\alpha_0\}$;} \\
              \cdot\,\, k_{\alpha}>0 & \hbox{if $\alpha=\alpha_0$;} \\
              \cdot\,\, k_{\alpha}=\ell_{\alpha} & \hbox{if $\alpha\in \Phi^{\alpha_0,-}_w$}
             \end{array}
           \right.
\right\}.
$$

Now we can define a polynomial
\begin{equation*}
f_{(\underline{\ell},\alpha_0)}(x)=\sum_{\underline{k}\in\Lambda(\underline{\ell},\alpha_0)} C_{\underline{k},w}x^{k_{\alpha_0}}\in\F_p[x]
\end{equation*}
for each tuple of integers $\underline{\ell}$. By definition, we have
\begin{equation*}
\sum_{\substack{\underline{k}\in\{0,\cdots,p-1\}^{|\Phi^+_w|}\\k_{\alpha_0}>0}} C_{\underline{k},w} S_{\underline{k},w}v_{\pi}
=
\sum_{\underline{\ell}\in\{0,\cdots,p-1\}^{|\Phi^{\alpha_0,-}_w|}}\left(\sum_{A\in U_w(\F_p)} \left(\prod_{\alpha\in\Phi^{\alpha_0,-}_w} A_{\alpha}^{\ell_{\alpha}}\right)f_{(\underline{\ell},\alpha_0)}(A_{\alpha_0})A\right)wv_{\pi}.
\end{equation*}
By the assumption on $v$ we know that $u_{\alpha_0}(t)v=v$ for all $t\in\F_p$. Using (\ref{equality}) and (\ref{two parts}) we have
\begin{equation*}
u_{\alpha_0}(t)\sum_{\substack{\underline{k}\in\{0,\cdots,p-1\}^{|\Phi^+_w|}\\k_{\alpha_0}>0}} C_{\underline{k},w} S_{\underline{k},w}v_{\pi}=\sum_{\substack{\underline{k}\in\{0,\cdots,p-1\}^{|\Phi^+_w|}\\k_{\alpha_0}>0}} C_{\underline{k},w} S_{\underline{k},w}v_{\pi}
\end{equation*}
and so
\begin{align*}
\sum_{\underline{\ell}\in\{0,\cdots,p-1\}^{|\Phi^{\alpha_0,-}_w|}}&\left(\sum_{A\in U_w(\F_p)}\left(\prod_{\alpha\in\Phi^{\alpha_0,-}_w}A_{\alpha}^{\ell_{\alpha}}\right)f_{(\underline{\ell},\alpha_0)}(A_{\alpha_0})A\right)wv_{\pi}\\
=&\,\,u_{\alpha_0}(t)\sum_{\underline{\ell}\in\{0,\cdots,p-1\}^{|\Phi^{\alpha_0,-}_w|}}\left(\sum_{A\in U_w(\F_p)}\left(\prod_{\alpha\in\Phi^{\alpha_0,-}_w}A_{\alpha}^{\ell_{\alpha}}\right)f_{(\underline{\ell},\alpha_0)}(A_{\alpha_0})A\right)wv_{\pi}\\
=&\sum_{\underline{\ell}\in\{0,\cdots,p-1\}^{|\Phi^{\alpha_0,-}_w|}}\left(\sum_{A\in U_w(\F_p)}\left(\prod_{\alpha\in\Phi^{\alpha_0,-}_w}A_{\alpha}^{\ell_{\alpha}}\right) f_{(\underline{\ell},\alpha_0)}(A_{\alpha_0}-t)A\right)wv_{\pi}
\end{align*}
where the last equality follows from a change of variable $A\leftrightarrow u_{\alpha_0}(t)A$.

By the linear independence of Jacobi sums from Proposition \ref{prop: basis}, we deduce an equality
\begin{multline*}
\left(\sum_{A\in U_w(\F_p)}\left(\prod_{\alpha\in\Phi^{\alpha_0,-}_w}A_{\alpha}^{\ell_{\alpha}}\right)f_{(\underline{\ell},\alpha_0)}(A_{\alpha_0})A\right)wv_{\pi}\\
=\left(\sum_{A\in U_w(\F_p)}\left(\prod_{\alpha\in\Phi^{\alpha_0,-}_w}A_{\alpha}^{\ell_{\alpha}}\right)f_{(\underline{\ell},\alpha_0)}(A_{\alpha_0}-t)A\right)wv_{\pi}
\end{multline*}
for each fixed tuple $\underline{\ell}$.

Therefore, again by the linear independence of Jacobi sum operators in Proposition \ref{prop: basis} we deduce that
\begin{equation*}
f_{(\underline{\ell},\alpha_0)}(A_{\alpha_0}-t)=f_{(\underline{\ell},\alpha_0)}(A_{\alpha_0})
\end{equation*}
for all $t\in\F_p$ and each $(\underline{\ell},\alpha_0)$.  We know that if $f\in\F_p[x]$ satisfies $\mathrm{deg}f\leq p-1$, $f(0)=0$ and $f(x-t)=f(x)$ for all $t\in\F_p$ then $f=0$. Thus we deduce that
\begin{equation*}
f_{(\underline{\ell},\alpha_0)}=0
\end{equation*}
for each tuple of integers $\underline{\ell}$, which is a contradiction to (\ref{property}) and so we have $k_{\alpha}=0$ for any $\alpha\in\Phi^\prime$ for each tuple of integers $\underline{k}$ such that $C_{\underline{k},w}\neq 0$.

As a result, we have shown that each vector in $\pi_w^{ U_{\Phi^{\prime}}(\F_p)}$ can be written as certain linear combination of vectors in (\ref{invariant basis}). On the other hand, by Proposition \ref{prop: basis} we know that vectors in (\ref{invariant basis}) are linear independent, and therefore they actually form a basis of $\pi_w^{ U_{\Phi^{\prime}}(\F_p)}$.
\end{proof}

\begin{coro}\label{prop: one}
Let $\mu_{\pi}=(d_1,\cdots,d_n)$ and fix a non-zero vector $v_{\pi}\in\pi^{ U(\F_p),\mu_{\pi}}$.  Given a weight $\mu=(\ell_1,\cdots,\ell_n)\in X_1(T)$ the space
$$\pi_{w_0}^{[ U(\F_p), U(\F_p)],\mu}$$
has a basis whose elements are of the form
$$S_{\underline{k},w_0}v_{\pi}$$
where $\underline{k}=(k_{\alpha})$ satisfies
$$\ell_r\equiv d_{n+1-r}+(1-\lfloor{1}/{r}\rfloor)k_{r-1,r}-(1-\lfloor{1}/{(n+1-r)}\rfloor)k_{r,r+1} \mbox{ mod }(p-1)$$ for all $1\leq r\leq n$
and $k_{\alpha}=0$ if $\alpha\in\Phi^+\setminus\Delta$.
\end{coro}

\begin{proof}
By a special case of Proposition \ref{invariant basis} when $\Phi^{\prime}=\{\sum_{r=i}^j\alpha_r\mid 1\leq i<j\leq n-1\}$, we know that
$$\{S_{\underline{k},w_0}v_{\pi}\mid k_{\alpha}=0 \mbox{ if } \alpha\in\Phi^+\setminus\Delta \}$$
forms a basis of $\pi_{w_0}^{[ U(\F_p), U(\F_p)]}$. On the other hand, we know from Proposition \ref{prop: basis} that the above basis is actually an $ T(\F_p)$-eigenbasis. Therefore the vectors in this basis with a fixed eigencharacter $\mu$ form a basis of the eigensubspace $\pi_{w_0}^{[ U(\F_p), U(\F_p)],\mu}$. Finally, using (i) of the second part of Lemma \ref{lemm: eigen} we conclude this lemma.
\end{proof}

\begin{coro}\label{prop: one1}
Let $\mu_{\pi}=(d_1,d_2,\cdots,d_n)$ and fix a non-zero vector $v_{\pi}\in\pi^{ U(\F_p),\mu_{\pi}}$.  Given a weight $\mu=(\ell_1,\cdots,\ell_n)\in X_1(T)$, the space $$\pi_{w_0}^{U_1(\F_p),\mu}$$ has a basis whose elements are of the form
$$S_{\underline{k},w_0}v_{\pi}$$
where $\underline{k}=(k_{i,j})_{i,j}$ satisfies
$$k_{1,j}\equiv\ell_j-d_{n+1-j}\mbox{ mod }(p-1)$$
for $2\leq j\leq n$ and $k_{i,j}=0$ for all $2\leq i<j\leq n$.
\end{coro}

\begin{proof}
By a special case of Proposition \ref{invariant basis} when $\Phi^{\prime}=\{\sum_{r=i}^j\alpha_r\mid 2\leq i\leq j\leq n-1\}$, we know that
$$\{S_{\underline{k},w_0}v_{\pi}\mid k_{i,j}=0 \mbox{ if } 2\leq i<j\leq n\}$$
forms a basis of $\pi_{w_0}^{U_1(\F_p)}$. On the other hand, we know from Proposition \ref{prop: basis} that the above basis is actually an $ T(\F_p)$-eigenbasis. Therefore the vectors in this basis with a fixed eigencharacter $\mu$ form a basis of the eigensubspace $\pi_{w_0}^{U_1(\F_p),\mu}$. Finally, using (ii) of the second part of Lemma \ref{lemm: eigen} we conclude this lemma.
\end{proof}

\subsection{Main results in characteristic $p$}\label{subsec: main results in char. p}
In this section, we state our main results on certain Jacobi sum operators in characteristic $p$. From now on we fix an $n$-tuple of integers $(a_{n-1},\cdots,a_0)$ which is assumed to be $2n$-generic in the lowest alcove (cf. Definition~\ref{defi: generic on tuples}).

We let
\begin{equation}\label{certain weights}
\left\{
  \begin{array}{ll}
   \mu_1:=(a_1,a_2,\cdots,a_{n-3},a_{n-2},a_{n-1},a_0);  & \hbox{} \\
   \mu'_1:=(a_{n-1},a_0,a_1,a_2,\cdots,a_{n-3},a_{n-2}).  & \hbox{}
  \end{array}
\right.
\end{equation}
We denote their corresponding principal series representations by
$$\pi_1\,\,\mbox{ and }\,\,\pi'_1$$ respectively and their non-zero fixed vectors by
$$v_1\in \pi_1^{ U(\F_p),\mu_1}\,\,\mbox{ and }\,\,v'_1\in (\pi'_1)^{ U(\F_p),\mu'_1}.$$
Finally, we define one more specific weight
\begin{equation}\label{certain weight 2}
\mu^{\ast}:=(a_{n-1}-n+2,a_{n-2},a_{n-3},\cdots,a_2,a_1,a_0+n-2)
\end{equation}
which will play a central role in Corollary~\ref{coro: isomorphism}.

We let $\underline{k}^1=(k^1_{i,j})$ and $\underline{k}^{1,\prime}=(k^{1,\prime}_{i,j})$, where
\begin{equation}\label{main exponent}
\left\{
\begin{array}{ll}
k^1_{i,i+1}\,=&[a_0-a_{n-i}]_1+n-2; \\
k^{1,\prime}_{i,i+1}\,=&[a_{n-i-1}-a_{n-1}]_1+n-2
\end{array}
\right.
\end{equation}
for $1\leq i\leq n-1$ and $k^1_{i,j}=k^{1,\prime}_{i,j}=0$ otherwise, and define two most important Jacobi sum operators $\mathcal{S}_n$ and $\mathcal{S}_n^{\prime}$ to be
\begin{equation}\label{sum}
\mathcal{S}_n:=S_{\underline{k}^1,w_0}\qquad\mbox{ and }\qquad \mathcal{S}_n^{\prime}:=S_{\underline{k}^{1,\prime},w_0}.
\end{equation}
We also let $V_1$ (resp. $V_1^{\prime}$) denote the sub-representation of $\pi_1$ (resp. of $\pi_1^{\prime}$) generated by $\mathcal{S}_nv_1$ (resp. by $\mathcal{S}_n^{\prime}v_1^{\prime}$).

The following theorem, which we usually call the \emph{non-vanishing theorem}, is a technical heart on the local automorphic side. The proofs of this non-vanishing theorem as well as the next theorem, which we usually call the \emph{multiplicity one theorem}, will occupy the following sections.
\begin{theo}\label{theo: main}
Assume that $(a_{n-1},\cdots,a_0)$ is $n$-generic in the lowest alcove.

Then we have
\begin{equation*}
F(\mu^{\ast})\in\mathrm{JH}(V_1)\cap\mathrm{JH}(V_1^{\prime}).
\end{equation*}
\end{theo}
\begin{proof}
This is an immediate consequence of Proposition~\ref{prop: reduction} and Theorem~\ref{weaknonvanishing}.
\end{proof}

We also have the following multiplicity one result.
\begin{theo}\label{conj: mult}
Assume that $(a_{n-1},\cdots,a_0)$ is $2n$-generic in the lowest alcove.

Then $F(\mu^{\ast})$ has multiplicity one in $\pi_1$ (or equivalently in $\pi_1^{\prime}$).
\end{theo}

\begin{proof}
This is a special case of Corollary~\ref{multiplicity one}: replace $\mu^{0,n-1}_{\pi}$ with $\mu^{\ast}$.
\end{proof}


By Theorem \ref{conj: mult}, we can find a unique quotient $\mathcal{V}$ of $\pi_1$ (resp., a unique quotient $\mathcal{V}^{\prime}$ of $\pi_1^{\prime}$) such that $\mathcal{V}$ (resp., $\mathcal{V}^{\prime}$) has $F(\mu^{\ast})$ as socle and $F(\mu_1)$ (resp., $F(\mu'_1)$) as cosocle. Here, by our choice of $\F$, $\mathcal{V}$ and $\mathcal{V}^{\prime}$ always exist.

Now we can state one of our main results on the local automorphic side, and prove it under the theorems above.

\begin{coro}\label{coro: isomorphism}
Assume that $(a_{n-1},\cdots,a_0)$ is $2n$-generic in the lowest alcove.

Then we have
\begin{equation*}
0\neq \mathcal{S}_n\left(\mathcal{V}^{ U(\F_p),\mu_1}\right)\subseteq \mathcal{V}\,\,\mbox{ and }\,\,
0\neq \mathcal{S}^{\prime}_n\left((\mathcal{V}^{\prime})^{U(\F_p),\mu_1^{\prime}}\right)\subseteq \mathcal{V}^{\prime}.
\end{equation*}
\end{coro}

\begin{proof}
On one hand, by definition of $\mathcal{V}$ and $\mathcal{V}^{\prime}$, we have two natural morphisms
\begin{equation*}
V_1\hookrightarrow\pi_1\twoheadrightarrow \mathcal{V}\,\,\mbox{ and }\,\, V_1^{\prime}\hookrightarrow\pi_1^{\prime}\twoheadrightarrow \mathcal{V}^{\prime}.
\end{equation*}
Then by Theorem \ref{conj: mult} and Theorem \ref{theo: main} we know that $V_1$ (resp. $V^{\prime}_1$) has $F(\mu^{\ast})$ as a Jordan--H\"older factor and the image of $V_1$ (resp. $V_1^{\prime}$) in $\mathcal{V}$ (resp. $\mathcal{V}^{\prime}$) is non-zero.
\end{proof}

\begin{rema}
It is known by \cite{HLM} (Proposition 3.1.2 and its proof) that both $\mathcal{V}$ and $\mathcal{V}^{\prime}$ are uniserial of length $2$ if $n=3$. In general, we can prove that the cosocle filtration of both $\mathcal{V}$ and $\mathcal{V}^{\prime}$ has length $\frac{(n-1)(n-2)}{2}+1$. In other words, $\mathcal{V}$ and $\mathcal{V}^{\prime}$ are very big in general. It is a bit surprising at first sight that even though these two representations are big, we can still prove some accurate non-vanishing theorem (Theorem \ref{theo: main}) for Jacobi sums.
\end{rema}

\subsection{Summary of results on Deligne--Lusztig representations}\label{subsec: Deligne Lusztig representations}
In this section, we recall standard facts on Deligne--Lusztig representations and fix their notation that will be used throughout this paper. We closely follow \cite{herzig-duke}. Throughout this article we will only focus the group $ G(\F_p)=\mathrm{GL}_n(\F_p)$, which is the fixed point set of the standard ($p$-power) Frobenius $F$ inside $\mathrm{GL}_n(\overline{\F}_p)$. We will identify a variety over $\overline{\F}_p$ with the set of its $\overline{\F}_p$-rational points for simplicity. Then our fixed maximal torus $\mathbb{T}$ is $F$-stable and split.

To each pair $(\mathbb{T},\theta)$ consisting of an $F$-stable maximal torus $\mathbb{T}$ and a homomorphism $\theta: \mathbb{T}^F\rightarrow\overline{\Q}_p^{\times}$, Deligne--Lusztig \cite{DeligneLusztig} associate a virtual representation $R^{\theta}_{\mathbb{T}}$ of $\mathrm{GL}_n(\F_p)$. (We restrict ourself to $\mathrm{GL}_n(\F_p)$ although the result in \cite{DeligneLusztig} is much more general.) On the other hand, given a pair $(w,\mu)\in W\times X(T)$, one can construct a pair $(\mathbb{T}_w,\theta_{w,\mu})$ by the method in the third paragraph of \cite{herzig-duke} Section 4.1. Then we denote by $R_w(\mu)$ the representation corresponding to $R^{\theta_{w,\mu}}_{\mathbb{T}_w}$ after multiplying a sign. This is the so-called Jantzen parametrization in \cite{Jantzen81} 3.1.

The representations $R^{\theta}_{\mathbb{T}}$ (resp. $R_w(\mu)$) can be isomorphic for different pairs $(\mathbb{T},\theta)$ (resp. $(w,\mu)$), and the explicit relation between is summarized in \cite{herzig-duke} Lemma 4.2. As each $p$-regular character $\mu\in X(T)/(p-1)X(T)$ of $ T(\F_p)$ can be lift to an element in $X^{\rm{reg}}_1(T)$ which is unique up to $(p-1)X_0(T)$, the representation $R_w(\mu)$ is well defined for each $w\in W$ and such a $\mu$.

We recall the notation $\Theta(\theta)$ for a cuspidal representation for $\mathrm{GL}_n(\F_p)$ from Section 2.1 of \cite{Herzigthesis} where $\theta$ is a \emph{primitive} character of $\F_{p^n}^{\times}$ as defined in \cite{herzig-duke}, Section~4.2. We refer further discussion about the basic properties and references of $\Theta(\theta)$ to Section 2.1 of \cite{Herzigthesis}. The relation between the notation $R_w(\mu)$ and the notation $\Theta(\theta)$ is summarized in the Lemma 4.7 of \cite{herzig-duke}. In this paper, we will use the notation $\Theta_m(\theta_m)$ for a cuspidal representation for $\mathrm{GL}_m(\F_p)$ where $\theta$ is a primitive character of $\F_{p^m}^{\times}$.

We emphasize that, as a special case of Lemma 4.7 of \cite{herzig-duke}, we have the natural isomorphism
$$\mathrm{Ind}^{ G(\F_p)}_{ B(\F_p)}\widetilde{\mu}\cong R_1(\mu)$$
for a $p$-regular character $\mu$ of $ T(\F_p)$, where $\widetilde{\mu}$ is the Teichm\"uler lift of $\mu$.

\subsection{Proof of Theorem \ref{conj: mult}}\label{subsec: Proof of multiplicity one}
The main target of this section is to prove Theorem~\ref{conj: mult}. In fact, we prove Corollary~\ref{multiplicity one} which is a generalization of Theorem~\ref{conj: mult}.

We recall some notation from \cite{Jantzen2003}.
We use the notation $\overline{G}_r$ for the $r$-th Frobenius kernel defined in \cite{Jantzen2003} Chapter I.9 as kernel of $r$-th iteration of Frobenius morphism on the group scheme $\overline{G}$ over $\F_p$. We will consider the subgroup scheme $\overline{G}_r\overline{T}$, $\overline{G}_r\overline{B}$, $\overline{G}_r\overline{B}^-$ of $\overline{G}$ in the following. Note that our $\overline{B}$ (resp. $\overline{B}^-$) is denoted by $B^+$ (resp. $B$) in \cite{Jantzen2003} Chapter II. 9.
We define
$$
\begin{array}{lll}
\widehat{Z}^{\prime}_r(\lambda)&:=&\mathrm{ind}^{\overline{G}_r\overline{B}^-}_{\overline{B}^-}\lambda;\\
\widehat{Z}_r(\lambda)&:=&\mathrm{coind}^{\overline{G}_r\overline{B}}_{\overline{B}}\lambda\\
\end{array}$$
where $\mathrm{ind}$ and $\mathrm{coind}$ are defined in I.3.3 and I.8.20 of \cite{Jantzen2003} respectively. By \cite{Jantzen2003} Proposition II.9.6 we know that there exists a simple $\overline{G}_r\overline{T}$-module $\widehat{L}_r(\lambda)$ satisfying
$$\mathrm{soc}_{\overline{G}_r}\left(\widehat{Z}^{\prime}_r(\lambda)\right)\cong \widehat{L}_r(\lambda)\cong \mathrm{cosoc}_{\overline{G}_r}\left(\widehat{Z}_r(\lambda)\right).$$ The properties of $\widehat{Z}^{\prime}_r(\lambda)$ and $\widehat{Z}_r(\lambda)$ are systematically summarized in \cite{Jantzen2003} II.9, and therefore we will frequently refer to results over there.

From now on we assume $r=1$ in this section.

Now we recall several well-known results from \cite{Jantzen81}, \cite{Jantzen84} and \cite{Jantzen2003}. We recall the definition of $\widetilde{W}^{res}$ from (\ref{restricted subset}).

\begin{theo}[\cite{Jantzen81}, Satz 4.3]\label{Jantzendecomposition}
Assume that $\mu+\eta$ is in the lowest restricted alcove and $2n$-generic (Definition~\ref{defi: generic on tuples}). Then we have
\begin{equation*}
\overline{R_w(\mu+\eta)}=\sum_{\substack{\widetilde{w}^{\prime}\in \widetilde{W}^{\rm{res}}\\\nu\in X(T)}}[\widehat{Z}_1(\mu-p\nu+p\eta):\widehat{L}_1(\widetilde{w}^{\prime}\cdot\mu)]F(\widetilde{w}^{\prime}\cdot(\mu+w\nu)).
\end{equation*}
\end{theo}

\begin{prop}\label{multiplicity one Weyl}
Let $\lambda\in X(T)_+$. Suppose $\mu\in X(T)$ is maximal for $\mu\uparrow\lambda$ and $\mu\neq\lambda$. If $\mu\in X(T)_+$ and if $\mu\neq\lambda-p\alpha$ for all $\alpha\in\Phi^+$, then
$$[H^0(\lambda):F(\mu)]=1.$$
\end{prop}

\begin{proof}
This is the Corollary II $6.24$ in \cite{Jantzen2003}.
\end{proof}

If $M$ is an arbitrary $\overline{G}$-module, we use the notation $M^{[1]}$ for the Frobenius twist of $M$ as defined in \cite{Jantzen2003}, I.9.10.
\begin{prop}[\cite{Jantzen2003}, Proposition II. 9.14]\label{Frobeniuskernel}
Let $\lambda\in X(T)_+$. Suppose each composition factor of $\widehat{Z}^{\prime}_1(\lambda)$ has the form $\widehat{L}_1(\mu_0+p\mu_1)$ with $\mu_0\in X_1(T)$ and $\mu_1\in X(T)$ such that
\begin{equation*}
\langle\mu_1+\eta,\beta^{\vee}\rangle\geq 0
\end{equation*}
for all $\beta\in\Delta$. Then $H^0(\lambda)$ has a filtration with factors of the form $F(\mu_0)\otimes H^0(\mu_1)^{[1]}$. Each such module occurs as often as $\widehat{L}_1(\mu_0+p\mu_1)$ occurs in a composition series of $\widehat{Z}^{\prime}_1(\lambda)$.
\end{prop}

\begin{rema}\label{Frobenius twist}
Note that if $\mu_1$ is in the lowest restricted alcove, then $F(\mu_0)\otimes H^0(\mu_1)^{[1]}=F(\mu)$.
\end{rema}

\begin{lemm}[\cite{Jantzen2003}, Lemma II. 9.18 (a)]\label{sufficientlydominant}
Let $\widehat{L}_1(\mu)$ be a composition factor of $\widehat{Z}^{\prime}_1(\lambda)$, and write
$$\lambda+\eta=p\lambda_1+\lambda_0\mbox{ and }\mu=p\mu_1+\mu_0$$
with $\lambda_0,\mu_0\in X_1(T)$ and $\lambda_1,\mu_1\in X(T)$.

If
\begin{equation}\label{dominantenough}
\langle\lambda,\alpha^{\vee}\rangle\geq n-2
\end{equation}
for all $\alpha\in\Phi^+$, then
$$\langle\mu_1+\eta,\beta^{\vee}\rangle\geq 0$$
for all $\beta\in\Phi^+$.
\end{lemm}
\begin{proof}
We only need to mention that $h_{\alpha}=n$ for all $\alpha\in\Phi^+$ and for our group $\overline{G}=\mathrm{GL}_{n/\F_p}$, where $h_{\alpha}$ is defined in \cite{Jantzen2003}, II.9.18.
\end{proof}

We define an element $s_{\alpha,m}\in \widetilde{W}$ by
$$s_{\alpha,m}\cdot \lambda=s_{\alpha}\cdot \lambda+mp\alpha$$
for each $\alpha\in\Phi^+$ and $m\in \Z$.
\begin{theo}\label{multiplicity one of second layer}
Let $\lambda,\mu\in X(T)$ such that
\begin{equation}\label{maximal linkage}
\mu=s_{\alpha,m}\cdot\lambda\,\,\mbox{ and }\,\, mp<\langle\lambda+\eta,\alpha^{\vee}\rangle<(m+1)p.
\end{equation}
Assume further that there exists $\nu\in X(T)$ such that $\lambda+p\nu$ satisfies the condition (\ref{dominantenough}) and that $\nu$ and $\mu_1+\nu$ are in the lowest restricted alcove.

Then we have
\begin{equation*}
[\widehat{Z}_1(\lambda):\widehat{L}_1(\mu)]=1.
\end{equation*}
\end{theo}

\begin{proof}
The condition (\ref{maximal linkage}) ensures that for any fixed $\nu\in X(T)$, $\mu+p\nu$ is maximal for $\mu+p\nu\uparrow\lambda+p\nu$ and $\mu+p\nu\neq\lambda+p\nu$. Notice that we have
$$[\widehat{Z}_1(\lambda):\widehat{L}_1(\mu)]=[\widehat{Z}^{\prime}_1(\lambda):\widehat{L}_1(\mu)]$$
by II $9.2 (3)$ in \cite{Jantzen2003}, as the character of a $\overline{G}_r\overline{T}$-module determine its Jordan--H\"{o}ler factors with multiplicities (or equivalently, determine the semisimplification of the $\overline{G}_r\overline{T}$-module).

By II $9.2 (5)$ and II $9.6 (6)$ in \cite{Jantzen2003} we have
$$[\widehat{Z}^{\prime}_1(\lambda):\widehat{L}_1(\mu)]=[\widehat{Z}^{\prime}_1(\lambda)\otimes p\nu:\widehat{L}_1(\mu)\otimes p\nu]=[\widehat{Z}^{\prime}_1(\lambda+p\nu):\widehat{L}_1(\mu+p\nu)],$$
and thus we may assume that
$$\langle\lambda,\alpha^{\vee}\rangle\geq n-2$$
for all $\alpha\in\Phi^+$ by choosing appropriate $\nu$ (which exists by our assumption) and replacing $\lambda$ by $\lambda+p\nu$ and $\mu$ by $\mu+p\nu$. Then by Lemma \ref{sufficientlydominant} we know that
$$\langle\mu^{\prime}_1+\eta,\beta^{\vee}\rangle\geq 0$$
for any $\mu^{\prime}=p\mu^{\prime}_1+\mu^{\prime}_0$ such that $\widehat{L}_1(\mu^{\prime})$ is a factor of $\widehat{Z}^{\prime}_1(\lambda)$.

Thus by Proposition \ref{Frobeniuskernel}, Proposition \ref{multiplicity one Weyl} and Remark \ref{Frobenius twist} we know that
$$[\widehat{Z}^{\prime}_1(\lambda):\widehat{L}_1(\mu)]=[H^0(\lambda):F(\mu_0)\otimes H^0(\mu_1)^{[1]}]=[H^0(\lambda):F(\mu)]=1$$
which finishes the proof.
\end{proof}

We pick an arbitrary principal series $\pi$ and write
$$\mu_{\pi}=(d_1,\cdots,d_n)$$
For each pair of integers $(i_1,j_1)$ satisfying $0\leq i_1<i_1+1<j_1\leq n-1$, we define
$$\mu_{\pi}^{i_1,j_1}:=(d^{i_1,j_1}_1,\cdots,d^{i_1,j_1}_n)$$
where
$$d^{i_1,j_1}_k=
\left\{
\begin{array}{ll}
d_k&\hbox{if $k\neq n-j_1$ and $k\neq n-i_1$};\\
d_{n-i_1}+j_1-i_1-1& \hbox{if $k=n-i_1$};\\
d_{n-j_1}-j_1+i_1+1& \hbox{if $k=n-j_1$}.
\end{array}
\right.$$

\begin{coro}\label{multiplicity one}
Assume that $\mu_{\pi}$ is $2n$-generic in the lowest alcove (cf. Definition~\ref{defi: generic on tuples}). Then $F(\mu^{i_1,j_1}_{\pi})$ has multiplicity one in $\pi$.
\end{coro}

\begin{proof}
We only need to apply Theorem \ref{multiplicity one of second layer} and Theorem \ref{Jantzendecomposition} to these explicit examples. We will follow the notation of Theorem \ref{Jantzendecomposition}. We fix $w=1$ in Theorem \ref{Jantzendecomposition} and take
$$\mu+\eta:=\mu_{\pi}=\mu^{i_1,j_1}_{\pi}+(j_1-i_1-1)\left(\sum_{r=n-j_1}^{n-1-i_1}\alpha_r\right).$$
We are considering the multiplicity of $F(\mu^{i_1,j_1}_{\pi})$ in $\pi=R_1(\mu+\eta)$ and therefore we take $\widetilde{w}^{\prime}:=1\in \widetilde{W}^{\rm{res}}$ and
$$\nu:=\eta-(j_1-i_1-1)\left(\sum_{r=n-j_1}^{n-1-i_1}\alpha_r\right).$$
By II. 9.2(4) and II.9.16 (4)  in \cite{Jantzen2003} we know that
\begin{equation}\label{equal mult}
[\widehat{Z}_1(\mu-p\nu+p\eta):\widehat{L}_1(\mu)]=[\widehat{Z}_1((n-j_1,n-i_1)\cdot(\mu-p\nu)+p\eta):\widehat{L}_1(\mu)].
\end{equation}

We observe that
\begin{align*}
(n&-j_1,n-i_1)\cdot(\mu-p\nu)+p\eta\\
&=(n-j_1,n-i_1)\cdot\mu+p\left(\eta-(n-j_1,n-i_1)\eta-(j_1-i_1-1)\left(\sum_{r=n-j_1}^{n-1-i_1}\alpha_r\right)\right)\\
&=(n-j_1,n-i_1)\cdot\mu+p\left(\sum_{r=n-j_1}^{n-1-i_1}\alpha_r\right).
\end{align*}
Therefore we have
$$p<\left\langle(n-j_1,n-i_1)\cdot(\mu-p\nu)+p\eta,\sum_{r=n-j_1}^{n-1-i_1}\alpha_r\right\rangle<2p$$
and that
$$\mu=s_{\sum_{r=n-j_1}^{n-1-i_1}\alpha_r,p}\cdot ((n-j_1,n-i_1)\cdot(\mu-p\nu)+p\eta).$$
Moreover, it is easy to see that
$$(n-j_1,n-i_1)\cdot(\mu-p\nu)+p\eta)+p\eta=(n-j_1,n-i_1)\cdot\mu+p\left(\sum_{r=n-j_1}^{n-1-i_1}\alpha_r\right)+p\eta$$
satisfies (\ref{dominantenough}).

Hence, replacing the $\lambda$ and $\mu$ in Theorem \ref{multiplicity one of second layer} by $(n-j_1,n-i_1)\cdot(\mu-p\nu)+p\eta$ and $\mu$ respectively, we conclude that
$$[\widehat{Z}_1((n-j_1,n-i_1)\cdot(\mu-p\nu)+p\eta):\widehat{L}_1(\mu)]=1$$
which finishes the proof by Theorem \ref{Jantzendecomposition} and (\ref{equal mult}).
\end{proof}

\subsection{Some technical formula}\label{subsec: Some technical formula}
In this section, we prove several technical formula that will be used in Section~\ref{subsec: Proof of non-vanishing}. The main results of this section are Lemma~\ref{explicit1}, Proposition~\ref{prop: reduction} and Proposition~\ref{prop: technical formula}.

We define
\begin{equation}\label{sign}
\varepsilon_k:=(-1)^{\frac{k(k-1)}{2}}
\end{equation}

Let $R$ be a $\F_p$-algebra, and $A\in\overline{G}(R)$ a matrix. For $J_1,J_2\subseteq\{1,2,\cdots,n-1,n\}$, we write $A_{J_1,J_2}$ for the submatrix of $A$ consisting of the entries of $A$ at the $(i,j)$-position for $i\in J_1,\,j\in J_2$. For $1\leq i\leq n-1$, we define
$$
\left\{
  \begin{array}{ll}
    J^i_1:=\{n,\cdots,n-i+2,n-i+1\}; & \hbox{} \\
    J^i_2:=\{1,2,\cdots,i-1,i\}; & \hbox{} \\
    J^{i,\prime}_2:=\{1,2,\cdots,i-2,i-1,i+1\}; & \hbox{} \\
    J^{i,\prime\prime}_2:=\{2,3,\cdots,i,i+1\}. & \hbox{}
  \end{array}
\right.
$$
Note that $| J^i_1 |= |J^i_2|=| J^{i,\prime}_2|=| J^{i,\prime\prime}_2|=i$, so that for $1\leq i\leq n-1$ $$D_i:=\varepsilon_i\mathrm{det}(A_{J^i_1, J^i_2}),\,\,D_i^{\prime}:=\varepsilon_i\mathrm{det}(A_{J^i_1, J^{i,\prime}_2}),\,\mbox{ and }\,\,D_i^{\prime\prime}:=\varepsilon_i\mathrm{det}(A_{J^i_1, J^{i,\prime\prime}_2})$$ are well-defined. We also set $D_n:=\varepsilon_n\mathrm{det}(A)$. Hence, $D_i$, $D^{\prime}_i$, and $D^{\prime\prime}_i$ are polynomials over the entries of $A$.

Given a weight $\lambda\in X_+(T)$, we now introduce an explicit model for the representation $H^0(\lambda)$, and then start some explicit calculation. Consider the space of polynomials on $\overline{G}_{/\F_p}$, which is denoted by $\cO(\overline{G})$. The space $\cO(\overline{G})$ has both a left action and a right action of $\overline{B}$ induced by right translation and left translation by $\overline{B}$ on $\overline{G}$ respectively. The fraction field of $\cO(\overline{G})$ is denoted by $\mathcal{M}(\overline{G})$.

Consider the subspace
\begin{equation*}
\cO(\lambda):=\{f\in\cO(\overline{G})\mid f\cdot b=w_0\lambda(b)f\quad \forall b\in\overline{B}\},
\end{equation*}
which has a natural left $\overline{G}$-action by right translation.  As the right action of $\overline{T}$ on $\cO(\overline{G})$ is semisimple (and normalizes $\overline{U}$), we have a decomposition of algebraic representations of $\overline{G}$:
\begin{equation}\label{de}
\cO(\overline{G})^{\overline{U}}:=\{f\in\cO(\overline{G})\mid f\cdot u=f\quad \forall u\in\overline{U}\}=\oplus_{\lambda\in X(\overline{T})}\cO(\lambda).
\end{equation}
It follows from the definition of the dual Weyl module as an algebraic induction that we have a natural isomorphism
\begin{equation}\label{explicit model}
H^0(\lambda)\cong\cO(\lambda).
\end{equation}
Note by \cite{Jantzen2003}, Proposition II.2.6 that $H^0(\lambda)\neq 0$ if and only if $\lambda\in X(T)_+$.

We often write the weight $\lambda$ explicitly as $(d_1,d_2,\cdots,d_n)$ where $d_i\in\Z$ for $1\leq i\leq n$. We will restrict our attention to a $p$-restricted and dominant $\lambda$, i.e., $d_1\geq d_2\geq...\geq d_n$ and $d_{i-1}-d_i< p$ for $2\leq i\leq n$. We recall from the beginning of Section~\ref{sec: local automorphic side} the notation $(\cdot)_{\lambda^{\prime}}$ for a weight space with respect to the weight $\lambda^{\prime}$.
\begin{lemm}\label{explicit}
Let $\lambda=(d_1,d_2,\cdots,d_n)\in X_1(T)$. For $\lambda^{\prime}\in X(T)$, we have
\begin{equation*}
\mathrm{dim}_{\F_p}H^0(\lambda)^{[\overline{U},\overline{U}]}_{\lambda^{\prime}}\leq 1.
\end{equation*}

Moreover, the set of $\lambda^{\prime}$ such that the above space is nontrivial is described explicitly as follows: consider the set $\Sigma^{\prime}$ of $(n-1)$-tuple of integers $\underline{m}=(m_1,...,m_{n-1})$ satisfying $0\leq m_i\leq d_i-d_{i+1}$ for $1\leq i\leq n-1$, and let
\begin{equation*}
v_{\underline{m}}^{\rm{alg},\prime}:=D_n^{d_n}\prod_{i=1}^{n-1}D_i^{d_i-d_{i+1}-m_i}(D_i^{\prime})^{m_i}.
\end{equation*}
Then the set $$\{v_{\underline{m}}^{\rm{alg},\prime}\mid \underline{m}\in\Sigma^{\prime}\}$$
forms a basis for the space $H^0(\lambda)^{[\overline{U},\overline{U}]}$, and the weight of the vector $v^{\rm{alg},\prime}_{\underline{m}}$ is $$(d_1-m_1,d_2+m_1-m_2,...,d_{n-1}+m_{n-2}-m_{n-1},d_n+m_{n-1}).$$
\end{lemm}

\begin{proof}
We define
\begin{equation*}
^{[\overline{U},\overline{U}]}\cO(\overline{G})^{\overline{U}} :=\{f\in\cO(\overline{G})\mid u_1\cdot f=f\cdot u=f\quad \forall u\in\overline{U}\,\,\&\,\,\forall u_1\in[\overline{U},\overline{U}]\}
\end{equation*}
and
\begin{equation*}
^{[\overline{U},\overline{U}]}\mathcal{M}(\overline{G})^{\overline{U}} :=\{f\in\mathcal{M}(\overline{G})\mid u_1\cdot f=f\cdot u=f\quad \forall u\in\overline{U}\,\,\&\,\,\forall u_1\in[\overline{U},\overline{U}]\}.
\end{equation*}
We consider a matrix $A$ such that its entries $A_{i,j}$ are indefinite variables. Then we can write
$$A=A^{(1)}A^{(2)}A^{(3)}$$
such that the entries of $A^{(1)}$, $A^{(2)}$, $A^{(3)}$ are rational functions of $A_{i,j}$ satisfying
$$A^{(1)}_{i,j}=\left\{\begin{array}{lll}
1&\mbox{ if }& i=j;\\
0&\mbox{ if }& i>j,\\
\end{array}\right.$$
$$A^{(2)}_{i,j}=\left\{\begin{array}{lll}
D_j(A)&\mbox{ if }& i+j=n+1;\\
D_{j-1}^{\prime}(A)&\mbox{ if }& i+j=n+2;\\
0&\mbox{ if }& i+j\neq n+1, n+2,\\
\end{array}\right.$$
$$A^{(3)}_{i,j}=\left\{\begin{array}{lll}
1&\mbox{ if }& i=j;\\
0&\mbox{ if }& i>j\mbox{ or }i=j-1.\\
\end{array}\right.$$

For each rational function $f\in ~^{[\overline{U},\overline{U}]}\mathcal{M}(\overline{G})^{\overline{U}}$, we notice that $f$ only depends on $A^{(2)}$, which means that $f$ is rational function of $D_i$ for $1\leq i\leq n$ and $D_i^{\prime}$ for $1\leq i\leq n-1$. In other word, we have
$$^{[\overline{U},\overline{U}]}\mathcal{M}(\overline{G})^{\overline{U}}=\F_p\left(D_1,\cdots,D_n, D_1^{\prime},\cdots, D^{\prime}_{n-1}\right)\subseteq \mathcal{M}(\overline{G}).$$
Then we define
\begin{equation*}
^{[\overline{U},\overline{U}],\lambda^{\prime}}\cO(\overline{G})^{\overline{U},\lambda} :=\{f\in ~ ^{[\overline{U},\overline{U}]}\cO(\overline{G})^{\overline{U}}\mid x\cdot f=\lambda^{\prime}(x)f, \mbox{ and } f\cdot x=\lambda(x)f\quad \forall x\in\overline{T}\}
\end{equation*}
and
\begin{equation*}
^{[\overline{U},\overline{U}],\lambda^{\prime}}\mathcal{M}(\overline{G})^{\overline{U},\lambda} :=\{f\in ~ ^{[\overline{U},\overline{U}]}\mathcal{M}(\overline{G})^{\overline{U}} \mid x\cdot f=\lambda^{\prime}(x)f, \mbox{ and } f\cdot x=\lambda(x)f\quad \forall x\in\overline{T}\}.
\end{equation*}
Note that we have and an obvious inclusion
$$^{[\overline{U},\overline{U}],\lambda^{\prime}}\cO(\overline{G})^{\overline{U},\lambda}\subseteq ~^{[\overline{U},\overline{U}],\lambda^{\prime}}\mathcal{M}(\overline{G})^{\overline{U},\lambda}.$$
We can also identify $^{[\overline{U},\overline{U}],\lambda^{\prime}}\cO(\overline{G})^{\overline{U},\lambda}$ with $H^0(\lambda)^{[\overline{U},\overline{U}]}_{\lambda^{\prime}}$ via the isomorphism (\ref{explicit model}).
By definition of $D_i$ (resp. $D_i^{\prime}$) we know that they are $\overline{T}$-eigenvector with eigencharacter $\sum_{k=1}^i\epsilon_k$ (resp. $\epsilon_{i+1}+\sum_{k=1}^{i-1}\epsilon_k$) for $1\leq i\leq n$ (resp. for $1\leq i\leq n-1$). Therefore we observe that $^{[\overline{U},\overline{U}],\lambda^{\prime}}\mathcal{M}(\overline{G})^{\overline{U},\lambda}$ is one dimensional for any $\lambda,\lambda^{\prime}\in X(T)$ and is spanned by
$$D_n^{d_n}\prod_{i=1}^{n-1}D_i^{d_i-d_{i+1}-m_i}(D_i^{\prime})^{m_i}$$
where $\lambda=(d_1,\cdots,d_n)$ and $\lambda^{\prime}=(d_1-m_1, d_2+m_1-m_2,\cdots,d_{n-1}+m_{n-2}-m_{n-1},d_n+m_{n-1})$. As $\cO(\overline{G})$ is a UFD and $D_i, D_i^{\prime}$ are irreducible, we deduce that
$$D_n^{d_n}\prod_{i=1}^{n-1}D_i^{d_i-d_{i+1}-m_i}(D_i^{\prime})^{m_i}\in \cO(\overline{G})$$
if and only if
$$0\leq m_i\leq d_i-d_{i+1}\mbox{ for all }1\leq i\leq n-1$$
if and only if
$$H^0(\lambda)^{[\overline{U},\overline{U}]}_{\lambda^{\prime}}\neq 0$$
which finishes the proof.
\end{proof}

We recall from Example \ref{good subgroup} the definition of $U_1$ and $\overline{U}_1$.
\begin{lemm}\label{explicit1}
Let $\lambda=(d_1,d_2,\cdots,d_n)\in X_1(T)$. For $\lambda^{\prime}\in X(T)$, we have
\begin{equation*}
\mathrm{dim}_{\F_p}H^0(\lambda)^{\overline{U}_1}_{\lambda^{\prime}}\leq 1.
\end{equation*}

Moreover, the set of $\lambda^{\prime}$ such that the space above is nontrivial is described explicitly as follows: consider the set $\Sigma^{\prime\prime}$ of $(n-1)$-tuple of integers $\underline{m}=(m_1,...,m_{n-1})$ satisfying $m_i\leq d_i-d_{i+1}$ for $1\leq i\leq n-1$, and let
\begin{equation*}
v^{\rm{alg},\prime\prime}_{\underline{m}}:=D_n^{d_n}\prod_{i=1}^{n-1}D_i^{d_i-d_{i+1}-m_i}(D_i^{\prime\prime})^{m_i}.
\end{equation*}
Then the set $$\{v^{\rm{alg},\prime\prime}_{\underline{m}}\mid\underline{m}\in\Sigma^{\prime\prime}\}$$
forms a basis of the space $H^0(\lambda)^{\overline{U}_1}$, and the weight of the vector $v^{\rm{alg},\prime\prime}_{\underline{m}}$ is $$(d_1-\sum_{i=1}^{n-1}m_i,d_2+m_1,...,d_{n-1}+m_{n-2},d_n+m_{n-1}).$$
\end{lemm}

\begin{proof}
Replacing $[\overline{U},\overline{U}]$ by $\overline{U}_1$ in the proof of Lemma~\ref{explicit}, we can define the following objects
$$^{\overline{U}_1}\cO(\overline{G})^{\overline{U}},\quad ^{\overline{U}_1}\mathcal{M}(\overline{G})^{\overline{U}}$$
and
$$^{\overline{U}_1,\lambda^{\prime}}\cO(\overline{G})^{\overline{U},\lambda},\quad ^{\overline{U}_1,\lambda^{\prime}}\mathcal{M}(\overline{G})^{\overline{U},\lambda}$$
for each $\lambda, \lambda^{\prime}\in X(T)$. Note that we have and an obvious inclusion
$$^{\overline{U}_1,\lambda^{\prime}}\cO(\overline{G})^{\overline{U},\lambda}\subseteq ~^{\overline{U}_1,\lambda^{\prime}}\mathcal{M}(\overline{G})^{\overline{U},\lambda}.$$
We can also identify $^{\overline{U}_1,\lambda^{\prime}}\cO(\overline{G})^{\overline{U},\lambda}$ with $H^0(\lambda)^{\overline{U}_1}_{\lambda^{\prime}}$ via the isomorphism (\ref{explicit model}).

We consider a matrix $A$ such that its entries $A_{i,j}$ are indefinite variables. Then we can write
$$A=A^{(1)}A^{(2)}A^{(3)}$$
such that the entries of $A^{(1)}$, $A^{(2)}$, $A^{(3)}$ are rational functions of $A_{i,j}$ satisfying
$$A^{(1)}_{i,j}=\left\{\begin{array}{lll}
1&\mbox{ if }& i=j;\\
0&\mbox{ if }& i>j,\\
\end{array}\right.$$
$$A^{(2)}_{i,j}=\left\{\begin{array}{lll}
D_j(A)&\mbox{ if }& i+j=n+1;\\
D_{j-1}^{\prime\prime}(A)&\mbox{ if }& i=n,\, j>1;\\
0&\mbox{ if }& i+j\neq n+1,\, i<n,\\
\end{array}\right.$$
$$A^{(3)}_{i,j}=\left\{\begin{array}{lll}
1&\mbox{ if }& i=j;\\
0&\mbox{ if }& i>j\mbox{ or }i=1<j.\\
\end{array}\right.$$

For each rational function $f\in ~^{\overline{U}_1}\mathcal{M}(\overline{G})^{\overline{U}}$, we notice that $f$ only depends on $A^{(2)}$, which means that $f$ is rational function of $D_i$ for $1\leq i\leq n$ and $D_i^{\prime\prime}$ for $1\leq i\leq n-1$. In other word, we have
$$^{\overline{U}_1}\mathcal{M}(\overline{G})^{\overline{U}}=\F_p\left(D_1,\cdots,D_n, D_1^{\prime\prime},\cdots, D^{\prime\prime}_{n-1}\right)\subseteq \mathcal{M}(\overline{G}).$$

By definition of $D_i$ (resp. $D_i^{\prime\prime}$) we know that they are $\overline{T}$-eigenvector with eigencharacter $\sum_{k=1}^i\epsilon_k$ (resp. $\sum_{k=2}^{i+1}\epsilon_k$) for $1\leq i\leq n$ (resp. for $1\leq i\leq n-1$). Therefore we observe that $^{\overline{U}_1,\lambda^{\prime}}\mathcal{M}(\overline{G})^{\overline{U},\lambda}$ is one dimensional for any $\lambda,\lambda^{\prime}\in X(T)$ and is spanned by
$$D_n^{d_n}\prod_{i=1}^{n-1}D_i^{d_i-d_{i+1}-m_i}(D_i^{\prime})^{m_i}$$
where $\lambda=(d_1,\cdots,d_n)$ and $\lambda^{\prime}=(d_1-\sum_{i=1}^{n-1}m_i,d_2+m_1,...,d_{n-1}+m_{n-2},d_n+m_{n-1})$. As $\cO(\overline{G})$ is a UFD and $D_i, D_i^{\prime\prime}$ are irreducible, we deduce that
$$D_n^{d_n}\prod_{i=1}^{n-1}D_i^{d_i-d_{i+1}-m_i}(D_i^{\prime})^{m_i}\in \cO(\overline{G})$$
if and only if
$$0\leq m_i\leq d_i-d_{i+1}\mbox{ for all }1\leq i\leq n-1$$
if and only if
$$H^0(\lambda)^{\overline{U}_1}_{\lambda^{\prime}}\neq 0$$
which finishes the proof.
\end{proof}
\begin{rema}
Lemma~\ref{explicit1} essentially describes the decomposition of an irreducible algebraic representation of $\mathrm{GL}_n$ after restricting to a maximal Levi subgroup which is isomorphic to $\mathrm{GL}_1\times\mathrm{GL}_{n-1}$. This classical result is crucial in the proof of Theorem~\ref{weaknonvanishing}.
\end{rema}

Given a principal series $\pi$ and an integer $r$ satisfying $1\leq r\leq n-1$, we consider the morphism $\overline{\mathcal{T}}^{\pi}_{s_r}: \pi\rightarrow\pi^{\prime}$ defined in (\ref{intertwining morphism}). We fix a vector $v_{\pi^{\prime}}\in(\pi^{\prime})^{ U(\F_p),\mu^{s_r}_{\pi}}$ such that
$$\overline{\mathcal{T}}^{\pi}_{s_r}(v_{\pi})=S_{\underline{0},s_r}v_{\pi^{\prime}}.$$
\begin{lemm}\label{lemm: induction process}
Let $1\leq r\leq n-1$, and let $\underline{k}=(k_{i,j})\in\{0,1,\cdots,p-1\}^{|\Phi_{w_0}^+|}$ such that $k_{n-r,n+1-r}<p-1$ and $k_{i,j}=0$ for all $1\leq i<i+1<j\leq n$.

Then we have
\begin{equation*}
\overline{\mathcal{T}}^{\pi}_{s_r}(S_{\underline{k},w_0}v_{\pi})=
\left\{
\begin{array}{ll}
c_{k_{n-r,n+1-r},[d_{r+1}-d_r]_1}S_{\underline{k}^{\prime},w_0}v_{\pi^{\prime}}&\hbox{if $k_{n-r,n+1-r}\geq[d_{r+1}-d_r]_1$;}\\
0&\hbox{if $k_{n-r,n+1-r}<[d_{r+1}-d_r]_1$,}
\end{array}\right.
\end{equation*}
where $\underline{k}^{\prime}=(k_{i,j}^{\prime})$ is defined by
$$k^{\prime}_{i,j}=\left\{
\begin{array}{ll}
k_{n-r,n+1-r}-[d_{r+1}-d_r]_1&\hbox{if $(i,j)=(n-r,n+1-r)$;}\\
k_{i,j}&\hbox{otherwise.}
\end{array}\right.
$$
\end{lemm}

\begin{proof}
Note that we have
\begin{equation*}
\overline{\mathcal{T}}^{\pi}_{s_r}(S_{\underline{k},w_0}v_{\pi})=S_{\underline{k},w_0}\bullet\overline{\mathcal{T}}^{\pi}_{s_r}(v_{\pi})=S_{\underline{k},w_0}\bullet S_{\underline{0},s_r}v_{\pi^{\prime}}
\end{equation*}
and
\begin{equation*}
S_{\underline{k},w_0}\bullet S_{\underline{0},s_r}=\sum_{A\in U(\F_p), t\in\F_p}\left(\prod_{1\leq i<j\leq n}A^{k_{i,j}}_{i,j}\right)Aw_0u_{\alpha_r}(t)s_r.
\end{equation*}
We also have the Bruhat decompositions: if $t=0$
$$Aw_0u_{\alpha_r}(0)s_r=A(w_0s_r)=A^{\prime\prime}w_0s_ru_{\alpha_r}(A_{n-r,n+1-r}),$$
and if $t\neq0$
\begin{equation*}
Aw_0u_{\alpha_r}(t)s_k=Au_{\alpha_{n-r}}(t^{-1})w_0\mathrm{diag}(1,\cdots,t,-t^{-1},\cdots,1)u_{\alpha_r}(t^{-1}).
\end{equation*}

Therefore, we have
\begin{multline*}
S_{\underline{k},w_0}\cdot S_{\underline{0},s_r}v_{\pi^{\prime}}=\sum_{A\in U(\F_p), t\in\F^{\times}_p}\left(\prod_{1\leq i<j\leq n}A^{k_{i,j}}_{i,j}\right)t^{d_{r+1}-d_r}Au_{\alpha_{n-r}}(t^{-1})w_0v_{\pi^{\prime}}\\
+\sum_{A\in U(\F_p)}\left(\prod_{1\leq i<j\leq n}A^{k_{i,j}}_{i,j}\right)Aw_0s_r v_{\pi^{\prime}}.
\end{multline*}
The summation
$$
\sum_{A\in U(\F_p), t\in\F_p}\left(\prod_{1\leq i<j\leq n}A^{k_{i,j}}_{i,j}\right)Aw_0s_r v_{\pi^{\prime}}$$
can be rewritten as
$$
\sum_{A^{\prime\prime}\in U_{w_0s_r}(\F_p)}\left(\prod_{1\leq i<j\leq n,(i,j)\neq (n-r,n+1-r)}A^{k_{i,j}}_{i,j}\right)\left(\sum_{A_{n-r,n+1-r}\in\F_p}A_{n-r,n+1-r}^{k_{n-r,n+1-r}}\right)A^{\prime\prime}w_0s_r v_{\pi^{\prime}}
$$
which is $0$ as we assume $k_{n-r,n+1-r}<p-1$. Hence, we have
\begin{equation*}
S_{\underline{k},w_0}\cdot S_{\underline{0},s_r}v_{\pi^{\prime}}=\sum_{A\in U(\F_p), t\in\F^{\times}_p}\left(\prod_{1\leq i<j\leq n}A^{k_{i,j}}_{i,j}\right)t^{d_{r+1}-d_r}Au_{\alpha_{n-r}}(t^{-1})w_0v_{\pi^{\prime}}.
\end{equation*}

On the other hand, after setting $A^{\prime}=Au_{\alpha_{n-r}}(t^{-1})$ we have
\begin{multline}\label{initial formula}
\sum_{A\in U(\F_p), t\in\F^{\times}_p}\left(\prod_{1\leq i<j\leq n}A^{k_{i,j}}_{i,j}\right)t^{d_{r+1}-d_r}Au_{\alpha_{n-r}}(t^{-1})w_0v_{\pi^{\prime}}\\
=\sum_{A^{\prime}\in U(\F_p), t\in\F^{\times}_p}\left(\prod_{1\leq i\leq n-1, i\neq n-r}(A^{\prime}_{i,i+1})^{k_{i,i+1}}\right) (A^{\prime}_{n-r,n+1-r}-t^{-1})^{k_{n-r,n+1-r}}t^{d_{r+1}-d_r}A^{\prime}w_0v_{\pi^{\prime}}
\end{multline}
since $k_{i,j}=0$ for all $1\leq i<i+1<j\leq n$. Note that for $\ell\neq0$ we have
$$\sum_{t\in\F^{\times}_p}t^{\ell}=
\left\{
\begin{array}{lll}
0&\hbox{if $p-1\nmid\ell$;}\\
-1&\hbox{if $p-1\mid \ell$.}
\end{array}
\right.$$
One can easily check that
\begin{align*}
\sum_{t\in\F^{\times}_p}&(A^{\prime}_{n-r,n+1-r}-t^{-1})^{k_{n-r,n+1-r}}t^{d_{r+1}-d_r}\\
&=\sum_{t\in\F_p}\left(\sum_{s=0}^{k_{n-r,n+1-r}}c_{k_{n-r,n+1-r},s}(-t^{-1})^s(A^{\prime}_{n-r,n+1-r})^{k_{n-r,n+1-r}-s}\right)t^{d_{r+1}-d_r}\\
&=\sum_{s=0}^{k_{n-r,n+1-r}}c_{k_{n-r,n+1-r},s}(-1)^s\left(\sum_{t\in\F_p}t^{d_{r+1}-d_r-s}\right)(A^{\prime}_{n-r,n+1-r})^{k_{n-r,n+1-r}-s},
\end{align*}
which can be rewritten as follows: if $k_{n-r,n+1-r}\geq[d_{r+1}-d_r]_1$ then it is
$$(-1)^{[d_{r+1}-d_r]_1+1}c_{k_{n-r,n+1-r},[d_{r+1}-d_r]_1}(A^{\prime}_{n-r,n+1-r})^{k_{n-r,n+1-r}-[d_{r+1}-d_r]_1}$$
and if $k_{n-r,n+1-r}<[d_{r+1}-d_r]_1$ then it is $0$.
Combining these computations with (\ref{initial formula}) finishes the proof.
\end{proof}

Recall the definition of $\mu_1$ and $\mu_1^{\prime}$ from (\ref{certain weights}). We recursively define sequences of elements in the Weyl group $W$ by
\begin{equation*}
\left\{
  \begin{array}{ll}
    w_1=1,\,\,w_m=s_{n-m}w_{m-1}; & \hbox{} \\
    w_1^{\prime}=1,\,\,w_m^{\prime}=s_mw_{m-1}^{\prime} & \hbox{}
  \end{array}
\right.
\end{equation*}
for all $2\leq m\leq n-1$, where $s_m$ are the reflection of the simple roots $\alpha_m$. We also define sequences of characters of $ T(\F_p)$
\begin{equation*}
\mu_m=\mu_1^{w_m}\,\,\mbox{ and }\,\,\mu_m^{\prime}=(\mu_1^{\prime})^{w_m^{\prime}}
\end{equation*}
for all $1\leq m\leq n-1$, and thus we have sequences of principal series representations
\begin{equation*}
\pi_m:=\mathrm{Ind}^{ G(\F_p)}_{ B(\F_p)}\mu_m\,\,\mbox{ and }\,\,\pi_m^{\prime}:=\mathrm{Ind}^{ G(\F_p)}_{ B(\F_p)}\mu_m^{\prime}
\end{equation*}
for all $1\leq m\leq n-1$. Moreover, we have the following sequences of non-zero morphisms by Frobenius reciprocity:
\begin{equation*}
\overline{\mathcal{T}}^{\pi_m}_{s_{n-m-1}}\,:\,\pi_m\rightarrow\pi_{m+1}\,\,\mbox{ and }\,\,\overline{\mathcal{T}}^{\pi^{\prime}_m}_{s_{m+1}}\,:\,\pi_m^{\prime}\rightarrow\pi_{m+1}^{\prime}
\end{equation*}
for all $1\leq m\leq n-2$.  We fix sequences of non-zero vectors $$v_m\in\pi_m^{ U(\F_p),\mu_m}\,\,\mbox{ and }\,\, v_m^{\prime}\in(\pi_m^{\prime})^{ U(\F_p),\mu_m^{\prime}}$$
such that
\begin{equation*}
\overline{\mathcal{T}}^{\pi_m}_{s_{n-m-1}}(v_m)=S_{\underline{0},s_{n-1-m}}v_{m+1}\,\,\mbox{ and }\,\, \overline{\mathcal{T}}^{\pi^{\prime}_m}_{s_{m+1}}(v^{\prime}_m)=S_{\underline{0},s_{m+1}}v^{\prime}_{m+1}.
\end{equation*}
We also define several families of Jacobi sums:
\begin{equation*}
S_{\underline{k}^m,w_0}\,\,\mbox{ and }\,\, S_{\underline{k}^{m,\prime},w_0}
\end{equation*}
for all integers $m$ with $1\leq m\leq n-1$, where $\underline{k}^m=(k^m_{i,j})$ satisfies
\begin{equation*}
k^m_{i,j}=\left\{
\begin{array}{ll}
n-2+[a_0-a_{n-1}]_1 & \hbox{if $1\leq i=j-1\leq m$;}\\
n-2+[a_0-a_{n-i}]_1 & \hbox{if $m+1\leq i=j-1\leq n-1$;}\\
0&\hbox{otherwise}\\
\end{array}\right.
\end{equation*}
and $\underline{k}^{m,\prime}=(k^{m,\prime}_{i,j})$ satisfies
\begin{equation*}
k^{m,\prime}_{i,j}=\left\{
\begin{array}{ll}
n-2+[a_{n-i-1}-a_{n-1}]_1&\hbox{if $1\leq i=j-1\leq n-m-1$;}\\
n-2+[a_{0}-a_{n-1}]_1&\hbox{if $n-m\leq i=j-1\leq n-1$;}\\
0&\hbox{otherwise.}
\end{array}\right.
\end{equation*}
Finally, we set
\begin{equation}\label{zero}
\left\{
\begin{array}{l}
\mu_0:=\mu_{n-1}=\mu_{n-1}^{\prime};\\
\pi_0:=\pi_{n-1}=\pi_{n-1}^{\prime};\\
\underline{k}^0:=\underline{k}^{n-1}=\underline{k}^{n-1,\prime}.
\end{array}
\right.
\end{equation}

\begin{lemm}\label{image under sequence}
Assume that $(a_{n-1},\cdots,a_0)$ is $n$-generic (Definition~\ref{defi: generic on tuples}). Then we have non-zero scalars $c^m, c^{m,\prime}\in\F_p^{\times}$ such that
$$
\overline{\mathcal{T}}^{\pi_m}_{s_{n-m-1}}(S_{\underline{k}^m,w_0}v_m)=c^mS_{\underline{k}^{m+1},w_0}v_{m+1}
$$
and $$
\overline{\mathcal{T}}^{\pi^{\prime}_m}_{s_{m+1}}(S_{\underline{k}^{m,\prime},w_0}v^{\prime}_m)=c^{m,\prime}S_{\underline{k}^{m+1,\prime},w_0}v^{\prime}_{m+1}
$$
for all $1\leq m\leq n-2$.
\end{lemm}

\begin{proof}
This is a direct corollary of Lemma~\ref{lemm: induction process}. 
If we apply Lemma~\ref{lemm: induction process} to $S_{\underline{k}^m,w_0}v_m$ and $r=n-1-m$, we note that
\begin{equation*}
k_{n-r,n+1-r}=k^m_{m+1,m+2}=[a_0-a_{n-1-m}]_1+n-2>(a_{n-1}-a_{n-1-m})=[d_{r+1}-d_r]_1,
\end{equation*}
and therefore the conclusion follows and we pick
\begin{equation*}
c^m=c_{k_{m+1,m+2},a_{n-1}-a_{n-1-m}}.
\end{equation*}
Similarly, if we apply Lemma~\ref{lemm: induction process} to $S_{\underline{k}^{m,\prime},w_0}v_m$ and $r=m+1$, we note that
\begin{equation*}
k_{n-r,n+1-r}=k^{m,\prime}_{n-1-m,n-m}=[a_m-a_{n-1}]_1+n-2>(a_m-a_0)=[d_{r+1}-d_r]_1,
\end{equation*}
and therefore the conclusion follows by picking $c^{m,\prime}=c_{k_{n-1-m,n-m},a_m-a_0}$.
\end{proof}

We define $V_m$ (resp. $V_m^{\prime}$) to be the sub-representation of $\pi_m$ (resp. of $\pi_m^{\prime}$) generated by $S_{\underline{k}^m,w_0}v_m$ (resp. by $S_{\underline{k}^{m,\prime},w_0}v_m^{\prime}$). By definition we know that $S_{\underline{k}^0,w_0}=S_{\underline{k}^{n-1},w_0}=S_{\underline{k}^{n-1,\prime},w_0}$ and therefore $V_0=V_{n-1}=V_{n-1}^{\prime}$. It follows easily from the definition that
\begin{equation*}
S_{\underline{k}^m,w_0}v_m\in\pi_m^{[ U(\F_p), U(\F_p)],\mu^{\ast}}\,\, \mbox{ and }\,\, S_{\underline{k}^{m,\prime},w_0}v_m\in(\pi_m^{\prime})^{[ U(\F_p), U(\F_p)],\mu^{\ast}}
\end{equation*}
for $1\leq m\leq n-1$.


\begin{prop}\label{prop: reduction}
Assume that $(a_{n-1},\cdots,a_0)$ is $n$-generic (cf. Definition~\ref{defi: generic on tuples}).

If $F(\mu^{\ast})\in\mathrm{JH}(V_0)$, then the statement of Theorem \ref{theo: main} is true.
\end{prop}

\begin{proof}
By Lemma \ref{image under sequence} we know that there are surjections
\begin{equation*}
V_m\twoheadrightarrow V_{m+1}\,\,\mbox{ and }\,\, V^{\prime}_m\twoheadrightarrow V^{\prime}_{m+1}
\end{equation*}
for each $1\leq m\leq n-2$. Therefore we know that
\begin{equation*}
\mathrm{JH}(V_{m+1})\subseteq \mathrm{JH}(V_m)\,\,\mbox{ and }\,\, \mathrm{JH}(V^{\prime}_{m+1})\subseteq\mathrm{JH}(V_m^{\prime})
\end{equation*}
As we have an identification $V_{n-1}=V^{\prime}_{n-1}=V_0$, we deduce that
$$F(\mu^{\ast})\in\mathrm{JH}(V_0)\subseteq\left(\mathrm{JH}(V_1)\cap\mathrm{JH}(V_1^{\prime})\right)$$
which completes the proof.
\end{proof}

From now on, we assume that $(a_{n-1},\cdots,a_0)$ is $n$-generic in the lowest alcove (cf. Definition~\ref{defi: generic on tuples}).
We need to do some elementary calculation of Jacobi sums. For this purpose we need to define the following group operators for $2\leq r\leq n-1$:
\begin{equation*}
X^+_r:=\sum_{t\in \F_p}t^{p-2}u_{\sum_{i=r}^{n-1}\alpha_i}(t)\in\F_p[ G(\F_p)],
\end{equation*}
and similarly
\begin{equation*}
X^-_r:=\sum_{t\in \F_p}t^{p-2}w_0u_{\sum_{i=r}^{n-1}\alpha_i}(t)w_0\in\F_p[ G(\F_p)].
\end{equation*}
We notice that by definition we have the identification $X^+_r=X_{\sum_{i=r}^{n-1}\alpha_i,1}$, where $X_{\sum_{i=r}^{n-1}\alpha_i,1}$ is defined in (\ref{common used operator}).

\begin{lemm}\label{plus formula}
For a tuple of integers $\underline{k}=(k_{i,j})\in\{0,1,\cdots,p-1\}^{|\Phi_{w_0}^{+}|}$, we have
\begin{equation*}
X^+_r\bullet S_{\underline{k},w_0}=k_{r,n}S_{\underline{k}^{r,n,+},w_0}
\end{equation*}
where $\underline{k}^{r,n,+}=(k^{r,n,+}_{i,j})$ satisfies $k^{r,n,+}_{r,n}=k_{r,n}-1$, and $k^{r,n,+}_{i,j}=k_{i,j}$ if $(i,j)\neq (r,n)$.
\end{lemm}
\begin{proof}
This is just a special case of Lemma \ref{main lemma} when $\alpha_0=\sum_{i=r}^n\alpha_i$ and $m=1$.
\end{proof}

For the following lemma, we set
\begin{equation*}
\mathbf{I}:=\{(i_1,i_2,\cdots,i_s)\mid 1\leq i_1<i_2<\cdots<i_s=n\mbox{ for some } 1\leq s\leq n-1\}.
\end{equation*}
to lighten the notation.
\begin{lemm}\label{lemm: determinant}
Let $X=(X_{i,j})_{1\leq i,j\leq n}$ be a matrix satisfying
\begin{equation*}
X_{i,j}=0 \mbox{ if }1\leq j<i\leq n-1.
\end{equation*}

Then the determinant of $X$ is
\begin{equation}\label{determinant formula}
\mathrm{det}(X)=\sum_{(i_1,\cdots,i_s)\in\mathbf{I}}(-1)^{s-1}X_{n,i_1}\left(\prod_{j\neq i_k,\,1\leq k\leq s}X_{j,j}\right)\left(\prod_{k=1}^{s-1}X_{i_k,i_{k+1}}\right).
\end{equation}
\end{lemm}

\begin{proof}
By the definition of determinant we know that
$$\mathrm{det}(X)=\sum_{w\in W}(-1)^{\ell(w)}\prod_{k=1}^nX_{k,w(k)}.$$
From the assumption on $X$, we know that each $w$ that appears in the sum satisfies
\begin{equation}\label{condition on w}
w(k)<k
\end{equation}
for all $2\leq k\leq n-1$.

Assume that $w$ has the decomposition into disjoint cycles
\begin{equation}\label{decomposition of w}
w=(i^1_1,i^1_2,\cdots, i^1_{n_1})\cdots(i^m_1,i^m_2,\cdots, i^m_{n_m})
\end{equation}
where $m$ is the number of disjoint cycles and $n_k\geq 2$ is the length for the $k$-th cycle appearing in the decomposition.



We observe that the largest integer in $\{i^k_j \mid 1\leq j\leq n_k\}$ must be $n$ for each $1\leq k\leq m$ by condition (\ref{condition on w}). Therefore we must have $m=1$ and we can assume without loss of generality that $i^1_{n_1}=n$. It follows from the condition (\ref{condition on w}) that
$$i^1_j<i^1_{j+1}$$
for all $1\leq j\leq n_1-1$. Hence we can set
$$s:=n_1,\quad i_1:=i^1_1,\cdots, i_s:=i^1_{n_1}.$$
We observe that $\ell(w)=s-1$ and the formula (\ref{determinant formula}) follows.
\end{proof}

Recall from the beginning of Section~\ref{subsec: Some technical formula} that we use the notation $A_{J_1,J_2}$ for the submatrix of $A$ consisting of the entries at the $(i,j)$-position with $i\in J_1, j\in J_2$, where $J_1,J_2$ are two subsets of $\{1,2,\cdots,n\}$ with the same cardinality. For a pair of integers $(m,r)$ with $1\leq m\leq r-1\leq n-2$, we let
\begin{equation*}
J_1^{m,r}:=\{m,r+1,r+2,\cdots,n\}.
\end{equation*}
We also recall from (\ref{sign}) that $\varepsilon_k=(-1)^{\frac{k(k-1)}{2}}$.

For a matrix $A\in U(\F_p)$, an element $t\in\F_p$, and a triple of integers $(m,r,\ell)$ satisfying $1\leq m\leq r-1\leq n-2$ and $1\leq \ell\leq n-1$, we define some polynomials as follows:
\begin{equation}\label{determinant polynomial}
\left\{
\begin{array}{ll}
D_{m,r}(A,t)=\varepsilon_{n+1-r}\mathrm{det}\left(w_0u_{\sum_{i=r}^{n-1}\alpha_i}(t)w_0Aw_0\right)_{J_1^{m,r},J_2^{n-r+1}}&\hbox{when $1\leq m\leq r-1$;}\\
D_r^{(\ell)}(A,t)=\varepsilon_{\ell}\mathrm{det}\left(w_0u_{\sum_{i=r}^{n-1}\alpha_i}(t)w_0Aw_0\right)_{J_1^{\ell},J_2^{\ell}} &\hbox{when $1\leq\ell\leq n-r$}
\end{array}\right.
\end{equation}

We define the following subsets of $\mathbf{I}$: for each $1\leq \ell\leq n-1$
$$
\mathbf{I}_{\ell}:=\{(i_1,i_2,\cdots,i_s)\in\mathbf{I}\mid n-\ell+1\leq i_1<i_2<\cdots<i_s=n \mbox{ for some } 1\leq s\leq \ell\}.
$$
Note that we have natural inclusions
$$\mathbf{I}_\ell\subseteq \mathbf{I}_{\ell^{\prime}}\subseteq \mathbf{I}$$
if $1\leq \ell\leq \ell^{\prime}\leq n-1$. In particular, $\mathbf{I}_1$ has a unique element $(n)$.
Similarly, for each $1\leq \ell^{\prime}\leq n-1$ we define
$$
\mathbf{I}^{\ell^{\prime}}:=\{(i_1,i_2,\cdots,i_{s})\mid 1\leq i_1<i_2<\cdots<i_{s-1}\leq n-\ell^{\prime}< i_s=n \mbox{ for some } 1\leq s\leq \ell^{\prime}\},
$$
and we set
$$\mathbf{I}^{\ell^{\prime}}_\ell:=\mathbf{I}_\ell\cap\mathbf{I}^{\ell^{\prime}}$$
for all $1\leq \ell^{\prime}\leq \ell-1\leq n-2$. We often write $\underline{i}=(i_1,\cdots,i_s)$ for an arbitrary element of $\mathbf{I}$, and define the sign of $\underline{i}$ by
$$\varepsilon(\underline{i}):=(-1)^{s}.$$

We emphasize that all the matrices $\left(w_0u_{\sum_{i=r}^{n-1}\alpha_i}(t)w_0Aw_0\right)_{J_1^{m,r},J_2^{n-r+1}}$ for $1\leq m\leq r-1$, and all the matrices $\left(w_0u_{\sum_{i=r}^{n-1}\alpha_i}(t)w_0Aw_0\right)_{J_1^{\ell},J_2^{\ell}}$ for $1\leq \ell\leq n-r$, after multiplying a permutation matrix, satisfy the conditions on the matrix $X$ in Lemma \ref{lemm: determinant}. Hence, by Lemma \ref{lemm: determinant} we notice that
\begin{equation}\label{polynomial of A}
\left\{
\begin{array}{ll}
D_{m,r}(A,t)=A_{m,r}+tf_{m,r}(A)&\hbox{when $1\leq m\leq r-1$;}\\
D_r^{(\ell)}(A,t)=1-tf_{r,n-\ell+1}(A)&\hbox{wehn $1\leq \ell\leq n-r$;}
\end{array}\right.
\end{equation}
where for all $1\leq m\leq r-1$
\begin{equation}\label{explicit polynomial}
f_{m,r}(A):=\sum_{\underline{i}\in\mathbf{I}_{n-r+1}}\left(\varepsilon(\underline{i})A_{m,i_1} \prod_{j=2}^sA_{i_{j-1},i_j}\right).
\end{equation}

Let $(m,r)$ be a tuple of integers with $1\leq m\leq r-1\leq n-2$. Given a tuple of integers $\underline{k}\in\{0,1,\cdots,p-1\}^{|\Phi^+_{w_0}|}$, $\underline{i}=(i_1,i_2,\cdots,i_s)\in\mathbf{I}_{n-r+1}$, and an integer $r^{\prime}$ satisfying $1\leq r^{\prime}\leq r$, we define two tuples of integers $$\underline{k}^{\underline{i},m,r}=(k^{\underline{i},m,r}_{i,j})\in\{0,1,\cdots,p-1\}^{|\Phi^+_{w_0}|}$$
and $$\underline{k}^{\underline{i},m,r,r^{\prime}}=(k^{\underline{i},m,r,r^{\prime}}_{i,j})\in\{0,1,\cdots,p-1\}^{|\Phi^+_{w_0}|}$$ as follows:
\begin{equation*}
k^{\underline{i},m,r}_{i,j}=\left\{\begin{array}{ll}
k_{m,r}-1&\hbox{if $(i,j)=(m,r)$ and $i_1>r$;}\\
k_{m,r}&\hbox{if $(i,j)=(m,r)$ and $i_1=r$;}\\
k_{i,j}+1&\hbox{if $(i,j)=(i_h,i_{h+1})$ for $1\leq h\leq s-1$;}\\
k_{i,j} & \hbox{otherwise }
\end{array}\right.
\end{equation*}
and
\begin{equation*}
k^{\underline{i},m,r,r^{\prime}}_{i,j}=
\left\{\begin{array}{ll}
k^{\underline{i},m,r}_{r^{\prime},n}-1&\hbox{if $(i,j)=(r^{\prime},n)$;}\\
k^{\underline{i},m,r}_{i,j}& \hbox{otherwise.}
\end{array}\right.
\end{equation*}
Finally, we define one more tuple of integers $\underline{k}^{r,+}=(k^{r,+}_{i,j})\in\{0,1,\cdots,p-1\}^{|\Phi^+_{w_0}|}$ by
\begin{equation*}
k^{r,+}_{i,j}:=
\left\{\begin{array}{ll}
k_{r,n}+1&\hbox{if $(i,j)=(r,n)$;}\\
k_{i,j}& \mbox{otherwise.}
\end{array}\right.
\end{equation*}

\begin{lemm}\label{lemm: minus formula}
Fix two integers $r$ and $m$ such that $1\leq m\leq r-1\leq n-2$, and let $\underline{k}=(k_{i,j})\in\{0,1,\cdots,p-1\}^{|\Phi_{w_0}^+|}$. Assume that $k_{i,j}=0$ for $r+1\leq j\leq n-1$ and that $k_{i,r}=0$ for $i\neq m$, and assume further that
\begin{equation*}
a_{n-r}-a_1+[a_1-a_{n-1}-\sum_{i=1}^{n-1}k_{i,n}]_1+k_{m,r}<p.
\end{equation*}

Then we have
\begin{multline*}
X^-_r\bullet S_{\underline{k},w_0}v_0
=k_{m,r}\sum_{\underline{i}\in\mathbf{I}_{n-r}}\varepsilon(\underline{i})S_{\underline{k}^{\underline{i},m,r},w_0}v_0\\
+([a_{n-r}-a_{n-1}-\sum_{i=1}^{n-1}k_{i,n}]_1+k_{m,r})S_{\underline{k}^{r,+},w_0}v_0\\
-\sum_{\ell=2}^{n-r}(a_{n-r}-a_{\ell-1}+k_{m,r})\left(\sum_{\underline{i}\in \mathbf{I}_{\ell}\setminus\mathbf{I}_{\ell-1}}\varepsilon(\underline{i})S_{\underline{k}^{\underline{i},r,n-\ell+1},w_0}v_0\right).
\end{multline*}
\end{lemm}

\begin{proof}
By the definition of $X^-_r$, we have
\begin{equation}\label{basic equality}
X^-_r\bullet S_{\underline{k},w_0}v_0=\sum_{A\in U(\F_p), t\in\F_p}\left(t^{p-2}\left(\prod_{1\leq i<j\leq n}A_{i,j}^{k_{i,j}}\right)w_0u_{\sum_{h=r}^{n-1}\alpha_h}(t)w_0Aw_0\right)v_0.
\end{equation}
For an element $w\in W$, we use $E_w$ to denote the subset of $ U(\F_p)\times\F_p$ consisting of all $(A,t)$ such that
$$w_0u_{\sum_{h=r}^{n-1}\alpha_h}(t)w_0Aw_0\in  B(\F_p)w B(\F_p).$$
It is not difficult to see that if $E_w\neq \varnothing$ then $ww_0(i)=i$ for all $1\leq i\leq r-1$.

We define $M_w$ to be
$$M_w:=\sum_{(A,t)\in E_w}\left(t^{p-2}\left(\prod_{1\leq i<j\leq n}A_{i,j}^{k_{i,j}}\right)w_0u_{\sum_{h=r}^{n-1}\alpha_h}(t)w_0Aw_0\right)v_0.$$
By the definition of $E_w$, we deduce that there exist $A^{\prime}\in U_w(\F_p)$, $A^{\prime\prime}\in  U(\F_p)$, and $T\in  T(\F_p)$ for each given $(A,t)\in E_w$ such that
\begin{equation}\label{B decomposition}
w_0u_{\sum_{h=r}^{n-1}\alpha_h}(t)w_0Aw_0=A^{\prime}w TA^{\prime\prime}.
\end{equation}
Here, the entries of $A^{\prime}$, $T$ and $A^{\prime\prime}$ are rational functions of $t$ and the entries of $A$. We can rewrite (\ref{B decomposition}) as
\begin{equation}\label{B decomposition 1}
w_0u_{\sum_{h=r}^{n-1}\alpha_h}(-t)w_0A^{\prime}w=Aw_0T^{-1}(T(A^{\prime\prime})^{-1}T^{-1})
\end{equation}
In other words, the right side of (\ref{B decomposition 1}) can also be viewed as the Bruhat decomposition of the left side. Therefore the entries of $A$, $T$, $A^{\prime\prime}$ can also be expressed as rational functions of the entries of $A^{\prime}$.

For each $A^{\prime}\in U_w(\F_p)$ and $w\in W$, we define
\begin{equation}\label{general determinant}
\left\{
\begin{array}{ll}
D^w_{m,r}(A^{\prime},t):= \varepsilon_{n+1-r}\mathrm{det}\left(\left(w_0u_{\sum_{i=r}^{n-1}\alpha_i}(t)w_0A^{\prime}w\right)_{J_1^{m,r},J_2^{n-r+1}}\right) &\hbox{when $1\leq m\leq r-1$;}\\
D_r^{w,(\ell)}(A^{\prime},t):=\varepsilon_\ell\mathrm{det} \left(\left(w_0u_{\sum_{i=r}^{n-1}\alpha_i}(t)w_0A^{\prime}w\right)_{J_1^{\ell},J_2^{\ell}}\right) &\hbox{when $1\leq\ell\leq n-r$.}
\end{array}\right.
\end{equation}
Note that if $w=w_0$, then the definition in (\ref{general determinant}) specializes to (\ref{determinant polynomial}). We notice that for a given matrix $A^{\prime}\in U_w(\F_p)$, the equality (\ref{B decomposition 1}) exists if and only if
\begin{equation}\label{nonvanishing determinant}
D_r^{w,(\ell)}(A^{\prime},-t)\neq 0\mbox{ for all }1\leq \ell \leq n-r.
\end{equation}
On the other hand, we also notice that given a matrix $A\in  U(\F_p)$, the equality (\ref{B decomposition 1}) exists if and only if (\ref{nonvanishing determinant}) holds.

By the Bruhat decomposition in (\ref{B decomposition 1}), we have
\begin{equation}\label{T formula}
T^{-1}=\mathrm{diag}\left(D_r^{w,(1)},\frac{D_r^{w,(2)}}{D_r^{w,(1)}},\cdots, \frac{D_r^{w,(n-r)}}{D_r^{w,(n-1-r)}},\frac{1}{D_r^{w,(n-r)}},1,\cdots,1\right)
\end{equation}
in which we write $D_r^{w,(i)}$ for $D_r^{w,(i)}(A^{\prime},-t)$ for brevity. We also have
\begin{equation}\label{change variable}
A_{i,j}=\left\{
\begin{array}{ll}
A_{i,j}^{\prime}&\hbox{if $1\leq i<j\leq n$ and $j\leq r-1$;}\\
D^w_{m,r}(A^{\prime},-t)&\hbox{if $(i,j)=(m,r)$;}\\
\frac{A^{\prime}_{i,n}}{D_r^{w,(1)}(A^{\prime},-t)}&\hbox{if $1\leq i\leq n-1$ and $j=n$.}\\
\end{array}\right.
\end{equation}

We apply (\ref{B decomposition}), (\ref{change variable}) and (\ref{T formula}) to $M_w$ and get
\begin{equation}\label{after change}
M_w=\sum_{(A,t)\in E_w}\left(F(A^{\prime},w,t)\left(\prod_{\substack{1\leq i<j\leq n\\j\leq r\mbox{ or }j=n}}(A^{\prime}_{i,j})^{k_{i,j}}\right)A^{\prime}w_0\right)v_0
\end{equation}
where
$$F(A^{\prime},w,t):= t^{p-2}\left((D^w_{m,r})^{k_{m,r}}(D_r^{w,(1)})^{a_1-a_{n-1}- \sum_{i=1}^{n-1}k_{i,n}}\prod_{s=2}^{n-r}(D_r^{w,(s)})^{a_s-a_{s-1}}\right)$$
in which we let $D^w_{m,r}:=D^w_{m,r}(A^{\prime},-t)$ and $D_r^{w,(s)}:=D_r^{w,(s)}(A^{\prime},-t)$ for brevity. We have discussed in (\ref{nonvanishing determinant}) that $(A,t)\in E_w$ is equivalent to $(A^{\prime},t)\in U_w(\F_p)\times\F_p$ satisfying the conditions in (\ref{nonvanishing determinant}). As each $D_r^{w,(s)}(A^{\prime},-t)$ appears in $F(A^{\prime},w,t)$ with a positive power, we can automatically drop the condition (\ref{nonvanishing determinant}) and reach
\begin{equation}\label{after change 1}
M_w=\sum_{(A,t)\in \overline{U}_w(\F_0)\times\F_p}\left(F(A^{\prime},w,t)\left(\prod_{\substack{1\leq i<j\leq n\\j\leq r\mbox{ or }j=n}}(A^{\prime}_{i,j})^{k_{i,j}}\right)A^{\prime}w_0\right)v_0.
\end{equation}
If $w\neq w_0$, then there exist a unique integer $i_0$ satisfying $r\leq i_0\leq n$ such that $ww_0(i_0)<i_0$ but $ww_0(i)=i$ for all $i_0+1\leq i\leq n$.

By applying Lemma \ref{lemm: determinant} to $D_r^{w,(n+1-i_0)}(A^{\prime},-t)$ (as $(w_0u_{\sum_{i=r}^{n-1}\alpha_i}(t)w_0A^{\prime}w)_{J_1^{h,n-\ell},J_2^{\ell}}$ satisfy the condition of Lemma \ref{lemm: determinant} after multiplying a permutation matrix), we deduce that
$$D_r^{w,(n+1-i_0)}(A^{\prime},-t)=tf(A^{\prime})$$
where $f(A^{\prime})$ is certain polynomial of entries of $A^{\prime}$.

Now we consider $F(A^{\prime},w,t)$ as a polynomial of $t$. The minimal degree of monomials of $t$ appearing in $F(A^{\prime},w,t)$ is at least
$$p-2+a_{n+1-i_0}-a_{n-i_0}>p-1,$$
and the maximal degree of monomials of $t$ appearing in $F(A^{\prime},w,t)$ is
\begin{align*}
p-2&+k_{m,r}+[a_1-a_{n-1}-\sum_{i=1}^{n-1}k_{i,n}]_1+\sum_{s=2}^{n-r}a_s-a_{s-1}\\
&=p-2+k_{m,r}+[a_1-a_{n-1}-\sum_{i=1}^{n-1}k_{i,n}]_1+a_{n-r}-a_1\\
&<2(p-1).
\end{align*}
As a result, the degree of each monomials of $t$ in $F(A^{\prime},w,t)$ is not divisible by $p-1$. Hence, we conclude that
\begin{equation*}
M_w=0\mbox{ for all }w\neq w_0
\end{equation*}
as we know that $\sum_{t\in\F_p}t^k\neq 0$ if and only if $p-1\mid k$ and $k\neq 0$.

Finally, we calculate $M_{w_0}$ explicitly using (\ref{after change 1}). Indeed, by applying (\ref{polynomial of A}), the monomials of $t$ appearing in $F(A^{\prime},w_0,t)$ is nothing else than
\begin{equation*}
t^{p-1}\left(-k_{m,r}f_{m,r}(A^{\prime})+[a_1-a_{n-1}-\sum_{i=1}^{n-1}k_{i,n}]_1f_{r,n}(A^{\prime})+\sum_{s=2}^{n-r}(a_s-a_{s-1})f_{r,n+1-s}(A^{\prime})\right).
\end{equation*}
As $\sum_{t\in\F_p}t^{p-1}=-1$, we conclude that
\begin{equation}\label{intermediate formula}
X^-_r\bullet S_{\underline{k},w_0}v_0=M_{w_0}=\sum_{A^{\prime}\in U(\F_p)}\left(F_0(A^{\prime})\left(\prod_{\substack{1\leq i<j\leq n\\j\leq r\mbox{ or }j=n}}(A^{\prime}_{i,j})^{k_{i,j}}\right)A^{\prime}w_0\right)v_0
\end{equation}
where
$$F_0(A^{\prime}):=k_{m,r}f_{m,r}(A^{\prime})-[a_1-a_{n-1}-\sum_{i=1}^{n-1}k_{i,n}]_1f_{r,n}(A^{\prime})-\sum_{s=2}^{n-r}(a_s-a_{s-1})f_{r,n+1-s}(A^{\prime}).$$
Recalling the explicit formula of $f_{m,r}$ and $f_{r,n+1-s}$ for $1\leq s\leq n-r$ from (\ref{explicit polynomial}) and then rewriting (\ref{intermediate formula}) as a sum of distinct monomials of entries of $A^{\prime}$ finishes the proof.
\end{proof}

\begin{prop}\label{prop: technical formula}
Keep the assumptions and the notation of Lemma~\ref{lemm: minus formula}.

Then we have
\begin{multline*}
X^+_r\bullet X^-_r\bullet S_{\underline{k},w_0}v_0
=k_{m,r}k_{r,n}\sum_{\underline{i}\in\mathbf{I}_{n-r}}\varepsilon(\underline{i})S_{\underline{k}^{\underline{i},m,r},w_0}v_0\\
+(k_{r,n}+1)([a_{n-r}-a_{n-1}-\sum_{i=1}^{n-1}k_{i,n}]_1+k_{m,r})S_{\underline{k}^{r,+},w_0}v_0\\
- k_{r,n}\sum_{\ell=2}^{n-r}(a_{n-r}-a_{\ell-1}+k_{m,r})\left(\sum_{\underline{i}\in\mathbf{I}_{\ell} \setminus\mathbf{I}_{\ell-1}}\varepsilon(\underline{i})S_{\underline{k}^{\underline{i},r,n-\ell+1},w_0}v_0\right).
\end{multline*}
\end{prop}

\begin{proof}
This is just a direct combination of Lemma \ref{lemm: minus formula} and Lemma \ref{plus formula}.
\end{proof}

\subsection{Proof of Theorem~\ref{theo: main}}\label{subsec: Proof of non-vanishing}
The main target of this section is to prove Theorem \ref{weaknonvanishing}, which immediately implies Theorem~\ref{theo: main} by Proposition~\ref{prop: reduction}. We start this section by introducing some notation.

We first define a subset $\Lambda_{w_0}$ of $\{0,\cdots,p-1\}^{|\Phi_{w_0}^+|}$ consisting of the tuples $\underline{k}=(k_{i,j})_{i,j}$ such that for each $1\leq r\leq n-1$
$$
\sum_{1\leq i\leq r< j\leq n}k_{i,j}=[a_0-a_{n-1}]_1+n-2.
$$
Note that the set $\Lambda_{w_0}$ embeds into $\pi_0$ by sending $\underline{k}$ to $S_{\underline{k},w_0}v_0$. Moreover, this family of vectors $\{S_{\underline{k},w_0}v_0\mid \underline{k}\in\Lambda_{w_0}\}$ shares the same eigencharacter by Lemma \ref{lemm: eigen}.

We define $\underline{k}^{\sharp}\in\Lambda_{w_0}$ where $\underline{k}^{\sharp}=(k^{\sharp}_{i,j})$ is defined by $k^{\sharp}_{1,n}=[a_0-a_{n-1}]_1+n-2$ and $k^{\sharp}_{i,j}=0$ otherwise. We define $V^{\sharp}$ to be the subrepresentation of $\pi_0$ generated by $S_{\underline{k}^{\sharp},w_0}v_0$. We also need to define several useful elements and subsets of $\Lambda_{w_0}$.  For each $1\leq r\leq n-1$, we define $\underline{k}^{\sharp,r}\in\Lambda_{w_0}$ where $\underline{k}^{\sharp,r}=(k^{\sharp,r}_{i,j})$ is defined by
\begin{equation*}
k^{\sharp,r}_{i,j}:=
\left\{
  \begin{array}{ll}
    n-2+[a_0-a_{n-1}]_1 & \hbox{if $2\leq j=i+1\leq r$;} \\
    n-2+[a_0-a_{n-1}]_1 & \hbox{if $(i,j)=(r,n)$;} \\
    0 & \hbox{otherwise.}
  \end{array}
\right.
\end{equation*}
In particular, we have
\begin{equation}\label{first identification}
\underline{k}^{\sharp,1}=\underline{k}^{\sharp}\mbox{ and }\underline{k}^{\sharp,n-1}=\underline{k}^{0}
\end{equation}
where $\underline{k}^{0}$ is defined in (\ref{zero}).

For each $1\leq r\leq n-2$ and $0\leq s\leq [a_0-a_{n-1}]_1+n-2$, we define a tuple $\underline{k}^{\sharp,r,s}\in\Lambda_{w_0}$ as follows:
\begin{equation*}
k^{\sharp,r,s}_{i,j}=\left\{
\begin{array}{ll}
n-2+[a_0-a_{n-1}]& \hbox{if $2\leq j=i+1\leq r$;}\\
n-2+[a_0-a_{n-1}]_1-s& \hbox{if $(i,j)=(r,r+1)$;}\\
s& \hbox{if $(i,j)=(r,n)$;}\\
n-2+[a_0-a_{n-1}]_1-s& \hbox{if $(i,j)=(r+1,n)$;}\\
0&\hbox{otherwise.}
\end{array}\right.
\end{equation*}
In particular, we have
\begin{equation}\label{second identification}
\underline{k}^{\sharp,r,0}=\underline{k}^{\sharp,r+1}\mbox{ and }\underline{k}^{\sharp,r,[a_0-a_{n-1}]_1+n-2}=\underline{k}^{\sharp,r}
\end{equation}
for each $1\leq r\leq n-2$.

We now introduce the rough idea of the proof of Theorem~\ref{weaknonvanishing}. The first obstacle to generalize the method of Proposition 3.1.2 in \cite{HLM} is that $V_0$ does not admit a structure as $\overline{G}$-representation in general. Our method to resolve this difficulty is to replace $S_{\underline{k}^0,w_0}v_0$ by $S_{\underline{k}^{\sharp},w_0}v_0$. We prove in Proposition \ref{identification} that $V^{\sharp}$ admits a structure as $\overline{G}$-representation and actually can be identified with a dual Weyl module $H^0(\mu_0^{w_0})$. (The notation $\mu_0^{w_0}$ will be clear in the following.) Now it remains to prove that
\begin{equation}\label{main inclusion}
S_{\underline{k}^{\sharp},w_0}v_0\in V_0
\end{equation}
to deduce Theorem \ref{weaknonvanishing}. We will prove in Proposition \ref{construct next vector} that
$$S_{\underline{k}^{\sharp,r,s-1},w_0}v_0\in V_0 \Longrightarrow S_{\underline{k}^{\sharp,r,s},w_0}v_0\in V_0$$
for each $1\leq r\leq n-2$ and $1\leq s\leq [a_0-a_{n-1}]_1+n-2$. As a result, we can thus pass from $S_{\underline{k}^0,w_0}v_0\in V_0$ to $S_{\underline{k}^{\sharp,r},w_0}v_0\in V_0$ for $r=n-2,n-3,\cdots,1$. The identification $\underline{k}^{\sharp}=\underline{k}^{\sharp,1}$ thus gives us (\ref{main inclusion}).


We firstly state three direct Corollaries of Proposition~\ref{prop: technical formula}. It is easy to check that each tuple $\underline{k}$ appearing in the following Corollaries satisfies the assumption in Proposition \ref{prop: technical formula}.
\begin{coro}\label{first formula}
For each $2\leq r\leq n-1$ and $0\leq s\leq [a_0-a_{n-1}]_1+n-3$, we have
\begin{align*}
X^+_r&\bullet X^-_r\bullet S_{\underline{k}^{\sharp,r-1,s},w_0}v_0
=([a_0-a_{n-1}]_1+n-2-s)^2\sum_{\underline{i}\in\mathbf{I}_{n-r}}\varepsilon(\underline{i})S_{(\underline{k}^{\sharp,r-1,s})^{\underline{i},m,r,r},w_0}v_0\\
&+([a_1-a_{n-1}]_1-s)([a_0-a_{n-1}]_1+n-1-s)S_{\underline{k}^{r-1,s},w_0}v_0\\
&-([a_0-a_{n-1}]_1+n-2-s)\sum_{\ell=2}^{n-r}\left(a_{n-r}-a_{\ell-1}+[a_0-a_{n-1}]_1+n-2-s\right)\\
&\qquad\qquad\qquad\qquad\qquad\qquad\cdot\left(\sum_{\underline{i}\in\mathbf{I}_{\ell}\setminus\mathbf{I}_{\ell-1}}\varepsilon(\underline{i})S_{(\underline{k}^{\sharp,r-1,s})^{\underline{i},r,n-\ell+1,r},w_0}v_0\right).
\end{align*}
\end{coro}

\begin{coro}\label{second formula}
Fix two integers $r$ and $m$ such that $1\leq m\leq r-1\leq n-2$, and let $\underline{k}=(k_{i,j})$ be a tuple of integers in $\Lambda_{w_0}$ such that
\begin{equation*}
k_{i,j}=\left\{
\begin{array}{ll}
0& \hbox{if $r+1\leq j\leq n-1$;}\\
0& \hbox{if $i\neq m$ and $j=r$;}\\
0& \hbox{if $r+1\leq i\leq n-1$ and $j=n$;}\\
1& \hbox{if $(i,j)=(m,r)$;}\\
1&\hbox{if $(i,j)=(r,n)$. }
\end{array}\right.
\end{equation*}
Then we have
\begin{multline*}
X^+_r\bullet X^-_r\bullet S_{\underline{k},w_0}v_0
=\sum_{\underline{i}\in\mathbf{I}_{n-r}}\varepsilon(\underline{i})S_{\underline{k}^{\underline{i},m,r,r},w_0}v_0 +2(a_{n-r}-a_0-n+3)S_{\underline{k},w_0}v_0\\
-\sum_{\ell=2}^{n-r}(a_{n-r}-a_{\ell-1}+1)\left(\sum_{\underline{i}\in\mathbf{I}_{\ell}\setminus\mathbf{I}_{\ell-1}}\varepsilon(\underline{i})S_{\underline{k}^{\underline{i},r,n-\ell+1,r},w_0}v_0\right).
\end{multline*}
\end{coro}

\begin{coro}\label{third formula}
Fix two integers $r$ and $m$ such that $1\leq m\leq r-1\leq n-2$, and let $\underline{k}=(k_{i,j})$ be a tuple of integers in $\Lambda_{w_0}$ such that
\begin{equation*}
k_{i,j}=\left\{
\begin{array}{ll}
0& \hbox{if $r\leq j\leq n-1$;}\\
0&\hbox{if $r\leq i\leq n-1$ and $j=n$. }
\end{array}\right.
\end{equation*}
Then we have
\begin{multline*}
X^+_r\bullet X^-_r\bullet S_{\underline{k},w_0}v_0
=(a_{n-r}-a_0-n+2)S_{\underline{k},w_0}v_0\\
-\sum_{\ell=2}^{n-r}(a_{n-r}-a_{\ell-1}+1)\left(\sum_{\underline{i}\in\mathbf{I}_{\ell}\setminus\mathbf{I}_{\ell-1}}\varepsilon(\underline{i})S_{\underline{k}^{\underline{i},r,n-\ell+1,r},w_0}v_0\right).
\end{multline*}
\end{coro}

We now define the following constants in $\F_p$:
\begin{equation*}
\left\{\begin{array}{lll}
\mathbf{c}_\ell&:=&\prod_{k=1}^{\ell-1}(a_k-a_0-n+2+k)^{2^{\ell-k-1}};\\
\mathbf{c}^{\prime}_\ell&:=&(a_\ell-a_0-n+3+\ell)\mathbf{c}_\ell\\
\end{array}\right.
\end{equation*}
for all $1\leq \ell\leq n-1$ where we understand $\mathbf{c}_1$ to be $1$. As the tuple $(a_{n-1},\cdots,a_0)$ is $n$-generic in the lowest alcove, we notice that
$\mathbf{c}_\ell\neq0\neq \mathbf{c}^{\prime}_\ell$ for all $1\leq \ell\leq n-1.$  By definition of $\mathbf{c}_k$ and $\mathbf{c}^{\prime}_k$ one can also easily check that
\begin{equation}\label{equation of c_l}
\prod_{k=1}^{\ell-1}(\mathbf{c}^{\prime}_k-\mathbf{c}_k)=\mathbf{c}_{\ell}.
\end{equation}
We also define inductively the constants: for each $1\leq\ell\leq n-1$
\begin{equation*}
\mathbf{d}_{\ell,\ell^{\prime}}:=
\left\{\begin{array}{ll}
2(a_{\ell}-a_0-n+3)&\hbox{if $\ell^{\prime}=0$;}\\
\mathbf{c}_{\ell^{\prime}}^{\prime}\mathbf{d}_{\ell,\ell^{\prime}-1}-(a_{\ell}-a_{\ell^{\prime}}+1)\mathbf{c}_{\ell^{\prime}}\prod_{k=1}^{\ell^{\prime}-1}(\mathbf{c}^{\prime}_k-\mathbf{c}_k)& \hbox{if $1\leq \ell^{\prime}\leq \ell-1$.}
\end{array}\right.
\end{equation*}

We further define inductively a sequence of group operators $\mathcal{Z}_\ell$ as follows:
\begin{equation*}
\mathcal{Z}_1:=\mathbf{d}_{1,0}\mathrm{Id}-X^+_{n-1}\bullet X^-_{n-1}\in\F_p[ G(\F_p)]
\end{equation*}
and
\begin{equation*}
\mathcal{Z}_\ell:=\mathbf{d}_{\ell,\ell-1}\mathrm{Id}-\left(\mathcal{Z}_{\ell-1}\bullet\cdots\bullet\mathcal{Z}_1\bullet X^+_{n-\ell}\bullet X^-_{n-\ell}\right)\in\F_p[ G(\F_p)]
\end{equation*}
for each $2\leq\ell\leq n-2$.

\begin{lemm}\label{equality of constants}
For $1\leq\ell\leq n-1$, we have the identity
$$\mathbf{d}_{\ell,\ell-1}=(a_\ell-a_0-n+2)\left(\prod_{k=1}^{\ell-1}\mathbf{c}^{\prime}_k\right)+\mathbf{c}^{\prime}_\ell.$$
\end{lemm}

\begin{proof}
During the proof of this lemma, we will keep using the following obvious identity with two variables
\begin{equation}\label{obvious identity}
ab=(a+1)(b-1)+a-b+1
\end{equation}
By definition of $\mathbf{d}_{\ell,\ell-1}$ we know that
$$\mathbf{d}_{\ell,\ell-1}=2(a_{\ell}-a_0-n+3)\prod_{k=1}^{\ell-1}\mathbf{c}^{\prime}_k-\sum_{\ell^{\prime}=1}^{\ell-1}\left((a_\ell-a_{\ell^{\prime}}+1)\mathbf{c}_{\ell^{\prime}}\left(\prod_{k=1}^{\ell^{\prime}-1}(\mathbf{c}^{\prime}_k-\mathbf{c}_k)\right)\left(\prod_{k=\ell^{\prime}+1}^{\ell-1}\mathbf{c}^{\prime}_k\right)\right)$$
and therefore
\begin{multline*}
\mathbf{d}_{\ell,\ell-1}-(a_\ell-a_0-n+2)\left(\prod_{k=1}^{\ell-1}\mathbf{c}^{\prime}_k\right)
=(a_{\ell}-a_0-n+4)\prod_{k=1}^{\ell-1}\mathbf{c}^{\prime}_k\\ -\sum_{\ell^{\prime}=1}^{\ell-1}\left((a_\ell-a_{\ell^{\prime}}+1)\mathbf{c}_{\ell^{\prime}} \left(\prod_{k=1}^{\ell^{\prime}-1}(\mathbf{c}^{\prime}_k-\mathbf{c}_k)\right) \left(\prod_{k=\ell^{\prime}+1}^{\ell-1}\mathbf{c}^{\prime}_k\right)\right).
\end{multline*}

Now we prove inductively that for each $1\leq j\leq \ell-1$
\begin{multline}\label{small induction step}
\mathbf{d}_{\ell,\ell-1}-(a_\ell-a_0-n+2)\left(\prod_{k=1}^{\ell-1}\mathbf{c}^{\prime}_k\right) =(a_{\ell}-a_0-n+3+j)\left(\prod_{k=1}^{j-1}(\mathbf{c}^{\prime}_k-\mathbf{c}_k)\right) \left(\prod_{k=j}^{\ell-1}\mathbf{c}^{\prime}_k\right)\\ -\sum_{\ell^{\prime}=j}^{\ell-1}\left((a_\ell-a_{\ell^{\prime}}+1)\mathbf{c}_{\ell^{\prime}}\left(\prod_{k=1}^{\ell^{\prime}-1}(\mathbf{c}^{\prime}_k-\mathbf{c}_k)\right)\left(\prod_{k=\ell^{\prime}+1}^{\ell-1}\mathbf{c}^{\prime}_k\right)\right).
\end{multline}

By the identity (\ref{obvious identity}), one can easily deduce that
\begin{align*}
(a_{\ell}&-a_0-n+3+j)\mathbf{c}^{\prime}_j - (a_\ell-a_{j}+1)\mathbf{c}_j\\
&=\left[(a_{\ell}-a_0-n+3+j)(a_j-a_0-n+3+j)-(a_\ell-a_{j}+1)\right]\mathbf{c}_j \\
&=(a_{\ell}-a_0-n+4+j)(a_j-a_0-n+2+j)\mathbf{c}_j\\
&=(a_{\ell}-a_0-n+4+j)(\mathbf{c}^{\prime}_j-\mathbf{c}_j).
\end{align*}
Hence, we get the identity:
\begin{multline}\label{induction for c}
\left[(a_{\ell}-a_0-n+3+j)\mathbf{c}^{\prime}_j- (a_\ell-a_{j}+1)\mathbf{c}_j\right]\left(\prod_{k=j+1}^{\ell-1}\mathbf{c}^{\prime}_k\right)\left(\prod_{k=1}^{j-1}(\mathbf{c}^{\prime}_k-\mathbf{c}_k)\right)\\
=(a_{\ell}-a_0-n+4+j)\left(\prod_{k=1}^{j}(\mathbf{c}^{\prime}_k-\mathbf{c}_k)\right)\left(\prod_{k=j+1}^{\ell-1}\mathbf{c}^{\prime}_k\right).
\end{multline}

Thus, if the equation (\ref{small induction step}) holds for $j$, we can deduce that it also holds for $j+1$. By taking $j=\ell-1$ and using the equation (\ref{induction for c}) once more, we can deduce that
$$\mathbf{d}_{\ell,\ell-1}-(a_\ell-a_0-n+2)\left(\prod_{k=1}^{\ell-1}\mathbf{c}^{\prime}_k\right)=(a_{\ell}-a_0-n+3+\ell)\left(\prod_{k=1}^{\ell-1}(\mathbf{c}^{\prime}_k-\mathbf{c}_k)\right).$$
Hence, by the equation (\ref{equation of c_l}), one finishes the proof.
\end{proof}

\begin{prop}\label{prop: operators}
Fix two integers $r$ and $m$ such that $1\leq m\leq r-1\leq n-2$.
\begin{enumerate}
\item Let $\underline{k}=(k_{i,j})$ be as in Corollary~\ref{second formula}. Then we have
    \begin{equation}\label{first part}
    \mathcal{Z}_{n-r}\bullet S_{\underline{k},w_0}=\mathbf{c}_{n-r}S_{\underline{k}^{\prime},w_0}
    \end{equation}
    where $\underline{k}'=(k'_{i,j})$ is defined as follows:
    \begin{equation*}
    k^{\prime}_{i,j}:=
    \left\{\begin{array}{ll}
    0& \hbox{if $(i,j)=(m,r)$ or $(i,j)=(r,n)$;}\\
    1& \hbox{if $(i,j)=(m,n)$;}\\
    k_{i,j}&\hbox{otherwise.}\\
    \end{array}\right.
    \end{equation*}
\item Let $\underline{k}=(k_{i,j})$ be as in Corollary~\ref{third formula}. Then we have
    \begin{equation}\label{second part}
    \mathcal{Z}_{n-r}\bullet S_{\underline{k},w_0}=\mathbf{c}^{\prime}_{n-r}S_{\underline{k},w_0}.
    \end{equation}
\end{enumerate}
\end{prop}

We prove this proposition by a series of lemmas.
\begin{lemm}\label{initial step for operator}
Proposition \ref{prop: operators} is true for $r=n-1$.
\end{lemm}
\begin{proof}
For part (i) of Proposition \ref{prop: operators}, by applying Corollary \ref{second formula} to the case $r=n-1$ we deduce that
$$X^+_{n-1}\bullet X^-_{n-1}\bullet S_{\underline{k},w_0}v_0=2(a_1-a_0-n+3)S_{\underline{k},w_0}v_0-S_{\underline{k}^{\underline{i}_0,m,n-1,n-1},w_0}v_0$$
where $\underline{i}_0=\{n-1,n\}$. Hence, part (i) of the proposition follows directly from the definition of $\mathcal{Z}_1$ and $\mathbf{c}_1$.

For part (ii) of Proposition \ref{prop: operators}, again by Corollary \ref{third formula} to the case $r=n-1$ we deduce that
$$X^+_{n-1}\bullet X^-_{n-1}\bullet S_{\underline{k},w_0}v_0=(a_1-a_0-n+2)S_{\underline{k},w_0}v_0.$$
Then we have
$$\mathcal{Z}_1\bullet S_{\underline{k},w_0}v_0=(a_1-a_0-n+4)S_{\underline{k},w_0}v_0$$
and part (ii) of the proposition follows directly from the definition of $\mathbf{c}^{\prime}_1$.
\end{proof}

\begin{lemm}\label{induction step for operator}
Let $\ell$ be an integer with $2\leq \ell\leq n-1$. If Proposition \ref{prop: operators} is true for $r\geq n-\ell+1$, then it is true for $r=n-\ell$.
\end{lemm}

\begin{proof}
We prove part (ii) first. Assume that (\ref{second part}) holds for $r\geq n-\ell+1$. In fact, for a Jacobi sum $S_{\underline{k},w_0}$ satisfying the conditions in the (\ref{second part}) for $r=n-\ell$, we have
$$X^+_{n-\ell}\bullet X^-_{n-\ell}\bullet S_{\underline{k},w_0}v_0=(a_\ell-a_0-n+2)S_{\underline{k},w_0}v_0$$
by Corollary \ref{third formula}. Then we can deduce
$$\mathcal{Z}_{\ell-1}\bullet\cdots\bullet\mathcal{Z}_1\bullet X^+_{n-\ell}\bullet X^-_{n-\ell}\bullet S_{\underline{k},w_0}v_0=(a_\ell-a_0-n+2)\left(\prod_{s=1}^{\ell-1}\mathbf{c}^{\prime}_s\right)S_{\underline{k},w_0}v_0$$
from the assumption of the Lemma. Hence, by definition of $\mathcal{Z}_\ell$, we have
\begin{align*}
\mathcal{Z}_\ell\bullet S_{\underline{k},w_0}v_0&=\mathbf{d}_{\ell,\ell-1}S_{\underline{k},w_0}v_0-\mathcal{Z}_{\ell-1}\bullet\cdots\bullet\mathcal{Z}_1\bullet X^+_{n-\ell}\bullet X^-_{n-\ell}\bullet S_{\underline{k},w_0}v_0\\
&=\left(\mathbf{d}_{\ell,\ell-1}-(a_\ell-a_0-n+2)\left(\prod_{s=1}^{\ell-1}\mathbf{c}^{\prime}_s\right)\right)S_{\underline{k},w_0}v_0\\
&=\mathbf{c}^{\prime}_\ell S_{\underline{k},w_0}v_0
\end{align*}
where the last equality follows from Lemma \ref{equality of constants}.

Now we turn to part (i). Assume that (\ref{first part}) holds for $r\geq n-\ell+1$. 
We will prove inductively that for each $\ell^{\prime}$ satisfying $1\leq \ell^{\prime}\leq\ell-1$, we have
\begin{multline}\label{induction for Z}
\mathcal{Z}_{\ell^{\prime}}\bullet\cdots\bullet\mathcal{Z}_1\bullet X^+_{n-\ell}\bullet X^-_{n-\ell}\bullet S_{\underline{k},w_0}v_0\\
=\mathbf{d}_{\ell,\ell^{\prime}}S_{\underline{k},w_0}v_0+\left(\prod_{s=1}^{\ell^{\prime}}(\mathbf{c}^{\prime}_s-\mathbf{c}_s)\right)\left(\sum_{\underline{i}\in\mathbf{I}^{\ell^{\prime}}_\ell}\varepsilon(\underline{i})S_{\underline{k}^{\underline{i},m,n-\ell,n-\ell},w_0}v_0
\right)\\
+\left(\prod_{s=1}^{\ell^{\prime}}(\mathbf{c}^{\prime}_s-\mathbf{c}_s)\right)\left(\sum_{h=\ell^{\prime}+1}^{\ell-1}(a_\ell-a_h+1)\sum_{\underline{i}\in\mathbf{I}^{\ell^{\prime}}_{h}\setminus\mathbf{I}^{\ell^{\prime}}_{h+1}}\varepsilon(\underline{i})S_{\underline{k}^{\underline{i},n-\ell,n-h,n-\ell},w_0}v_0\right)
\end{multline}

We begin with studying some basic properties of the index sets $\mathbf{I}^{\ell^{\prime}}_h$. First of all, the set $\mathbf{I}^{\ell^{\prime}}_{\ell^{\prime}+1}\setminus \mathbf{I}^{\ell^{\prime}}_{\ell^{\prime}+2}$ has a unique element, which is precisely $\underline{i}=\{n-\ell^{\prime}-1,n\}$. Furthermore, there is a natural map of sets
$$\mathrm{res}_{\ell^{\prime}}:\mathbf{I}^{\ell^{\prime}}_h\rightarrow\mathbf{I}^{\ell^{\prime}+1}_h$$
for all $\ell^{\prime}+2\leq h\leq \ell$ defined by eliminating the element $n-\ell^{\prime}$ from $\underline{i}\in\mathbf{I}^{\ell^{\prime}}_h$ if $n-\ell^{\prime}\in\underline{i}$. In other words, for each $\underline{i}\in\mathbf{I}^{\ell^{\prime}+1}_h$, we have
$$\mathrm{res}_{\ell^{\prime}}^{-1}(\{\underline{i}\})=\{\underline{i}, \underline{i}\cup\{n-\ell^{\prime}\}\}\subseteq \mathbf{I}^{\ell^{\prime}}_h.$$
We use the shorten notation
$$\underline{i}^{\ell^{\prime}}:=\underline{i}\cup \{n-\ell^{\prime}\}$$
for each $\underline{i}\in\mathbf{I}^{\ell^{\prime}+1}_h$. Note in particular that $\varepsilon(\underline{i})=-\varepsilon(\underline{i}^{\ell^{\prime}})$.

Given an arbitrary $\underline{i}\in\mathbf{I}^{\ell^{\prime}+1}_h$ for $\ell^{\prime}+2\leq h\leq \ell-1$, then $S_{\underline{k}^{\underline{i},n-\ell,n-h,n-\ell},w_0}$ (resp. $S_{\underline{k}^{\underline{i}^{\ell^{\prime}},n-\ell,n-h,n-\ell},w_0}$) satisfies the conditions before the equation (\ref{second part}) (resp. (\ref{first part})). As a result, by the assumption that Proposition \ref{prop: operators} is true for $r=n-\ell^{\prime}-1$, we deduce that
\begin{equation}\label{case 1}
\mathcal{Z}_{\ell^{\prime}+1}\bullet\left(S_{\underline{k}^{\underline{i},n-\ell,n-h,n-\ell},w_0}v_0-S_{\underline{k}^{\underline{i}^{\ell^{\prime}},n-\ell,n-h,n-\ell},w_0}v_0\right)
=\left(\mathbf{c}^{\prime}_{\ell^{\prime}+1}-\mathbf{c}_{\ell^{\prime}+1}\right)S_{\underline{k}^{\underline{i},n-\ell,n-h,n-\ell},w_0}v_0.
\end{equation}
Similarly, we have
\begin{equation}\label{case 2}
\mathcal{Z}_{\ell^{\prime}+1}\bullet\left(S_{\underline{k}^{\underline{i},m,n-\ell,n-\ell},w_0}v_0- S_{\underline{k}^{\underline{i}^{\ell^{\prime}},m-\ell,n-\ell},w_0}v_0\right)
=\left(\mathbf{c}^{\prime}_{\ell^{\prime}+1}-\mathbf{c}_{\ell^{\prime}+1}\right)S_{\underline{k}^{\underline{i},m,n-\ell,n-\ell},w_0}v_0
\end{equation}
for each $\underline{i}\in\mathbf{I}^{\ell^{\prime}+1}_\ell$. We also have
\begin{equation}\label{case 3}
\mathcal{Z}_{\ell^{\prime}+1}\bullet S_{\underline{k},w_0}v_0=\mathbf{c}^{\prime}_{\ell^{\prime}+1}S_{\underline{k},w_0}v_0
\end{equation}
by (\ref{second part}) for $r=n-\ell^{\prime}-1$, and
\begin{equation}\label{case 4}
\mathcal{Z}_{\ell^{\prime}+1}\bullet S_{\underline{k}^{\underline{i}_0,n-\ell,n-\ell^{\prime}-1,n-\ell},w_0}v_0=\mathbf{c}_{\ell^{\prime}+1}S_{\underline{k},w_0}v_0
\end{equation}
by (\ref{first part}) for $r=n-\ell^{\prime}-1$ where $\underline{i}_0=\{n-\ell^{\prime}-1,n\}$.

Now assume that (\ref{induction for Z}) is true for some $1\leq \ell^{\prime}\leq \ell-2$. Then by combing (\ref{case 1}), (\ref{case 2}), (\ref{case 3}) and (\ref{case 4}), we have
\begin{multline*}
\mathcal{Z}_{\ell^{\prime}+1}\bullet\cdots\bullet\mathcal{Z}_1\bullet X^+_{n-\ell}\bullet X^-_{n-\ell}\bullet S_{\underline{k},w_0}v_0\\
=\mathbf{d}_{\ell,\ell^{\prime}}\mathcal{Z}_{\ell^{\prime}+1}\bullet S_{\underline{k},w_0}v_0+\left(\prod_{s=1}^{\ell^{\prime}}(\mathbf{c}^{\prime}_s-\mathbf{c}_s)\right)\mathcal{Z}_{\ell^{\prime}+1}\bullet\left(\sum_{\underline{i}\in\mathbf{I}^{\ell^{\prime}}_\ell}\varepsilon(\underline{i})S_{\underline{k}^{\underline{i},m,n-\ell,n-\ell},w_0}v_0
\right)\\
+\left(\prod_{s=1}^{\ell^{\prime}}(\mathbf{c}^{\prime}_s-\mathbf{c}_s)\right)\mathcal{Z}_{\ell^{\prime}+1}\bullet\left(\sum_{h=\ell^{\prime}+1}^{\ell-1}(a_\ell-a_h+1)\sum_{\underline{i}\in\mathbf{I}^{\ell^{\prime}}_{h}\setminus\mathbf{I}^{\ell^{\prime}}_{h+1}}\varepsilon(\underline{i})S_{\underline{k}^{\underline{i},n-\ell,n-h,n-\ell},w_0}v_0\right)
\end{multline*}
which is the same as
\begin{equation}\label{equation step main}
\mathbf{c}^{\prime}_{\ell^{\prime}}\mathbf{d}_{\ell,\ell^{\prime}}S_{\underline{k},w_0}v_0+ \left(\prod_{s=1}^{\ell^{\prime}}(\mathbf{c}^{\prime}_s-\mathbf{c}_s)\right)\left(X+Y+Z\right)
\end{equation}
where
$$X=(a_\ell-a_{\ell^{\prime}}+1)\mathcal{Z}_{\ell^{\prime}+1}\bullet S_{\underline{k}^{\underline{i}_0,n-\ell,n-\ell^{\prime}-1,n-\ell},w_0}v_0,$$
$$Y=\sum_{\underline{i}\in\mathbf{I}^{\ell^{\prime}+1}_\ell} \varepsilon(\underline{i})\mathcal{Z}_{\ell^{\prime}+1}\bullet \left(S_{\underline{k}^{\underline{i},m,n-\ell,n-\ell},w_0}v_0- S_{\underline{k}^{\underline{i}^{\ell^{\prime}},m,n-\ell,n-\ell},w_0}v_0\right),$$
and $$Z=\sum_{h=\ell^{\prime}+2}^{\ell-1}(a_\ell-a_h+1)\sum_{\underline{i}\in\mathbf{I}^{\ell^{\prime}+1}_{h} \setminus\mathbf{I}^{\ell^{\prime}+1}_{h+1}}\varepsilon(\underline{i})\mathcal{Z}_{\ell^{\prime}+1} \bullet\left(S_{\underline{k}^{\underline{i},n-\ell,n-h,n-\ell},w_0}v_0- S_{\underline{k}^{\underline{i}^{\ell^{\prime}},n-\ell,n-h,n-\ell},w_0}v_0\right).
$$
One can also readily check that (\ref{equation step main}) is also the same as
\begin{align*}
&\left(\mathbf{c}^{\prime}_{\ell^{\prime}+1}\mathbf{d}_{\ell,\ell^{\prime}}+\mathbf{c}_{\ell^{\prime}+1}\left(\prod_{s=1}^{\ell^{\prime}}(\mathbf{c}^{\prime}_s-\mathbf{c}_s)\right)(a_\ell-a_{\ell^{\prime}}+1)\right)S_{\underline{k},w_0}v_0\\
&\qquad+\left(\prod_{s=1}^{\ell^{\prime}+1}(\mathbf{c}^{\prime}_s-\mathbf{c}_s)\right)\left(\sum_{\underline{i}\in\mathbf{I}^{\ell^{\prime}+1}_\ell}\varepsilon(\underline{i})S_{\underline{k}^{\underline{i},m,n-\ell,n-\ell},w_0}v_0
\right)\\
&\qquad+\left(\prod_{s=1}^{\ell^{\prime}+1}(\mathbf{c}^{\prime}_s-\mathbf{c}_s)\right)\left(\sum_{h=\ell^{\prime}+2}^{\ell-1}(a_\ell-a_h+1)\sum_{\underline{i}\in\mathbf{I}^{\ell^{\prime}+1}_{h}\setminus\mathbf{I}^{\ell^{\prime}+1}_{h+1}}\varepsilon(\underline{i})S_{\underline{k}^{\underline{i},n-\ell,n-h,n-\ell},w_0}v_0\right),
\end{align*}
which finishes the proof of (\ref{induction for Z}), as we have $$\mathbf{d}_{\ell,\ell^{\prime}+1}=\mathbf{c}^{\prime}_{\ell^{\prime}+1}\mathbf{d}_{\ell,\ell^{\prime}}+\mathbf{c}_{\ell^{\prime}+1}\left(\prod_{s=1}^{\ell^{\prime}}(\mathbf{c}^{\prime}_s-\mathbf{c}_s)\right)(a_\ell-a_{\ell^{\prime}}+1)$$ by definition.

Note that (\ref{induction for Z}) for each $1\leq \ell^{\prime}\leq \ell-1$ then follows from Corollary \ref{second formula} for $r=n-\ell$. Note that the case $\ell^{\prime}=\ell-1$ for (\ref{induction for Z}) is just the following
$$\mathcal{Z}_{\ell-1}\bullet\cdots\bullet\mathcal{Z}_1\bullet X^+_{n-\ell}\bullet X^-_{n-\ell}\bullet S_{\underline{k},w_0}v_0=\mathbf{d}_{\ell,\ell-1}S_{\underline{k},w_0}v_0-\left(\prod_{s=1}^{\ell-1}(\mathbf{c}^{\prime}_s-\mathbf{c}_s)\right)S_{\underline{k}^{\underline{i}_1,m,n-\ell,n-\ell},w_0}v_0$$
where $\underline{i}_1=\{n\}$.

Finally, (\ref{first part}) for $r=n-\ell$ follows from the equation above together with the definition of $\mathcal{Z}_\ell$ and the identity (\ref{equation of c_l}).
\end{proof}

\begin{proof}[Proof of Proposition \ref{prop: operators}]
Proposition~\ref{prop: operators} follows easily from Lemma \ref{initial step for operator} and Lemma \ref{induction step for operator}.
\end{proof}

\begin{prop}\label{construct next vector}
For each $1\leq r\leq n-2$ and $1\leq s\leq [a_0-a_{n-1}]_1+n-2$, if $S_{\underline{k}^{\sharp,r,s-1},w_0}v_0\in V_0$, then $S_{\underline{k}^{\sharp,r,s},w_0}v_0\in V_0$.
\end{prop}

\begin{proof}
By Proposition \ref{prop: operators} and its proof, we can deduce the following equalities
$$\mathcal{Z}_{n-2-r}\bullet\cdots\bullet\mathcal{Z}_1\bullet S_{\underline{k}^{\sharp, r,s-1},w_0}v_0=\left(\prod_{\ell=1}^{n-2-r}\mathbf{c}^{\prime}_\ell\right)S_{\underline{k}^{\sharp, r,s-1},w_0}v_0,$$
$$\mathcal{Z}_{n-2-r}\bullet\cdots\bullet\mathcal{Z}_1\bullet\left(\sum_{\underline{i}\in\mathbf{I}_{n-1-r}}\varepsilon(\underline{i})S_{(\underline{k}^{\sharp,r,s-1})^{\underline{i},r,r+1,r+1},w_0}v_0\right)=-\left(\prod_{\ell=1}^{n-2-r}(\mathbf{c}^{\prime}_\ell-\mathbf{c}_\ell)\right)S_{\underline{k}^{\sharp, r,s},w_0}v_0,$$
and
\begin{multline*}
\mathcal{Z}_{n-2-r}\bullet\cdots\bullet\mathcal{Z}_1\bullet\left(\sum_{\underline{i}\in\mathbf{I}_{\ell} \setminus\mathbf{I}_{\ell-1}}\varepsilon(\underline{i})S_{(\underline{k}^{\sharp,r,s-1})^{\underline{i},r+1,n-\ell+1,r+1},w_0}v_0\right)\\
=\mathbf{c}_\ell\left(\prod_{h=1}^{\ell-1}(\mathbf{c}^{\prime}_h-\mathbf{c}_h)\right)\left(\prod_{h=\ell+1}^{n-2-r}\mathbf{c}^{\prime}_h\right)S_{\underline{k}^{\sharp, r,s-1},w_0}v_0
\end{multline*}
for each $1\leq \ell\leq n-2-r$.
Therefore by replacing $(r,s)$ in Corollary \ref{first formula} by $(r+1, s-1)$, we can deduce that
\begin{align*}
&\mathcal{Z}_{n-2-r}\bullet\cdots\bullet\mathcal{Z}_1\bullet X^+_{r+1}\bullet X^-_{r+1}\bullet S_{\underline{k}^{\sharp,r,s-1},w_0}v_0\\
&\qquad=-([a_0-a_{n-1}]_1+n-1-s)^2\left(\prod_{\ell=1}^{n-2-r}(\mathbf{c}^{\prime}_\ell-\mathbf{c}_\ell)\right)S_{\underline{k}^{\sharp,r,s},w_0}v_0+\mathbf{C}S_{\underline{k}^{\sharp, r,s-1},w_0}v_0\\
&\qquad=-([a_0-a_{n-1}]_1+n-1-s)^2\mathbf{c}_{n-1-r}S_{\underline{k}^{\sharp, r,s},w_0}v_0+\mathbf{C}S_{\underline{k}^{\sharp, r,s-1},w_0}v_0
\end{align*}
for certain constant $\mathbf{C}\in \F_p$. Note that we use the identity (\ref{equation of c_l}) for the last equality .

By our assumption, we know that $S_{\underline{k}^{\sharp,r,s-1},w_0}v_0\in V_0$. Hence we can deduce
$$S_{\underline{k}^{\sharp,r,s},w_0}v_0\in V_0$$
since $([a_0-a_{n-1}]_1+n-1-s)^2\mathbf{c}_{n-1-r}\neq 0$.
\end{proof}

\begin{coro}\label{inclusion1}
We have $S_{\underline{k}^{\sharp},w_0}v_0\in V_0$.
\end{coro}

\begin{proof}
By (\ref{second identification}) and Proposition \ref{construct next vector} we deduce that
$$S_{\underline{k}^{\sharp, r}}v_0\in V_0\Rightarrow S_{\underline{k}^{\sharp, r-1}}v_0\in V_0$$
for each $2\leq r\leq n-1$. Then by (\ref{first identification}) and the definition of $V_0$, we finish the proof.
\end{proof}

We write $\beta$ for $\sum_{r=1}^{n-1}\alpha_r$ to lighten the notation.
\begin{lemm}\label{easy calculation}
Given a Jacobi sum $S_{\underline{k},w_0}$, we have
$$X_{\beta,k_{1,n}}\bullet S_{\underline{k},w_0}=(-1)^{k_{1,n}+1}S_{\underline{k}^{\prime},w_0}$$
where $\underline{k}^{\prime}=(k'_{i,j})$ satisfies $k^{\prime}_{1,n}=0$ and $k^{\prime}_{i,j}=k_{i,j}$ otherwise.
\end{lemm}

\begin{proof}
This is a special case of Lemma \ref{main lemma} when $\alpha_0=\beta$ and $m=k_{1,n}$.
\end{proof}

From now on, whenever we want to view the notation $\mu_0^{w_0}$ as a weight, namely to fix a lift of $\mu_0^{w_0}\in X(T)/(p-1)X(T)$ into $X^{\rm{reg}}_1(T)$, we always mean
$$\mu^{w_0}_0=(a_0+p-1,a_{n-2},\cdots,a_1,a_{n-1}-p+1)\in X(T).$$
In particular, we have
$$(1,n)\cdot \mu^{w_0}_0+p\beta=\mu^{\ast}.$$

We recall the operators $X^{\rm{alg}}_{\beta,k}$ from the beginning of Section~\ref{sec: local automorphic side}.
\begin{lemm}\label{algebraic}
For $1\leq r\leq n-1$, we have the following equalities on $H^0(\mu_0^{w_0})_{\mu^{\ast}}$:
\begin{equation*}
X_{\beta,k}=-X^{\rm{alg}}_{\beta,k}
\end{equation*}
for all $1\leq k\leq p-1$.
\end{lemm}

\begin{proof}

Note that we have
\begin{equation*}
\mu_0^{w_0}-(\mu^{\ast}+k\beta)=([a_0-a_{n-1}]_1+n-2-k,0,\cdots,0,k-([a_0-a_{n-1}]_1+n-2)).
\end{equation*}
Therefore $\mu_0^{w_0}-(\mu^{\ast}+k\beta)\notin\sum_{\alpha\in\Phi^+}\Z_{\geq0}\alpha$ as long as $k>[a_0-a_{n-1}]_1+n-2$. As $(a_{n-1},\cdots,a_0)$ is assumed to be $n$-generic in the lowest alcove throughout this section, we deduce that
\begin{equation}\label{negative}
\mu_0^{w_0}-(\mu^{\ast}+k\beta)\notin\sum_{\alpha\in\Phi^+}\Z_{\geq0}\alpha\mbox{ for all }k\geq p-1.
\end{equation}
On the other hand, by the definition (\ref{Taylor expansion}), the image of $X^{\rm{alg}}_{\beta,k}$ lies inside $H^0(\mu_0^{w_0})_{\mu^{\ast}+k\beta}$, which is zero by (\ref{negative}) assuming $k\geq p-1$. Hence we deduce that
$$X^{\rm{alg}}_{\beta,k}=0\mbox{ on }H^0(\mu_0^{w_0})_{\mu^{\ast}}\mbox{ for all }k\geq p-1.$$
Then the conclusion of this lemma follows from the equality (\ref{relation between operators}).
\end{proof}

We have a natural embedding $H^0(\mu_0^{w_0})\hookrightarrow\pi_0$ by the definition of algebraic induction and parabolic induction. Recall that we have defined $U_1$ in Example \ref{good subgroup}.
\begin{lemm}\label{identification of space}
We have
\begin{equation}\label{inside2}
\F_p[S_{\underline{k}^{\sharp},w_0}v_0]=H^0(\mu_0^{w_0})^{\overline{U}_1}_{\mu^{\ast}}.
\end{equation}
In particular,
$$V^{\sharp}\subseteq H^0(\mu_0^{w_0}).$$
\end{lemm}

\begin{proof}

On one hand, by Lemma \ref{explicit1} we know that
\begin{equation*}
\mathrm{dim}_{\F_p}H^0(\mu_0^{w_0})^{\overline{U}_1}_{\mu^{\ast}}=1,
\end{equation*}
and this space is generated by $v^{\rm{alg},\prime\prime}_{\underline{m}^{\sharp}}$ where
\begin{equation}\label{special tuple}
\underline{m}^{\sharp}=(m^{\sharp}_1,\cdots, m^{\sharp}_{n-1}):=(0,\cdots,0,[a_0-a_{n-1}]_1+n-2).
\end{equation}

We now need to identify the vector $v^{\rm{alg},\prime\prime}_{\underline{m}^{\sharp}}$ with certain linear combination of Jacobi sums. Note that by Lemma \ref{explicit1} we have
$$v^{\rm{alg},\prime\prime}_{\underline{m}^{\sharp}}=D_n^{d_n}D_{n-1}^{a_1-a_0-n+2}(D^{\prime\prime}_{n-1})^{[a_0-a_{n-1}]_1+n-2}\prod^{n-2}_{i=1}D_i^{d_i-d_{i+1}}.$$
Given a matrix $A\in G(\F_p)$, then $D_i(A)\neq 0$ for all $1\leq i\leq n-1$ if and only if
$$A\in  B(\F_p)w_0 B(\F_p),$$
and thus the support of $v^{\rm{alg},\prime\prime}_{\underline{m}^{\sharp}}$ is contained in $ B(\F_p)w_0 B(\F_p)$.
As a result, by the proof of Proposition \ref{prop: basis}, we know that $v^{\rm{alg},\prime\prime}_{\underline{m}^{\sharp}}$ is a linear combination of vectors of the form
$$S_{\underline{k},w_0}v_0.$$
As $v^{\rm{alg},\prime\prime}_{\underline{m}^{\sharp}}$ is $\overline{U}_1$-invariant, and in particular $U_1(\F_p)$-invariant, then by Proposition \ref{prop: one1} we know that it is a linear combination of vectors of the form
\begin{equation}\label{special vector}
S_{\underline{k},w_0}v_0
\end{equation}
such that $k_{1,n}=[a_0-a_{n-1}]_1+n-2$, $k_{1,j}=0\mbox{ or }p-1$ for $2\leq j\leq n-1$ and $k_{i,j}=0$ for all $2\leq i<j\leq n$.

Finally, note that
$$u_{\beta}(t)~ v^{\rm{alg},\prime\prime}_{\underline{m}^{\sharp}}=D_n^{d_n}D_{n-1}^{a_1-a_0-n+2}(D^{\prime\prime}_{n-1}+t D_{n-1})^{[a_0-a_{n-1}]_1+n-2}\prod^{n-2}_{i=1}D_i^{d_i-d_{i+1}}$$
is a polynomial of $t$ with degree $[a_0-a_{n-1}]+n-2$, we conclude that
$$X^{\rm{alg}}_{\beta,[a_0-a_{n-1}]_1+n-2}~ v^{\rm{alg},\prime\prime}_{\underline{m}^{\sharp}}=v^{\rm{alg},\prime\prime}_{\underline{0}}$$
where $\underline{0}$ is the $(n-1)$-tuple with all entries zero.

By Lemma \ref{algebraic} and the fact that
$$\F_p[v^{\rm{alg},\prime\prime}_{\underline{0}}]=\F_p[S_{\underline{0},w_0}v_0]=\pi_0^{ U(\F_p),\mu_0^{w_0}},$$
we deduce that
$$X_{\beta,[a_0-a_{n-1}]_1+n-2}~ v^{\rm{alg},\prime\prime}_{\underline{m}^{\sharp}}=c^{\prime} S_{\underline{0},w_0}v_0$$
for some non-zero constant $c^{\prime}$. By Lemma \ref{easy calculation} and the linear independence of Jacobi sums proved in Proposition \ref{prop: basis}, we know that only $S_{\underline{k}^{\sharp},w_0}v_0$ can appear in the linear combination \ref{special vector}. In other words, we have shown that
$$v^{\rm{alg},\prime\prime}_{\underline{m}^{\sharp}}=c^{\prime\prime}S_{\underline{k}^{\sharp},w_0}v_0$$
for some non-zero constant $c^{\prime\prime}$, and thus we finish the proof.
\end{proof}

\begin{lemm}\label{length two}
The dual Weyl module $H^0(\mu_0^{w_0})$ is uniserial with length two with socle $F(\mu_0^{w_0})$ and cosocle $F(\mu^{\ast})$.
\end{lemm}

\begin{proof}
By \cite{Jantzen2003} Proposition II.2.2 we know that $\mathrm{soc}_{\overline{G}}\left(H^0(\mu_0^{w_0})\right)$ is irreducible and can be identified with $F(\mu^{w_0}_0)$ (which is in fact the definition of $F(\mu^{w_0}_0)$). Therefore it suffices to show that $H^0(\mu_0^{w_0})$ has only two Jordan--H\"{o}lder factor $F(\mu^{w_0}_0)$ and $F(\mu^{\ast})$, each of which has multiplicity one.

By \cite{Jantzen2003} II.2.13 (2) it is harmless for us to replace $H^0(\mu_0^{w_0})$ by the Weyl module $V(\mu_0^{w_0})$ (defined in \cite{Jantzen2003} II.2.13 ) and show that $V(\mu_0^{w_0})$ has only two Jordan--H\"{o}lder factor $F(\mu^{w_0}_0)$ and $F(\mu^{\ast})$ and each of them has multiplicity one.
As
$$\left\{\begin{array}{lllll}
p&<&\left\langle \mu_0^{w_0}, (\sum_{i=1}^{n-1}\alpha_i)^{\vee}\right\rangle&<&2p;\\
0&<&\left\langle \mu_0^{w_0}, (\sum_{i=1}^{n-2}\alpha_i)^{\vee}\right\rangle&<&p;\\
0&<&\left\langle \mu_0^{w_0}, (\sum_{i=2}^{n-1}\alpha_i)^{\vee}\right\rangle&<&p,
\end{array}\right.
$$
we deduce that the only dominant alcove lying below the one $\mu_0^{w_0}$ lies in is the lowest $p$-restricted alcove. In particular, the only dominant weight which is linked to and strictly smaller than $\mu_0^{w_0}$ is $\mu^{\ast}$.

By \cite{Jantzen2003} Proposition II. 8.19, we know the existence of a filtration of subrepresentation
$$V(\mu_0^{w_0})\supseteq V_1(\mu_0^{w_0})\supseteq \cdots$$
such that the following equality in Grothendieck group holds
$$\sum_{i>0}V_i(\mu_0^{w_0})=F(\mu^{\ast}).$$
This equality implies that
$$V_1(\mu_0^{w_0})=F(\mu^{\ast})$$
and
$$V_i(\mu_0^{w_0})=0\mbox{ for all }i\geq 2.$$
By \cite{Jantzen2003} II.8.19 (2) we also know that
$$V(\mu_0^{w_0})/V_1(\mu_0^{w_0})\cong F(\mu_0^{w_0}),$$
and thus we have shown that
$$V(\mu_0^{w_0})=F(\mu_0^{w_0})+F(\mu^{\ast})$$
in the Grothendieck group.
\end{proof}

\begin{prop}\label{identification}
We have $V^{\sharp}=H^0(\mu_0^{w_0})$.
\end{prop}

\begin{proof}
By Lemma \ref{length two}, we have the natural surjection
$$H^0(\mu_0^{w_0})\twoheadrightarrow F(\mu^{\ast})$$
which induces a morphism
\begin{equation}\label{weightspace}
H^0(\mu_0^{w_0})_{\mu_{\ast}}\rightarrow F(\mu^{\ast})_{\mu_{\ast}}.
\end{equation}
Now we consider $H^0(\mu_0^{w_0})$ as a $\overline{L}$-representation where $L\cong\mathbb{G}_m\times\mathrm{GL}_{n-1}$ is the standard Levi subgroup of $G$ which contains $U_1$ as a maximal unipotent subgroup. For any $\lambda\in X_L(T)_+$ (cf.~(\ref{L dominant})) we use the notation $H^0_L(\lambda)$ for the $\overline{L}$-dual Weyl module defined at the beginning of Section~\ref{sec: local automorphic side}. The dual Weyl module $H^0(\mu_0^{w_0})$ is the mod $p$ reduction of a lattice $V_{\Z_p}$ in the unique irreducible algebraic representation $V_{\Q_p}$ of $G$ such that $\left(V_{\Q_p}^{U}\right)_{\mu^{w_0}_0}\neq 0$. As the category of finite dimensional algebraic representations of $L$ in characteristic $0$ is semisimple, $V$ decomposes into a direct sum of characteristic $0$ irreducible representations of $L$. More precisely, we have the decomposition
$$V_{\Q_p}|_{L}=\bigoplus_{\substack{\lambda\in X_L(T)_+\\ (V_{\Q_p})^{U_1}_{\lambda}\neq 0}}m_{\lambda}V_L(\lambda)$$
where $V_L(\lambda)$ is the unique (up to isomorphism) irreducible algebraic representation of $L$ such that $\left(V_L(\lambda)^{U_1}\right)_{\lambda}\neq 0$ and
$$m_{\lambda}:=\mathrm{dim}_{\Q_p}\left(V^{U_1}_{\Q_p}\right)_{\lambda}.$$
Therefore in the Grothendieck group of algebraic representations of $\overline{L}$ over $\F_p$, we have
\begin{equation}\label{factordecomposition}
[H^0(\mu_0^{w_0})]|_{\overline{L}}=\bigoplus_{\substack{\lambda\in X_L(T)_+\\ H^0(\mu_0^{w_0})^{U_1}_{\lambda}\neq 0}}m_{\lambda}[H^0_L(\lambda)]
\end{equation}
as by Lemma \ref{explicit1} we know that $H^0(\mu_0^{w_0})^{\overline{U}_1}$ is the mod $p$ reduction of $V_{\Z_p}^{U_1}$ and that $V_{\Z_p}^{U_1}\otimes_{\Z_p}\Q_p=V_{\Q_p}^{U_1}$.

We say that
$$\mu^{\ast}\uparrow_L\lambda$$
if there exists $\widetilde{w}\in\widetilde{W}^L$ (see the beginning of Section~\ref{sec: local-global}) such that
$$\lambda=\widetilde{w}\cdot \mu^{\ast}\mbox{ and }\mu^{\ast}\leq \lambda.$$
Assume that there exists a $\lambda\in X_L(T)_+$ such that $\mu^{\ast}\uparrow_L\lambda$ and that $H^0(\mu_0^{w_0})^{\overline{U}_1}_{\lambda}\neq 0$. We denote by $v^{\rm{alg},\prime\prime}_{\underline{m}}$ the vector in $H^0(\mu_0^{w_0})^{\overline{U}_1}_{\lambda}\neq 0$ given by Lemma \ref{explicit1}. We note that by Lemma \ref{explicit1} the vector in $H^0(\mu_0^{w_0})^{\overline{U}_1}_{\mu^{\ast}}$ is $v^{\rm{alg},\prime\prime}_{\underline{m}^{\sharp}}$ (see (\ref{special tuple})). As $\mu^{\ast}\uparrow_L\lambda$, we must firstly have $\sum_{i=1}^{n-1}m_i=[a_0-a_{n-1}]_1+n-2$. By the last statement in Lemma \ref{explicit1}, we have
\begin{multline}\label{dorminant equation}
\lambda=\left(a_0+p-1-\sum_{i=1}^{n-1}m_i,a_{n-2}+m_1,\cdots,a_1+m_{n-2},a_{n-1}-p+1+m_{n-1}\right)\\
=(a_{n-1}-n+2,a_{n-2}+m_1,\cdots,a_1+m_{n-2},a_{n-1}-p+1+m_{n-1}).
\end{multline}
Recall $\eta=(n-1,n-2,\cdots,1,0)$.
We notice that $\mu^{\ast}-\eta$ lies in the lowest restricted $\overline{L}$-alcove in the sense that
\begin{equation}\label{lowest L alcove}
0<\langle\mu^{\ast},\alpha^{\vee}\rangle<p\mbox{ for all }\alpha\in\Phi^+_L
\end{equation}
where $\Phi^+_L$ is the positive roots of $L$ defined at the beginning of Section~\ref{sec: local-global}.

As we assume that $(a_{n-1},\cdots,a_0)$ is $n$-generic, it is easy to see the following
$$
\left\{
\begin{array}{ll}
a_{n-2}+m_1-(a_{n-1}-p+1+m_{n-1})\leq p+1+a_{n-2}-a_{n-1}+m_1<2p;&\\
a_{n-2}+m_1-(a_1+m_{n-2})\leq a_{n-2}+m_1-a_1\leq [a_0-a_1]_1<p;&\\
a_{n-3}+m_2-(a_{n-1}-p+1+m_{n-1})\leq [a_{n-3}-a_{n-1}]_1+m_2\leq [a_{n-2}-a_{n-1}]_1<p,&
\end{array}
\right.
$$
so that we know that $\lambda-\eta$ lies in either the lowest $\overline{L}$-alcove in the sense of (\ref{lowest L alcove}) (if we replace $\mu^{\ast}$ by $\lambda$) or the $p$-restricted $\overline{L}$-alcove described by the conditions
$$\left\{\begin{array}{llllll}
p&<&\left\langle\lambda,\left(\sum_{i=2}^{n-1}\alpha_i\right)^{\vee}\right\rangle&<&2p&\\
0&<&\left\langle\lambda,\left(\sum_{i=2}^{n-2}\alpha_i\right)^{\vee}\right\rangle&<&p&\\
0&<&\left\langle\lambda,\left(\sum_{i=3}^{n-1}\alpha_i\right)^{\vee}\right\rangle&<&p&\\
\end{array}\right.$$
and
$$0<\langle\lambda,\alpha^{\vee}\rangle<p\mbox{ for all }\alpha\in\Delta_L$$
where $\Delta_L:=\{\alpha_i\mid 2\leq i\leq n-1\}$ is the positive simple roots in $\Phi^+_L$.

In the first case, if $\lambda-\eta$ lies in the lowest $\overline{L}$-alcove, as we assume that $\mu^{\ast}\uparrow_L\lambda$, we must have $\lambda=\mu^{\ast}$; in the second case, we must have
$$\lambda=(2,n)\cdot\mu^{\ast}+p\left(\sum_{i=2}^{n-1}\alpha_i\right)=(a_{n-1}-n+2,a_0+p,a_{n-3},\cdots,a_1,a_{n-2}+n-2-p)$$
which means by (\ref{dorminant equation}) that
$$\underline{m}=(m_1,\cdots,m_{n-1})=([a_0-a_{n-2}]_1+1,0,\cdots,0,a_{n-2}-a_{n-1}+n-3).$$
This implies $a_{n-2}-a_{n-1}+n-1=m_{n-1}\geq 0$, which is a contradiction to the $n$-generic assumption on $(a_{n-1},\cdots,a_0)$. Therefore we must have $\lambda=\mu^{\ast}$. Hence we deduce by (\ref{factordecomposition}) and the strong linkage principle \cite{Jantzen2003} II.2.12 (1) that $F^L(\mu^{\ast})$ (see the beginning of Section~\ref{sec: local-global} for notation) has multiplicity one in $\mathrm{JH}_{\overline{L}}(H^0(\mu_0^{w_0})|_{\overline{L}})$ and is actually a direct summand.

On the other hand, as $F^L(\mu^{\ast})$ is obviously an $\overline{L}$-subrepresentation of $F(\mu^{\ast})$, we know that the surjection of $\overline{G}$-representation $H^0(\mu_0^{w_0})\twoheadrightarrow F(\mu^{\ast})$ induces an isomorphism of $\overline{L}$-representation on the direct summand $F^L(\mu^{\ast})$ on both sides with multiplicity one, by restriction from $\overline{G}$ to $\overline{L}$. In particular, we know that the map
$$H^0(\mu_0^{w_0})^{\overline{U}_1}_{\mu^{\ast}}\rightarrow F(\mu^{\ast})_{\mu^{\ast}}$$
is a bijection, and therefore the composition
$$V^{\sharp}\hookrightarrow H^0(\mu_0^{w_0})\twoheadrightarrow F(\mu^{\ast})$$
is non-zero as
$$H^0(\mu_0^{w_0})^{\overline{U}_1}_{\mu^{\ast}}=\F_p[v^{\rm{alg},\prime\prime}_{\underline{m}^{\sharp}}]=\F_p[S_{\underline{k}^{\sharp},w_0}v_0]$$
by Lemma \ref{identification of space}.
Hence, we have a surjection
\begin{equation}\label{final surjection}
V^{\sharp}\twoheadrightarrow F(\mu^{\ast}).
\end{equation}
Combining the surjection (\ref{final surjection}) with the injection
$$V^{\sharp}\hookrightarrow H^0(\mu_0^{w_0}),$$
we finish the proof by Lemma \ref{length two}.
\end{proof}

\begin{theo}\label{weaknonvanishing}
Assume that $(a_{n-1},\cdots,a_0)$ is $n$-generic in the lowest alcove (cf. Definition~\ref{defi: generic on tuples}). Then $H^0(\mu_0^{w_0})\subseteq V_0$. In particular, we have
$$F(\mu^{\ast})\in\mathrm{JH}(V_0).$$
\end{theo}
\begin{proof}
The first inclusion is a direct consequence of Proposition \ref{identification} together with Corollary \ref{inclusion1}. The second inclusion follows from the first as we have $F(\mu^{\ast})\in\mathrm{JH}(H^0(\mu_0^{w_0}))$.
\end{proof}

Before we end this section, we need several remarks to summarize the proof, and to clarify the necessity for all the constructions.
\begin{rema}
If we assume that for all $2\leq k\leq n-2$
\begin{equation}\label{simplest case}
[a_0-a_{n-1}]_1+n-2<a_k-a_{k-1},
\end{equation}
then we can actually show that
$$S_{\underline{k}^0,w_0}v_0\in H^0(\mu_0^{w_0})^{[\overline{U},\overline{U}]}_{\mu^{\ast}}$$
using Corollary~\ref{prop: one} and Lemma~\ref{explicit}, and thus
$$V_0=H^0(\mu_0^{w_0}).$$
Moreover, under the condition (\ref{simplest case}), we can even prove that the set
$$\{S_{\underline{k},w_0}v_0\mid \underline{k}\in\Lambda_{w_0}\}$$
forms a basis for $H^0(\mu_0^{w_0})_{\mu^{\ast}}$.

On the other hand, if we have
$$[a_0-a_{n-1}]_1+n-2\geq a_k-a_{k-1}$$
for some $2\leq k\leq n-2$, then we can use Lemma~\ref{lemm: induction process} to prove that
$$F(\mu_0^{s_kw_0})\in\mathrm{JH}(V_0)$$
which means that the inclusion
$$H^0(\mu_0^{w_0})\subseteq V_0$$
is actually strict.

In fact, through the proof of Proposition \ref{construct next vector}, the subrepresentation of $\pi_0$ generated by $S_{\underline{k}^{\sharp,r,s}}v_0$ is shrinking if $r$ is fixed and $s$ is growing. Therefore the subrepresentation of $\pi_0$ generated by $S_{\underline{k}^{\sharp,r}}v_0$ shrinks as $r$ decreases. Finally, we succeeded in shrinking from $V_0$ to $V^{\sharp}$ which can be identified with $H^0(\mu_0^{w_0})$.
\end{rema}

\begin{rema}
We need to emphasize that the choice of the operators $X^+_r$ and $X^-_r$ for $1\leq r\leq n-1$ are crucial. For example, the operator
$$\sum_{t\in\F_p}t^{p-2}w_0u_{\alpha_r}(t)w_0\in\F_p[ G(\F_p)]$$
for some $2\leq r\leq n-2$ does not work in general. The reason is that, as one can check by explicit computation, applying such operator to $S_{\underline{k}w_0}v_0$ for some $\underline{k}\in\Lambda_{w_0}$ will generally give us a huge linear combination of Jacobi sum operators. From our point of view, it is basically impossible to compute such a huge linear combination explicitly and systematically. Instead, as stated in Proposition \ref{prop: technical formula}, our operators $X^+_r$ and $X^-_r$ can be computed systematically, even though the computation is still complicated.

The motivation of the choice of operators $X^+_r$ and $X^-_r$ can be roughly explained as follows. First of all, we need one `weight raising operator' $X^+$ and one `weight lowering operator' $X^-$. These are two operators lying in a subalgebra $\F_p\langle X^+, X^-\rangle$ of $\F_p[ G(\F_p)]$ such that
$$\F_p\langle X^+, X^-\rangle\cong\F_p[\mathrm{GL}_2(\F_p)].$$
We start with the vector $S_{\underline{k},w_0}v_0$ for some $\underline{k}\in\Lambda_{w_0}$. We apply the operator $X^-$ once and then $X^+$ once, the result is a vector with the same $ T(\F_p)$-eigencharacter $\mu^{\ast}$. We observe that $S_{\underline{k},w_0}v_0$ is in general not an eigenvector of the operator $X^+\bullet X^-$ because the representation $\pi_0$, after restricting from $\F_p[ G(\F_p)]$ to $\F_p\langle X^+, X^-\rangle$, is highly non-semisimple. The naive expectation is that we just take the difference
$$X^+\bullet X^-\bullet S_{\underline{k},w_0}v_0-cS_{\underline{k},w_0}v_0$$
for some constant $c\in\F_p$, and then repeat the procedure by applying some other operators similar to $X^+$ and $X^-$.

The case $n=3$ is easy. In the case $n=4$, the operator
$$\sum_{t\in\F_p}t^{p-2}w_0u_{\alpha_2}(t)w_0\in\F_p[\mathrm{GL}_4(\F_p)]$$
is not well behaved as we explained in this remark, and therefore we are forced to use our $X^-_2$ to replace $\sum_{t\in\F_p}t^{p-2}w_0u_{\alpha_2}(t)w_0$.

Now we consider the general case, and it is possible for us to carry on an induction step. We have a sequence of growing subgroups of $\overline{G}$
$$\overline{P}_{\{n-1\}}\subsetneq \overline{P}_{\{n-2,n-1\}}\subsetneq \cdots\subsetneq \overline{P}_{\{2,\cdots,n-1\}}$$
and
$$\overline{L}_{\{n-1\}}\subsetneq \overline{L}_{\{n-2,n-1\}}\subsetneq \cdots\subsetneq \overline{L}_{\{2,\cdots,n-1\}}$$
where $\overline{P}_{\{r,\cdots, n-1\}}$ is the standard parabolic subgroup corresponding to the simple roots $\alpha_k$ for $r\leq k\leq n-1$ and $\overline{L}_{\{r,\cdots, n-1\}}$ is its standard Levi subgroup. Technically speaking, constructing the vector $S_{\underline{k}^{\sharp,r+1},w_0}v_0$ (for some $1\leq r\leq n-2$) from $S_{\underline{k}^0,w_0}v_0$ should be reduced to Corollary \ref{inclusion1} when we replace $\overline{G}$ by its Levi subgroup $\overline{L}_{\{r+1,\cdots, n-1\}}$. In other words, to construct $S_{\underline{k}^{\sharp,r+1},w_0}v_0$ from $S_{\underline{k}^0,w_0}v_0$ we only need the operators
$$X^+_k, X^-_k\in\F_p[\overline{L}_{\{r+2,\cdots, n-1\}}(\F_p)]\subsetneq \F_p[\overline{L}_{\{r+1,\cdots, n-1\}}(\F_p)]$$
for all $r+2\leq k\leq n-1$.

In order to construct $S_{\underline{k}^{\sharp,r},w_0}v_0$ from $S_{\underline{k}^{\sharp,r+1},w_0}v_0$, we only need to prove Proposition \ref{construct next vector}. Then we summarize the proof of Proposition \ref{construct next vector} as the following: for some $a\in\F_p^{\times}$ and $b\in\F_p$
$$X^+_{r+1}\bullet X^-_{r+1}\bullet S_{\underline{k}^{\sharp, r,s-1},w_0}v_0\equiv a S_{\underline{k}^{\sharp,r,s},w_0}v_0+ b S_{\underline{k}^{\sharp,r,s-1},w_0}v_0+\text{error terms}$$
and the error terms can be killed by combinations of the operators $X^+_k, X^-_k$ for $r+2\leq k\leq n-1$.
\end{rema}

\subsection{Jacobi sums in characteristic $0$}\label{subsec: Jacobi sums in char. 0}
In this section, we establish an intertwining identity for lifts of Jacobi sums in characteristic~$0$ in Theorem~\ref{theo: identity}, which is one of the main ingredients of the proof of Theorem~\ref{theo: lgc}. All of our calculations here are in the setting of $G(\Q_p)=\GL_n(\Q_p)$. We first fix some notation.

Let $A\in  G(\F_p)$. By $\lceil A\rceil$ we mean the matrix in $G(\Q_p)$ whose entries are the classical Teichm\"uller lifts of the entries of $A$.
The map $A\mapsto\lceil A\rceil$ is obviously not a group homomorphism but only a map between sets. On the other hand, we use the notation $\widetilde{\mu}$ for the Teichm\"uller lift of a character $\mu$ of $T(\F_p)$. 

We denote the standard lifts of simple reflections in $G(\Q_p)$ by
\begin{equation*}
s_i=\left(
\begin{array}{cccc}
\mathrm{Id}_{i-1}&&&\\
&&1&\\
&1&&\\
&&&\mathrm{Id}_{n-i-1}\\
\end{array}\right)
\end{equation*}
for $1\leq i\leq n-1$. We also use the following notation
\begin{equation*}
t_i=\left(
\begin{array}{cc}
p \mathrm{Id}_{i}&\\
&\mathrm{Id}_{n-i}\\
\end{array}\right)
\end{equation*}
for $1\leq i\leq n$. Let
\begin{equation}\label{generator of normalizer}
\Xi_n:=w^{\ast}t_1,
\end{equation}
where $w^{\ast}:=s_{n-1}\bullet ...\bullet s_1$. We recall the Iwahori subgroup $I$ and the pro-$p$ Iwahori subgroup $I(1)$ from the beginning of Section~\ref{sec: local automorphic side}. Note that the operator $\Xi_n$ and the group $I$ actually generates the normalizer of $I$ inside $G(\Q_p)$. One easily sees that $\Xi_n$ is nothing else than the following matrix:
$$\Xi_n=
\begin{pmatrix}
0 & 1 & 0 & \cdots & 0 & 0 & 0 \\
0 & 0 & 1 & \cdots & 0 & 0 & 0 \\
0 & 0 & 0 & \cdots & 0 & 0 & 0 \\
\vdots & \vdots & \vdots & \ddots & \vdots & \vdots & \vdots \\
0 & 0 & 0 & \cdots & 0 & 1 & 0 \\
0 & 0 & 0 & \cdots & 0 & 0 & 1 \\
p & 0 & 0 & \cdots & 0 & 0 & 0
\end{pmatrix}\in G(\Q_p).
$$

For each $1\leq i\leq n-1$, we consider the maximal parabolic subgroup $P_i^-$ of $G$ containing lower-triangular Borel subgroup $B^-$ such that its Levi subgroup can be chosen to be $\GL_i\times\GL_{n-i}$ which embeds into $G$ in the standard way.
We denote the unipotent radical of $P^-_i$ by $N^-_i$. Then we introduce
\begin{equation}\label{operator}
U^{i}_n=\sum_{A\in N^-_i(\F_p)}t^{-1}_i\lceil A\rceil\mbox{ for each }1\leq i\leq n-1.
\end{equation}
Note that each $A\in N^-_i$ has the form
\begin{equation*}
\left(\begin{array}{cc}
\mathrm{Id}_i & 0_{(n-i)\times i}\\
\ast_{i\times (n-i)}& \mathrm{Id}_{n-i}
\end{array}\right).
\end{equation*}
for each $1\leq i\leq n-1$.

We recall the tuples $\underline{k}^1$ and $\underline{k}^{1,\prime}$ from (\ref{main exponent}), and consider the characteristic $0$ lift of Jacobi sums $\mathcal{S}_n$ and $\mathcal{S}^{\prime}_n$ as follows:
\begin{equation}\label{char 0 Jacobi sum}
\left\{\begin{array}{ll}
\widehat{\mathcal{S}}_n&=\,\,\sum_{A\in U(\F_p)} \left(\prod_{i=1}^{n-1}\lceil A\rceil_{i,i+1}^{k^1_{i,i+1}}\right)\lceil A\rceil~ w_0;\\
\widehat{\mathcal{S}}^{\prime}_n&=\,\,\sum_{A\in U(\F_p)} \left(\prod_{i=1}^{n-1}\lceil A\rceil_{i,i+1}^{k_{i,i+1}^{1,\prime}}\right)\lceil A\rceil~ w_0.\\
\end{array}\right.
\end{equation}

The main result of this section is the following, which is a generalization of the case $n=3$ in (3.2.1) of \cite{HLM}.
\begin{theo}\label{theo: identity}
Assume that the $n$-tuple of integers $(a_{n-1},\cdots,a_0)$ is $n$-generic in the lowest alcove, and let $$\Pi_p:=\mathrm{Ind}^{G(\Q_p)}_{B(\Q_p)}(\chi_1\otimes\chi_2\otimes\chi_3\otimes ...\otimes\chi_{n-2}\otimes\chi_{n-1}\otimes\chi_0)$$ be a tamely ramified principal series representation where the $\chi_i: \Q_p^{\times}\rightarrow E^{\times}$ are smooth characters satisfying $\chi_i|_{\Z_p^{\times}}=\widetilde{\omega}^{a_i}$ for $0\leq i\leq n-1$.

On the $1$-dimensional subspace $\Pi_p^{I(1), (a_1,a_2,...,a_{n-1},a_0)}$ we have the identity:
\begin{equation*}
\widehat{\mathcal{S}}_n^{\prime}\bullet(\Xi_n)^{n-2}=p^{n-2}\kappa_n\left(\prod^{n-2}_{k=1}\chi_{k}(p)\right)\widehat{\mathcal{S}}_n
\end{equation*}
for some $\kappa_n\in \cO_E^{\times}$ such that
$$\kappa_n\equiv \varepsilon^{\ast}\mathcal{P}_n(a_{n-1},\cdots,a_0)\mbox{ mod }\varpi_E$$
where $\varepsilon^{\ast}=\pm1$ is a sign defined in (\ref{the sign}) that depends only on $(a_{n-1},\cdots,a_0)$ and $\mathcal{P}_n$ is an explicit rational function defined in (\ref{rational function}).
\end{theo}

Firstly, we need a lemma, which is a direct generalization of Lemma $3.2.5$ in \cite{HLM}.
\begin{lemm}\label{lemm: direct}
Pick a non-zero element $\widehat{v}\in\Pi_p^{I(1),(a_1,a_2,...,a_{n-1},a_0)}$. Then
we have
\begin{equation*}
U^{n-2}_n\widehat{v}=\left(\prod^{n-2}_{k=1}\chi_k(p)\right)^{-1}\widehat{v}
\end{equation*}
and moreover
\begin{equation*}
(\Xi_n)^{n-2}\bullet U^{n-2}_n\widehat{v}=\sum_{\textbf{B}\in (w^{\ast})^{n-2}N^-_{n-2}(\F_p)}\lceil\textbf{B}\rceil\widehat{v}
\end{equation*}
\end{lemm}
Note that $\textbf{B}\in (w^{\ast})^{n-2}N^-_{n-2}(\F_p)$ is equivalent to that $\textbf{B}$ is running through the matrices in $ G(\F_p)$ of the form
\begin{equation}\label{matrix1}
\left(\begin{array}{cc}
\ast_{(n-2\times 2)}& \mathrm{Id}_2\\
\mathrm{Id}_{n-2} & 0_{(2\times n-2)}\\
\end{array}\right).
\end{equation}

\begin{proof}
The proof of this lemma is an immediate calculation which is parallel to that of \cite{HLM}, Lemma 3.2.5.
\end{proof}

From now on, we fix a matrix $\textbf{B}$ of the form in (\ref{matrix1}), so that we may have
\begin{equation}\label{explicit matrix}
w_0 ~\textbf{B}=
\left(\begin{array}{ccccccc}
0&0&\cdots&0&1&0&0\\
0&0&\cdots&1&0&0&0\\
\vdots&\vdots&\ddots&\vdots&\vdots&\vdots&\vdots\\
0&1&\cdots&0&0&0 & 0 \\
1&0&\cdots&0&0&0&0\\
\lambda_{1,1}&\lambda_{1,2}&\cdots&\lambda_{1,n-3}&\lambda_{1,n-2}&0&1\\
\lambda_{2,1}&\lambda_{2,2}&\cdots&\lambda_{2,n-3}&\lambda_{2,n-2}&1&0\\
\end{array}\right).
\end{equation}

We now compute the Bruhat decomposition of the matrix $w_0 \textbf{B}$. We apply the definition of $D_i$, $D_i^{\prime}$ at the beginning of Section~\ref{subsec: Some technical formula} as polynomials of entries of matrices to the matrix $w_0\textbf{B}$, namely we define
$$
D_i:=\left\{
\begin{array}{ll}
\lambda_{2,1} & \hbox{if $i=1$;}\\
\lambda_{2,i-1}\lambda_{1,i}-\lambda_{1,i-1}\lambda_{2,i} & \hbox{if $2\leq i\leq n-2$;}\\
-\lambda_{1,n-2} & \hbox{if $i=n-1$}
\end{array}
\right.
$$
and
$$
D_i^{\prime}:=\left\{
\begin{array}{ll}
\lambda_{1,1} & \hbox{if $i=1$;}\\
-\lambda_{2,2} & \hbox{if $i=2$;}\\
\lambda_{1,i-2}\lambda_{2,i}-\lambda_{2,i-2}\lambda_{1,i} & \hbox{if $3\leq i\leq n-2$;}\\
\lambda_{1,n-3} & \hbox{if $i=n-1$.}
\end{array}
\right.
$$

Assume first that $D_i\neq 0$ for $1\leq i\leq n-1$, and let
\begin{equation*}
T_{\textbf{B}}=\mathrm{diag}\left(D_1,\frac{D_2}{D_1},...,\frac{D_k}{D_{k-1}},...,\frac{1}{D_{n-1}}\right)
\end{equation*}
and
\begin{equation*}
U_{\textbf{B}}=\left(\begin{array}{cccccccccc}
1&\frac{D^{\prime}_{n-1}}{D_{n-1}}&\cdots & \ast &\ast&\ast&\ast&\cdots&\ast&\ast \\
&1&\cdots&\ast&\ast&\ast&\ast&\cdots&\ast&\ast\\
&&\ddots&\vdots&\vdots&\vdots&\vdots&\ddots&\vdots&\vdots\\
&&&1&\frac{D^{\prime}_{k+1}}{D_{k+1}}&\ast&\ast&\cdots&\ast&\ast\\
&&&&1&\frac{D^{\prime}_k}{D_k}&\ast &\cdots&\ast &\ast \\
&&&&&1&\frac{D^{\prime}_{k-1}}{D_{k-1}}&\cdots&\ast&\ast \\
&&&&&&1&\cdots&\ast&\ast\\
&&&&&&&\ddots&\vdots&\vdots\\
&&&&&&&&1&\frac{D^{\prime}_1}{D_1}\\
&&&&&&&&&1\\
\end{array}\right).
\end{equation*}
By a direct computation, we have
\begin{equation*}
(U_{\textbf{B}} w_0 T_{\textbf{B}})^{-1}w_0 \textbf{B}\in  U(\F_p),
\end{equation*}
so that we may write
$$w_0 \textbf{B}=U_{\textbf{B}} w_0 T_{\textbf{B}}U_{\textbf{B}}^{\prime}$$
for some matrix $U_{\textbf{B}}^{\prime}$ in $ U(\F_p)$ (whose explicit form is not important for our purpose). We notice that
$$w_0\textbf{B}\in B(\F_p)w_0 B(\F_p)\mbox{ if and only if }D_i\neq 0\mbox{ for all }1\leq i\leq n-1.$$
In general, if $w_0\textbf{B}\in U_w(\F_p)w B(\F_p)$, we write
$$w_0\textbf{B}=U_{\textbf{B}}^wwT^w_{\textbf{B}}U^{w,\prime}_{\textbf{B}}$$
for $U_{\textbf{B}}^w\in U_w(\F_p)$, $T^w_{\textbf{B}}\in T(\F_p)$ and $U^{w,\prime}_{\textbf{B}}\in U(\F_p)$.

As a result, we deduce that
\begin{equation*}
\widehat{\mathcal{S}}^{\prime}_n\bullet(\Xi_n)^{n-2}\bullet U^{n-2}_n\widehat{v}=\sum_{w\neq w_0}\widehat{S}_{w}\widehat{v}+\sum_{A\in U(\F_p),\textbf{B}\in (w^{\ast})^{n-2}N^-_{n-2}(\F_p)}\left(\prod_{i=1}^{n-1}\lceil A_{i,i+1}\rceil^{k^{1,\prime}_{i,i+1}}\right)\lceil AU_{\textbf{B}}w_0 T_{\textbf{B}}\rceil\widehat{v}
\end{equation*}
where
$$\widehat{S}_w:=\left(\sum_{A\in U(\F_p)}\left(\prod_{i=1}^{n-1}\lceil A_{i,i+1}\rceil^{k^{1,\prime}_{i,i+1}}\right)\lceil A\rceil\right)\cdot\left(\sum_{\textbf{B}\in w_0 U_w(\F_p)w B(\F_p)\cap (w^{\ast})^{n-2}N^-_{n-2}(\F_p)}\lceil U_{\textbf{B}}^wwT^w_{\textbf{B}}U^{w,\prime}_{\textbf{B}}\rceil\right).$$

\begin{lemm}\label{lemm: vanish1}
We have $\widehat{S}_w\widehat{v}=0$ for each $w\neq w_0$.
\end{lemm}

The following proof is a direct generalization of \textbf{Case}~1 of Lemma $3.2.6$ in \cite{HLM}.
\begin{proof}
We notice that
$$\widehat{S}_w\widehat{v}=\sum_{A\in U(\F_p)}\left(\left(\prod_{i=1}^{n-1}\lceil A_{i,i+1}\rceil^{k^{1,\prime}_{i,i+1}}\right)\cdot\left(\sum_{\textbf{B}\in w_0 U_w(\F_p)w B(\F_p)\cap (w^{\ast})^{n-2}N^-_{n-2}(\F_p)}\lceil AU_{\textbf{B}}^wwT^w_{\textbf{B}}\rceil\right)\right)\widehat{v}$$
as $\widehat{v}$ is $I(1)$-invariant.
After changing the variable $A^w:=AU_{\textbf{B}}^w$, we deduce
$$\widehat{S}_w\widehat{v}=\sum_{A^w\in U(\F_p)}\left(\left(\prod_{i=1}^{n-1}\lceil A_{i,i+1}\rceil^{k^{1,\prime}_{i,i+1}}\right)\cdot\left(\sum_{\textbf{B}\in w_0 U_w(\F_p)w B(\F_p)\cap (w^{\ast})^{n-2}N^-_{n-2}(\F_p)}\lceil A^wwT^w_{\textbf{B}}\rceil\right)\right)\widehat{v}$$
where $A_{i,i+1}$ is viewed as a rational function of $A^w_{i,i+1}$ and the entries of $\textbf{B}$.

For the given Weyl element $w\neq w_0$, we know that $\Delta\cap(\Phi^+\setminus\Phi^+_w)\neq\varnothing$. If $\alpha_i\in\Delta\cap(\Phi^+\setminus\Phi^+_w)$ for some $1\leq i\leq n-1$, then we have
$$A^w_{i,i+1}=A_{i,i+1}$$
for the choice of $i$ above.

By the definition of $U_w$ we have the set theoretical decomposition
$$ U(\F_p)=U_{w_0s_i}(\F_p)\times U_{s_i}(\F_p)$$
and thus we can write
$$A^w=A^w_1\cdot A^w_2$$
for $A^w_1\in U_{w_0s_i}(\F_p)$ and $A^w_2\in U_{s_i}(\F_p)$ uniquely determined by $A^w$.

As $A^w_2w\in w U(\F_p)$, we deduce that
$$\lceil A^wwT^w_{\textbf{B}}\rceil\widehat{v}=\lceil A^w_1wT^w_{\textbf{B}}\rceil\widehat{v}$$
and thus
\begin{align*}
\widehat{S}_w\widehat{v}&=\left(\sum_{A^w_2\in U_{s_i}(\F_p)}\lceil A_{i,i+1}\rceil^{k^{1,\prime}_{i,i+1}}\right) \sum_{A^w_1\in U_{w_0s_i}(\F_p)}\left(\prod_{\substack{1\leq j\leq n-1\\ j\neq i}}\lceil A_{j,j+1}\rceil^{k^{1,\prime}_{j,j+1}}\right)\\
&\qquad\qquad\qquad\qquad\qquad\qquad \cdot\left(\sum_{\textbf{B}\in w_0 U_w(\F_p)w B(\F_p)\cap (w^{\ast})^{n-2}N^-_{n-2}(\F_p)}\lceil A^w_1wT^w_{\textbf{B}}\rceil\right)\widehat{v}\\
&=0.
\end{align*}
Note that the sum $\sum_{A_{i,i+1}\in\F_p}\lceil A_{i,i+1}\rceil^{k^{1,\prime}_{i,i+1}}$ is zero as the $\lceil A_{i,i+1}\rceil$ are chosen to be Teichm\"uller lifts.
\end{proof}

By Lemma~\ref{lemm: vanish1}, we may and do assume that $D_i\neq 0$ for $1\leq i\leq n-1$ from now on. As
\begin{equation*}
C:=AU_{\textbf{B}}=\left(\begin{array}{cccccccccc}
1&A_{1,2}+\frac{D^{\prime}_{n-1}}{D_{n-1}}&\cdots  &\ast&\ast&\cdots&\ast&\ast \\
&1&\cdots&\ast&\ast&\cdots&\ast&\ast\\
&&\ddots&\vdots&\vdots&\ddots&\vdots&\vdots\\
&&&1&A_{k,k+1}+\frac{D^{\prime}_{n-k}}{D_{n-k}}&\cdots&\ast &\ast \\
&&&&1&\cdots&\ast&\ast\\
&&&&&\ddots&\vdots&\vdots\\
&&&&&&1&A_{n-1,n}+\frac{D^{\prime}_1}{D_1}\\
&&&&&&&1\\
\end{array}\right),
\end{equation*}
we actually change the variable from $A_{i,i+1}$ to $C_{i,i+1}$ through $C_{i,i+1}=A_{i,i+1}+\frac{D^{\prime}_{n-i}}{D_{n-i}}$ for $1\leq i\leq n-1$. In other words, we have the equality
\begin{equation}\label{AA}
\widehat{\mathcal{S}}^{\prime}_n\bullet(\Xi_n)^{n-2}\bullet U^{n-2}_n\widehat{v}=\sum_{\substack{C\in  U(\F_p),\textbf{B}\in(w^{\ast})^{n-2}N^-_{n-2}(\F_p)\\D_i\neq 0\mbox{ for }1\leq i\leq n-1}}\left(\prod_{i=1}^{n-1}\left\lceil C_{i,i+1}-\frac{D^{\prime}_{n-i}}{D_{n-i}}\right\rceil^{k_{i,i+1}^{1,\prime}}\right)\lceil Cw_0T_{\textbf{B}}\rceil\widehat{v}.
\end{equation}

Note that we have
\begin{equation}\label{BB}
\lceil T_{\textbf{B}}\rceil\widehat{v}=\lceil D_1\rceil^{a_1}\lceil D_{n-1}\rceil^{-a_0}\prod_{k=2}^{n-1}\left\lceil\frac{D_{k}}{D_{k-1}}\right\rceil^{a_k}\widehat{v}=\lceil D_{n-1}\rceil^{a_{n-1}-a_0}\prod_{k=1}^{n-2}\lceil D_k\rceil^{a_k-a_{k+1}}\widehat{v}.
\end{equation}
Combining (\ref{AA}) with (\ref{BB}), we obtain
\begin{equation}\label{sum1}
\widehat{\mathcal{S}}_n^{\prime}\bullet(\Xi_n)^{n-2}\bullet U^{n-2}_n\widehat{v} =\sum_{\substack{C\in  U(\F_p),\textbf{B}\in(w^{\ast})^{n-2}N^-_{n-2}(\F_p)\\D_i\neq 0,\mbox{ for }1\leq i\leq n-1}} X_0 \lceil Cw_0\rceil\widehat{v}
\end{equation}
where
$$
X_0:=\left(\prod_{i=1}^{n-1}\left\lceil C_{i,i+1}-\frac{D^{\prime}_{n-i}}{D_{n-i}}\right\rceil^{k_{i,i+1}^{1,\prime}}\right) \left(\lceil D_{n-1}\rceil^{a_{n-1}-a_0}\prod_{k=1}^{n-2}\lceil D_k\rceil^{a_k-a_{k+1}}\right).
$$

Our main target in the rest of this section is to calculate (\ref{sum1}) explicitly. The result (c.f. Theorem \ref{theo: identity}) is simple and clean. However, the intermediate step is a bit complicated. The sum (\ref{sum1}) is essentially an exponential sum over $\F_p$-points of an affine variety. In our case, it is possible for us to introduce an induction step to finally reduce the calculation of (\ref{sum1}) to the special case $n=4$. In other words, the induction step in Proposition \ref{main induction step} is trying to reduce the calculation of an exponential sum with many variables to another one with less variables. The main subtlety of the induction is to carefully manipulate the affine varieties where the sums lie and to change the variables systematically.

Before we go into the calculation of (\ref{sum1}), we start with recalling some standard facts about Jacobi sums and Gauss sums. We fix a primitive $p$-th root of unity $\xi\in E$ and set $\epsilon:=\xi-1$.
For each pair of integers $(a,b)$ with $0\leq a,b\leq p-1$, we set
\begin{equation}\label{definition of Jacobi sum}
J(a,b):=\sum_{\lambda\in\F_p}\lceil \lambda\rceil^a\lceil1-\lambda\rceil^b.
\end{equation}
We also set
$$G(a):=\sum_{\lambda\in\F_p}\lceil \lambda\rceil^a\xi^{\lambda}$$
for each integers $a$ with $0\leq a\leq p-1$. For example, we have $G(p-1)=-1$.

It is known by section $1.1$, $GS3$ of \cite{Lang} that if $a+b\not\equiv0$ mod $(p-1)$, we have
\begin{equation}\label{relation}
J(a,b)=\frac{G(a)G(b)}{G(a+b)}.
\end{equation}
It is also obvious from the definition that if $a,b,a+b\not\equiv0$ mod $(p-1)$ then
\begin{equation}\label{transform}
J(b,a)=J(a,b)=(-1)^bJ(b,[-a-b]_1)=(-1)^aJ(a,[-a-b]_1).
\end{equation}
By Stickelberger's theorem (\cite{Lang} Section $1.2$, Theorem $2.1$), we know that
\begin{equation}\label{congruence}
\left\{
\begin{array}{ll}
\mathrm{ord}_p(G(a))=1-\frac{a}{p-1}&\\
\frac{G(a)}{\epsilon^{p-1-a}}\equiv a!& \pmod{p}.\\
\end{array}\right.
\end{equation}

We introduce further notation.  It is easy to see from (\ref{explicit matrix}) that there is an isomorphism of schemes
$$(w^{\ast})^{n-2}N^-_{n-2}\cong\textbf{M}_{2,n-2}$$
over $\Z$ where the right side is the $(2n-4)$-dimensional affine space, which can be viewed as the space of all $2\times (n-2)$-matrices. As a result, we can replace the subscript $\textbf{B}\in(w^{\ast})^{n-2}N^-_{n-2}(\F_p)$ in (\ref{sum1}) by $\underline{\lambda}\in\textbf{M}_{2,n-2}(\F_p)$ where
$$\underline{\lambda}:=\left(
\begin{array}{lll}
\lambda_{1,2}&\cdots&\lambda_{1,n-2}\\
\lambda_{2,2}&\cdots&\lambda_{2,n-2}\\
\end{array}\right).$$
For each integer $2\leq m\leq n-2$, we consider the space $\textbf{M}_{2,m}$ of $2\times m$-matrices, and denote an arbitrary $\F_p$-point of $\textbf{M}_{2,m}$ by
$$\underline{\lambda}^m:=\left(
\begin{array}{lll}
\lambda_{1,2}&\cdots&\lambda_{1,m}\\
\lambda_{2,2}&\cdots&\lambda_{2,m}\\
\end{array}\right).$$
Hence, for each $3\leq m\leq n-2$ we have a natural restriction map
\begin{equation*}
\mathrm{pr}_{m,m-1}: \textbf{M}_{2,m}(\F_p)\twoheadrightarrow\textbf{M}_{2,m-1}(\F_p)
\end{equation*}
by sending
\begin{equation*}
\left(\begin{array}{lll}
\lambda_{1,2}&\cdots&\lambda_{1,m}\\
\lambda_{2,2}&\cdots&\lambda_{2,m}\\
\end{array}\right)\mapsto\left(
\begin{array}{lll}
\lambda_{1,2}&\cdots&\lambda_{1,m-1}\\
\lambda_{2,2}&\cdots&\lambda_{2,m-1}\\
\end{array}\right).
\end{equation*}

We define
\begin{equation}\label{open locus}
\textbf{U}_m:=\{\underline{\lambda}^m\in\textbf{M}_{2,m}(\F_p) \mid \lambda_{1,m}\neq 0\,\,\,\&\,\,\, D_i\neq 0\mbox{ for }1\leq i\leq m\}
\end{equation}
and thus $\textbf{U}_m$ is the set of $\F_p$-points of the open subscheme of $\textbf{M}_{2,m}$ defined by the equations in (\ref{open locus}).


For each subset $\textbf{U}\subseteq \textbf{U}_m$, we also define
\begin{equation}\label{big sum}
\mathcal{L}_m(\textbf{U}):=X_1\cdot\left(\sum_{C\in  U(\F_p),\underline{\lambda}^m\in\textbf{U}} Y_1\cdot Z_1 \cdot\lceil Cw_0\rceil\widehat{v}\right)
\end{equation}
where
$$X_1:=(-1)^{(n-1-m)(a_0-a_{n-1})}\prod_{\ell=m+1}^{n-2} J([a_0-a_\ell]_1,[a_\ell-a_{\ell+1}]_1+n-2)J(a_{n-1}-a_\ell, [a_{\ell-1}-a_{n-1}]_1+n-2)$$
$$
Y_1:=\left\lceil C_{n-1-m,n-m}+\frac{\lambda_{1,m-1}}{\lambda_{1,m}}\right\rceil^{[a_m-a_{m+1}]_1+n-2}\left(\prod_{i=1}^{n-2-m}\lceil C_{i,i+1}\rceil^{[a_0-a_{n-i}]_1+n-2}\right)
$$
and
$$Z_1:=\lceil\lambda_{1,m}\rceil^{a_{n-1}-a_0}\lceil D_m\rceil^{a_m-a_{n-1}}\left(\prod_{k=1}^{m-1}\lceil D_k\rceil^{a_k-a_{k+1}}\right)\cdot \prod_{i=n-m}^{n-1}\left\lceil C_{i,i+1}-\frac{D^{\prime}_{n-i}}{D_{n-i}}\right\rceil^{[a_{n-i-1}-a_{n-1}]_1+n-2}.$$
It follows easily from this definition that, if $\textbf{U}$ and $\textbf{U}^{\prime}$ are two subsets of $\textbf{U}_m$ satisfying $\textbf{U}\cap\textbf{U}^{\prime}=\varnothing$, then for the disjoint union $\textbf{U}\sqcup\textbf{U}^{\prime}\subseteq \textbf{U}_m$ we have
\begin{equation}\label{additive}
\mathcal{L}_m(\textbf{U}\sqcup\textbf{U}^{\prime})=\mathcal{L}_m(\textbf{U})+\mathcal{L}_m(\textbf{U}^{\prime}).
\end{equation}

\begin{prop}\label{main conclusion}
We have an equality
\begin{equation}\label{induction sum}
\widehat{\mathcal{S}}_n^{\prime}\bullet(\Xi_n)^{n-2}\bullet U^{n-2}_n\widehat{v}=\mathcal{L}_m(\textbf{U}_m)
\end{equation}
for each $2\leq m\leq n-2$.
\end{prop}
We prove this proposition by a series of Lemmas.

\begin{lemm}\label{initial step}
The equality (\ref{induction sum}) is true for $m=n-2$.
\end{lemm}

\begin{proof}
This is simply a reformulation of (\ref{sum1}).
\end{proof}

\begin{prop}\label{main induction step}
If the equality (\ref{induction sum}) is true for $m$, then it is also true for $m-1$.
\end{prop}

Before we prove Proposition \ref{main induction step}, we need to define some further notation. We fix an integer $3\leq m\leq n-2$ until the end of the proof of Proposition \ref{main induction step}. We define
$$\textbf{U}^{m-1}_m:=\mathrm{pr}_{m,m-1}^{-1}(\textbf{U}_{m-1})\subseteq \textbf{U}_m\,\,\mbox{ and }\,\, \textbf{U}^{m-1,\prime}_m:=\textbf{U}_m\backslash \textbf{U}^{m-1}_m$$
or more explicitly
$$\textbf{U}^{m-1}_m=\{\underline{\lambda}^m\in\textbf{U}_m \mid \lambda_{1,m-1}\neq 0\}\,\,\mbox{ and }\,\, \textbf{U}^{m-1,\prime}_m=\{\underline{\lambda}^m\in\textbf{U}_m \mid \lambda_{1,m-1}=0\}.$$
We also define
\begin{multline*}
f_m(C,\underline{\lambda}^m):= \lceil\lambda_{1,m}\rceil^{a_{n-1}-a_0}\lceil D_m\rceil^{a_m-a_{n-1}} \left\lceil C_{n-1-m,n-m}+\frac{\lambda_{1,m-1}}{\lambda_{1,m}}\right\rceil^{[a_m-a_{m+1}]_1+n-2}\\
\cdot \left\lceil C_{n-m,n-m+1}-\frac{D^{\prime}_{m}}{D_{m}}\right\rceil^{[a_{m-1}-a_{n-1}]_1+n-2}.
\end{multline*}
Notice that $f_m(C,\underline{\lambda}^m)$ is actually the Teichm\"{u}ller lift of a rational function of $C_{n-1-m,n-m}$, $C_{n-m,n-m+1}$, $\lambda_{1,m}$, $\lambda_{2,m}$, $\lambda_{1,m-1}$, $\lambda_{2,m-1}$, $\lambda_{1,m-2}$ and $\lambda_{2,m-2}$.

Now we can rewrite (\ref{big sum}) as
\begin{equation}\label{sum step}
\mathcal{L}_m(\textbf{U})=X_1\cdot\left(\sum_{C\in  U(\F_p),\underline{\lambda}^{m-1}\in\mathrm{pr}_{m,m-1}(\textbf{U})} Y_2\cdot Z_2 \cdot\lceil Cw_0\rceil\widehat{v}\right)
\end{equation}
for each subset $\textbf{U}\subseteq\textbf{U}_m$, where
\begin{multline*}
Y_2:=\left(\prod_{k=1}^{m-1}\lceil D_k\rceil^{a_k-a_{k+1}}\right) \left(\prod_{i=1}^{n-2-m}\lceil C_{i,i+1}\rceil^{[a_0-a_{n-i}]_1+n-2}\right)\\
\cdot\left(\prod_{i=n-m+1}^{n-1}\left\lceil C_{i,i+1}-\frac{D^{\prime}_{n-i}}{D_{n-i}}\right\rceil^{[a_{n-i-1}-a_{n-1}]_1+n-2}\right)
\end{multline*}
and
$$Z_2:=\sum_{\underline{\lambda}^m\in\mathrm{pr}^{-1}_{m,m-1}(\underline{\lambda}^{m-1})\cap\textbf{U}}f_m(C,\underline{\lambda}^m).$$
We emphasize that $Y_2$ only depends on $C$ and $\mathrm{pr}_{m,m-1}(\underline{\lambda}^m)$.  It is natural that the calculation of $\mathcal{L}_m(\textbf{U})$ for each $\textbf{U}\subseteq\textbf{U}_m$ start with the calculation of $Z_2$.

\begin{lemm}\label{case one}
We have
$$\mathcal{L}_m(\textbf{U}^{m-1}_m)=\mathcal{L}_{m-1}(\textbf{U}_{m-1}).$$
\end{lemm}

\begin{proof}
For each $\underline{\lambda}^m\in\textbf{U}_m^{m-1}$, we have
\begin{align*}
f_m(C,\underline{\lambda})
&=\lceil\lambda_{1,m}\rceil^{a_{n-1}-a_0} {\lceil\lambda_{2,m-1}\lambda_{1,m}-\lambda_{1,m-1}\lambda_{2,m}\rceil}^{a_m-a_{n-1}}\\
&\qquad\qquad\cdot {\left\lceil C_{n-1-m,n-m}+\frac{\lambda_{1,m-1}}{\lambda_{1,m}}\right\rceil}^{[a_m-a_{m+1}]_1+n-2}\\
& \qquad\qquad\qquad\cdot{\left\lceil C_{n-m,n-m+1}+\frac{\lambda_{1,m-2}\lambda_{2,m}-\lambda_{2,m-2}\lambda_{1,m}}{\lambda_{1,m-1}\lambda_{2,m}-\lambda_{2,m-1}\lambda_{1,m}}\right\rceil}^{[a_{m-1}-a_{n-1}]_1+n-2}.
\end{align*}
By a change of variable $x=\lambda_{1,m}\in\F_p^{\times}$ and $y=\frac{\lambda_{2,m}}{\lambda_{1,m}}\in\F_p\backslash\{\frac{\lambda_{2,m-1}}{\lambda_{1,m-1}}\}$, we deduce that
\begin{align*}
f_m(C,\underline{\lambda})
&=\lceil x\rceil^{a_m-a_0}\lceil \lambda_{2,m-1}-\lambda_{1,m-1}y\rceil^{a_m-a_{n-1}}\\
&\qquad\qquad\cdot\left\lceil C_{n-1-m,n-m}+\frac{\lambda_{1,m-1}}{x}\right\rceil^{[a_m-a_{m+1}]_1+n-2}\\
&\qquad\qquad\qquad\qquad\cdot\left\lceil C_{n-m,n-m+1}+\frac{\lambda_{1,m-2}y-\lambda_{2,m-2}}{\lambda_{1,m-1}y-\lambda_{2,m-1}}\right\rceil^{[a_{m-1}-a_{n-1}]_1+n-2}.
\end{align*}

If $C_{n-1-m,n-m}=0$, then
$$\sum_{x\in\F_p^{\times}}f_m(C,\underline{\lambda})=\left(\sum_{x\in\F_p^{\times}}\lceil x\rceil^{a_{m+1}-a_0-n+2}\right)\cdot\left(\ast\right)=0$$
where $\ast$ is a certain term which is independent of $x$.
If $C_{n-1-m,n-m}\neq0$, then we deduce that
$$
\sum_{x\in\F_p^{\times}}f_m(C,\underline{\lambda})= \sum_{x\in\F_p^{\times}} X_3\cdot Y_3\cdot Z_3
$$
where
$$X_3:=\lceil\lambda_{1,m-1}\rceil^{a_m-a_0}\lceil C_{n-1-m,n-m}\rceil^{[a_0-a_{m+1}]_1+n-2},$$
$$Y_3:=\lceil\lambda_{2,m-1}-\lambda_{1,m-1}y\rceil^{a_m-a_{n-1}}\left\lceil C_{n-m,n-m+1}+ \frac{\lambda_{1,m-2}y-\lambda_{2,m-2}}{\lambda_{1,m-1}y-\lambda_{2,m-1}}\right\rceil^{[a_{m-1}-a_{n-1}]_1+n-2},$$
and
$$Z_3:=\left\lceil\frac{C_{n-1-m,n-m}x}{\lambda_{1,m-1}}\right\rceil^{a_m-a_0}
\left\lceil 1+\frac{\lambda_{1,m-1}}{C_{n-1-m,n-m}x}\right\rceil^{[a_m-a_{m+1}]_1+n-2}.$$
Therefore, by (\ref{definition of Jacobi sum}) we deduce that
\begin{equation}\label{case one sum one}
\sum_{x\in\F_p^{\times}}f_m(C,\underline{\lambda})=(-1)^{a_0-a_m}J([a_0-a_m]_1, [a_m-a_{m+1}]_1+n-2)\cdot X_3\cdot Y_3.
\end{equation}
We emphasize that (\ref{case one sum one}) still holds even if $C_{n-1-m,n-m}=0$ as $X_3=0$ in that case.

One can rewrite $Y_3$ as follows:
\begin{multline*}
Y_3=\lceil\lambda_{2,m-1}-\lambda_{1,m-1}y\rceil^{a_m-a_{n-1}}\\
\cdot\left\lceil\left(C_{n-m,n-m+1}+\frac{\lambda_{1,m-2}}{\lambda_{1,m-1}}\right)+\frac{\frac{\lambda_{1,m-2}\lambda_{2,m-1}}{\lambda_{1,m-1}}-\lambda_{2,m-2}}{\lambda_{1,m-1}y-\lambda_{2,m-1}}\right\rceil^{[a_{m-1}-a_{n-1}]_1+n-2}.
\end{multline*}
If $C_{n-m,n-m+1}+\frac{\lambda_{1,m-2}}{\lambda_{1,m-1}}=0$, then we deduce that
$$\sum_{y\neq \frac{\lambda_{2,m-1}}{\lambda_{1,m-1}}}Y_3=\left(\sum_{y\neq \frac{\lambda_{2,m-1}}{\lambda_{1,m-1}}}\lceil\lambda_{1,m-1}y-\lambda_{2,m-1}\rceil^{a_m-a_{m-1}-n+2}\right)\left(\ast\right)=0$$
where $\ast$ is a certain term which is independent of $y$.  If $C_{n-m,n-m+1}+\frac{\lambda_{1,m-2}}{\lambda_{1,m-1}}\neq0$, then we deduce that
$$\sum_{y\neq \frac{\lambda_{2,m-1}}{\lambda_{1,m-1}}}Y_3=\left(\sum_{y\neq \frac{\lambda_{2,m-1}}{\lambda_{1,m-1}}}\lceil X_4\rceil^{a_{n-1}-a_m}\lceil1-X_4\rceil^{[a_{m-1}-a_{n-1}]_1+n-2}\right)\cdot Y_4$$
where
$$X_4:=\frac{\frac{\lambda_{1,m-2}\lambda_{2,m-1}}{\lambda_{1,m-1}}-\lambda_{2,m-2}}{(\lambda_{2,m-1}-\lambda_{1,m-1}y)(C_{n-m,n-m+1}+\frac{\lambda_{1,m-2}}{\lambda_{1,m-1}})}$$
and $$Y_4:=\left\lceil C_{n-m,n-m+1}+\frac{\lambda_{1,m-2}}{\lambda_{1,m-1}}\right\rceil^{[a_{m-1}-a_m]_1+n-2}\left\lceil \frac{\lambda_{1,m-2}\lambda_{2,m-1}}{\lambda_{1,m-1}}-\lambda_{2,m-2}\right\rceil^{a_m-a_{n-1}}.$$

By (\ref{definition of Jacobi sum}) we obtain that
\begin{equation}\label{case one sum two}
\sum_{y\neq \frac{\lambda_{2,m-1}}{\lambda_{1,m-1}}}Y_3=J(a_{n-1}-a_m, [a_{m-1}-a_{n-1}]_1+n-2)Y_4.
\end{equation}
Combining (\ref{case one sum one}) with (\ref{case one sum two}) we deduce that
\begin{multline}\label{case one sum three}
\sum_{x\in\F_p^{\times}, y\neq \frac{\lambda_{2,m-1}}{\lambda_{1,m-1}}}f_m(C,\underline{\lambda}) =(-1)^{a_0-a_m}  J([a_0-a_m]_1, [a_m-a_{m+1}]_1+n-2)\\
\cdot J(a_{n-1}-a_m, [a_{m-1}-a_{n-1}]_1+n-2)\cdot X_3\cdot Y_4.
\end{multline}

Now applying (\ref{case one sum three}) back to (\ref{sum step}), and then recalling the definition of $\mathcal{L}_{m-1}(\textbf{U}_{m-1})$ from (\ref{big sum}) by replacing $m$ by $m-1$, we conclude that
$\mathcal{L}_m(\textbf{U}^{m-1}_m)=\mathcal{L}_{m-1}(\textbf{U}_{m-1})$,
which completes the proof.
\end{proof}

\begin{lemm}\label{case two}
We have
$$\mathcal{L}_m(\textbf{U}^{m-1,\prime}_m)=0.$$
\end{lemm}

\begin{proof}
For each $\underline{\lambda}^m\in\textbf{U}_m^{m-1,\prime}$, we have
\begin{multline*}
f_m(C,\underline{\lambda})
=\lceil\lambda_{1,m}\rceil^{a_{m+1}-a_0}\lceil\lambda_{2,m-1}\lambda_{1,m}\rceil^{a_m-a_{n-1}} \lceil C_{n-1-m,n-m}\rceil^{[a_m-a_{n-1}]_1+n-2}\\
\cdot\left\lceil C_{n-m,n-m+1}+\frac{\lambda_{1,m-2}\lambda_{2,m}-\lambda_{2,m-2}\lambda_{1,m}}{-\lambda_{2,m-1}\lambda_{1,m}}\right\rceil^{[a_{m-1}-a_{n-1}]_1+n-2}.
\end{multline*}
By a change of variable $x=\lambda_{1,m}\in\F_p^{\times}$ and $y=\frac{\lambda_{2,m}}{\lambda_{1,m}}\in\F_p$, we deduce that
\begin{multline*}
f_m(C,\underline{\lambda})
=\lceil x\rceil^{a_m+a_{m+1}-a_0-a_{n-1}}\lceil\lambda_{2,m-1}\rceil^{a_m-a_{n-1}}
\lceil C_{n-1-m,n-m}\rceil^{[a_m-a_{n-1}]_1+n-2}\\
\cdot\left\lceil C_{n-m,n-m+1}+\frac{\lambda_{1,m-2}y-\lambda_{2,m-2}}{-\lambda_{2,m-1}}\right\rceil^{[a_{m-1}-a_{n-1}]_1+n-2}.
\end{multline*}
Hence, we obtain
$$\sum_{y\in\F_p}f_m(C,\underline{\lambda})=\left(\sum_{y\in\F_p}\left\lceil C_{n-m,n-m+1}+\frac{\lambda_{1,m-2}y-\lambda_{2,m-2}}{-\lambda_{2,m-1}}\right\rceil^{[a_{m-1}-a_{n-1}]_1+n-2}\right)\left(\ast\right)=0$$
where $\ast$ is a certain term which is independent of $y$, and thus $\mathcal{L}_m(\textbf{U}_m^{m-1,\prime})=0$.
\end{proof}

\begin{proof}[Proof of Proposition \ref{main induction step}]
By Lemma \ref{case one} and Lemma \ref{case two},  (\ref{additive}) and the decomposition $\textbf{U}_m=\textbf{U}^{m-1}_m \sqcup\textbf{U}^{m-1,\prime}_m$ we deduce that
\begin{equation*}
\mathcal{L}_{m}(\textbf{U}_m)=\mathcal{L}_{m}(\textbf{U}^{m-1}_m)+\mathcal{L}_{m}(\textbf{U}^{m-1,\prime}_m)=\mathcal{L}_{m-1}(\textbf{U}_{m-1})
\end{equation*}
which finishes the proof of Proposition \ref{main induction step}.
\end{proof}

\begin{proof}[Proof of Proposition \ref{main conclusion}]
This follows directly from Proposition \ref{main induction step} and Lemma \ref{initial step} by induction.
\end{proof}

We define
\begin{equation}\label{coefficient}
\kappa_n:=(-1)^{(n-2)(a_0-a_{n-1})}p^{2-n}\prod_{m=1}^{n-2}\Upsilon_m
\end{equation}
where
$$
\Upsilon_m:=J([a_0-a_m]_1, [a_m-a_{m+1}]_1+n-2)\cdot J(a_{n-1}-a_m, [a_{m-1}-a_{n-1}]_1+n-2).
$$

\begin{prop}\label{last case}
We have
$$\mathcal{L}_2(\textbf{U}_2)=p^{n-2}\kappa_n\widehat{\mathcal{S}}_n\widehat{v}.$$
\end{prop}

\begin{proof}
By the case $m=2$ of (\ref{big sum}), we have
\begin{equation*}
\mathcal{L}_2(\textbf{U}_2):=\sum_{C\in  U(\F_p),\underline{\lambda}^2\in\textbf{U}_2}X_5\cdot Y_5\cdot Z_5 \cdot\lceil Cw_0\rceil\widehat{v}
\end{equation*}
where
$$X_5:=\left(\prod_{m=3}^{n-2}\varepsilon^m \Upsilon_m \right) \cdot\left(\prod_{i=1}^{n-4}\lceil C_{i,i+1}\rceil^{[a_0-a_{n-i}]_1+n-2}\right),
$$
\begin{equation*}
Y_5:=\lceil\lambda_{2,1}\rceil^{a_1-a_{2}} \lceil\lambda_{1,2}\rceil^{a_{n-1}-a_0} \left\lceil\lambda_{1,2}\lambda_{2,1}-\lambda_{1,1}\lambda_{2,2}\right\rceil^{a_2-a_{n-1}}\left\lceil C_{n-3,n-2}+\frac{\lambda_{1,1}}{\lambda_{1,2}}\right\rceil^{[a_2-a_{3}]_1+n-2},
\end{equation*}
and
$$
Z_5:=\left\lceil C_{n-1,n}-\frac{\lambda_{1,1}}{\lambda_{2,1}}\right\rceil^{[a_{0}-a_{n-1}]_1+n-2}\left\lceil C_{n-2,n-1}-\frac{\lambda_{2,2}}{\lambda_{1,1}\lambda_{2,2}-\lambda_{1,2}\lambda_{2,1}}\right\rceil^{[a_{1}-a_{n-1}]_1+n-2}.
$$

We define
$$\left\{\begin{array}{lll}
\textbf{U}_2^1&:=&\{\underline{\lambda}^2\in\textbf{U}_2\mid \lambda_{1,1}\neq 0\};\\
\textbf{U}_2^{1,\prime}&:=&\{\underline{\lambda}^2\in\textbf{U}_2\mid \lambda_{1,1}=0\}.
\end{array}\right.$$
It is obvious that
$\textbf{U}_2=\textbf{U}_2^1\bigsqcup\textbf{U}_2^{1,\prime}$ and so
$$\mathcal{L}_2(\textbf{U}_2)=\mathcal{L}_2(\textbf{U}_2^1)+\mathcal{L}_2(\textbf{U}_2^{1,\prime}).$$

We start with the calculation of $\mathcal{L}_2(\textbf{U}_2^{1,\prime})$.
In this case we have
\begin{multline*}
Y_5\cdot Z_5=\lceil\lambda_{2,1}\rceil^{a_1-a_{2}}\lceil\lambda_{1,2}\rceil^{a_{n-1}-a_0}\left\lceil\lambda_{1,2}\lambda_{2,1}\right\rceil^{a_2-a_{n-1}}\\
\cdot \lceil C_{n-3,n-2}\rceil^{[a_2-a_{3}]_1+n-2}\lceil C_{n-1,n}\rceil^{[a_{0}-a_{n-1}]_1+n-2}\left\lceil C_{n-2,n-1}+\frac{\lambda_{2,2}}{\lambda_{1,2}\lambda_{2,1}}\right\rceil^{[a_{1}-a_{n-1}]_1+n-2}
\end{multline*}
and thus
$$\sum_{\lambda_{2,2}\in\F_p}Y_5\cdot Z_5=\left(\sum_{\lambda_{2,2}\in\F_p}\left\lceil C_{n-2,n-1}+\frac{\lambda_{2,2}}{\lambda_{1,2}\lambda_{2,1}}\right\rceil^{[a_{1}-a_{n-1}]_1+n-2}\right)\left(\ast\right)=0$$
where $\ast$ is a certain term which is independent of $\lambda_{2,2}$. Hence we conclude that
\begin{equation*}
\mathcal{L}_2(\textbf{U}_2^{1,\prime})=0.
\end{equation*}

We now compute $\mathcal{L}_2(\textbf{U}_2^{1})$. By a change of variable $x=\lambda_{2,1}\in\F_p^{\times}$ and $y=\frac{\lambda_{2,2}}{\lambda_{2,1}}\in\F_p\backslash\{\frac{\lambda_{1,2}}{\lambda_{1,1}}\}$ we can rewrite $Y_5\cdot Z_5$ as
\begin{multline*}
Y_5\cdot Z_5=\lceil x\rceil^{a_1-a_{n-1}}\lceil\lambda_{1,2}\rceil^{a_{n-1}-a_0}\left\lceil\lambda_{1,2}-\lambda_{1,1}y\right\rceil^{a_2-a_{n-1}}\left\lceil C_{n-3,n-2}+\frac{\lambda_{1,1}}{\lambda_{1,2}}\right\rceil^{[a_2-a_{3}]_1+n-2}\\
\cdot\left\lceil C_{n-1,n}-\frac{\lambda_{1,1}}{x}\right\rceil^{[a_{0}-a_{n-1}]_1+n-2}\left\lceil C_{n-2,n-1}-\frac{y}{\lambda_{1,1}y-\lambda_{1,2}}\right\rceil^{[a_{1}-a_{n-1}]_1+n-2}.
\end{multline*}
If $C_{n-1,n}=0$, then
$$\sum_{x\in\F_p^{\times}}Y_5\cdot Z_5=\left(\sum_{x\in\F_p^{\times}}\lceil x\rceil^{a_1-a_0-n+2}\right)\left(\ast\right)=0$$
where $\ast$ is a certain term that is independent of $x$.
If $C_{n-1,n}\neq0$, then
we deduce from (\ref{definition of Jacobi sum}) that
\begin{equation}\label{last case sum one}
\sum_{x\in\F_p^{\times}}Y_5\cdot Z_5=J(a_{n-1}-a_1,[a_0-a_{n-1}]_1+n-2)\lceil C_{n-1,n}\rceil^{[a_{0}-a_1]_1+n-2}\cdot X_6\cdot Y_6
\end{equation}
where
$$X_6:=\left\lceil\lambda_{1,2}-\lambda_{1,1}y\right\rceil^{a_2-a_{n-1}}\left\lceil C_{n-2,n-1}-\frac{y}{\lambda_{1,1}y-\lambda_{1,2}}\right\rceil^{[a_{1}-a_{n-1}]_1+n-2}$$
and
$$Y_6:=\lceil\lambda_{1,1}\rceil^{a_1-a_{n-1}}\lceil\lambda_{1,2}\rceil^{a_{n-1}-a_0}\left\lceil C_{n-3,n-2}+\frac{\lambda_{1,1}}{\lambda_{1,2}}\right\rceil^{[a_2-a_{3}]_1+n-2},$$
as
\begin{multline*}
\lceil x\rceil^{a_1-a_{n-1}}\left\lceil C_{n-1,n}-\frac{\lambda_{1,1}}{x}\right\rceil^{[a_{0}-a_{n-1}]_1+n-2}\\
=\lceil C_{n-1,n}\rceil^{[a_{0}-a_1]_1+n-2}\lceil\lambda_{1,1}\rceil^{a_1-a_{n-1}}\left\lceil\frac{C_{n-1,n}x}{\lambda_{1,1}}\right\rceil^{a_1-a_{n-1}}\left\lceil1-\frac{\lambda_{1,1}}{C_{n-1,n}x}\right\rceil^{[a_{0}-a_{n-1}]_1+n-2}.
\end{multline*}
We emphasize that (\ref{last case sum one}) still holds if $C_{n-1,n}=0$.

Now we can rewrite $X_6$ as
$$X_6=\left\lceil\lambda_{1,2}-\lambda_{1,1}y\right\rceil^{a_2-a_{n-1}}\left\lceil\left(C_{n-2,n-1}-\frac{1}{\lambda_{1,1}}\right)-\frac{\frac{\lambda_{1,2}}{\lambda_{1,1}}}{\lambda_{1,1}y-\lambda_{1,2}}\right\rceil^{[a_{1}-a_{n-1}]_1+n-2}.$$
If $C_{n-2,n-1}-\frac{1}{\lambda_{1,1}}=0$, then we deduce
$$\sum_{y\in\F_p\backslash\{\frac{\lambda_{1,2}}{\lambda_{1,1}}\}}X_6=\left(\sum_{y\in\F_p\backslash\{\frac{\lambda_{1,2}}{\lambda_{1,1}}\}}\left\lceil\lambda_{1,1}y-\lambda_{1,2}\right\rceil^{a_2-a_1-n+2}\right)\left(\ast\right)=0$$
where $\ast$ is a certain term which is independent of $y$. If $C_{n-2,n-1}-\frac{1}{\lambda_{1,1}}\neq0$, then we deduce
$$X_6=\left\lceil C_{n-2,n-1}-\frac{1}{\lambda_{1,1}}\right\rceil^{[a_1-a_2]_1+n-2}\left\lceil\frac{\lambda_{1,2}}{\lambda_{1,1}}\right\rceil^{a_2-a_{n-1}}\lceil X_7\rceil^{a_{n-1}-a_2}\lceil 1+X_7\rceil^{[a_1-a_{n-1}]_1+n-2}$$
where
$$X_7:=\frac{\frac{\lambda_{1,2}}{\lambda_{1,1}}}{(\lambda_{1,2}-\lambda_{1,1}y)\left(C_{n-2,n-1}+\frac{1}{\lambda_{1,1}}\right)}.$$
Therefore, by (\ref{definition of Jacobi sum}) we obtain that
\begin{multline}\label{last case sum two}
\sum_{y\in\F_p\backslash\{\frac{\lambda_{1,2}}{\lambda_{1,1}}\}}X_6=(-1)^{a_{n-1}-a_2}J(a_{n-1}-a_2,[a_1-a_{n-1}]_1+n-2)\\ \cdot\left\lceil C_{n-2,n-1}-\frac{1}{\lambda_{1,1}}\right\rceil^{[a_1-a_2]_1+n-2} \left\lceil\frac{\lambda_{1,2}}{\lambda_{1,1}}\right\rceil^{a_2-a_{n-1}}.
\end{multline}
We emphasize that (\ref{last case sum two}) still holds even if $C_{n-2,n-1}+\frac{1}{\lambda_{1,1}}=0$.

Hence, we deduce that
\begin{align*}
\sum_{y\in\F_p\backslash\{\frac{\lambda_{1,2}}{\lambda_{1,1}}\}}X_6\cdot Y_6 &=(-1)^{a_{n-1}-a_2}J(a_{n-1}-a_2,[a_1-a_{n-1}]_1+n-2)\lceil\lambda_{1,1}\rceil^{a_1-a_2}\lceil\lambda_{1,2}\rceil^{a_2-a_0}\\
&\qquad\qquad\qquad\cdot \left\lceil C_{n-3,n-2}+\frac{\lambda_{1,1}}{\lambda_{1,2}}\right\rceil^{[a_2-a_{3}]_1+n-2}\left\lceil C_{n-2,n-1}-\frac{1}{\lambda_{1,1}}\right\rceil^{[a_1-a_2]_1+n-2}\\
&=(-1)^{a_{n-1}-a_2}J(a_{n-1}-a_2,[a_1-a_{n-1}]_1+n-2)\cdot X_8\cdot Y_8
\end{align*}
where
$$X_8:=\lceil\lambda_{1,1}\rceil^{a_1-a_0}\left\lceil C_{n-2,n-1}-\frac{1}{\lambda_{1,1}}\right\rceil^{[a_1-a_2]_1+n-2}$$
and
$$Y_8:=\left\lceil\frac{\lambda_{1,2}}{\lambda_{1,1}}\right\rceil^{a_2-a_0}\left\lceil C_{n-3,n-2}+\frac{\lambda_{1,1}}{\lambda_{1,2}}\right\rceil^{[a_2-a_{3}]_1+n-2}.$$
It is not difficult to see that
\begin{equation}\label{last case sum three}
\left\{
\begin{array}{l}
\sum_{\lambda_{1,1}\in\F_p^{\times}}X_8=\lceil C_{n-2,n-1}\rceil^{[a_0-a_2]_1+n-2}J([a_0-a_1]_1, [a_1-a_2]_1+n-2);\\
\sum_{\frac{\lambda_{1,2}}{\lambda_{1,1}}\in\F_p^{\times}}Y_8=(-1)^{a_2-a_0}\lceil C_{n-3,n-2}\rceil^{[a_0-a_3]_1+n-2}J([a_0-a_2]_1, [a_2-a_3]_1+n-2).\\
\end{array}\right.
\end{equation}
Combining (\ref{last case sum one}), (\ref{last case sum two}) and (\ref{last case sum three}), we finish the proof.
\end{proof}

We define
\begin{equation}\label{rational function}
\mathcal{P}_n:=\prod_{k=1}^{n-2}\prod_{j=1}^{n-2}\frac{[a_k-a_{n-1}]_1+j}{[a_0-a_k]_1+j}=\prod_{k=1}^{n-2}\prod_{j=0}^{n-3}\frac{a_k-a_{n-1}+j}{a_0-a_k+j}\in\Z_p^{\times}
\end{equation}
and
\begin{equation}\label{the sign}
\varepsilon^{\ast}:=\prod_{m=1}^{n-2}(-1)^{a_0-a_m}.
\end{equation}

\begin{lemm}\label{last congruence}
We have
\begin{equation*}
\left\{
\begin{array}{ll}
\mathrm{ord}_p(\kappa_n)=0; &\\
\kappa_n\equiv \varepsilon^{\ast}\mathcal{P}_n&\pmod{p}.
\end{array}\right.
\end{equation*}
\end{lemm}

\begin{proof}
By (\ref{relation}), we deduce that
$$
\left\{
\begin{array}{l}
\mathrm{ord}_p\left(J([a_0-a_m]_1,[a_m-a_{m+1}]_1+n-2)\right)=0;\\
\mathrm{ord}_p\left(J([a_{n-1}-a_m]_1,[a_{m-1}-a_{n-1}]_1+n-2)\right)=1,\\
\end{array}\right.$$
and thus
$\mathrm{ord}_p\left(\kappa_n\right)=0$. On the other hand, still by (\ref{relation}), we obtain that
$$
\left\{
\begin{array}{l}
J([a_0-a_m]_1,[a_m-a_{m+1}]_1+n-2)\equiv\frac{[a_0-a_m]_1!([a_m-a_{m+1}]_1+n-2)!}{([a_0-a_{m+1}]_1+n-2)!}\pmod{p};\\
p^{-1}J([a_{n-1}-a_m]_1,[a_{m-1}-a_{n-1}]_1+n-2)\equiv -\frac{[a_{n-1}-a_m]_1!([a_{m-1}-a_{n-1}]_1+n-2)!}{([a_{m-1}-a_{m}]_1+n-2)!} \pmod{p}\\
\end{array}\right.$$
for each $1\leq m\leq n-2$, and thus
\begin{align*}
\kappa_n&\equiv(-1)^{(n-2)(a_0-a_{n-1})}\left(\prod_{m=1}^{n-2}\frac{[a_0-a_m]_1!([a_m-a_{m+1}]_1+n-2)!}{([a_0-a_{m+1}]_1+n-2)!}\right)\\
&\qquad\qquad\qquad\qquad\qquad\qquad\cdot\left(\prod_{m=1}^{n-2}-\frac{[a_{n-1}-a_m]_1!([a_{m-1}-a_{n-1}]_1+n-2)!}{([a_{m-1}-a_{m}]_1+n-2)!}\right)\pmod{p}\\
&\equiv(-1)^{(n-2)(a_0-a_{n-1})}\prod_{m=1}^{n-2}\left((-1)^{a_{n-1}-a_m}\frac{[a_0-a_m]_1!([a_{m}-a_{n-1}]_1+n-2)!}{([a_0-a_{m}]_1+n-2)![a_{m}-a_{n-1}]_1!}\right)\pmod{p}\\
&\equiv\varepsilon^{\ast}\mathcal{P}_n\pmod{p},
\end{align*}
which completes the proof.
\end{proof}

\begin{proof}[Proof of Theorem {\ref{theo: identity}}]
Theorem \ref{theo: identity} follows from the combination of Lemma \ref{lemm: direct}, Proposition \ref{main conclusion}, Proposition \ref{last case} and Lemma \ref{last congruence}.
\end{proof}

\section{Mod $p$ local-global compatibility}\label{sec: local-global}
In this section, we state and prove our main results on mod $p$ local-global compatibility, which is a global application of our local results of Sections~\ref{sec: local Galois side} and~\ref{sec: local automorphic side}. In the first two sections, we recall some necessary known results on algebraic autormophic forms and Serre weights, for which we closely follow \cite{EGH}, \cite{HLM}, and \cite{BLGG}.

We first fix some notation for the whole section. Let $P\supseteq B$ be an arbitrary standard parabolic subgroup and $N$ its unipotent radical. We denote the opposite parabolic by $P^-:=w_0Pw_0$ with corresponding unipotent radical $N^-:=w_0Nw_0$. We fix a standard choice of Levi subgroup $L:=P\cap P^-\subseteq G$. We denote the positive roots of $L$ defined by the pair $(B\cap L, T)$ by $\Phi^+_L$. We use
\begin{equation}\label{L dominant}
X_L(T)_+:=\{\lambda\in X(T)\mid \langle\lambda, \alpha^{\vee}\rangle >0\mbox{ for all }\alpha\in\Phi^+_L\}
\end{equation}
to denote the set of dominant weights with respect to the pair $(B\cap L, T)$. We denote the Weyl group of $L$ by $W^L$ and identify it with a subgroup of $W$. The longest Weyl element in $W^L$ is denoted by $w_0^L$. We define the affine Weyl group $\widetilde{W}^L$ of $L$ as the semi-direct product of $W^L$ and $X(T)$ with respect to the natural action of $W^L$ on $X(T)$. Therefore $\widetilde{W}^L$ has a natural embedding into $\widetilde{W}$. We define the subgroups $\overline{P}$, $\overline{L}$, $\cdots$ of $\overline{G}$ in the obvious similar fashion.

We also need to define several open compact subgroups of $L(\Q_p)$. We define
$$K^L:=L(\Z_p),$$
and via the mod $p$ reduction map
$$\mathrm{red}^L: K^L=L(\Z_p)\twoheadrightarrow  L(\F_p)$$
we also define $K^L(1)$, $I^L(1)$, and $I^L$ as follows:
\begin{equation}\label{open compact subgroups}
\begin{array}{ll}
K^L(1):=(\mathrm{red}^L)^{-1}(1) &\subseteq\,\, I^L(1):=(\mathrm{red}^L)^{-1}( U(\F_p)\cap L(\F_p))\\
&\subseteq\,\, I^L:=(\mathrm{red}^L)^{-1}( B(\F_p)\cap L(\F_p)).
\end{array}
\end{equation}

For any dominant weight $\lambda\in X(T)_+$, we let
$$H^0_L(\lambda):=\left(\mathrm{Ind}^{\overline{L}}_{\overline{B}\cap\overline{L}}w_0^L\lambda\right)^{\rm{alg}}_{/\F_p}$$
be the associated dual Weyl module of $L$. We also write $F^L(\lambda):=\mathrm{soc}_{\overline{L}}\left(H^0_L(\lambda)\right)$ for its irreducible socle as an algebraic representation of $\overline{L}$. Through a similar argument presented at the beginning of Section~\ref{sec: local automorphic side}, the notation $F^L(\lambda)$ is well defined as an irreducible representation of $L(\F_p)$ if $\lambda\in T(\F_p)$ is $p$-regular, namely lies in the image of $X^{\rm{reg}}_1(T)\rightarrow X(T)/(p-1)X(T)$. We will sometimes abuse the notation $F^L(\lambda)$ for $F^L(\lambda)\otimes_{\F_p}\F$ or $F^L(\lambda)$ for $F^L(\lambda)\otimes_{\F_p}\overline{\F}_p$ in the literature. We will emphasize the abuse of the notation $F^L(\lambda)$ each time we do so.

We introduce some specific standard parabolic subgroups of $G$. Fix integers $i_0$ and $j_0$ such that $0\leq j_0<j_0+1<i_0\leq n-1$, and let $i_1$ and $j_1$ be the integers determined by the equation
\begin{equation}\label{definition of i_1 and j_1}
i_0+i_1=j_0+j_1=n-1.
\end{equation}
We let $P_{i_1,j_1}\supset B$ be the standard parabolic subgroup of $G=\GL_n$ corresponding to the subset $\{\alpha_k\mid j_0+1\leq k\leq i_0 \}$ of $\Delta$. By specifying the notation for general $P$ to $P_{i_1,j_1}$, we can define $P_{i_1,j_1}^-$, $L_{i_1,j_1}$, $N_{i_1,j_1}$ and $N_{i_1,j_1}^{-}$. We can naturally embeds $\GL_{j_1-i_1+1}$ into $G$ with its image denoted by $G_{i_1,j_1}$ such that $L_{i_1,j_1}=G_{i_1,j_1}T$:
\begin{equation}\label{embedding for induction step}
\GL_{j_1-i_1+1}\overset{\sim}\rightarrow G_{i_1,j_1}\hookrightarrow L_{i_1,j_1}\hookrightarrow P_{i_1,j_1}\hookrightarrow G.
\end{equation}
We define $T_{i_1,j_1}$ to be the maximal tori of $G_{i_1,j_1}$ that is contained in $T$, and define $X(T_{i_1,j_1})$ to be the character group of $T_{i_1,j_1}$. If $i_1$ and $j_1$ are clear from the context (or equivalently $i_0$ and $j_0$ are clear) then we often write $P$, $P^-$ $L$, $N$, and $N^{-}$ for $P_{i_1,j_1}$, $P_{i_1,j_1}^-$, $L_{i_1,j_1}$, $N_{i_1,j_1}$, and $N_{i_1,j_1}^{-}$, respectively.

\subsection{The space of algebraic autormophic forms}\label{subsec: auto forms}
Let $F/\Q$ be a CM field with maximal totally real subfield $F^+$. We write $c$ for the generator of $\Gal(F/F^+)$, and let $S_p^+$ (resp. $S_p$) be the set of places of $F^+$ (resp. $F$) above $p$. For $v$ (resp. $w$) a finite place of $F^+$ (resp. $F$) we write $k_v$ (resp. $k_w$) for the residue field of $F_v^+$ (resp. $F_w)$.

From now on, we assume that
\begin{itemize}
\item $F/F^+$ is unramified at all finite places;
\item $p$ splits completely in $F$.
\end{itemize}
Note that the first assumption above excludes $F^+=\Q$. We also note that the second assumption is not essential in this section, but it is harmless since we are only interested in $G_{\Q_p}$-representations in this paper. Every place $v$ of $F^+$ above $p$ further decomposes and we often write $v=w w^c$ in $F$.

There exists a reductive group $G_{n /F^+}$ satisfying the following properties (cf. \cite{BLGG}, Section 2):
\begin{itemize}
\item $G_n$ is an outer form of $\GL_n$ with $G_{n/F}\cong\GL_{n/F}$,
\item $G_n$ is a quasi-split at any finite place of $F^+$;
\item $G_n(F^{+}_{v})\simeq U_n(\R)$ for all $v|\infty$.
\end{itemize}
By \cite{CHT}, Section~3.3, $G_n$ admits an integral model $\cG_n$ over $\cO_{F^+}$ such that $\cG_n\times_{\cO_{F^+}} \cO_{F^+_v}$ is reductive if $v$ is a finite place of $F^+$ which splits in $F$. If $v$ is such a place and $w$ is a place of $F$ above~$v$, then we have an isomorphism
\begin{equation}\label{iso integral}
\iota_w: \cG_n(\cO_{F^+_v})\stackrel{\sim}{\rightarrow}\cG_n(\cO_{F_w})\stackrel{\sim}{\rightarrow}\GL_n(\cO_{F_w}).
\end{equation}
We fix this isomorphism for each such place $v$ of $F^+$.

Define $F_p^+:= F^+\otimes_{\Q}\Qp$ and $\cO_{F^+,p}:= \cO_{F^+}\otimes_\Z\Zp$. If $W$ is an $\cO_E$-module endowed with an action of $\cG_n(\cO_{F^+,p})$ and $U\subset G_n(\bA_{F^+}^{\infty,p})\times\cG_n(\cO_{F^+,p})$ is a compact open subgroup,  the space of algebraic automorphic forms on $G_n$ of level $U$ and coefficients in $W$, which is also an $\cO_E$-module, is defined as follows:
\begin{equation*}
S(U,W):= \left\{f:\,G_n(F^{+})\backslash G_n(\bA^{\infty}_{F^{+}})\rightarrow W\mid f(gu)=u_p^{-1}f(g)\,\,\forall\,\,g\in G_n(\bA^{\infty}_{F^{+}}), u\in U\right\}
\end{equation*}
with the usual notation $u=u^pu_p$ for the elements in $U$.

We say that the level $U$ is \emph{sufficiently small} if $$t^{-1} G_n(F^+) t \cap U$$
has finite order prime to $p$ for all $t \in G_n(\bA^{\infty}_{F^+})$.  We say that $U$ is \emph{unramified} at a finite place $v$ of $F^+$ if it has a decomposition $$U=\cG_n(\cO_{F_v^+})U^{v}$$ for some compact open $U^v\subset G_n(\bA^{\infty,v}_{F^+})$. If $w$ is a finite place of $F$, then we say, by abuse of notation, that $w$ is an unramified place for $U$ or $U$ is unramified at $w$ if $U$ is unramified at~$w\vert_{F^+}$.

For a compact open subgroup $U$ of $G_n(\bA_{F^+}^{\infty,p})\times\cG_n(\cO_{F^+,p})$, we let $\cP_U$ denote the set consisting of finite places $w$ of $F$ such that
\begin{itemize}
\item $w\vert_{F^+}$ is split in~$F$,
\item $w\notin S_p$,
\item $U$ is unramified at $w$.
\end{itemize}
For a subset $\cP\subseteq \cP_U$ of finite complement and closed with respect to complex conjugation we write $\bT^{\cP}=\cO_E[T^{(i)}_w,\,\,w\in\cP,\,i\in\{0,1,\cdots,n\}]$ for the universal Hecke algebra on $\cP$, where the Hecke operator $T_w^{(i)}$ acts on $S(U,W)$ via the usual double coset operator
$$
\iota_w^{-1}\left[ \GL_n(\cO_{F_w}) \left(\begin{matrix}
      \varpi_{w}\mathrm{Id}_i & 0 \cr 0 & \mathrm{Id}_{n-i} \end{matrix} \right)
\GL_n(\cO_{F_w}) \right]
$$
where $\varpi_w$ is a uniformizer of $\cO_{F_w}$ and $\mathrm{Id}_i$ is the identity matrix of size $i$. The Hecke algebra $\bT^{\mathcal{P}}$ naturally acts on $S(U,W)$.


Recall that we assume that $p$ splits completely in $F$. Following \cite{EGH}, Section 7.1 we consider the subset $(\Z^n_+)_0^{S_p}$ consisting of dominant weights $\underline{a}=(\underline{a}_w)_w$ where $\underline{a}_w=(a_{1,w},a_{2,w},\cdots,a_{n,w})$ satisfying
\begin{equation}\label{simmetry1}
a_{i,w}+a_{n+1-i,w^c}=0
\end{equation}
for all $w\in S_p$ and $1\leq i\leq n$. We let $$W_{\underline{a}_w}:=M_{\underline{a}_w}(\cO_{F_w})\otimes_{\cO_{F_w}}\cO_E$$ where the $M_{\underline{a}_w}(\cO_{F_w})$ is $\cO_{F_w}$-specialization of the dual Weyl module associated to $\underline{a}_w$ (cf. \cite{EGH}, Section 4.1.1); by condition~(\ref{simmetry1}), one deduces an isomorphism of $\cG_n(\cO_{F^+_v})$-representations $W_{\underline{a}_w}\circ \iota_w\cong W_{\underline{a}_{w^c}}\circ\iota_{w^c}$.  Therefore, by letting $W_{\underline{a}_v}:= W_{\underline{a}_w}\circ\iota_{w}$ for any place $w| v$, the $\cO_E$-representation of $\cG_n(\cO_{F^+,p})$
$$
W_{\underline{a}}:= \bigotimes_{v|p}W_{\underline{a}_{v}}
$$
is well-defined.

For a weight $\underline{a}\in (\Z^n_+)_0^{S_p}$, let us write $S_{\underline{a}}(\Qpbar)$ to denote the inductive limit of the spaces $S(U,W_{\underline{a}})\otimes_{\cO_E}\Qpbar$ over the compact open subgroups $U\subset G_n(\bA^{\infty,p}_{F^+})\times \cG_n(\cO_{F^+,p})$. (Note that the transition maps are induced, in a natural way, from the inclusions between levels $U$.) Then $S_{\underline{a}}(\Qpbar)$ has a natural left action of $G_n(\bA^{\infty}_{F^+})$ induced by right translation of functions.

We briefly recall the relation between the space $\cA$ of classical automorphic forms and the previous spaces of algebraic automorphic forms in the particular case which is relevant to~us. Fix an isomorphism $\imath:\Qpbar\stackrel{\sim}{\rightarrow}\bC$ for the rest of the paper. As we did for the $\cO_{F_w}$-specialization of the dual Weyl modules, we define a finite dimensional $G_n(F^+\otimes_\Q\R)$-representation $\sigma_{\underline{a}}\cong \underset{v|\infty}{\bigoplus} \sigma_{\underline{a}_v}$ with $\bC$-coefficients. (We refer to \cite{EGH}, Section 7.1.4 for the precise definition of $\sigma_{\underline{a}}$.)

\begin{lemm}[\cite{EGH}, Lemma~7.1.6]\label{Lemm: automorphic, alg vs classic}
The isomorphism $\imath:\Qpbar\stackrel{\sim}{\rightarrow}\bC$ induces an isomorphism of smooth
$G_n(\bA_{F^+}^\infty)$-representations
$$
S_{\underline{a}}(\Qpbar)\otimes_{\Qpbar,\imath}\bC\stackrel{\imath}{\longrightarrow} \mathrm{Hom}_{G_n(F^+\otimes_\Q\R)}({\sigma_{\underline{a}}^{\vee}},\cA)
$$
for any $\underline{a}\in(\Z^n_+)_0^{S_p}$
\end{lemm}

The following theorem guarantees the existence of Galois representations attached to automorphic forms on the unitary group $G_n$. We let $|\,\,\,|^{\frac{1-n}{2}}:F^{\times}\rightarrow \Q_p^{\times}$ denote the unique square root of $|\,\,\,|^{1-n}$ whose composite with $\iota:\Qpbar\stackrel{\sim}{\rightarrow}\bC$ takes positive values.

\begin{theo}[\cite{EGH}, Theorem~7.2.1]\label{theo: local/global compatibility}
Let $\Pi$ be an irreducible $G_n(\bA_{F^+}^{\infty})$-subrepresentation of $S_{\underline{a}}(\Qpbar)$.

Then there exists a continuous semisimple representation
$$
r_{\Pi} :G_F\rightarrow \GL_n(\Qpbar)
$$
such that
\begin{enumerate}
\item $r_{\Pi}^c\otimes\varepsilon^{n-1}\cong
  r_{\Pi}^\vee$;
\item for each place $w$ above $p$, the representation $r_{\Pi}|_{G_{F_w}}$ is de Rham with Hodge--Tate weights
$$
\mathrm{HT}(r_{\Pi}|_{G_{F_w}})=\{a_{1,w}+(n-1),a_{2,w}+(n-2),\cdots,a_{n,w}\};
$$
\item if $w|p$ is a place of $F$ and $v:= w|_{F^+}$ splits in $F$, then
$$
\WD(r_\Pi|_{G_{F_w}})^{\mathrm{F-ss}}\cong
 \rec_w((\Pi_v\circ\iota_w^{-1})\otimes|\cdot|^{\frac{1-n}{2}}).
$$
\end{enumerate}
\end{theo}

We note that the fact that (iii) holds without semi-simplification on the automorphic side is one of the main results of \cite{Cara}. We also note that property (iii) says that the restriction to $G_{F_w}$ is compatible with the local Langlands correspondence at $w$, which is denoted by $\rec_w$.

\subsection{Serre weights and potentially crystalline lifts}\label{subsec: Serre weigts and pot crys lifts}
In this section, we recall the relation of Serre weights and potentially crystalline lifts via (inertial) local Langlands correspondence.

\begin{defi}
A \emph{Serre weight} for $\cG_n$ is an isomorphism class of an absolutely irreducible smooth $\overline{\F}_p$-representation $V$ of $\cG_n(\cO_{F^+,p})$. If $v$ is a place of $F^+$ above $p$, then a \emph{Serre weight at $v$} is an isomorphism class of an absolutely irreducible $\overline{\F}_p$-smooth representation $V_v$ of $\cG_n(\cO_{F^+_v})$. Finally, if $w$ is a place of $F$ above $p$, a \emph{Serre weight at $w$} is an isomorphism class of an absolutely irreducible $\overline{\F}_p$-smooth representation $V_w$ of $\GL_n(\cO_{F_w})$.
\end{defi}

We will often say a Serre weight for a Serre weight for $\cG_n$ if $\cG_n$ is clear from the context. Note that if $V_v$ is a Serre weight at $v$, there is an associated Serre weight at $w|v$ defined by $V_{v}\circ \iota_{w}^{-1}$.

As explained in \cite{EGH}, Section 7.3, a Serre weight $V$ admits an explicit description in terms of $\GL_n(k_w)$-representations.  More precisely, let $w$ be a place of $F$ above~$p$ and write $v:= w|_{F^+}$.  For any $n$-tuple of integers $\underline{a}_w:= (a_{1,w},a_{2,w},\cdots,a_{n,w})\in \Z^n_+$, that is restricted (i.e., $0\leq a_{i,w}-a_{i+1,w}\leq p-1$ for $i=1,2,\cdots,n-1$), we consider the Serre weight $F(\underline{a}_w):= F(a_{1,w},a_{2,w},\cdots,a_{n,w})$, as defined in \cite{EGH}, Section~4.1.2. It is an irreducible $\overline{\F}_p$-representation of $\GL_n(k_w)$ and of $\cG_n(k_v)$ via the isomorphism $\iota_w$.
Note that $F(a_{1,w},a_{2,w},\cdots,a_{n,w})^{\vee}\circ\iota_{w^c}\cong F(a_{1,w},a_{2,w},\cdots,a_{n,w}) \circ\iota_w$ as $\cG_n(k_v)$-representations, i.e. $F(\underline{a}_{w^c})\circ\iota_{w^c}\cong F(\underline{a}_w)\circ\iota_{w}$ if $a_{i,w}+a_{n+1-i,w^c}=0$ for all $1\leq i\leq n$. Hence, if $\underline{a}=(\underline{a}_{w})_w\in (\Z^n_+)_0^{S_p}$ that is restricted, then we can set $F_{\underline{a}_{v}}:= F(\underline{a}_{w})\circ\iota_{w}$ for $w|v$. We also set
$$
F_{\underline{a}}:= \underset{v\vert p}{\bigotimes} F_{\underline{a}_v}
$$
which is a Serre weight for $\cG_n(\cO_{F^+,p})$. From \cite{EGH}, Lemma 7.3.4 if $V$ is a Serre weight for $\cG_n$, there exists a restricted weight $\underline{a}=(\underline{a}_w)_w\in (\Z^n_+)_0^{S_p}$ such that $V$ has a decomposition $V\cong \underset{v\vert p}{\bigotimes} V_v$ where the $V_v$ are Serre weights at $v$ satisfying $V_v\circ\iota_w^{-1}\cong F(\underline{a}_w)$.

Recall that we write $\F$ for the residue field of $E$.
\begin{defi}\label{definition modularity}
Let $\rbar:G_F\rightarrow \GL_n(\F)$ be an absolutely irreducible continuous Galois representation and let $V$ be a Serre weight for $\cG_n$. We say that $\overline{r}$ is \emph{automorphic of weight $V$} (or that $V$ \emph{is a Serre weight of $\overline{r}$}) if there exists a compact open subgroup $U$ in $G_n(\bA^{\infty,p}_F)\times \cG_n(\cO_{F^+,p})$ unramified above $p$ and a cofinite subset $\cP\subseteq \cP_U$  such that $\rbar$ is unramified at each place of $\cP$ and
$$
S(U,V)_{\mathfrak{m}_{\rbar}}\neq0
$$
where $\mathfrak{m}_{\rbar}$ is the kernel of the system of Hecke eigenvalues $\overline{\alpha}:\bT^{\cP}\rightarrow \F$ associated to $\overline{r}$, i.e.
$$
\det\left(1-\overline{r}^{\vee}(\mathrm{Frob}_w)X\right)=\sum_{j=0}^n (-1)^j(\mathbf{N}_{F/\Q}(w))^{\binom{j}{2}}\overline{\alpha}(T_w^{(j)})X^j
$$
for all $w\in \cP$.
\end{defi}

We write $W(\rbar)$ for the set of automorphic Serre weights of $\rbar$. Let $w$ be a place of $F$ above $p$ and $v=w|_{F^+_p}$. We also write $W_{w}(\rbar)$ for the set of Serre weights $F(\underline{a}_w)$ such that $$\left(F(\underline{a}_w)\circ\iota_w\right)\otimes \left(\bigotimes_{v'\in S^+_p\setminus\{v\}}V_{v'}\right)\in W(\rbar)$$ where $V_{v'}$ are Serre weights of $\cG_n(\cO_{F^+_{v'}})$ for all $v'\in S^+_p\setminus\{v\}$. We often write $W(\rbar|_{G_{F_w}})$ and $W_w(\rbar|_{G_{F_w}})$ for $W(\rbar)$ and $W_w(\rbar)$ respectively, when the given $\rbar|_{G_{F_w}}$ is clearly a restriction of an automorphic representation $\rbar$ to $G_{F_w}$.

Fix a place $w$ of $F$ above $p$ and let $v=w|_{F^+_p}$. We also fix a compact open subgroup $U$ of $G_n(\bA^{\infty,p}_F)\times \cG_n(\cO_{F^+,p})$ which is sufficiently small and unramified above $p$. We may write $U=\cG_n(\cO_{F^+_v})\times U^v$. If $W'$ is an $\cO_E$-module with an action of $\prod_{v'\in S^+_p\setminus\{v\}}\cG_n(\cO_{F^+_{v'}})$, we define
$$S(U^v,W'):= \underset{\underset{U_v}\longrightarrow}\lim\,\, S(U^v\cdot U_v, W')$$
where the limit runs over all compact open subgroups $U_v$ of $\cG_n(\cO_{F^+_{v}})$, endowing $W'$ with a trivial $\cG_n(\cO_{F^+_v})$-action. Note that $S(U^v,W')$ has a smooth action of $\cG_n(F^+_v)$ (given by right translation) and hence of $\GL_n(F_w)$ via $\iota_w$. We also note that $S(U^v,W')$ has an action of $\bT^{\mathcal{P}}$ commuting with the smooth action of $\cG_n(F^+_v)$, where $\mathcal{P}$ is a cofinite subset of~$\mathcal{P}_U$.

\begin{lemm}[\cite{EGH}, Lemma 7.4.3]
Let $U$ be a compact open subgroup of $G_n(\bA^{\infty,p}_F)\times \cG_n(\cO_{F^+,p})$ which is sufficiently small and unramified above $p$, and $\mathcal{P}$ a cofinite subset of $\mathcal{P}_U$. Fix a place $w$ of $F$ above $p$ and let $v=w|_{F^+_p}$. Let $V\cong\bigotimes_{v'\in S^+_p}V_{v'}$ be a Serre weight for $\cG_n$. Then there is a natural isomorphism of $\bT^{\mathcal{P}}$-modules
$$
\Hom_{\cG_n(\cO_{F^+_v})}\left(V_v^{\vee},\,\, S(U^v,V')\right)\,\overset{\sim}\rightarrow\, S(U,V)
$$
where $V':=\bigotimes_{v'\in S^+_p\setminus\{v\}}V_{v'}$.
\end{lemm}

We now recall some formalism related to  Deligne--Lusztig representations from Section~\ref{subsec: Deligne Lusztig representations}. Let $w$ be a place of $F$ above $p$. For a positive integer $m$, let  $k_{w,m}/k_w$ be an extension satisfying $[k_{w,m}:k_w]=m$, and let $\mathbb{T}$ be a $F$-stable maximal torus in ${\GL_n}_{/k_w}$ where $F$ is the Frobenius morphism. We have an identification from \cite{herzig-duke}, Lemma~4.7
\begin{equation*}
\label{torus}
\mathbb{T}(k_w) \stackrel{\sim}{\longrightarrow} \prod_{j}k_{w,n_j}^{\times}
\end{equation*}
where $n\geq n_j>0$ and $\sum_jn_j=n$; the isomorphism is unique up to $\prod_j \Gal(k_{w,n_j}/k_w)$-conjugacy.
In particular, any character $\theta: \mathbb{T}(k_w)\rightarrow \overline{\Q}_p^{\times}$ can be written as $\theta=\otimes_j\theta_j$ where $\theta_j:k_{w,n_j}^{\times}\rightarrow\overline{\Q}_p^{\times}$.

Given a $F$-stable maximal torus $\mathbb{T}$ and a primitive character $\theta$, we consider the Deligne-Lusztig representation $R^{\theta}_{\mathbb{T}}$ of $\GL_n(k_w)$ over $\overline{\Q}_p$ defined in Section~\ref{subsec: Deligne Lusztig representations}. Recall from Section~\ref{subsec: Deligne Lusztig representations} that $\Theta(\theta_j)$ is cuspidal representation of $\GL_{n_j}(k_w)$ associated to the primitive character $\theta_j$, we have
$$
R^{\theta}_{\mathbb{T}}\cong (-1)^{n-r}\cdot \mathrm{Ind}_{P_{\underline{n}}(k_w)}^{\GL_n(k_w)} {\left(\otimes_{j}\Theta(\theta_j)\right)}
$$
where $P_{\underline{n}}$ is the standard parabolic subgroup containing the Levi $\prod_{j}\GL_{n_j}$ and $r$ denotes the number of its Levi factors.

Let $F_{w,m}:= W(k_{w,m})[\frac 1p]$ for a positive integer $m$. We consider $\theta_j$ as a character on $\cO_{F_{w,n_{j}}}^{\times}$ by inflation and we define the following Galois type $\rec(\theta):I_{F_w}\rightarrow\GL_n(\overline{\Q}_p)$ as follows:
$$\rec(\theta):=\bigoplus_{j=1}^{r}\left(\underset{\sigma\in \Gal(k_{w,n_j}/k_w)}{\bigoplus}\sigma\left(\theta_j\circ\Art^{-1}_{F_{w,n_j}}\right) \right)$$
where $\theta_j$ is a primitive character on $k_{w,n_j}^{\times}$ of niveau $n_j$ for each $j=1,\cdots,r$. Recall that $\Art_{F_{w,n_j}} : F_{w,n_j}^{\times}\rightarrow W_{F_{w,n_j}}^{ab}$ is the isomorphism of local class field theory, normalized by sending the uniformizers to the geometric Frobenius.

We quickly review inertial local Langlands correspondence.
\begin{theo}[\cite{CEGGPS}, Theorem~3.7 and \cite{LLL}, Proposition~2.3.4]\label{theo: inertial Langlands}
Let $\tau:I_{\Q_p}\rightarrow \GL_n(\overline{\Q}_p)$ be a Galois type. Then there exists a finite dimensional irreducible smooth $\overline{\Q}_p$-representation $\sigma(\tau)$ of $\GL_n(\Z_p)$ such that if $\pi$ is any irreducible smooth $\overline{\Q}_p$-representation of $\GL_n(\Q_p)$ then $\pi|_{\GL_n(\Z_p)}$ contains a unique copy of $\sigma(\tau)$ as a subrepresentation if and only if $\rec_{\Q_p}(\pi)|_{I_{\Q_p}}\cong\tau$ and $N=0$ on $\rec_{\Q_p}(\pi)$.

Moreover, if $\tau\cong \oplus_{j=1}^{r}\tau_j$ and the $\tau_j$ are pairwise distinct, then $\sigma(\tau)\cong R_{\mathbb{T}}^{\theta}$ and $\tau\cong \rec(\theta)$ for a maximal torus $\mathbb{T}$ in ${\GL_n}_{/\F_p}$ and a primitive character $\theta: \mathbb{T}(\F_p)\rightarrow \overline{\Q}_p^{\times}$.
\end{theo}

The following theorem provides a connection between Serre weights and potentially crystalline lifts, which will be useful for the main result, Theorem~\ref{theo: lgc}.
\begin{theo}[\cite{LLL}, Proposition~4.2.5]\label{theo: global lifting theorem}
Let $w$ be a place of $F$ above $p$, $\mathbb{T}$ a maximal torus in ${\GL_n}_{/k_w}$, $\theta=\bigotimes_{j=1}^r\theta_j: \mathbb{T}(k_w)\rightarrow \overline{\Q}_p^{\times}$ a primitive character such that $\theta_j$ are pairwise distinct, and $V_w$ a Serre weight at $w$ for a Galois representation $\rbar:G_F\rightarrow \GL_n(\F)$.

Assume that $V_w$ is a Jordan-H\"older constituent in the mod $p$ reduction of the Deligne--Lusztig representation $R_{\mathbb{T}}^{\theta}$ of $\GL_n(k_w)$.  Then $\overline{r}\vert_{G_{F_w}}$ has a potentially crystalline lift with Hodge--Tate weights $\{-(n-1),-(n-2),\cdots,0\}$ and Galois type $\mathrm{rec}(\theta)$.
\end{theo}

For a given automorphic Galois representation $\rbar:G_F\rightarrow\GL_n(\F)$, it is quite difficult to determine if a given Serre weight is a Serre weight of $\rbar$. Thanks to the work of \cite{BLGG}, we have the following theorem, in which we refer the reader to \cite{BLGG} for the unfamiliar terminology.
\begin{theo}[\cite{BLGG}, Theorem~4.1.9]\label{theo: modularity of Serre weights}
Assume that if $n$ is even then so is $\frac{n[F^+:\Q]}{2}$, that $\zeta_p\not\in F$, and that $\rbar:G_{F}\rightarrow\GL_n(\F)$ is an absolutely irreducible representation with split ramification. Assume further that there is a $RACSDC$ automorphic representation $\Pi$ of $\GL_n(\mathbf{A}_F)$ such that
\begin{itemize}
\item $\rbar\simeq\rbar_{\Pi}$;
\item For each place $w|p$ of $F$, $r_{\Pi}|_{G_{F_w}}$ is potentially diagonalizable;
\item $\rbar(G_{F(\zeta_p)})$ is adequate.
\end{itemize}

If $\underline{a}=(\underline{a}_{w})_w\in (\Z^n_+)_0^{S_p}$ and for each $w\in S_p$ $\rbar|_{G_{F_w}}$ has a potentially diagonalizable crystalline lift with Hodge--Tate weights $\{a_{1,w}+(n-1),a_{2,w}+(n-2),\cdots,a_{n-1,w}+1,a_{n,w}\}$, then a Jordan--H\"older factor of $W_{\underline{a}}\otimes_{\Z_p}\F$ is a Serre weight of $\rbar$.
\end{theo}

\subsection{Weight elimination and automorphy of a Serre weight}\label{subsec: weight elimination}
In this section, we state our main Conjecture for weight elimination (Conjecture~\ref{conj: weight elimination}) which will be a crucial assumption in the proof of Theorem~\ref{theo: lgc}. This conjecture is now known by Bao V. Le Hung~\cite{LeH}. We also prove the automorphy of a certain obvious Serre weight under the assumptions of Taylor--Wiles type.

Throughout this section, we assume that $\rhobar_0$ is always a restriction of an automorphic representation $\rbar:G_F\rightarrow\GL_n(\F)$ to $G_{F_w}$ for a fixed place $w$ above $p$ and is generic (c.f. Definition~\ref{definition: genericity condition}). Recall that for $0\leq j_0<j_0+1<i_0\leq n-1$ we have defined a tuple of integers $(r^{i_0,j_0}_{n-1},\cdots,r^{i_0,j_0}_1,r^{i_0,j_0}_0)$ in (\ref{Special Galois type in general}), which determines the Galois types as in (\ref{Intro, Galois types}). In many cases, we will consider the dual of our Serre weights, so that we define a pair of integers $(i_1,j_1)$ by the equation~(\ref{definition of i_1 and j_1}). We also let $$b_k:=-c_{n-1-k}$$ for all $0\leq k\leq n-1$. We will keep the notation $(i_1,j_1)$ and $b_k$ for the rest of the paper.

For the rest of the this section, we are mainly interested in the following characters of $ T(\F_p)$: let
$$\mu^{\square}:=(b_{n-1},\cdots,b_0),$$
and let
$$\mu^{i_1,j_1}:=(x_{n-1},x_{n-2},\cdots,x_1,x_0),$$
$$\mu^{i_1,j_1,\prime}:=(x'_{n-1},x'_{n-2},\cdots,x'_1,x'_0),$$
and
$$\mu^{\square,i_1,j_1}:=(y_{n-1},y_{n-2},\cdots,y_1,y_0)$$
where
$$x_j=
\left\{
\begin{array}{ll}
b_j & \hbox{if $j>j_1$ or $i_1>j$;}\\
b_{j_1+i_1+1-j} & \hbox{if $j_1\geq j>i_1+1$;}\\
b_{j_1}+j_1-i_1-1 & \hbox{if $j=i_1+1$;}\\
b_{i_1}-j_1+i_1+1 & \hbox{if $j=i_1$,}
\end{array}
\right.
$$
$$x'_j=
\left\{
\begin{array}{ll}
b_j & \hbox{if $j>j_1$ or $i_1>j$;}\\
b_{j_1+i_1-1-j} & \hbox{if $j_1-1> j\geq i_1$;}\\
b_{j_1}+j_1-i_1-1 & \hbox{if $j=j_1$;}\\
b_{i_1}-j_1+i_1+1 & \hbox{if $j=j_1-1$,}
\end{array}
\right.
$$
and
$$y_j=
\left\{
\begin{array}{ll}
b_j & \hbox{if $j\not\in\{j_1, i_1\}$;}\\
b_{i_1}-j_1+i_1+1 & \hbox{if $j=j_1$;}\\
b_{j_1}+j_1-i_1-1 & \hbox{if $j=i_1$.}
\end{array}
\right.
$$
As $\rhobar_0$ is generic, each of the characters above is $p$-regular and thus uniquely determines a $p$-restricted weight up to a twist in $(p-1)X_0(T)$, and, by abuse of notation, we write $\mu^{\square},\,\mu^{i_1,j_1},\,\mu^{i_1,j_1,\prime},\,\mu^{\square,i_1,j_1}$ for those corresponding $p$-restricted weights, respectively. We will clarify the twist in $(p-1)X_0(T)$ whenever necessary. We also define two principal series representations
$$\pi^{i_1,j_1}:=\mathrm{Ind}^{ G(\F_p)}_{ B(\F_p)}\mu^{i_1,j_1}\qquad\mbox{and} \qquad \pi^{i_1,j_1,\prime}:=\mathrm{Ind}^{ G(\F_p)}_{ B(\F_p)}\mu^{i_1,j_1,\prime}.$$

We now state necessary results of weight elimination to our proof of the main results, Theorem~\ref{theo: lgc}, in this paper.
\begin{conj}\label{conj: weight elimination}
Let $\rbar:G_F\rightarrow\GL_n(\F)$ be a continuous automorphic Galois representation with $\overline{r}|_{G_{F_w}}\cong\rhobar_0$ as in (\ref{ordinary representation}). Fix a pair of integers $(i_0,j_0)$ such that $0\leq j_0<j_0+1<i_0\leq n-1$, and assume that $\rhobar_{i_0,j_0}$ is Fontaine--Laffaille generic and that $\mu^{\square,i_1,j_1}$ is $2n$-generic.

Then we have
\begin{equation*}
 W_w(\rbar)\cap\mathrm{JH}((\pi^{i_1,j_1})^{\vee})\subseteq \{F(\mu^{\square})^{\vee}, F(\mu^{\square,i_1,j_1})^{\vee}\}.
\end{equation*}
\end{conj}

Recently, we are informed that Bao V. Le Hung proved Conjecture~\ref{Intro: weight elimination} completely in his forthcoming paper~\cite{LeH}. Therefore, Conjecture~\ref{Intro: weight elimination} becomes a theorem based on the results in~\cite{LeH}.

Finally, we prove the automorphy of the Serre weight $F(\mu^{\square})^{\vee}$.
\begin{prop}\label{prop: modularity}
Keep the assumptions and notation of Conjecture~\ref{conj: weight elimination}. Assume further that if $n$ is even then so is $\frac{n[F^+:\Q]}{2}$, that $\zeta_p\not\in F$, that $\rbar:G_{F}\rightarrow\GL_n(\F)$ is an irreducible representation with split ramification, and that there is a $RACSDC$ automorphic representation $\Pi$ of $\GL_n(\mathbf{A}_F)$ such that
\begin{itemize}
\item $\rbar\simeq\rbar_{\Pi}$;
\item for each place $w'|p$ of $F$, $r_{\Pi}|_{G_{F_{w'}}}$ is potentially diagonalizable;
\item $\rbar(G_{F(\zeta_p)})$ is adequate.
\end{itemize}

Then
$$\{F(\mu^{\square})^{\vee}\}\subseteq W_w(\rbar)\cap\mathrm{JH}((\pi^{i_1,j_1})^{\vee}).$$
\end{prop}

\begin{proof}
We prove that $F(\mu^{\square})^{\vee}=F(c_{n-1},c_{n-2},\cdots,c_0)\in W_w(\overline{r})$ as well as $F(\mu^{\square})^{\vee}\in\mathrm{JH}\left((\pi^{i_1,j_1})^{\vee}\right)$. Note that $(c_{n-1},\cdots,c_0)$ is in the lowest alcove as $\rhobar_0$ is generic, so that by Theorem~\ref{theo: modularity of Serre weights} it is enough to show that $\rhobar_0$ has a potentially diagonalizable crystalline lift with Hodge--Tate weights $\{c_{n-1}+(n-1),\cdots,c_1+1,c_0\}$. Since $\rhobar_0$ is generic, by \cite{BLGGT}, Lemma~1.4.3 it is enough to show that $\rhobar_0$ has an ordinary crystalline lift with those Hodge--Tate weights. The existence of such a crystalline lift is immediate by \cite{GHLS}, Proposition~2.1.10. On the other hand, we have
$F(\mu^{\square})^{\vee}\in\mathrm{JH}((\pi^{i_1,j_1})^{\vee}$ which is a direct corollary of Theorem~\ref{generalmult}. Therefore, we conclude that $F(\mu^{\square})^{\vee}\in W_w(\overline{r})\cap\mathrm{JH}\left((\pi^{i_1,j_1})^{\vee}\right)$.
\end{proof}

\subsection{Some application of Morita theory}\label{subsec: Morita theory}
In this section, we will recall standard results from Morita theory to prove Proposition~\ref{classification of lattice} and Corollary~\ref{classification of lattice dual} which will be useful for the proof of Propostion~\ref{lattice induction} and Corollary~\ref{lattice coinduction} in the next section. We fix here an arbitrary finite group $H$ and a finite dimensional irreducible $E$-representation $V$ of $H$. By the Proposition 16.16 in \cite{CurtisReiner}, we know that for any $\cO_E$-lattice $V^{\circ}\subseteq V$, the set $\mathrm{JH}_{\F[H]}(V^{\circ}\otimes_{\cO_E}\F)$ depends only on $V$ and is independent of the choice of $V^{\circ}$, and thus we will use the notation $\mathrm{JH}_{\F[H]}(V)$ from now on. Let $\mathcal{C}$ be the category of all finitely generated $\cO_E$-module with a $H$-action which are isomorphic to subquotients of $\cO_E$-lattice in $V^{\oplus k}$ for some $k\geq 1$. The irreducible objects of $\mathcal{C}$ are $\sigma\in\mathrm{JH}_{\F[H]}(V)$. If $\sigma$ has multiplicity one in $V$, then we use $V^{\sigma}$ to denote a lattice (unique up to homothety by following the proof of Lemma 4.4.1 of \cite{EGS} as it actually requires only the multiplicity one of $\sigma$ in our notation) with cosocle~$\sigma$.

By repeating the proof of Lemma 2.3.1, Lemma 2.3.2 and Proposition 2.3.3 in \cite{Daniel15}, we deduce the following.
\begin{prop}\label{projectivity of lattice}
If $\sigma$ has multiplicity one in $V$, then the lattice $V^{\sigma}$ is a projective object in $\mathcal{C}$.
\end{prop}

We need to emphasize that the proof of Proposition 2.3.3 in \cite{Daniel15} requires only the multiplicity one of $\sigma$, although it is necessary for all Jordan--H\"older $\sigma$ to have multiplicity one to have Proposition 2.3.4 in \cite{Daniel15}.

\begin{coro}\label{projective1}
Let $\Sigma$ be a subset of $\mathrm{JH}_{\F[H]}(V)$ such that each $\sigma\in\Sigma$ has multiplicity one in $V$. If a $\cO_E$-lattice $V^{\circ}\subseteq V$ satisfies
\begin{equation}\label{cosocle of lattice}
\mathrm{cosoc}_{H}(V^{\circ}\otimes_{\cO_E}\F)=\bigoplus_{\sigma\in\Sigma}\sigma
\end{equation}
then we have a surjection
\begin{equation}\label{lift to projective}
\bigoplus_{\sigma\in\Sigma}V^{\sigma}\twoheadrightarrow V^{\circ}.
\end{equation}
\end{coro}
\begin{proof}
By (\ref{cosocle of lattice}) we have a surjection
$$V^{\circ}\twoheadrightarrow \bigoplus_{\sigma\in\Sigma}\sigma.$$
By Proposition \ref{projectivity of lattice} we know that $\bigoplus_{\sigma\in\Sigma}V^{\sigma}$ is a projective object in $\mathcal{C}$. By the definition of $V^{\sigma}$ we know that there is a surjection
$$\bigoplus_{\sigma\in\Sigma}V^{\sigma}\twoheadrightarrow \bigoplus_{\sigma\in\Sigma}\sigma$$
which can be lifted by projectiveness to (\ref{lift to projective}).
\end{proof}
Note in particular that (\ref{lift to projective}) implies automatically the surjection
\begin{equation}\label{mod p projective lift}
\bigoplus_{\sigma\in\Sigma}V^{\sigma}\otimes_{\cO_E}\F\twoheadrightarrow V^{\circ}\otimes_{\cO_E}\F.
\end{equation}

\begin{prop}\label{classification of lattice}
For a given $\Sigma$ as in Corollary~\ref{projective1}, there are a finite number of lattices (up to homothety) such that (\ref{cosocle of lattice}) holds. Moreover, if $V^{\circ}$ is such a lattice, then we have
\begin{equation*}
\mathrm{Hom}_{\cO_E[H]}\left(\bigoplus_{\sigma\in\Sigma}V^{\sigma}, V^{\circ}\right)\cong \bigoplus_{\sigma\in\Sigma}\left(\mathrm{Hom}_{\cO_E[H]}(V^{\sigma}, V^{\circ})\right)\cong \cO_E^{\mid\Sigma\mid}
\end{equation*}
\end{prop}
\begin{proof}
We fix an embedding
$$V^{\circ}\hookrightarrow V.$$
By (\ref{lift to projective}) we have a surjection
\begin{equation*}
\bigoplus_{\sigma\in\Sigma}V^{\sigma}\twoheadrightarrow V^{\circ},
\end{equation*}
and thus we have the composition
$$V^{\sigma}\hookrightarrow\bigoplus_{\sigma\in\Sigma}V^{\sigma}\twoheadrightarrow V^{\circ}\hookrightarrow V.$$
We identify $V^{\sigma}$ with its image in $V$ via this composition, and hence we have
$$V^{\circ}=\sum_{\sigma}V^{\sigma}\subseteq V.$$
In particular, we have an inclusion
$$V^{\sigma}\subseteq V^{\circ}$$
for each $\sigma\in\Sigma$.

If $V^{\sigma_1}\subseteq V^{\sigma_2}$ for some $\sigma_1\neq \sigma_2\in\Sigma$, then we have
$$V^{\circ}=\sum_{\sigma\in\Sigma, \sigma\neq \sigma_1}V^{\sigma},$$
and thus
$$\mathrm{cosoc}_{H}(V^{\circ}\otimes_{\cO_E}\F)\hookrightarrow\bigoplus_{\sigma\in\Sigma, \sigma\neq\sigma_1}\sigma$$
which is a contradiction to (\ref{cosocle of lattice}).
As a result, we deduce that
\begin{equation}\label{full cosocle}
V^{\sigma_1}\nsubseteq V^{\sigma_2}\mbox{ for each }\sigma_1\neq \sigma_2\in\Sigma.
\end{equation}

We notice that for each $\sigma_1\neq \sigma_2\in\Sigma$ and each $V^{\sigma_1}$, $V^{\sigma_2}$, there exists an integer $n\geq 1$ such that
\begin{equation}\label{bounded}
\varpi_E^{n}V^{\sigma_1}\subseteq V^{\sigma_2}\subseteq \varpi_E^{-n}V^{\sigma_1}.
\end{equation}
We define the set
$$\mathcal{E}:=\left\{(V^\sigma)_{\sigma\in\Sigma}\right\}/\sim$$
where $V^{\sigma}$ runs through lattices in $V$ with cosocle $\sigma$, and $\sim$ is the equivalence defined by
$$(V^\sigma)_{\sigma\in\Sigma}\sim (V^{\sigma,\prime})_{\sigma\in\Sigma}\Longleftrightarrow V^{\sigma,\prime}=\varpi_E^n V^\sigma\mbox{ for all }\sigma\in\Sigma \mbox{ and some }n\in\Z.$$
Then we can define
$$\mathcal{E}^{\prime}:=\left\{(V^\sigma)_{\sigma\in\Sigma}\in \mathcal{E}\mbox{ that satisfies (\ref{full cosocle})}\right\}$$
as the condition (\ref{full cosocle}) is preserved by the equivalence $\sim$.

Now we can summarize that there exists a surjective map from the set $\mathcal{E}^{\prime}$ to the set of homothety class of lattices $V^{\circ}$ satisfying (\ref{cosocle of lattice}). Therefore we only need to show that the set $\mathcal{E}^{\prime}$ is finite. By the equivalence $\sim$, we only always fix a $V^{\sigma_0}$ for a fixed element $\sigma_0\in\Sigma$ in advance. Then for each $\sigma\in\Sigma$ such that $\sigma\neq \sigma_0$, we have only finite number of choices of $V^{\sigma}$ by (\ref{bounded}), and hence $\mathcal{E}^{\prime}$ is finite.
\end{proof}

If $\sigma$ has multiplicity one in $V$, then we use $V_{\sigma}$ to denote a lattice (unique up to homothety by following the proof of Lemma 4.4.1 of \cite{EGS}) with socle $\sigma$.
\begin{coro}\label{classification of lattice dual}
Let $\Sigma$ be as in Corollary~\ref{projective1}. There are a finite number of lattices $V^{\circ}$ (up to homothety) such that
\begin{equation}\label{socle of lattice}
\mathrm{soc}_{H}(V^{\circ}\otimes_{\cO_E}\F)=\bigoplus_{\sigma\in\Sigma}\sigma
\end{equation}
holds. Moreover, if $V^{\circ}$ is such a lattice, then we have
\begin{equation*}
\mathrm{Hom}_{\cO_E[H]}\left(V^{\circ}, \bigoplus_{\sigma\in\Sigma}V^{\sigma}\right)\cong \bigoplus_{\sigma\in\Sigma}\left(\mathrm{Hom}_{\cO_E[H]}(V^{\circ}, V^{\sigma})\right)\cong \cO_E^{\mid\Sigma\mid}
\end{equation*}
\end{coro}
\begin{proof}
We only need to notice that $\mathrm{Hom}_{\cO_E}(V^{\circ},\cO_E)$ and $\mathrm{Hom}_{\cO_E}(V_{\sigma},\cO_E)$ are $\cO_E$-lattices in $\mathrm{Hom}_{E}(V,E)$ for all $\sigma\in\Sigma$. Moreover the cosocle of $\mathrm{Hom}_{\cO_E}(V_{\sigma},\cO_E)\otimes_{\cO_E}\F=\mathrm{Hom}_{\F}(V_{\sigma}\otimes_{\cO_E}\F,\F)$ is irreducible and has multiplicity one in $\mathrm{Hom}_{E}(V,E)$ (since this is essentially the dual version of Proposition~\ref{classification of lattice}). Therefore, we can apply Proposition \ref{classification of lattice} to $\mathrm{Hom}_{\cO_E}(V^{\circ},\cO_E)$ and then apply the functor $\mathrm{Hom}_{\cO_E}(\cdot, \cO_E)$ to deduce this corollary.
\end{proof}

\begin{rema}\label{description of lattice}
Similar to the proof of Proposition \ref{classification of lattice}, we can define the set
$$\mathcal{E}^{\prime\prime}:=\left\{(V_\sigma)_{\sigma\in\Sigma}\mbox{ such that }V_{\sigma_1}\nsubseteq V_{\sigma_2}\mbox{ for each }\sigma_1\neq \sigma_2\in\Sigma\right\}/\sim$$
where $\sim$ is the equivalence defined by simultaneous homothety as in the definition of $\mathcal{E}$. Then there is a surjective map from the finite set $\mathcal{E}^{\prime\prime}$ to the set of homothety class of lattices $V^{\circ}$ satisfying (\ref{socle of lattice}) by sending $(V_\sigma)_{\sigma\in\Sigma}$ to $\bigcap_{\sigma\in\Sigma}V_\sigma$.
\end{rema}

\subsection{Complementary results on the local automorphic side}\label{subsec: supplementary results on auto side}
In this section, we establish further results on the local automorphic side that will be used in the proof of Proposition~\ref{general nonvanishing} and Theorem~\ref{theo: lgc}. One of the main result of this section is Corollary~\ref{lattice coinduction} (or rather Proposition~\ref{lattice induction}), which will be used in the proof of Theorem~\ref{theo: lgc} to deduce that certain lattice in a principal series comes from coinduction.

In this section, we will use the notation $P$ (resp. $N$, $L$, $P^-$ $\cdots$) for general standard parabolic subgroup (resp. unipotent radical, Levi, opposite parabolic subgroup, $\cdots$) as introduced at the beginning of Section~\ref{sec: local-global}.

We use our standard notation
\begin{equation*}
\pi=\mathrm{Ind}^{ G(\F_p)}_{ B(\F_p)}\mu_{\pi}
\end{equation*}
for a principal series representation where
\begin{equation*}
\mu_{\pi}=(d_1,d_2,\cdots,d_n).
\end{equation*}
We will also consider the representation
\begin{equation*}
\pi^L:=\mathrm{Ind}^{ L(\F_p)}_{ B(\F_p)\cap  L(\F_p)}\mu_{\pi}.
\end{equation*}
We note that by the definition of these principal series we have natural surjection of $ L(\F_p)$ representation
\begin{equation}\label{surjection from coinvariant}
\pi|_{ L(\F_p)}\twoheadrightarrow\pi_{ N(\F_p)}\twoheadrightarrow\pi^L
\end{equation}
where the left side is the $ N(\F_p)$-coinvariant of $\pi$.
We fix a non-zero vector $v_{\pi}\in\pi^{ U(\F_p),\mu_{\pi}}$ and denote its image in $\pi^L$ by $v_{\pi}^L$.

\begin{lemm}\label{passtocovariant}
Fix an element $w\in W^L$. The surjection (\ref{surjection from coinvariant}) maps $S_{\underline{k},w}v_{\pi}$ to $S_{\underline{k},w}v^L_{\pi}$ and induces a bijection between the following two sets
\begin{equation*}
\{S_{\underline{k},w}v_{\pi}\mid \underline{k}=(k_{\alpha})_{\alpha\in\Phi^+_{w}}\} \longleftrightarrow \{S_{\underline{k},w}v^L_{\pi}\mid \underline{k}=(k_{\alpha})_{\alpha\in\Phi^+_{w}}\},
\end{equation*}
where $S_{\underline{k},w}$ on the right side is interpreted as an element in $\F_p[ L(\F_p)]$ and $S_{\underline{k},w}$ on the left side is interpreted as an element in $\F_p[ G(\F_p)]$ which is the image of the Jacobi sum on the right side via the natural embedding $\F_p[ L(\F_p)]\hookrightarrow\F_p[ G(\F_p)]$. 
\end{lemm}

\begin{proof}
We recall from (\ref{consequence of Bruhat}) the decomposition
$$\pi=\oplus_{w\in W}\pi_w.$$
Similarly, we also have
$$\pi^L=\oplus_{w\in W^L}\pi^L_w.$$
We also recall from the proof of Proposition \ref{prop: basis} the following decompostion
$$\pi_w=\oplus_{A\in U_w(\F_p)}\pi_{w,A}$$
and similarly we have
$$\pi^L_w=\oplus_{A\in U_w(\F_p)}\pi^L_{w,A}$$
where $\pi^L_{w,A}$ is the subspace of $\pi^L$ consisting of functions supported in $B\cap L(\F_p)w^{-1}A^{-1}$.

Notice that we have the following equality of set
\begin{equation*}
 B(\F_p)w^{-1}A^{-1}= B(\F_p)\cap L(\F_p)\cdot N(\F_p)w^{-1}A^{-1} = B(\F_p)\cap L(\F_p)w^{-1}A^{-1} N(\F_p)
\end{equation*}
as both $w^{-1}$ and $A^{-1}$ normalize $ N(\F_p)$. Hence, by the definition of $\pi_{w,A}$ and $\pi^L_{w,A}$, we deduce that the morphism (\ref{surjection from coinvariant}) maps $\pi_{w,A}$ to $\pi^L_{w,A}$. Then by the definition of Jacobi sum operators
$$S_{\underline{k},w}\in\F_p[ L(\F_p)]\hookrightarrow\F_p[ G(\F_p)]$$
we conclude that (\ref{surjection from coinvariant}) maps $S_{\underline{k},w}v_{\pi}$ to $S_{\underline{k},w}v^L_{\pi}$.

By Proposition \ref{prop: basis} we know that (\ref{surjection from coinvariant}) maps a basis of $\pi_w$ to a basis of $\pi^L_w$ and both space has dimension $|\Phi^+_w|$, and thus (\ref{surjection from coinvariant}) actually induces an isomorphism
$$\pi_w|_{ L(\F_p)}\xrightarrow{\sim}\pi^L_w$$
and a bijection of basis stated in this lemma.
\end{proof}

\begin{lemm}\label{secondreciprocity}
For a representation $V$ of $ G(\F_p)$ and a representation $W$ of $ L(\F_p)$, we have the following form of Frobenius reciprocity
\begin{equation*}
\mathrm{Hom}_{ L(\F_p)}(V_{ N(\F_p)},W)=\mathrm{Hom}_{ G(\F_p)}(V,\mathrm{coInd}_{ P(\F_p)}^{ G(\F_p)}W)
\end{equation*}
where
\begin{equation*}
\mathrm{coInd}_{ P(\F_p)}^{ G(\F_p)}W:=(\mathrm{Ind}_{ P(\F_p)}^{ G(\F_p)}W^{\vee})^{\vee}
\end{equation*}
Here, $(\cdot)^{\vee}$ is the dual.
\end{lemm}

\begin{proof}
It is easy to prove by chasing the definition:
\begin{align*}
\mathrm{Hom}_{ L(\F_p)}(V_{ N(\F_p)},W)&=\mathrm{Hom}_{ L(\F_p)}(W^{\vee},(V_{ N(\F_p)})^{\vee})\\
&=\mathrm{Hom}_{ G(\F_p)}(\mathrm{Ind}_{ P(\F_p)}^{ G(\F_p)}W^{\vee},V^{\vee})\\
&=\mathrm{Hom}_{ G(\F_p)}(V,\mathrm{coInd}_{ P(\F_p)}^{ G(\F_p)}W).
\end{align*}
This completes the proof.
\end{proof}
\begin{rema}
In fact, $\mathrm{coInd}_{ P(\F_p)}^{ G(\F_p)}W$ and $\mathrm{Ind}_{ P(\F_p)}^{ G(\F_p)}W$ has the same Jordan Holder factors with the same multiplicities. The relation between them is essentially that the graded pieces of the socle filtraion of each of them is the graded pieces of the cosocle filtration of the other one. In fact, we also have the identification
$$\mathrm{coInd}_{ P(\F_p)}^{ G(\F_p)}\left(\cdot\right)\cong\mathrm{Ind}_{ P^-(\F_p)}^{ G(\F_p)}\left(\cdot\right).$$
\end{rema}

We use the notation $\mathrm{Inj}_{ G(\F_p)}(\cdot)$ (resp. $\mathrm{Inj}_{ L(\F_p)}(\cdot)$) for the injective envelop in the category of finite dimensional $\overline{\F}_p$-representation of $ G(\F_p)$ (resp, $ L(\F_p)$). We will abuse the shorten the notation $F(\lambda)\otimes_{\F_p}\overline{\F}_p$ (resp. $F^L(\lambda)\otimes_{\F_p}\overline{\F}_p$) to $F(\lambda)$ (resp. $F^L(\lambda)$) in the following Lemma \ref{induction of injection} and Lemma \ref{socle coinduction}.
\begin{lemm}\label{induction of injection}
Fix a $\lambda\in X^{\rm{reg}}_1(T)$. Then there are a surjection
\begin{equation}\label{surjection injective envelop}
\mathrm{Inj}_{ G(\F_p)}F(\lambda)\twoheadrightarrow \mathrm{Ind}^{ G(\F_p)}_{ P(\F_p)}\left(\mathrm{Inj}_{ L(\F_p)}F^L(\lambda)\right)\
\end{equation}
and an injection
\begin{equation}\label{injection injective envelop}
\mathrm{coInd}^{ G(\F_p)}_{ P(\F_p)}\left(\mathrm{Inj}_{ L(\F_p)}F^L(\lambda)\right)\hookrightarrow\mathrm{Inj}_{ G(\F_p)}F(\lambda).
\end{equation}
\end{lemm}
\begin{proof}
We notice that the injection (\ref{injection injective envelop}) is just the dual of the surjection (\ref{surjection injective envelop}), so that we only need to prove the existence of the surjection (\ref{surjection injective envelop}).

As $\mathrm{Inj}_{ G(\F_p)}F(\lambda)$ is indecomposable and also the projective with cosocle $F(\lambda)$, we deduce that the existence of a surjection
$$\mathrm{Inj}_{ G(\F_p)}F(\lambda)\twoheadrightarrow V$$
for a $\overline{\F}_p$-representation $V$ of $ G(\F_p)$ is equivalent to the fact
\begin{equation}\label{cosocle criterion}
\mathrm{cosoc}_{ G(\F_p)}V=F(\lambda).
\end{equation}
Now we pick any $\mu\in X_1(T)$. By Frobenius reciprocity we have
$$\mathrm{Hom}_{ G(\F_p)} \left(\mathrm{Ind}^{ G(\F_p)}_{ P(\F_p)}\left(\mathrm{Inj}_{ L(\F_p)}F^L(\lambda)\right),F(\mu)\right)=\mathrm{Hom}_{ L(\F_p)}\left(\mathrm{Inj}_{ L(\F_p)}F^L(\lambda),F(\mu)^{ N(\F_p)}\right).$$ As we know that
$$\mathrm{cosoc}_{ L(\F_p)}(\mathrm{Inj}_{ L(\F_p)}F^L(\lambda))=F^L(\lambda),$$
we deduce that
$$\mathrm{Hom}_{ L(\F_p)}\left(\mathrm{Inj}_{ L(\F_p)}F^L(\lambda),F(\mu)^{ N(\F_p)}\right)\neq 0\mbox{ if and only if }F^L(\lambda)\in\mathrm{JH}_{ L(\F_p)}(F(\mu)^{ N(\F_p)}).$$

Then by Lemma 2.3 and 2.3 in \cite{Herzigirr}, we can identify $F(\mu)^{ N(\F_p)}$ with $F^L(\mu)$, and hence $$\mathrm{Hom}_{ L(\F_p)}\left(\mathrm{Inj}_{ L(\F_p)}F^L(\lambda),F(\mu)^{ N(\F_p)}\right)\neq 0$$
implies $\lambda=\mu$. In other word, we have shown that
$$\mathrm{cosoc}_{ G(\F_p)}\left(\mathrm{Ind}^{ G(\F_p)}_{ P(\F_p)}\left(\mathrm{Inj}_{ L(\F_p)}F^L(\lambda)\right)\right)=F(\lambda),$$
and thus we finish the proof by (\ref{cosocle criterion}).
\end{proof}

\begin{lemm}\label{socle coinduction}
Fix a $\lambda\in X^{\rm{reg}}_1(T)$. For any finite dimensional $\overline{\F}_p$-representation $V$ of $ L(\F_p)$, if
$$\mathrm{soc}_{ L(\F_p)}V=F^L(\lambda),$$
then we have
$$\mathrm{soc}_{ G(\F_p)}\left(\mathrm{coInd}^{ G(\F_p)}_{ P(\F_p)}V\right)=F(\lambda).$$
\end{lemm}

\begin{proof}
By assumption we have an injection
$$V\hookrightarrow\mathrm{Inj}_{ L(\F_p)}F^L(\lambda).$$
By applying the exact functor $\mathrm{coInd}^{ G(\F_p)}_{ P(\F_p)}$ we deduce
$$\mathrm{coInd}^{ G(\F_p)}_{ P(\F_p)}V\hookrightarrow \mathrm{coInd}^{ G(\F_p)}_{ P(\F_p)}\left(\mathrm{Inj}_{ L(\F_p)}F^L(\lambda)\right).$$
We finish the proof by the second part of Lemma~\ref{induction of injection} and by observing $\mathrm{soc}_{ G(\F_p)}\left(\mathrm{Inj}_{ G(\F_p)}F(\lambda)\right)=F(\lambda)$.
\end{proof}

\begin{rema}\label{cosocle induction}
Of course, we have a similar statement for the cosocle of an induction, which is just the dual statement of Lemma \ref{socle coinduction}.
\end{rema}

\begin{rema}
In the statement of Lemma \ref{induction of injection} and Lemma \ref{socle coinduction} the coefficient of each representation is $\overline{\F}_p$. In our future application, as long as the representation $V$ in Lemma \ref{socle coinduction} is given, we can fix a sufficiently large finite extension $\F$ of $\F_p$ such that the two equalities in Lemma \ref{socle coinduction} are defined over $\F$.
\end{rema}

We consider a principal series $\pi$ and together with the characteristic $0$ principal series $\widetilde{\pi}:=\mathrm{Ind}^{ G(\F_p)}_{ B(\F_p)}\widetilde{\mu}_{\pi}$, where $\widetilde{\mu}_{\pi}$ is the Teichm\"uler lift of $\mu_{\pi}$. Here $\widetilde{\pi}$ is a $\Q_p$-representation of $ G(\F_p)$ by definition. We use the notation $\widetilde{\pi}^{\circ}$ for a lattice in $\widetilde{\pi}$, which is a $\Z_p$-subrepresentation of $\widetilde{\pi}$ such that
$$\widetilde{\pi}^{\circ}\otimes_{\Z_p}\Q_p=\widetilde{\pi}.$$
We also introduce similar notation $\widetilde{\pi}^L$ and $(\widetilde{\pi}^L)^{\circ}$ by replacing $\pi$ by $\pi^L$.

\begin{prop}\label{lattice induction}
Let $\Sigma$ be a subset of $\mathrm{JH}_{ G(\F_p)}(\pi)$. Assume that $F(\lambda)$ has multiplicity one in $\pi$ for each $F(\lambda)\in\Sigma$. Assume further that
\begin{equation*}
F^L(\lambda)\in\mathrm{JH}_{ L(\F_p)}(\pi^L)
\end{equation*}
for all $\lambda$ satisfying $F(\lambda)\in\Sigma$.

If a lattice $\widetilde{\pi}^{\circ}$ satisfies
\begin{equation*}
\mathrm{cosoc}_{ G(\F_p)}\left(\widetilde{\pi}^{\circ}\otimes_{\Z_p}\F_p\right)=\bigoplus_{F(\lambda)\in\Sigma}F(\lambda),
\end{equation*}
then there exists a lattice $(\widetilde{\pi}^L)^{\circ}$ of $\widetilde{\pi}^L$ such that
\begin{equation*}
\widetilde{\pi}^{\circ}=\mathrm{Ind}^{ G(\F_p)}_{ P(\F_p)}(\widetilde{\pi}^L)^{\circ}.
\end{equation*}
Moreover, we have
\begin{equation*}
\mathrm{cosoc}_{ L(\F_p)}\left((\widetilde{\pi}^L)^{\circ}\otimes_{\Z_p}\F_p\right)=\bigoplus_{F(\lambda)\in\Sigma}F^L(\lambda).
\end{equation*}
\end{prop}

\begin{proof}
We will continue to use the notation in Proposition \ref{classification of lattice}. Hence $\widetilde{\pi}^{F(\lambda)}$ is a lattice in $\widetilde{\pi}$ with cosocle $F(\lambda)$. We use the notation $\mathcal{E}^{\prime}_{\widetilde{\pi}}$ for the set $\mathcal{E}^{\prime}$ defined in Proposition \ref{classification of lattice} if we replace $V$ by $\widetilde{\pi}$. By Proposition \ref{classification of lattice}, we deduce the existence of an element $(\widetilde{\pi}^{F(\lambda)})_{F(\lambda)\in\Sigma}\in\mathcal{E}^{\prime}_{\widetilde{\pi}}$ such that
$$\widetilde{\pi}^{\circ}=\sum_{F(\lambda)\in\Sigma}\widetilde{\pi}^{F(\lambda)}\subseteq \widetilde{\pi}.$$
On the other hand, as $F(\lambda)$ has multiplicity one in $\pi$, $F^L(\lambda)$ must have multiplicity one in $\pi^L$, and thus we have a unique (up to homothety) lattice $(\widetilde{\pi}^L)^{F^L(\lambda)}$ in $\widetilde{\pi}^L$ with cosocle $F^L(\lambda)$. Now we consider the lattice
$$\mathrm{Ind}^{ G(\F_p)}_{ P(\F_p)}(\widetilde{\pi}^L)^{F^L(\lambda)}.$$
By applying Remark \ref{cosocle induction} to $\left(\mathrm{Ind}^{ G(\F_p)}_{ P(\F_p)}(\widetilde{\pi}^L)^{F^L(\lambda)}\right)\otimes_{\cO_E}\F$ we deduce that
$$\mathrm{cosoc}_{ G(\F_p)}\left(\mathrm{Ind}^{ G(\F_p)}_{ P(\F_p)}(\widetilde{\pi}^L)^{F^L(\lambda)}\right)=F(\lambda).$$
Hence by the uniqueness of $\widetilde{\pi}^{F(\lambda)}$ up to homothety, we conclude the existence of $(\widetilde{\pi}^L)^{F^L(\lambda)}$ satisfying
\begin{equation}\label{exact induction}
\widetilde{\pi}^{F(\lambda)}=\mathrm{Ind}^{ G(\F_p)}_{ P(\F_p)}(\widetilde{\pi}^L)^{F^L(\lambda)}
\end{equation}
for a given lattice $\widetilde{\pi}^{F(\lambda)}$.

Therefore for each element $(\widetilde{\pi}^{F(\lambda)})_{F(\lambda)\in\Sigma}\in\mathcal{E}^{\prime}_{\widetilde{\pi}}$, there exists an element $((\widetilde{\pi}^L)^{F^L(\lambda)})_{F(\lambda)\in\Sigma}\in\mathcal{E}^{\prime}_{\widetilde{\pi}^L}$ such that (\ref{exact induction}) holds for all $F(\lambda)\in\Sigma$, where $\mathcal{E}^{\prime}_{\widetilde{\pi}^L}$ is the finite set defined in Proposition \ref{classification of lattice} if we replace $V$ by $\widetilde{\pi}^L$.

Finally, by the exactness of the functor $\mathrm{Ind}^{ G(\F_p)}_{ P(\F_p)}$, we deduce the equality
$$\widetilde{\pi}^{\circ}=\sum_{F(\lambda)\in\Sigma}\widetilde{\pi}^{F(\lambda)}=\sum_{F(\lambda)\in\Sigma}\left(\mathrm{Ind}^{ G(\F_p)}_{ P(\F_p)}(\widetilde{\pi}^L)^{F^L(\lambda)}\right)=\mathrm{Ind}^{ G(\F_p)}_{ P(\F_p)}\left(\sum_{F(\lambda)\in\Sigma}(\widetilde{\pi}^L)^{F^L(\lambda)}\right).$$
Hence, letting
$$(\widetilde{\pi}^L)^{\circ}:=\sum_{F(\lambda)\in\Sigma}(\widetilde{\pi}^L)^{F^L(\lambda)}$$
completes the proof.
\end{proof}

\begin{coro}\label{lattice coinduction}
Keep the notation and the assumption of Proposition~\ref{lattice induction}.

If a lattice $\widetilde{\pi}^{\circ}$ satisfies
\begin{equation*}
\mathrm{soc}_{ G(\F_p)}\left(\widetilde{\pi}^{\circ}\otimes_{\Z_p}\F_p\right)=\bigoplus_{F(\lambda)\in\Sigma}F(\lambda),
\end{equation*}
then there exists a lattice $(\widetilde{\pi}^L)^{\circ}$ of $\widetilde{\pi}^L$ such that
\begin{equation*}
\widetilde{\pi}^{\circ}=\mathrm{coInd}^{ G(\F_p)}_{ P(\F_p)}(\widetilde{\pi}^L)^{\circ}.
\end{equation*}
Moreover, we have
\begin{equation*}
\mathrm{soc}_{ L(\F_p)}\left((\widetilde{\pi}^L)^{\circ}\otimes_{\Z_p}\F_p\right)=\bigoplus_{F(\lambda)\in\Sigma}F^L(\lambda).
\end{equation*}
\end{coro}

\begin{proof}
This is just the dual version of Proposition \ref{lattice induction}.
\end{proof}

\begin{lemm}\label{saturated lattice}
Let $H$ be an arbitrary finite group. The $p$-adic field $E$ is sufficiently large such that all irreducible representation of $H$ over $\overline{\Q}_p$ are defined over $E$.

If we have an injection $V^{\circ}\hookrightarrow W^{\circ}$ of finite rank $\cO_E$-representations of $H$, then the induced morphism
\begin{equation}\label{reduction morphism}
V^{\circ}\otimes_{\cO_E}\F\rightarrow W^{\circ}\otimes_{\cO_E}\F
\end{equation}
is injective if and only if
\begin{equation}\label{intersection lattice}
V^{\circ}=V^{\circ}\otimes_{\cO_E}E\cap W^{\circ}\hookrightarrow W^{\circ}\otimes_{\cO_E}E
\end{equation}
is.
\end{lemm}
Note that the injection $V^{\circ}\hookrightarrow W^{\circ}$ always induces a natural injection
$$V^{\circ}\otimes_{\cO_E}E\hookrightarrow W^{\circ}\otimes_{\cO_E}E$$
by the flatness of $E$ over $\cO_E$.

\begin{proof}
We notice that (\ref{intersection lattice}) holds if and only if $W^{\circ}/V^{\circ}$ is $\varpi_E$-torsion free. On the other hand, by tensoring
$$V^{\circ}\hookrightarrow W^{\circ}\twoheadrightarrow W^{\circ}/V^{\circ}$$
with $E$, we deduce that the torsion free part of $W^{\circ}/V^{\circ}$ has rank the dimension of $E$-space $W^{\circ}\otimes_{\cO_E}E/V^{\circ}\otimes_{\cO_E}E$.

As we have
\begin{align*}
\mathrm{dim}_{\F}&\left(\mathrm{Ker}(V^{\circ}\otimes_{\cO_E}\F\rightarrow W^{\circ}\otimes_{\cO_E}\F)\right)\\
&=\mathrm{dim}_{\F}(V^{\circ}\otimes_{\cO_E}\F)+\mathrm{dim}_{\F}\left((W^{\circ}/V^{\circ})\otimes_{\cO_E}\F\right)-\mathrm{dim}_{\F}(W^{\circ}\otimes_{\cO_E}\F)\\
&=\mathrm{dim}_{E}(V^{\circ}\otimes_{\cO_E}E)+\mathrm{dim}_{E}\left((W^{\circ}/V^{\circ})\otimes_{\cO_E}E\right)\\
&\qquad\qquad\qquad\qquad\qquad\qquad +\mathrm{dim}_{\F}\left((W^{\circ}/V^{\circ})^{\rm{tor}}\otimes_{\cO_E}\F\right) -\mathrm{dim}_{E}(W^{\circ}\otimes_{\cO_E}E)\\
&=\mathrm{dim}_{\F}\left((W^{\circ}/V^{\circ})^{\rm{tor}}\otimes_{\cO_E}\F\right)
\end{align*}
where $(W^{\circ}/V^{\circ})^{\rm{tor}}$ is the $\varpi_E$-torsion part of $W^{\circ}/V^{\circ}$, (\ref{reduction morphism}) is injective if and only if
$$(W^{\circ}/V^{\circ})^{\rm{tor}}=0$$
or equivalently (\ref{intersection lattice}) holds.
\end{proof}

\subsection{Generalization of Section \ref{subsec: main results in char. p}}
In this section, we fix a pair of integers $(i_0,j_0)$ satisfying $0\leq j_0<j_0+1<i_0\leq n-1$, and determine $(i_1,j_1)$ by the equation~(\ref{definition of i_1 and j_1}). We will use the shorten notation $P$ (resp. $N$, $L$, $P^-$ $\cdots$) for $P_{i_1,j_1}$ (resp. $N_{i_1,j_1}$, $L_{i_1,j_1}$, $P^-_{i_1,j_1}$, $\cdots$) as introduced at the beginning of Section~\ref{sec: local-global}. The main target of this section is to prove Proposition~\ref{general nonvanishing} using Corollary~\ref{coro: isomorphism} and results in Section~\ref{subsec: supplementary results on auto side}. Proposition~\ref{general nonvanishing} is crucial for the proof of Theorem~\ref{theo: lgc}.

Recall $\mathcal{S}_n$, $\mathcal{S}_n^{\prime}$ (cf. (\ref{sum})), whose definitions are completely determined by fixing the data $n$ and $(a_{n-1},\cdots,a_0)$. We define $\mathcal{S}_{i_1,j_1}, \mathcal{S}_{i_1,j_1}^{\prime}\in\F_p[\GL_{j_1-i_1+1}(\F_p)]$ by replacing $n$ and $(a_{n-1},\cdots,a_1,a_0)$ by $$j_1-i_1+1\,\,\mbox{ and }\,\,(b_{j_1}+j_1-i_1-1,b_{j_1-1},\cdots,b_{i_1+1},b_{i_1}-j_1+i_1+1)$$ respectively with $b_k$ as at the beginning of Section~\ref{subsec: weight elimination}. Via the natural embedding~(\ref{embedding for induction step})
we also define $\mathcal{S}^{i_1,j_1}$ (resp. $\mathcal{S}^{i_1,j_1,\prime}$) to be the image of $\mathcal{S}_{i_1,j_1}$ (resp. of $\mathcal{S}_{i_1,j_1}^{\prime}$) in $\F_p[ G(\F_p)]$.
More precisely, we have
$$\mathcal{S}_{i_1,j_1}:=S_{\underline{k}^{i_1,j_1},w_0^L}\,\,\mbox{ and }\,\,\mathcal{S}_{i_1,j_1,\prime}:=S_{\underline{k}^{i_1,j_1,\prime},w_0^L}$$
where $\underline{k}^{i_1,j_1}=(k^{i_1,j_1}_{i,j})_{i,j}\in \{0,\cdots,p-1\}^{|\Phi^+_{w_0^L}|}$ and $\underline{k}^{i_1,j_1,\prime}=(k^{i_1,j_1,\prime}_{i,j})_{i,j}\in \{0,\cdots,p-1\}^{|\Phi^+_{w_0^L}|}$ are defined by
$$k^{i_1,j_1}_{i,j}:=
\left\{
\begin{array}{ll}
[b_{i_1}-b_{n-i}]_1 & \hbox{if $n-j_1+1\leq i=j-1\leq n-i_1-1$};\\
i_1-j_1+1+[b_{i_1}-b_{j_1}]_1 & \hbox{if $i=j-1=n-j_1$};\\
0 & \hbox{if $j\geq i+2$}
\end{array}\right.$$
and
$$k^{i_1,j_1,\prime}_{i,j}:=\left\{
\begin{array}{ll}
[b_{n-1-i}-b_{j_1}]_1 & \hbox{if $n-j_1\leq i=j-1\leq n-i_1-2$};\\
i_1-j_1+1+[b_{i_1}-b_{j_1}]_1 & \hbox{if $i=j-1=n-i_1-1$};\\
0 & \hbox{if $j\geq i+2$}.
\end{array}\right.$$
We will also need the tuple $\underline{k}^{i_1,j_1,0}=(k^{i_1,j_1,0}_{i,j})_{i,j}\in \{0,\cdots,p-1\}^{|\Phi^+_{w_0^L}|}$ defined by
\begin{equation}\label{systematical tuple}
k^{i_1,j_1,0}_{i,j}:=\left\{
\begin{array}{ll}
i_1-j_1+1+[b_{i_1}-b_{j_1}]_1 & \hbox{if $n-j_1\leq i=j-1\leq n-i_1-1$};\\
0 & \hbox{if $j\geq i+2$}.
\end{array}\right.
\end{equation}

We state here a generalization of the Theorem \ref{conj: mult}.
\begin{theo}\label{generalmult}
Assume that $\mu^{\square,i_1,j_1}$ is $2n$-generic (cf. Definition~\ref{defi: generic on tuples}).
Then the constituent $F(\mu^{\square})$ has multiplicity one in $\pi^{i_1,j_1}$ (or equivalently in $\pi^{i_1,j_1,\prime}$).
\end{theo}

\begin{proof}
This is Corollary \ref{multiplicity one} if we replace $\mu^{i_1,j_1}_{\pi}$ by $\mu^{\square}$.
\end{proof}

We define
\begin{equation*}
w^{i_1,j_1}:=s_{n-j_1}\cdots s_{n-i_1-2}\in W^L\quad\mbox{and}\quad w^{i_1,j_1,\prime}:=s_{n-i_1-1}\cdots s_{n-j_1+1}\in W^L
\end{equation*}
and notice that
$$(\mu^{i_1,j_1})^{w^{i_1,j_1}}=(\mu^{\square,i_1,j_1})^{w^L_0}=(\mu^{i_1,j_1,\prime})^{w^{i_1,j_1,\prime}}.$$
We also define the following principal series
\begin{equation*}
\begin{array}{lllll}
\pi^{i_1,j_1}_{\ast}&:=\mathrm{Ind}^{ G(\F_p)}_{ B(\F_p)}(\mu^{\square,i_1,j_1})^{w_0}&\mbox{and}&
\pi^{i_1,j_1}_0&:=\mathrm{Ind}^{ G(\F_p)}_{ B(\F_p)}(\mu^{\square,i_1,j_1})^{w^L_0},
\end{array}
\end{equation*}
and recall the length function $\ell(\cdot)$ on $W$ defined at the beginning of Section~\ref{sec: local automorphic side}.  Notice that we have
$$\ell(w_0w^L_0w^{i_1,j_1})= \ell(w_0w^L_0)+\ell(w^{i_1,j_1})=\ell(w_0w^L_0w^{i_1,j_1,\prime})=\ell(w_0w^L_0)+\ell(w^{i_1,j_1,\prime}).$$
By (\ref{intertwining morphism}), we have the following morphisms of principal series
\begin{equation*}
\overline{\mathcal{T}}^{\pi^{i_1,j_1}}_{w^{i_1,j_1}}: \pi^{i_1,j_1}\rightarrow\pi^{i_1,j_1}_0, \quad\overline{\mathcal{T}}^{\pi^{i_1,j_1,\prime}}_{w^{i_1,j_1,\prime}}:\pi^{i_1,j_1,\prime}\rightarrow\pi^{i_1,j_1}_0,
\quad\mbox{and}\quad
\overline{\mathcal{T}}^{\pi^{i_1,j_1}_0}_{w_0w^L_0}: \pi^{i_1,j_1}_0\rightarrow\pi^{i_1,j_1}_{\ast}.
\end{equation*}

We define $V^{i_1,j_1}_1$ to be the subrepresentation of $\pi^{i_1,j_1}$ generated by $\mathcal{S}^{i_1,j_1}\left((\pi^{i_1,j_1})^{ U(\F_p),\mu^{i_1,j_1}}\right)$. Similarly, we define $V^{i_1,j_1,\prime}_1$ and $V^{i_1,j_1}_0$ to be the subrepresentations of $\pi^{i_1,j_1,\prime}$ and $\pi^{i_1,j_1}_0$ generated by $\mathcal{S}^{i_1,j_1,\prime}\left((\pi^{i_1,j_1,\prime})^{ U(\F_p),\mu^{i_1,j_1,\prime}}\right)$ and  $S_{\underline{k}^{i_1,j_1,0},w^L_0}\left((\pi^{i_1,j_1}_0)^{ U(\F_p),(\mu^{\square,i_1,j_1})^{w^L_0}}\right)$ respectively.

\begin{lemm}\label{non-vanishing1}
Assume that $\mu^{\square,i_1,j_1}$ is $n$-generic in the lowest alcove. Then we have
$$\mathrm{dim}_{\F_p}(\pi^{i_1,j_1}_{\ast})^{ U(\F_p),\mu^{i_1,j_1}}=\mathrm{dim}_{\F_p}(\pi^{i_1,j_1}_{\ast})^{ U(\F_p),\mu^{i_1,j_1,\prime}}=1$$
and
$$\mathcal{S}^{i_1,j_1}\left((\pi^{i_1,j_1}_{\ast})^{ U(\F_p),\mu^{i_1,j_1}}\right)= \mathcal{S}^{i_1,j_1,\prime}\left((\pi^{i_1,j_1}_{\ast})^{ U(\F_p),\mu^{i_1,j_1,\prime}}\right)\neq0.$$
\end{lemm}

\begin{proof}
By direct generalization of arguments in Lemma \ref{image under sequence} and  Proposition ~\ref{prop: reduction} we can deduce that
\begin{equation}\label{equality of two spaces}
\begin{array}{ll}
\overline{\mathcal{T}}^{\pi^{i_1,j_1}}_{w^{i_1,j_1}}\left(\mathcal{S}^{i_1,j_1}\left((\pi^{i_1,j_1})^{ U(\F_p),\mu^{i_1,j_1}}\right)\right) &=\overline{\mathcal{T}}^{\pi^{i_1,j_1,\prime}}_{w^{i_1,j_1,\prime}}\left(\mathcal{S}^{i_1,j_1,\prime}\left((\pi^{i_1,j_1,\prime})^{ U(\F_p),\mu^{i_1,j_1,\prime}}\right)\right)\\
&=S_{\underline{k}^{i_1,j_1,0},w^L_0}\left((\pi^{i_1,j_1}_0)^{ U(\F_p),(\mu^{\square,i_1,j_1})^{w^L_0}}\right).\\
\end{array}
\end{equation}
In other words, we have the surjections
\begin{equation}\label{surjection of sub}
V^{i_1,j_1}_1\twoheadrightarrow V^{i_1,j_1}_0\mbox{ and }V^{i_1,j_1,\prime}_1\twoheadrightarrow V^{i_1,j_1}_0.
\end{equation}
To lighten the notation, we pick a vector $v\in(\pi^{i_1,j_1}_0)^{ U(\F_p),(\mu^{\square,i_1,j_1})^{w^L_0}}$.

By Lemma \ref{simple formula} we can deduce that
$$S_{\underline{0},w^L_0}=X\cdot S_{\underline{k}^{i_1,j_1,0},w_0^L}$$
for some $X\in\F_p[ U(\F_p)]$, and thus
\begin{equation}\label{inclusion in sub}
S_{\underline{0},w^L_0}v\in V^{i_1,j_1}_0.
\end{equation}
On the other hand, by Lemma \ref{Uinvariant} we know that
$$\F_p[S_{\underline{0},w^L_0}v]=(\pi^{i_1,j_1}_0)^{ U(\F_p),\mu^{\square,i_1,j_1}},$$
and thus by Frobenius reciprocity we have a non-zero morphism
$$\overline{\mathcal{T}}^{\pi^{\prime}}_{w^L_0}: \pi^{\prime}\rightarrow\pi^{i_1,j_1}_0\mbox{ for }\pi^{\prime}:=\mathrm{Ind}^{ G(\F_p)}_{ B(\F_p)}\mu^{\square,i_1,j_1}$$
such that
$$\F_p[S_{\underline{0},w^L_0}v]=\overline{\mathcal{T}}^{\pi^{\prime}}_{w^L_0}\left((\pi^{\prime})^{ U(\F_p),\mu^{\square,i_1,j_1}}\right).$$
By (\ref{composition of intertwine}), we know that
$$\overline{\mathcal{T}}^{\pi^{i_1,j_1}_0}_{w_0w^L_0}\bullet\overline{\mathcal{T}}^{\pi^{\prime}}_{w^L_0}=c\overline{\mathcal{T}}^{\pi^{\prime}}_{w_0}$$
for some $c\in\F_p^{\times}$, and thus
\begin{equation}\label{nonzero Uinvariant}
\F_p\left[\overline{\mathcal{T}}^{\pi^{i_1,j_1}_0}_{w_0w^L_0}\left(S_{\underline{0},w^L_0}v\right)\right]=(\pi^{i_1,j_1}_{\ast})^{ U(\F_p),\mu^{\square,i_1,j_1}}
\end{equation}
by (\ref{image of morphism}) applied to $\overline{\mathcal{T}}^{\pi^{\prime}}_{w_0}$.

Combining (\ref{nonzero Uinvariant}) and (\ref{inclusion in sub}) we deduce that
$$\overline{\mathcal{T}}^{\pi^{i_1,j_1}_0}_{w_0w^L_0}\left(V^{i_1,j_1}_0\right)\neq 0\mbox{ or equivalently }\overline{\mathcal{T}}^{\pi^{i_1,j_1}_0}_{w_0w^L_0}\left(S_{\underline{k}^{i_1,j_1,0},w^L_0}v\right)\neq 0.$$
We finish the proof by the following observation
\begin{align*}
\mathcal{S}^{i_1,j_1}\left((\pi^{i_1,j_1}_{\ast})^{ U(\F_p),\mu^{i_1,j_1}}\right)&=\mathcal{S}^{i_1,j_1}\left(\overline{\mathcal{T}}^{\pi^{i_1,j_1}}_{w_0w^L_0w^{i_1,j_1}}\left((\pi^{i_1,j_1})^{ U(\F_p),\mu^{i_1,j_1}}\right)\right)\\
&=\overline{\mathcal{T}}^{\pi^{i_1,j_1}_0}_{w_0w^L_0}\bullet\overline{\mathcal{T}}^{\pi^{i_1,j_1}}_{w^{i_1,j_1}}\left(\mathcal{S}^{i_1,j_1}\left((\pi^{i_1,j_1})^{ U(\F_p),\mu^{i_1,j_1}}\right)\right)\\
&=\F_p\left[\overline{\mathcal{T}}^{\pi^{i_1,j_1}_0}_{w_0w^L_0}\left(S_{\underline{k}^{i_1,j_1,0},w^L_0}v\right)\right]
\end{align*}
and a similar observation for $\mathcal{S}^{i_1,j_1,\prime}\left((\pi^{i_1,j_1}_{\ast})^{ U(\F_p),\mu^{i_1,j_1,\prime}}\right)$.
\end{proof}

We recall the notation $F^L(\lambda)$ from the beginning of Section~\ref{sec: local-global}. 
We define the representation $$\pi_{i_1,j_1,0}:=\mathrm{Ind}^{ L(\F_p)}_{ B(\F_p)\cap L(\F_p)}(\mu^{\square,i_1,j_1})^{w^L_0},$$
and also define $\mathcal{V}^{i_1,j_1}$ (resp. $\mathcal{V}^{i_1,j_1,\prime}$) to be the unique (up to isomorphism) quotient of $\pi^{i_1,j_1}$ (resp. $\pi^{i_1,j_1,\prime}$) with socle $F(\mu^{\square})$, whose existence is ensured by Theorem \ref{generalmult}.

Note that $\mu^{\square,i_1,j_1}$ is a permutation of both $\mu^{i_1,j_1}$ and $\mu^{i_1,j_1,\prime}$ and thus $F(\mu^{\square,i_1,j_1})$ has multiplicity one in both $\pi^{i_1,j_1}$ and $\pi^{i_1,j_1,\prime}$. We define $V^{i_1,j_1}$ (resp. $V^{i_1,j_1,\prime}$) as the unique (up to isomorphism) quotient of $\pi^{i_1,j_1}$ (resp. $\pi^{i_1,j_1,\prime}$) with socle $F(\mu^{\square,i_1,j_1})$.

\begin{lemm}\label{lemm: non-vanishing2}
Assume that $\mu^{\square}$ is $3n$-generic in the lowest alcove.  Then we have
\begin{equation*}
0\neq \mathcal{S}^{i_1,j_1}\left((\mathcal{V}^{i_1,j_1})^{ U(\F_p),\mu^{i_1,j_1}}\right)\subseteq \mathcal{V}^{i_1,j_1}
\end{equation*}
and
\begin{equation*}
0\neq \mathcal{S}^{i_1,j_1,\prime}\left((\mathcal{V}^{i_1,j_1,\prime})^{ U(\F_p),\mu^{i_1,j_1,\prime}}\right)\subseteq \mathcal{V}^{i_1,j_1,\prime}.
\end{equation*}
We also have
\begin{equation}\label{inclusion factor}
F(\mu^{\square})\in\mathrm{JH}(V^{i_1,j_1})\cap\mathrm{JH}(V^{i_1,j_1,\prime}).
\end{equation}
\end{lemm}

\begin{proof}
By the same argument as in the proof of Corollary~\ref{coro: isomorphism}, we only need to show the inclusion
\begin{equation*}
F(\mu^{\square})\in\mathrm{JH}(V^{i_1,j_1}_1)\cap \mathrm{JH}(V^{i_1,j_1,\prime}_1).
\end{equation*}
By (\ref{surjection of sub}) we only need to show that
\begin{equation}\label{inclusion of JH factor}
F(\mu^{\square})\in\mathrm{JH}(V^{i_1,j_1}_0).
\end{equation}

To lighten the notation, we fix a vector $v\in(\pi^{i_1,j_1}_0)^{ U(\F_p),(\mu^{\square,i_1,j_1})^{w^L_0}}$ and denote its image under the composition
\begin{equation}\label{composition coinvariant}
\pi^{i_1,j_1}\twoheadrightarrow(\pi^{i_1,j_1})_{ N(\F_p)}\twoheadrightarrow\pi_{i_1,j_1},
\end{equation}
by $v^L$.  We recall the definition of the tuple $\underline{k}^{i_1,j_1,0}$ from (\ref{systematical tuple}).
We define $V_{i_1,j_1,0}$ to be the subrepresentation of $\pi_{i_1,j_0}$ generated by $S_{\underline{k}^{i_1,j_1,0},w^L_0}v^L$.
By Lemma~\ref{passtocovariant} we know that the vector $S_{\underline{k}^{i_1,j_1,0},w^L_0}v$ is sent to $S_{\underline{k}^{i_1,j_1,0},w^L_0}v^L$ under the composition (\ref{composition coinvariant}), and thus we have natural surjections
\begin{equation*}
V^{i_1,j_1}_0|_{ L(\F_p)}\twoheadrightarrow (V^{i_1,j_1}_0)_{ N(\F_p)}\twoheadrightarrow V_{i_1,j_1,0}.
\end{equation*}
On the other hand, by replacing $(a_{n-1},\cdots,a_1,a_0)$ with $(b_{j_1}+j_1-i_1-1,b_{j_1-1},\cdots,b_{i_1+1},b_{i_1}-j_1+i_1+1)$ in Theorem~\ref{weaknonvanishing}, we have an inclusion
\begin{equation*}
F^L(\mu^{\square})\in\mathrm{JH}_{ L(\F_p)}(V_{i_1,j_1,0}).
\end{equation*}

We use the notation $V$ to denote the unique quotient of $V_{i_1,j_1,0}$ with $ L(\F_p)$-socle $F^L(\mu^{\square})$, and hence we have a surjection
$$(V^{i_1,j_1}_0)_{ N(\F_p)}\twoheadrightarrow V.$$
By Lemma~\ref{secondreciprocity} this gives a non-zero morphism
\begin{equation}\label{composition}
V^{i_1,j_1}_0\rightarrow\mathrm{coInd}^{ G(\F_p)}_{ P(\F_p)}V.
\end{equation}
Now by Lemma \ref{socle coinduction} we know that
$$\mathrm{soc}_{ G(\F_p)}\left(\mathrm{coInd}^{ G(\F_p)}_{ P(\F_p)}V\right)=F(\mu^{\square}).$$
As the morphism (\ref{composition}) is non-zero, we thus deduce (\ref{inclusion of JH factor}).

By definition of $V^{i_1,j_1}$, $V^{i_1,j_1,\prime}$ and $V^{i_1,j_1}_0$, we notice that we have inclusions
$$V^{i_1,j_1}_0\subseteq V^{i_1,j_1}\mbox{ and }V^{i_1,j_1}_0\subseteq V^{i_1,j_1,\prime}$$
and thus (\ref{inclusion factor}) also follows from (\ref{inclusion of JH factor}).
\end{proof}

\begin{prop}\label{general nonvanishing}
Let $(\widetilde{\pi}^{i_1,j_1})^{\circ}$ be a lattice in $\widetilde{\pi}^{i_1,j_1}$ satisfying
$$\mathrm{soc}_{ G(\F_p)}\left((\widetilde{\pi}^{i_1,j_1})^{\circ}\otimes_{\cO_E}\F\right)\hookrightarrow F(\mu^{\square})\oplus F(\mu^{\square,i_1,j_1}).$$
Then we have
\begin{equation}\label{dimension of invariant}
\mathrm{dim}_{\F}((\widetilde{\pi}^{i_1,j_1})^{\circ}\otimes_{\cO_E}\F)^{ U(\F_p),\mu^{i_1,j_1}}= \mathrm{dim}_{\F}((\widetilde{\pi}^{i_1,j_1})^{\circ}\otimes_{\cO_E}\F)^{ U(\F_p),\mu^{i_1,j_1,\prime}}=1
\end{equation}
and
\begin{equation}\label{nonvanishing of image}
\mathcal{S}^{i_1,j_1}\left(((\widetilde{\pi}^{i_1,j_1})^{\circ}\otimes_{\cO_E}\F)^{ U(\F_p),\mu^{i_1,j_1}}\right) =\mathcal{S}^{i_1,j_1,\prime}\left(((\widetilde{\pi}^{i_1,j_1})^{\circ}\otimes_{\cO_E}\F)^{ U(\F_p),\mu^{i_1,j_1,\prime}}\right)\neq0.
\end{equation}
\end{prop}

\begin{proof}
By Bruhat decomposition (\ref{Bruhat}), we have
$$\mathrm{dim}_{E}(\widetilde{\pi}^{i_1,j_1})^{ U(\F_p),\widetilde{\mu}^{i_1,j_1}}=\mathrm{dim}_{E}(\widetilde{\pi}^{i_1,j_1})^{ U(\F_p),\widetilde{\mu}^{i_1,j_1,\prime}}=1$$
Therefore by taking the intersection of $(\widetilde{\pi}^{i_1,j_1})^{\circ}$ with the two one dimension $E$-spaces above and then taking reduction mod $\varpi_E$,  we deduce that
$$\mathrm{dim}_{\F}((\widetilde{\pi}^{i_1,j_1})^{\circ}\otimes_{\cO_E}\F)^{ U(\F_p),\mu^{i_1,j_1}}\geq 1\,\,\mbox{ and }\,\, \mathrm{dim}_{\F}((\widetilde{\pi}^{i_1,j_1})^{\circ}\otimes_{\cO_E}\F)^{ U(\F_p),\mu^{i_1,j_1,\prime}}\geq 1.$$
Then by Frobenius reciprocity and the fact that $F(\mu^{i_1,j_1})$ and $F(\mu^{i_1,j_1,\prime})$ have multiplicity one in $(\widetilde{\pi}^{i_1,j_1})^{\circ}\otimes_{\cO_E}\F$, we can deduce the equality (\ref{dimension of invariant}).

If $\mathrm{soc}_{ G(\F_p)}\left((\widetilde{\pi}^{i_1,j_1})^{\circ}\otimes_{\cO_E}\F\right)$ is $F(\mu^{\square,i_1,j_1})$, then (\ref{nonvanishing of image}) reduce to Lemma \ref{non-vanishing1}. Hence we may assume
\begin{equation}\label{fixed socle}
F(\mu^{\square})\hookrightarrow\mathrm{soc}_{ G(\F_p)} \left((\widetilde{\pi}^{i_1,j_1})^{\circ}\otimes_{\cO_E}\F\right)\hookrightarrow F(\mu^{\square})\oplus F(\mu^{\square,i_1,j_1}).
\end{equation}
from now on.

As $F(\mu^{i_1,j_1})$, $F(\mu^{i_1,j_1,\prime})$ and $F(\mu^{\square,i_1,j_1})$ have multiplicity one in $(\widetilde{\pi}^{i_1,j_1})^{\circ}\otimes_{\cO_E}\F$, we can define $V^{\square,i_1,j_1}$ (resp. $V^{\square,i_1,j_1,\prime}$) to be the unique subquotient of $(\widetilde{\pi}^{i_1,j_1})^{\circ}\otimes_{\cO_E}\F$ with cosocle $F(\mu^{i_1,j_1})$ (resp. $F(\mu^{i_1,j_1,\prime})$) and socle $F(\mu^{\square,i_1,j_1})$  if either of them exists and zero otherwise. By Frobenius reciprocity, we have surjections
$$\pi^{i_1,j_1}\twoheadrightarrow V^{\square,i_1,j_1}\,\,\mbox{ and }\,\, \pi^{i_1,j_1,\prime}\twoheadrightarrow V^{\square,i_1,j_1,\prime}.$$
If $V^{\square,i_1,j_1}$ is non-zero, then by the definition of $V^{i_1,j_1}$ before Lemma \ref{lemm: non-vanishing2}, we deduce that
$$V^{\square,i_1,j_1}\cong V^{i_1,j_1}$$
and thus by (\ref{inclusion factor}), we have
$$F(\mu^{\square})\in\mathrm{JH}(V^{\square,i_1,j_1})$$
which contradicts the fact $F(\mu^{\square})$ has multiplicity one in $(\widetilde{\pi}^{i_1,j_1})^{\circ}\otimes_{\cO_E}\F$ and actually lies in the socle of $(\widetilde{\pi}^{i_1,j_1})^{\circ}\otimes_{\cO_E}\F$. This contradiction means that we have
$$V^{\square,i_1,j_1}=0.$$
Similarly, one can show that
$$V^{\square,i_1,j_1,\prime}=0.$$

As $F(\mu^{i_1,j_1})$, $F(\mu^{i_1,j_1,\prime})$ and $F(\mu^{\square})$ have multiplicity one in $(\widetilde{\pi}^{i_1,j_1})^{\circ}\otimes_{\cO_E}\F$, we can define $V^{\square}$ (resp. $V^{\square,\prime}$) to be the unique subquotient of $(\widetilde{\pi}^{i_1,j_1})^{\circ}\otimes_{\cO_E}\F$ with cosocle $F(\mu^{i_1,j_1})$ (resp. $F(\mu^{i_1,j_1,\prime})$) and socle $F(\mu^{\square})$. By (\ref{fixed socle}) and the vanishing of $V^{\square,i_1,j_1}$ and $V^{\square,i_1,j_1,\prime}$, we deduce that
$$V^{\square}\neq 0\,\,\mbox{ and }\,\,V^{\square,\prime}\neq 0,$$
and that both $V^{\square}$ and $V^{\square,\prime}$ are actually subrepresentation of $(\widetilde{\pi}^{i_1,j_1})^{\circ}\otimes_{\cO_E}\F$.

In fact, we obviously have the isomorphism
\begin{equation}\label{identification of invariant1}
(V^{\square})^{ U(\F_p),\mu^{i_1,j_1}}\xrightarrow{\sim}((\widetilde{\pi}^{i_1,j_1})^{\circ}\otimes_{\cO_E}\F)^{ U(\F_p),\mu^{i_1,j_1}}
\end{equation}
and
\begin{equation}\label{identification of invariant2} (V^{\square,\prime})^{ U(\F_p),\mu^{i_1,j_1,\prime}}\xrightarrow{\sim}((\widetilde{\pi}^{i_1,j_1})^{\circ}\otimes_{\cO_E}\F)^{ U(\F_p),\mu^{i_1,j_1,\prime}}.
\end{equation}
By Frobenius reciprocity again, we have surjections
$$\pi^{i_1,j_1}\twoheadrightarrow V^{\square}\,\,\mbox{ and }\,\, \pi^{i_1,j_1,\prime}\twoheadrightarrow V^{\square,\prime}$$
and thus we deduce the isomorphisms
$$V^{\square}\cong \mathcal{V}^{i_1,j_1}\,\,\mbox{ and }\,\,V^{\square,\prime}\cong \mathcal{V}^{i_1,j_1}$$
by the definition of $\mathcal{V}^{i_1,j_1}$ and $\mathcal{V}^{i_1,j_1}$ before Lemma \ref{lemm: non-vanishing2}. Therefore we can deduce
\begin{equation}\label{nonvanishing in general lattice}
\mathcal{S}^{i_1,j_1}\left((V^{\square})^{ U(\F_p),\mu^{i_1,j_1}}\right)\neq 0\mbox{ and }\mathcal{S}^{i_1,j_1,\prime}\left((V^{\square,\prime})^{ U(\F_p),\mu^{i_1,j_1,\prime}}\right)\neq 0
\end{equation}
from the first part of Lemma~\ref{lemm: non-vanishing2}.

By Corollary \ref{classification of lattice dual} we deduce the existence of two lattices $\widetilde{\pi}^{i_1,j_1}_{F(\mu^{\square})}$ and $\widetilde{\pi}^{i_1,j_1}_{F(\mu^{\square,i_1,j_1})}$ in $\widetilde{\pi}^{i_1,j_1}$ such that
$$
\left\{
\begin{array}{ll}
&\mathrm{soc}_{ G(\F_p)}\left(\widetilde{\pi}^{i_1,j_1}_{F(\mu^{\square})}\otimes_{\cO_E}\F\right)=F(\mu^{\square});\\
&\mathrm{soc}_{G(\F_p)}\left(\widetilde{\pi}^{i_1,j_1}_{F(\mu^{\square,i_1,j_1})}\otimes_{\cO_E}\F\right)=F(\mu^{\square,i_1,j_1});\\
&
(\widetilde{\pi}^{i_1,j_1})^{\circ}=\widetilde{\pi}^{i_1,j_1}_{F(\mu^{\square})}\cap \widetilde{\pi}^{i_1,j_1}_{F(\mu^{\square,i_1,j_1})}\subseteq \widetilde{\pi}^{i_1,j_1}.
\end{array}
\right.
$$
Note in particular that we have an isomorphism
\begin{equation}\label{identification of representation}
\widetilde{\pi}^{i_1,j_1}_{F(\mu^{\square,i_1,j_1})}\otimes_{\cO_E}\F\cong \pi^{i_1,j_1}_{\ast}
\end{equation}
by the uniqueness (up to homothety) of lattices with socle $F(\mu^{\square,i_1,j_1})$.

The inclusion $(\widetilde{\pi}^{i_1,j_1})^{\circ}\hookrightarrow \widetilde{\pi}^{i_1,j_1}_{F(\mu^{\square})}$ induce a non-zero morphism
\begin{equation}\label{non-zero morphism}
(\widetilde{\pi}^{i_1,j_1})^{\circ}\otimes_{\cO_E}\F\hookrightarrow \widetilde{\pi}^{i_1,j_1}_{F(\mu^{\square})}\otimes_{\cO_E}\F
\end{equation}
as (\ref{non-zero morphism}) sends $F(\mu^{\square})$ in the left side into the socle of the right side by the proof of Corollary \ref{classification of lattice dual} and Proposition \ref{classification of lattice}. By the definition of $V^{\square}$ and $V^{\square,\prime}$, both of them are sent injectively into $\widetilde{\pi}^{i_1,j_1}_{F(\mu^{\square})}\otimes_{\cO_E}\F$ by (\ref{non-zero morphism}).

On the other hand, there exist a unique integer $k\geq 1$ such that
$$\varpi_E^k\widetilde{\pi}^{i_1,j_1}_{F(\mu^{\square,i_1,j_1})}\subseteq \widetilde{\pi}^{i_1,j_1}_{F(\mu^{\square})}\,\,\mbox{ and }\,\,\varpi_E^k\widetilde{\pi}^{i_1,j_1}_{F(\mu^{\square,i_1,j_1})}\nsubseteq \varpi_E\widetilde{\pi}^{i_1,j_1}_{F(\mu^{\square})}.$$
The inclusion $\varpi_E^k\widetilde{\pi}^{i_1,j_1}_{F(\mu^{\square,i_1,j_1})}\hookrightarrow \widetilde{\pi}^{i_1,j_1}_{F(\mu^{\square})}$ induces a non-zero morphism
\begin{equation}\label{non-zero morphism1}
\varpi_E^k\widetilde{\pi}^{i_1,j_1}_{F(\mu^{\square,i_1,j_1})}\otimes_{\cO_E}\F\subseteq \widetilde{\pi}^{i_1,j_1}_{F(\mu^{\square})}\otimes_{\cO_E}\F
\end{equation}
as we have $\varpi_E^k\widetilde{\pi}^{i_1,j_1}_{F(\mu^{\square,i_1,j_1})}\nsubseteq \varpi_E\widetilde{\pi}^{i_1,j_1}_{F(\mu^{\square})}$.

By (\ref{identification of representation}), the image of (\ref{non-zero morphism1}) can be identified with the unique quotient of $\pi^{i_1,j_1}_{\ast}$ with socle $F(\mu^{\square})$, which will be denoted by $V^{i_1,j_1}_{\ast}$. Then by (\ref{inclusion factor}) and the definition of $V^{i_1,j_1}$ and $V^{i_1,j_1,\prime}$, we deduce that
\begin{equation}\label{inclusion two factors}
F(\mu^{i_1,j_1}), F(\mu^{i_1,j_1,\prime})\in\mathrm{JH}(V^{i_1,j_1}_{\ast}),
\end{equation}
with multiplicity one, and thus we have the embeddings
\begin{equation}\label{two embeddings}
V^{\square}\hookrightarrow V^{i_1,j_1}_{\ast}\mbox{ and }V^{\square,\prime}\hookrightarrow V^{i_1,j_1}_{\ast}
\end{equation}
by the definition of $V^{\square}$ and $V^{\square,\prime}$.

As $V^{i_1,j_1}_{\ast}$ is a quotient of $\pi^{i_1,j_1}_{\ast}$, we deduce
$$\mathrm{dim}_{\F}(V^{i_1,j_1}_{\ast})^{ U(\F_p),\mu^{i_1,j_1}}=\mathrm{dim}_{\F}(V^{i_1,j_1}_{\ast})^{ U(\F_p),\mu^{i_1,j_1,\prime}}=1$$
from
$$\mathrm{dim}_{\F}(\pi^{i_1,j_1}_{\ast})^{ U(\F_p),\mu^{i_1,j_1}}=\mathrm{dim}_{\F}(\pi^{i_1,j_1}_{\ast})^{ U(\F_p),\mu^{i_1,j_1,\prime}}=1$$
and (\ref{inclusion two factors}) together with Frobenius reciprocity.

Then the embeddings (\ref{two embeddings}) induce the isomorphisms
\begin{equation}\label{isomorphism of invariant}
(V^{\square})^{ U(\F_p),\mu^{i_1,j_1}}\xrightarrow{\sim} (V^{i_1,j_1}_{\ast})^{ U(\F_p),\mu^{i_1,j_1}}\mbox{ and }(V^{\square,\prime})^{ U(\F_p),\mu^{i_1,j_1,\prime}}\xrightarrow{\sim} (V^{i_1,j_1}_{\ast})^{ U(\F_p),\mu^{i_1,j_1,\prime}}.
\end{equation}
Then by Lemma \ref{non-vanishing1} we can deduce that
\begin{equation}\label{nonvanishing of image1} \mathcal{S}^{i_1,j_1}\left((V^{i_1,j_1}_{\ast})^{ U(\F_p),\mu^{i_1,j_1}}\right)=\mathcal{S}^{i_1,j_1,\prime}\left((V^{i_1,j_1}_{\ast})^{ U(\F_p),\mu^{i_1,j_1,\prime}}\right).
\end{equation}
Combine (\ref{nonvanishing of image1}) and (\ref{isomorphism of invariant}) we obtain that
$$\mathcal{S}^{i_1,j_1}\left((V^{\square})^{ U(\F_p),\mu^{i_1,j_1}}\right)=\mathcal{S}^{i_1,j_1,\prime}\left((V^{\square,\prime})^{ U(\F_p),\mu^{i_1,j_1,\prime}}\right)\subseteq (\widetilde{\pi}^{i_1,j_1})^{\circ}\otimes_{\cO_E}\F$$
which finishes the proof of (\ref{nonvanishing of image}) by applying (\ref{nonvanishing in general lattice}), (\ref{identification of invariant1}), and (\ref{identification of invariant2}).
\end{proof}

\subsection{Main results}\label{subsec: Main results on l-g comp}
In this section, we state and prove our main results on mod $p$ local-global compatibility. Throughout this section, $\rhobar_0$ is always assumed to be a restriction of a global representation $\rbar:G_{F}\rightarrow\GL_n(\F)$ to $G_{F_w}$ for a fixed place $w$ of $F$ above $p$. Let $v:=w|_{F^+}$, and assume further that $\rbar$ is automorphic of a Serre weight $V=\bigotimes_{v'}V_{v'}$ with $V_w:=V_v\circ \iota_w^{-1}\cong F(\mu^{\square})^{\vee}$. We may write $V_{v'}\circ \iota_{w'}^{-1}\cong F(\underline{a}_{w'})^{\vee}$ for a dominant weight $\underline{a}_{w'}\in \Z^n_+$ where $w'$ is a place of $F$ above $v'$, and define
\begin{equation}\label{prime to v}
V':=\bigotimes_{v'\neq v}V_{v'}\quad\mbox{and}\quad\widetilde{V}':=\bigotimes_{v'\neq v}W_{\underline{a}_{v'}}.
\end{equation}
From now on, we also assume that \emph{$\underline{a}_{w'}$ is in the lowest alcove for each place $w'$ of $F$ above~$p$}, so that $V'\cong\widetilde{V}'\otimes_{\cO_E}\F$.

Let $U$ be a compact open subgroup of $G_n(\bA^{\infty,p}_F)\times \cG_n(\cO_{F^+,p})$, which is sufficiently small and unramified above $p$, such that $S(U,V)[\mathfrak{m}_{\rbar}]\not=0$ where $\mathfrak{m}_{\rbar}$ is the maximal ideal of $\bT^{\cP}$ attached to $\rbar$ for a cofinite subset $\cP$ of $\cP_{U}$.

We fix a pair of integers $(i_0,j_0)$ such that $0\leq j_0<j_0+1<i_0\leq n-1$, and determine a pair inters $(i_1,j_1)$ by the equation~(\ref{definition of i_1 and j_1}). For brevity, we will use the general notation $P$ (resp. $N$, $L$, $P^-$ $\cdots$) for the specific groups $P_{i_1,j_1}$ (resp. $N_{i_1,j_1}$, $L_{i_1,j_1}$, $P^-_{i_1,j_1}$ $\cdots$) throughout this section, where $P_{i_1,j_1}$, $N_{i_1,j_1}$, $L_{i_1,j_1}$, $P^-_{i_1,j_1}$ $\cdots$ are defined at the beginning of Section~\ref{sec: local-global}.

Recall $\widehat{\mathcal{S}}_n$, $\widehat{\mathcal{S}}^{\prime}_n$ (cf. (\ref{char 0 Jacobi sum})), $\kappa_n$ (cf. (\ref{coefficient})), $\varepsilon^{\ast\ast}$ (cf. (\ref{the sign})), and $\mathcal{P}_n$ (cf. (\ref{rational function})), whose definitions are completely determined by fixing the data $n$ and $(a_{n-1},\cdots,a_0)$. We define $\widehat{\mathcal{S}}_{i_1,j_1}, \widehat{\mathcal{S}}^{\prime}_{i_1,j_1}\in\Q_p[\GL_{j_1-i_1+1}(\Q_p)]$, $\kappa_{i_1,j_1}\in\Z_p^{\times}$, $\varepsilon^{i_1,j_1}=\pm1$ and $\mathcal{P}_{i_1,j_1}\in\Z_p^{\times}$ by replacing $n$ and $(a_{n-1},\cdots,a_1,a_0)$ by $j_1-i_1+1$ and $(b_{j_1}+j_1-i_1-1,b_{j_1-1},\cdots,b_{i_1+1},b_{i_1}-j_1+i_1+1)$ respectively with $b_k$ as at the beginning of Section~\ref{subsec: weight elimination}. Via the natural embedding~(\ref{embedding for induction step})
we also define $\widehat{\mathcal{S}}^{i_1,j_1}$ (resp. $\widehat{\mathcal{S}}^{i_1,j_1,\prime}$) to be the image of $\widehat{\mathcal{S}}_{i_1,j_1}$ (resp. $\widehat{\mathcal{S}}_{i_1,j_1}^{\prime}$) in $\Q_p[G(\Q_p)]$. Note that $\widehat{\mathcal{S}}^{i_1,j_1}$ (resp. $\widehat{\mathcal{S}}^{i_1,j_1,\prime}$) is a Teichm\"uler lift of $\mathcal{S}^{i_1,j_1}$ (resp. $\mathcal{S}^{i_1,j_1,\prime}$).

We recall the operator $\Xi_{j_i-i_1+1}\in\mathrm{GL}_{j_1-i_1+1}(\Q_p)$ from (\ref{generator of normalizer}) except that here we replace $n$ by $j_1-i_1+1$. Then we define
\begin{equation}\label{the operator Xi_i_1,j_1}
\Xi_{i_1,j_1}:=(\Xi_{j_i-i_1+1})^{j_1-i_1-1}
\end{equation}
and denote the image of $\Xi_{i_1,j_1}$ via the embedding $$\mathrm{GL}_{j_1-i_1+1}(\Q_p)\cong G_{i_1,j_1}(\Q_p)\hookrightarrow L(\Q_p)\hookrightarrow\mathrm{GL}_n(\Q_p)$$ by $\Xi^{i_1,j_1}$.

We define
$$
\left\{
\begin{array}{lll}
M&:=&S(U^v,\widetilde{V}^{\prime})_{\mathfrak{m}}\\
M^{i_1,j_1}&:=&S(U^v,\widetilde{V}^{\prime})_{\mathfrak{m}_{\overline{r}}}^{I(1),\widetilde{\mu}^{i_1,j_1}}\\
M^{i_1,j_1,\prime}&:=&S(U^v,\widetilde{V}^{\prime})_{\mathfrak{m}_{\overline{r}}}^{I(1),\widetilde{\mu}^{i_1,j_1,\prime}}\\
\end{array}\right.$$
then $M^{i_1,j_1}$ (resp. $M^{i_1,j_1,\prime}$) is a free $\cO_E$-module of finite rank as $M$ is a smooth admissible representation of $G(\Q_p)$ which is $\varpi_E$-torsion free.  For any $\cO_E$-algebra $A$, we write $M^{i_1,j_1}_A$ for $M^{i_1,j_1}\otimes_{\cO_E}A$. We similarly define $M^{i_1,j_1,\prime}_A$ and $M_A$.

\begin{defi}\label{definition: connected}
Two vectors $v^{i_1,j_1}\in M^{i_1,j_1}_\F[\mathfrak{m}_{\overline{r}}]$ and $v^{i_1,j_1,\prime}\in M^{i_1,j_1,\prime}_\F[\mathfrak{m}_{\overline{r}}]$ are said to be \emph{connected} if there exists
$$\widehat{v}^{i_1,j_1}\in M^{i_1,j_1}\,\,\,\mbox{ and }\,\,\, \widehat{v}^{i_1,j_1,\prime}\in M^{i_1,j_1,\prime}$$
that lifts $v^{i_1,j_1}$ and $v^{i_1,j_1,\prime}$ respectively such that $\widehat{v}^{i_1,j_1,\prime}$ and $\Xi^{i_1,j_1}\widehat{v}^{i_1,j_1}$ has the the same image in $\left(M_E\right)_{N^-(\Q_p)}$ via the coinvariant morphism
$$M_E\twoheadrightarrow \left(M_E\right)_{N^-(\Q_p)}.$$
We also say that $v^{i_1,j_1,\prime}$ is a connected vector to $v^{i_1,j_1}$ if $v^{i_1,j_1}$ and $v^{i_1,j_1,\prime}$ are connected.
\end{defi}

Let $\bT$ be the $\cO_E$-module that is the image of $\bT^{\cP}$ in $\mathrm{End}_{\cO_E}(M^{i_1,j_1})$. Then $\bT$ is a local $\cO_E$-algebras with the maximal ideal $\mathfrak{m}_{\rbar}$, where, by abuse of notation, we write $\mathfrak{m}_{\rbar}\subseteq\bT$ for the image of $\mathfrak{m}_{\rbar}$ of $\bT^{\cP}$. As the level $U$ is sufficiently small, by passing to a sufficiently large $E$ as in the proof of Theorem~4.5.2 of \cite{HLM}, we may assume that $\bT_E\cong E^r$ for some $r>0$. For any $\cO_E$-algebra $A$ we write $\bT_A$ for $\bT\otimes_{\cO_E}A$. Similarly, we define $\bT^{\prime}$ and $\bT^{\prime}_A$ by replacing $M^{i_1,j_1}$ by $M^{i_1,j_1,\prime}$.

We have $M^{i_1,j_1}_E=\bigoplus_{\mathfrak{p}}M^{i_1,j_1}_E[\mathfrak{p}_E]$, where the sum runs over the minimal primes $\mathfrak{p}$ of $\bT$ and $\mathfrak{p}_E:=\mathfrak{p}\bT_E$. Note that for any such $\mathfrak{p}$ $\bT_E/\mathfrak{p}_E\cong E$. By abuse of notation, we also write $\mathfrak{p}$ (resp. $\mathfrak{p}_E$) for its inverse image in $\bT^{\mathcal{P}}$ (resp. $\bT^{\mathcal{P}}_E$) and for the corresponding minimal prime ideal of $\bT^{\prime}$ (resp. $\bT^{\prime}_E$). We also note that for any such $\mathfrak{p}$ we have a surejection $M[\mathfrak{p}]\twoheadrightarrow M_{\F}[\mathfrak{m}_{\overline{r}}]$ as $\mathfrak{m}_{\overline{r}}=\mathfrak{p}+\varpi_E\bT^{\cP}$. 

\begin{defi}\label{definition: primitive}
A non-zero vector $v^{i_1,j_1}\in M^{i_1,j_1}_{\F}$ is said to be \emph{primitive} if there exists a vector $\widehat{v}^{i_1,j_1}\in M^{i_1,j_1}[\mathfrak{p}]$ that lifts $v^{i_1,j_1}$, for certain minimal prime $\mathfrak{p}$ of $\bT$.
\end{defi}
Note that the $G(\Q_p)$-subrepresentation of $M_E$ generated by $\widehat{v}^{i_1,j_1}$ is irreducible and actually lies in $M_E[\mathfrak{p}_E]$.

We have to be careful that we only know that there is an inclusion
\begin{equation}\label{unknown inclusion}
\bigoplus_{\mathfrak{p}} M^{i_1,j_1}[\mathfrak{p}]\subseteq M^{i_1,j_1}
\end{equation}
of $\cO_E$-modules, but we do not know if the equality holds. As we can always pick a minimal prime $\mathfrak{p}$ of $\bT$ and then pick an arbitrary vector $\widehat{v}^{i_1,j_1}\in M^{i_1,j_1}[\mathfrak{p}]$ such that $\widehat{v}^{i_1,j_1}\notin \varpi_EM^{i_1,j_1}[\mathfrak{p}]$, we can define $v^{i_1,j_1}$ as the image of $\widehat{v}^{i_1,j_1}$ and then deduce that $v^{i_1,j_1}$ is primitive. In other word, we have shown that a primitive vector in $M^{i_1,j_1}$ always exists, but as (\ref{unknown inclusion}) might not be an equality, we do not know in general if all primitive vectors span the whole $\F$-space $M^{i_1,j_1}_{\F}$.

Now we can state our main results in this paper. Recall that by $\rhobar_0$ we always mean an $n$-dimensional ordinary representation of $G_{\Q_p}$ as described in (\ref{ordinary representation}). We will shorten the notation $F(\lambda)_\F$ (resp. $F^L(\lambda)_\F$) to $F(\lambda)$ (resp. $F^L(\lambda)$) in the statement of the theorem and its proof.

\begin{theo}\label{theo: lgc}
Fix a pair of integers $(i_0,j_0)$ satisfying $0\leq j_0<j_0+1<i_0\leq n-1$, and let $(i_1,j_1)$ be a pair of integers such that $i_0+i_1=j_0+j_1=n-1$. We also let $\rbar:G_{F}\rightarrow\GL_n(\F)$ be an irreducible automorphic representation with $\rbar|_{G_{F_w}}\cong\rhobar_0$. Assume that
\begin{itemize}
\item $\mu^{\square,i_1,j_1}$ is $2n$-generic;
\item $\rhobar_{i_0,j_0}$ is Fontaine--Laffaille generic.
\end{itemize}
Assume further that
\begin{equation}\label{weight elimination in main results}
\{F(\mu^{\square})^{\vee}\}\subseteq W_w(\rbar)\cap\mathrm{JH}((\pi^{i_1,j_1})^{\vee})\subseteq \{F(\mu^{\square})^{\vee}, F(\mu^{\square,i_1,j_1})^{\vee}\}.
\end{equation}

Then there exists a primitive vector in $M_{\F}[\mathfrak{m}_{\overline{r}}]^{I(1),\mu^{i_1,j_1}}$. Moreover, for each primitive vector $v^{i_1,j_1}\in M_{\F}[\mathfrak{m}_{\overline{r}}]^{I(1),\mu^{i_1,j_1}}$  there exists a connected vector $v^{i_1,j_1,\prime}\in M_{\F}[\mathfrak{m}_{\overline{r}}]^{I(1),\mu^{i_1,j_1,\prime}}$ to $v^{i_1,j_1}$ such that $\mathcal{S}^{i_1,j_1,\prime}v^{i_1,j_1,\prime}\neq0$ and
\begin{equation*}
\mathcal{S}^{i_1,j_1,\prime}v^{i_1,j_1,\prime}=\varepsilon^{i_1,j_1}\mathcal{P}_{i_1,j_1}(b_{n-1},\cdots,b_0) \cdot \mathrm{FL}^{i_0,j_0}_n(\rbar|_{G_{F_w}})\cdot \mathcal{S}^{i_1,j_1}v^{i_1,j_1}
\end{equation*}
where
$$\varepsilon^{i_1,j_1}=\prod_{k=i_1+1}^{j_1-1}(-1)^{b_{i_1}-b_k-j_1+i_1+1}$$
and
$$\mathcal{P}_{i_1,j_1}(b_{n-1},\cdots,b_0)=\prod_{k=i_1+1}^{j_1-1}\prod_{j=1}^{j_1-i_1-1}\frac{b_k-b_{j_1}-j}{b_{i_1}-b_k-j}\in\Z_p^{\times}.$$
\end{theo}

The right inclusion of (\ref{weight elimination in main results}) is just Conjecture~\ref{conj: weight elimination}, which is now a theorem of Bao V. Le Hung~\cite{LeH} (cf. Remark~\ref{LeH}). We also give an evidence for the left inclusion of (\ref{weight elimination in main results}) in Proposition~\ref{prop: modularity} under some assumption of Taylor--Wiles type. As a result, the condition (\ref{weight elimination in main results}) can be removed under some standard Taylor--Wiles conditions.

\begin{proof}
We firstly point out that $M^{i_1,j_1}\neq 0$ (resp. $M^{i_1,j_1,\prime}\neq 0$), as $S(U, (F(\mu^{\square})^{\vee}\circ\iota_w)\otimes V')_{\mathfrak{m}_{\overline{r}}}\neq 0$ and $F(\mu^{\square})$ is a factor of $\overline{\mathrm{Ind}^K_{I}\widetilde{\mu}^{i_1,j_1}}=\mathrm{Ind}^{ G(\F_p)}_{ B(\F_p)}\mu^{i_1,j_1}$ (resp. $\overline{\mathrm{Ind}^K_{I}\widetilde{\mu}^{i_1,j_1,\prime}}$).

Picking an embedding $E\hookrightarrow\overline{\Q}_p$, as well as an isomorphism $\iota:\overline{\Q}_p\xrightarrow{\sim}\mathbb{C}$, we see that
\begin{equation}\label{decomposition}
M^{i_1,j_1}_{\overline{\Q}_p}\cong\bigoplus_{\Pi} m(\Pi)\cdot \Pi^{I(1),\widetilde{\mu}^{i_1,j_1}}_v\otimes(\Pi^{\infty,v})^{U^v},
\end{equation}
where the sum runs over irreducible representations $\Pi\cong\Pi_{\infty}\otimes\Pi_v\otimes\Pi^{\infty,v}$ of $G_n(\mathbb{A}_{F^+})$ over $\overline{\mathbb{Q}}_p$ such that $\Pi\otimes_{\iota}\mathbb{C}$ is a cuspidal automorphic representation of multiplicity $m(\Pi)\in\Z_{>0}$ with $\Pi_{\infty}\otimes_{\iota}\mathbb{C}$ being determined by the algebraic representation $(\widetilde{V}^{\prime})^{\vee}$ and with associated Galois representation $r_{\Pi}$ lifting $\overline{r}^{\vee}$ (cf. Lemma~\ref{Lemm: automorphic, alg vs classic}).

We write $\delta$ for the modulus character of $B(\Q_p)$:
\begin{equation*}
\delta:=\mid\,\,\,\mid^{n-1}\otimes\mid\,\,\,\mid^{n-2}\otimes\cdots\otimes\mid\,\,\,\mid\otimes 1
\end{equation*}
where $\mid\,\,\,\mid$ is the (unramfied) norm character sending $p$ to $p^{-1}$. For any $\Pi$ contributing to (\ref{decomposition}), we have
\begin{enumerate}
\item $\Pi_v\cong\mathrm{Ind}^{G(\Q_p)}_{B(\Q_p)}(\psi\otimes\delta)$ for some smooth character $$\psi=\psi_{n-1}\otimes\psi_{n-2}\otimes\cdots\otimes\psi_1\otimes\psi_0$$ of $T(\Q_p)$ such that $\psi|_{ T(\Z_p)}=\widetilde{\mu}^{i_1,j_1}|_{T(\Z_p)}$, where $\psi_k$ are the smooth characters of~$\Q_p^{\times}$.
\item $r^{\vee}_{\Pi}|_{G_{F_w}}$ is a potentially crystalline lift of $\rbar$ with Hodge--Tate weights $\{-(n-1),-(n-2),\cdots,-1,0\}$ and $\mathrm{WD}(r^{\vee}_{\Pi}|_{G_{F_w}})^{\mathrm{F-ss}}\cong\oplus_{k=0}^{n-1}\psi_k^{-1}$.
\end{enumerate}
Here, part (i) follows from \cite{EGH}, Propositions 2.4.1 and 7.4.4, and part (ii) follows from classical local-global compatibility (cf. Theorem~\ref{theo: local/global compatibility}). 
Moreover, by Corollary~\ref{coro: main thm}, we have
\begin{equation}\label{Frob}
\mathrm{FL}^{i_0,j_0}_n(\rhobar_0)= \overline{\frac{\prod_{k=j_0+1}^{i_0-1}\psi_{i_1-j_0+1+k}(p)}{p^{\frac{(i_0+j_0)(i_0-j_0-1)}{2}}}}.
\end{equation}
(Note that we may identify $\psi_{i_1-j_0+1+k}$ with $\Omega_{k}^{-1}$ for $j_0<k<i_0$, where $\Omega_{k}$ is defined in Corollary~\ref{coro: main thm}.)
We use the shorten notation
$$\widetilde{C}(\chi):=\frac{\prod_{k=j_0+1}^{i_0-1}\psi_{i_1-j_0+1+k}(p)}{p^{\frac{(i_0+j_0)(i_0-j_0-1)}{2}}}$$
for any smooth character $\chi:=\psi\otimes\delta$, and we notice that
\begin{equation}\label{equality of char 0}
\widetilde{C}(\chi)=\widetilde{C}(\chi^{\prime})\mbox{ if }\chi|_{T_{i_1,j_1}(\Q_p)}=\chi^{\prime}|_{T_{i_1,j_1}(\Q_p)}
\end{equation}
for any two smooth characters $\chi, \chi^{\prime}: T(\Q_p)\rightarrow E^{\times}$.

Now we pick an arbitrary primitive vector $v^{i_1,j_1}\in M^{i_1,j_1}_{\F}[\mathfrak{m}_{\overline{r}}]$ with a lift $\widehat{v}^{i_1,j_1}\in M^{i_1,j_1}[\mathfrak{p}]$. We set
$$\widetilde{\pi}^{i_1,j_1}:=\langle K\widehat{v}^{i_1,j_1}\rangle_E\subseteq M_E[\mathfrak{p}_E]\mbox{ and }(\widetilde{\pi}^{i_1,j_1})^{\circ}:=\widetilde{\pi}^{i_1,j_1}\cap M[\mathfrak{p}],$$
and thus $(\widetilde{\pi}^{i_1,j_1})^{\circ}$ is a $\cO_E$-lattice in $\widetilde{\pi}^{i_1,j_1}$. Note that $M^{i_1,j_1}_E[\mathfrak{p}_E]\otimes_E\overline{\Q}_p$ is a direct summand of (\ref{decomposition}) where $\Pi$ runs over a subset of automorphic representations in (\ref{decomposition}). The same argument as in the paragraph above (4.5.7) of \cite{HLM} using Cebotarev density shows us that the local component $\Pi_v$ of each $\Pi$ occurring in this direct summand does not depend on $\Pi$.

By Lemma \ref{saturated lattice} and the definition of $(\widetilde{\pi}^{i_1,j_1})^{\circ}$, we obtain an injection
\begin{equation}\label{injection into global}
(\widetilde{\pi}^{i_1,j_1})^{\circ}\otimes_{\cO_E}\F\hookrightarrow \left(M[\mathfrak{p}]\right)\otimes_{\cO_E}\F=M_\F[\mathfrak{m}_{\overline{r}}]
\end{equation}
as $\mathfrak{p}+\varpi_E\bT^{\mathcal{P}}=\mathfrak{m}_{\overline{r}}$. By the assumption~(\ref{weight elimination in main results}) (cf. Conjecture \ref{conj: weight elimination}), we deduce that
$$\mathrm{JH}\left(\mathrm{soc}_{ G(\F_p)}\left(M_\F[\mathfrak{m}_{\overline{r}}]\right)\right)\subseteq\{F(\mu^{\square}), F(\mu^{\square,i_1,j_1})\}$$
and therefore by (\ref{injection into global}) we have
$$\mathrm{JH}\left(\mathrm{soc}_{ G(\F_p)}\left((\widetilde{\pi}^{i_1,j_1})^{\circ}\otimes_{\cO_E}\F\right)\right)\subseteq\{F(\mu^{\square}), F(\mu^{\square,i_1,j_1})\}.$$

Pick an arbitrary vector
$$\widehat{v}^{i_1,j_1,\prime\prime}\in\left((\widetilde{\pi}^{i_1,j_1})^{\circ}\right)^{ U(\F_p),\widetilde{\mu}^{i_1,j_1,\prime}} 
\setminus\varpi_EM^{i_1,j_1,\prime}[\mathfrak{p}]$$
and denote its image in $\left((\widetilde{\pi}^{i_1,j_1})^{\circ}\otimes_{\cO_E}\F\right)^{ U(\F_p),\widetilde{\mu}^{i_1,j_1,\prime}}$ by $v^{i_1,j_1,\prime\prime}$. Then, by Proposition~\ref{general nonvanishing}, we obtain
\begin{equation}\label{nonvanishing before comparison}
0\neq \F[\mathcal{S}^{i_1,j_1}v^{i_1,j_1}]=\F[\mathcal{S}^{i_1,j_1,\prime}v^{i_1,j_1,\prime\prime}]\subseteq (\widetilde{\pi}^{i_1,j_1})^{\circ}\otimes_{\cO_E}\F\hookrightarrow M_\F[\mathfrak{m}_{\overline{r}}]
\end{equation}

We recall the open compact subgroups $K^L$, $K^L(1)$, $I^L$ and $I^L(1)$ of $L(\Q_p)$ from (\ref{open compact subgroups}).  By Corollary \ref{lattice coinduction} we know that there exists a $\cO_E$-lattice $(\widetilde{\pi}^{i_1,j_1,L})^{\circ}$ in $$\widetilde{\pi}^{i_1,j_1,L}:=\mathrm{Ind}^{ L(\F_p)}_{B\cap L(\F_p)}\widetilde{\mu}^{i_1,j_1}$$ as a $K^L$-representation such that
$$(\widetilde{\pi}^{i_1,j_1})^{\circ}=\mathrm{coInd}^{ G(\F_p)}_{ P(\F_p)}(\widetilde{\pi}^{i_1,j_1,L})^{\circ}$$
and
$$\mathrm{JH}\left(\mathrm{soc}_{ L(\F_p)}\left((\widetilde{\pi}^{i_1,j_1,L})^{\circ}\otimes_{\cO_E}\F\right)\right)\subseteq\{F^L(\mu^{\square}), F^L(\mu^{\square,i_1,j_1})\}.$$
Since $\mathrm{coInd}^{ G(\F_p)}_{ P(\F_p)}(\,\cdot\,)\cong\mathrm{Ind}^{ G(\F_p)}_{ P^-(\F_p)}(\,\cdot\,)$ and $(\,\cdot\,)_{ N^-(\F_p)}$ are left and right adjoint functors of each other, we deduce the existence of surjections of $\cO_E$-representations of $ L(\F_p)$
\begin{equation}\label{composition to L}
(\widetilde{\pi}^{i_1,j_1})^{\circ}|_{ L(\F_p)}\twoheadrightarrow(\widetilde{\pi}^{i_1,j_1})^{\circ}_{ N^-(\F_p)}\twoheadrightarrow (\widetilde{\pi}^{i_1,j_1,L})^{\circ}.
\end{equation}
We denote the composition (\ref{composition to L}) by $\mathrm{pr}$.

If we write explicitly
$$\mathrm{coInd}^{ G(\F_p)}_{ P(\F_p)}(\widetilde{\pi}^{i_1,j_1,L})^{\circ}=\{f:  G(\F_p)\rightarrow  (\widetilde{\pi}^{i_1,j_1,L})^{\circ}\mid f(p^-g)=p^-\cdot f(g)\,\,\forall p^-\in  P^-(\F_p)\}$$
where $p^-$ acts on $(\widetilde{\pi}^{i_1,j_1,L})^{\circ}$ through its image in $ L(\F_p)$, we can express $\mathrm{pr}$ by
\begin{equation}\label{explicit morphism}
\mathrm{pr}: (\widetilde{\pi}^{i_1,j_1})^{\circ}|_{ L(\F_p)}\twoheadrightarrow(\widetilde{\pi}^{i_1,j_1,L})^{\circ},\qquad f\mapsto f(1).
\end{equation}
By (\ref{explicit morphism}) we obtain the following equalities
\begin{equation*}
\left\{
\begin{array}{lll}
\cO_E\left[\mathrm{pr}(\widehat{v}^{i_1,j_1})\right]&=&\left((\widetilde{\pi}^{i_1,j_1,L})^{\circ}\right)^{{U}\cap L(\F_p),\widetilde{\mu}^{i_1,j_1}};\\
\cO_E\left[\mathrm{pr}(\widehat{v}^{i_1,j_1,\prime\prime})\right]&=&\left((\widetilde{\pi}^{i_1,j_1,L})^{\circ}\right)^{{U}\cap L(\F_p),\widetilde{\mu}^{i_1,j_1,\prime}}.
\end{array}
\right.
\end{equation*}
By applying Proposition \ref{general nonvanishing} to $(\widetilde{\pi}^{i_1,j_1,L})^{\circ}\otimes_{\cO_E}\F$ we deduce that
\begin{equation}\label{Jacquet nonvanishing}
0\neq \F[\mathcal{S}_{i_1,j_1}\left((\mathrm{pr}\otimes_{\cO_E}\F)v^{i_1,j_1}\right)]=\F[\mathcal{S}_{i_1,j_1}^{\prime}\left((\mathrm{pr}\otimes_{\cO_E}\F)v^{i_1,j_1,\prime\prime}\right)]\subseteq (\widetilde{\pi}^{i_1,j_1,L})^{\circ}\otimes_{\cO_E}\F.
\end{equation}
By Theorem \ref{theo: identity} we have
$$0\neq E[\widehat{\mathcal{S}}_{i_1,j_1}\left(\mathrm{pr}(\widehat{v}^{i_1,j_1})\right)]=E[\widehat{\mathcal{S}}_{i_1,j_1}^{\prime}\left(\mathrm{pr}(\widehat{v}^{i_1,j_1,\prime\prime})\right)]\subseteq \widetilde{\pi}^{i_1,j_1,L},$$
and thus together with (\ref{Jacquet nonvanishing}) we deduce that
\begin{equation}\label{to be compared}
\varpi_E(\widetilde{\pi}^{i_1,j_1,L})^{\circ}\nsupseteq \cO_E[\widehat{\mathcal{S}}_{i_1,j_1}\left(\mathrm{pr}(\widehat{v}^{i_1,j_1})\right)]=\cO_E[\widehat{\mathcal{S}}_{i_1,j_1}^{\prime}\left(\mathrm{pr}(\widehat{v}^{i_1,j_1,\prime\prime})\right)]\subseteq (\widetilde{\pi}^{i_1,j_1,L})^{\circ}.
\end{equation}

We define
$$\Pi^{i_1,j_1}:=\langle G(\Q_p)\widehat{v}^{i_1,j_1}\rangle_E.$$
As $\widehat{v}^{i_1,j_1}$ is primitive, by Definition \ref{definition: primitive} we deduce that $\Pi^{i_1,j_1}$ is irreducible and there exists a smooth character $\chi: T(\Q_p)\rightarrow E^{\times}$ satisfying $\chi|_{T(\Z_p)}=\widetilde{\mu}^{i_1,j_1}$ such that
$$\Pi^{i_1,j_1}\cong \mathrm{Ind}^{G(\Q_p)}_{B(\Q_p)}\chi.$$
In particular, we notice that
\begin{equation}\label{sub invariant}
(\Pi^{i_1,j_1})^{K(1)}=\widetilde{\pi}^{i_1,j_1}.
\end{equation}
We define
$$B^{\prime}:=N^-\cdot(B\cap L),$$
and thus $B^{\prime}$ is a Borel subgroup of $G$ as it is conjugated to $B$ via $w_0w^L_0$.

By the intertwining between generic smooth principal series in characteristic zero in \cite{Shahidi}, Chapter 4, we deduce the existence of a smooth character $\chi^{\prime}: T(\Q_p)\rightarrow E^{\times}$ such that
$$\mathrm{Ind}^{G(\Q_p)}_{B(\Q_p)}\chi\cong \mathrm{Ind}^{G(\Q_p)}_{B^{\prime}(\Q_p)}\chi^{\prime}.$$
As $T_{i_1,j_1}$ commutes with $w_0w^L_0$, we observe from the above intertwining isomorphism that
\begin{equation}\label{equality of ch}
\chi^{\prime}|_{T_{i_1,j_1}(\Q_p)}=\chi|_{T_{i_1,j_1}(\Q_p)}.
\end{equation}
Then we define
$$\Pi^{i_1,j_1,L}:=\mathrm{Ind}^{L(\Q_p))}_{B\cap L(\Q_p)}\chi^{\prime}$$
and thus
$$\Pi^{i_1,j_1}\cong \mathrm{Ind}^{G(\Q_p)}_{B(\Q_p)}\chi\cong \mathrm{Ind}^{G(\Q_p)}_{B^{\prime}(\Q_p)}\chi^{\prime}=\mathrm{Ind}^{G(\Q_p)}_{P^-(\Q_p)}\Pi^{i_1,j_1,L}.$$

In particular, we also have
\begin{equation}\label{sub invariant1}
(\Pi^{i_1,j_1,L})^{K^L(1)}=\widetilde{\pi}^{i_1,j_1,L}
\end{equation}
As $\mathrm{Ind}^{G(\Q_p)}_{P^-(\Q_p)}(\,\cdot\,)$ and $(\,\cdot\,)_{N^-(\Q_p)}$ are left and right adjoint functor of each other, we have surjections of $L(\Q_p)$-representation
\begin{equation}\label{composition char 0}
\Pi^{i_1,j_1}|_{L(\Q_p)}\twoheadrightarrow (\Pi^{i_1,j_1})_{N^-(\Q_p)}\twoheadrightarrow \Pi^{i_1,j_1,L},
\end{equation}
and we denote the composition (\ref{composition char 0}) by $\mathrm{Pr}$.

If we write explicitly
$$\mathrm{Ind}^{G(\Q_p)}_{P^-(\Q_p)}\Pi^{i_1,j_1,L}=\{f: G(\Q_p)\rightarrow  \Pi^{i_1,j_1,L}\mid f(p^-g)=p^-\cdot f(g)\mbox{ for all }p^-\in P^-(\Q_p)\}$$
where $p^-$ acts on $\Pi^{i_1,j_1,L}$ through its image in $L(\Q_p)$, we can express $\mathrm{Pr}$ by
\begin{equation}\label{explicit morphism1}
\mathrm{Pr}: \Pi^{i_1,j_1}|_{L(\Q_p)}\twoheadrightarrow\Pi^{i_1,j_1,L},\qquad f\mapsto f(1).
\end{equation}
By (\ref{sub invariant}), (\ref{sub invariant1}), (\ref{explicit morphism}) and (\ref{explicit morphism1}), the morphism $\mathrm{pr}$ and $\mathrm{Pr}$ fit into the following commutative diagram:
\begin{equation}\label{main diagram in the proof of main theorem}
\xymatrix{
(\widetilde{\pi}^{i_1,j_1})^{\circ}\otimes_{\cO_E}\F\ar@{->>}[rr]^{\mathrm{pr}\otimes_{\cO_E}\F}&&(\widetilde{\pi}^{i_1,j_1,L})^{\circ}\otimes_{\cO_E}\F \\
(\widetilde{\pi}^{i_1,j_1})^{\circ}\ar@{->>}[rr]^{\mathrm{pr}}\ar@{^{(}->}[d]_{}\ar@{->>}[u]_{}&&(\widetilde{\pi}^{i_1,j_1,L})^{\circ}\ar@{->>}[u]_{} \ar@{^{(}->}[d]_{} \\
\widetilde{\pi}^{i_1,j_1}=(\Pi^{i_1,j_1})^{K(1)}\ar@{->>}[rr]^{\mathrm{pr}\otimes_{\cO_E}E}\ar@{^{(}->}[d]_{}&&\widetilde{\pi}^{i_1,j_1,L}=(\Pi^{i_1,j_1,L})^{K^L(1)} \ar@{^{(}->}[d]_{} \\
\Pi^{i_1,j_1}\ar@{->>}[rr]^{\mathrm{Pr}}&&\Pi^{i_1,j_1,L}. }
\end{equation}
It is clear from the commutative diagram (\ref{main diagram in the proof of main theorem}) that we can use the notation $\mathrm{Pr}(v)$ instead of $\mathrm{pr}(v)$ for any $v\in(\widetilde{\pi}^{i_1,j_1})^{\circ}$.


Since $\Xi_{i_1,j_1}\in L(\Q_p)$
lies in the normalizer of $I^L(1)$ in $L(\Q_p)$, we deduce that
\begin{equation*}
\mathrm{Pr}\left(\Xi^{i_1,j_1}(\widehat{v}^{i_1,j_1})\right)=\Xi_{i_1,j_1}\left(\mathrm{Pr}(\widehat{v}^{i_1,j_1})\right) \in \Xi_{i_1,j_1}(\Pi^{i_1,j_1,L})^{I^L(1),\widetilde{\mu}^{i_1,j_1}}.
\end{equation*}
Note that
$$\Xi_{i_1,j_1}(\Pi^{i_1,j_1,L})^{I^L(1),\widetilde{\mu}^{i_1,j_1}} =(\Pi^{i_1,j_1,L})^{I^L(1),\widetilde{\mu}^{i_1,j_1,\prime}} =(\widetilde{\pi}^{i_1,j_1,L})^{{U}\cap L(\F_p),\widetilde{\mu}^{i_1,j_1,\prime}}.$$
As a result, we have
\begin{equation}\label{up to scalar}
E\left[\mathrm{Pr}\left(\Xi^{i_1,j_1}(\widehat{v}^{i_1,j_1})\right)\right]=E\left[\mathrm{Pr}(\widehat{v}^{i_1,j_1,\prime\prime})\right]
\end{equation}

By applying Theorem \ref{theo: identity} to $\Pi^{i_1,j_1,L}$ we deduce that
\begin{equation}\label{to be compared1}
\widehat{S}^{\prime}_{i_1,j_1}\bullet \Xi_{i_1,j_1}\left(\mathrm{Pr}(\widehat{v}^{i_1,j_1})\right)=\kappa_{i_1,j_1}\widetilde{C}(\chi^{\prime})\widehat{S}_{i_1,j_1}\left(\mathrm{Pr}(\widehat{v}^{i_1,j_1})\right)
\end{equation}
for some $\widetilde{C}(\chi^{\prime})\in\cO_E^{\times}$ and for $\kappa_{i_1,j_1}$ satisfying
\begin{equation}\label{congruence of constant}
\kappa_{i_1,j_1}\equiv\varepsilon^{i_1,j_1}\mathcal{P}^{i_1,j_1}(b_{n-1},\cdots,b_0) \pmod{\varpi_E}.
\end{equation}
Comparing (\ref{to be compared1}) with (\ref{to be compared}) we deduce the existence of $C_2\in\cO_E^{\times}$ such that
$$\widehat{S}^{\prime}_{i_1,j_1}\bullet \Xi_{i_1,j_1}\left(\mathrm{Pr}(\widehat{v}^{i_1,j_1})\right)=C_2\widehat{S}^{\prime}_{i_1,j_1}\left(\mathrm{Pr}(\widehat{v}^{i_1,j_1,\prime\prime})\right)$$
and thus by (\ref{up to scalar}) we obtain
$$\Xi_{i_1,j_1}\left(\mathrm{Pr}(\widehat{v}^{i_1,j_1})\right)= C_2\mathrm{Pr}(\widehat{v}^{i_1,j_1,\prime\prime})\in (\widetilde{\pi}^{i_1,j_1,L})^{{U}\cap L(\F_p),\widetilde{\mu}^{i_1,j_1,\prime}}.$$

Now we let $$\widehat{v}^{i_1,j_1,\prime}:=C_2\widehat{v}^{i_1,j_1,\prime\prime}\in(\widetilde{\pi}^{i_1,j_1})^{\circ}$$ and denote by $v^{i_1,j_1,\prime}$ the image of $\widehat{v}^{i_1,j_1,\prime}$ in $(\widetilde{\pi}^{i_1,j_1})^{\circ}\otimes_{\cO_E}\F$. Then by Definition \ref{definition: connected}, we know that $v^{i_1,j_1}$ and $v^{i_1,j_1,\prime}$ are connected. Moreover, by definition of $\widehat{v}^{i_1,j_1,\prime}$ we have
\begin{equation}\label{comparison again1}
\Xi_{i_1,j_1}\left(\mathrm{Pr}(\widehat{v}^{i_1,j_1})\right)=\mathrm{Pr}(\widehat{v}^{i_1,j_1,\prime})\in(\widetilde{\pi}^{i_1,j_1,L})^{{U}\cap L(\F_p),\widetilde{\mu}^{i_1,j_1,\prime}},
\end{equation}
and we deduce from (\ref{nonvanishing before comparison}) the equality
\begin{equation}\label{nonvanishing local global}
\mathcal{S}^{i_1,j_1}v^{i_1,j_1}=C\mathcal{S}^{i_1,j_1,\prime}v^{i_1,j_1,\prime}
\end{equation}
for some $C\in\F^{\times}$.  Hence, we can lift the equality (\ref{nonvanishing local global}) into $(\widetilde{\pi}^{i_1,j_1})^{\circ}$ as
\begin{equation}\label{lift of nonvanishing}
\widehat{\mathcal{S}}^{i_1,j_1}\widehat{v}^{i_1,j_1}=\widetilde{C}\widehat{\mathcal{S}}^{i_1,j_1,\prime}\widehat{v}^{i_1,j_1,\prime}+\varpi_Ev\in (\widetilde{\pi}^{i_1,j_1})^{\circ}
\end{equation}
for some $\widetilde{C}\in\cO_E^{\times}$ that lifts $C$ and for some $v\in(\widetilde{\pi}^{i_1,j_1})^{\circ}$.

We consider the image of (\ref{lift of nonvanishing}) under the morphism $\mathrm{Pr}$ (or rather $\mathrm{pr}$):
$$\mathrm{Pr}\left(\widehat{\mathcal{S}}^{i_1,j_1}\widehat{v}^{i_1,j_1}\right)=\widetilde{C}\mathrm{Pr}\left(\widehat{\mathcal{S}}^{i_1,j_1,\prime}\widehat{v}^{i_1,j_1,\prime}\right)+\varpi_E\mathrm{Pr}(v)\in(\widetilde{\pi}^{i_1,j_1,L})^{\circ}$$
or equivalently
\begin{equation}\label{comparison again}
\widehat{\mathcal{S}}_{i_1,j_1}\left(\mathrm{Pr}(\widehat{v}^{i_1,j_1})\right)= \widetilde{C}\widehat{\mathcal{S}}_{i_1,j_1}^{\prime}\left(\mathrm{Pr}(\widehat{v}^{i_1,j_1,\prime})\right)+\varpi_E\mathrm{Pr}(v)\in(\widetilde{\pi}^{i_1,j_1,L})^{\circ}.
\end{equation}
Comparing (\ref{comparison again}) with (\ref{comparison again1}), we deduce that
$$
\mathrm{Pr}(v)\in \mathrm{Pr}((\widetilde{\pi}^{i_1,j_1})^{\circ})\cap E\left[\widehat{\mathcal{S}}_{i_1,j_1}\left(\mathrm{Pr}(\widehat{v}^{i_1,j_1})\right)\right].
$$
Note that
\begin{align*}
\mathrm{Pr}((\widetilde{\pi}^{i_1,j_1})^{\circ})\cap E\left[\widehat{\mathcal{S}}_{i_1,j_1}\left(\mathrm{Pr}(\widehat{v}^{i_1,j_1})\right)\right]&
=(\widetilde{\pi}^{i_1,j_1,L})^{\circ}\cap E\left[\widehat{\mathcal{S}}_{i_1,j_1}\left(\mathrm{Pr}(\widehat{v}^{i_1,j_1})\right)\right] \\ &=\cO_E\left[\widehat{\mathcal{S}}_{i_1,j_1}\left(\mathrm{Pr}(\widehat{v}^{i_1,j_1})\right)\right].
\end{align*}
By~(\ref{to be compared1}) we deduce
$$\widetilde{C}\equiv \kappa_{i_1,j_1}\widetilde{C}(\chi^{\prime})\pmod{\varpi_E}.$$
Therefore it follows from (\ref{equality of ch}) and (\ref{equality of char 0}) as well as the congruences (\ref{Frob}) and (\ref{congruence of constant}) that
$$C=\overline{\kappa_{i_1,j_1}\widetilde{C}(\chi^{\prime})}=\overline{\kappa_{i_1,j_1}\widetilde{C}(\chi)}=\varepsilon^{i_1,j_1}\mathcal{P}^{i_1,j_1}(b_{n-1}, \cdots,b_0)\mathrm{FL}^{i_0,j_0}_n(\rhobar_0),$$
which completes the proof.
\end{proof}

\begin{rema}\label{main remark}
In Theorem \ref{theo: lgc}, we construct a vector $v^{i_1,j_1,\prime}$ starting with a primitive vector $v^{i_1,j_1}$ such that $v^{i_1,j_1}$ and $v^{i_1,j_1,\prime}$ are connected. However, the definition of 'connected' (c.f. Definition \ref{definition: connected}) and the construction of $v^{i_1,j_1,\prime}$ involves the lifts $\widehat{v}^{i_1,j_1}$ (resp. $\widehat{v}^{i_1,j_1,\prime}$) of $v^{i_1,j_1}$ (resp. $v^{i_1,j_1,\prime}$) in characteristic zero. We emphasize that our proof of Theorem \ref{theo: lgc} automatically implies that $v^{i_1,j_1}$ is independent of the choice of $\widehat{v}^{i_1,j_1}$ and the lift of the action of $\Xi^{i_1,j_1}$ on $M_\F[\mathfrak{m}_{\rbar}]$ into characteristic zero, although we do not show how to construct $v^{i_1,j_1,\prime}$ from $v^{i_1,j_1}$ without lifting.
\end{rema}

\begin{coro}
Keep the notation of Theorem~\ref{theo: lgc} and assume that each assumption in Theorem~\ref{theo: lgc} holds for all $(i_0,j_0)$ such that $0\leq j_0<j_0+1<i_0\leq n-1$.

Then the Galois representation $\rhobar_0$ is determined by $M_\F[\mathfrak{m}_{\rbar}]$ in the sense of Remark \ref{main remark}.
\end{coro}

\begin{proof}
We follow the notation in Section $3.4$ of \cite{BH}. As $\rhobar_0$ is ordinary, we can view it as a morphism
$$\rhobar_0:\quad G_{\Q_p}\rightarrow\widehat{B}(\F)\subseteq\widehat{G}(\F)$$
where $\widehat{B}$ (resp. $\widehat{G}$) is the dual group of $B$ (resp. $G$). The local class field theory gives us a bijection between smooth characters of $\Q_p^{\times}$ and the smooth characters of the Weil group of $\Q_p$ in characteristic zero. This bijection restricts to a bijection between smooth characters of $\Q_p^{\times}$ and smooth characters of $\mathrm{Gal}(\overline{\Q}_p/\Q_p)$ both with values in $\mathcal{O}_E^{\times}$. Taking mod $p$ reduction and then taking products we reach a bijection between smooth $\F$-characters of $T(\Q_p)$ and $\mathrm{Hom}\left(\mathrm{Gal}(\overline{\Q}_p/\Q_p), \widehat{T}(\F)\right)$. We can therefore define $\chi_{\rhobar_0}$ as the character of $T(\Q_p)$ corresponding to the composition
$$\widehat{\chi}_{\rhobar_0}:\quad \mathrm{Gal}(\overline{\Q}_p/\Q_p)\rightarrow\widehat{B}(\F)\twoheadrightarrow\widehat{T}(\F).$$
In \cite{BH}, a closed subgroup $C_{\rhobar_0}\subseteq B$ (at the beginning of section 3.2) and a subset $W_{\rhobar_0}$ ($(2)$ before Lemma 2.3.6) of $W$ is defined.

As we are assuming that $\rhobar_0$ is maximally non-split, we observe that $C_{\rhobar_0}=B$ and $W_{\rhobar_0}=\{1\}$ in our case. Therefore by the definition of $\Pi^{ord}(\rhobar_0)$ in \cite{BH} before Definition 3.4.3, we know that it is indecomposable with socle
$$\mathrm{Ind}^{G(\Q_p)}_{B^-(\Q_p)}\chi_{\rhobar_0}\cdot(\omega^{-1}\circ\theta)$$
where $\theta\in X(T)$ is a twist character defined after Conjecture 3.1.2 in \cite{BH} which can be chosen to be $\eta$ in our notation. Then as a Corollary of Theorem 4.4.7 in \cite{BH}, we deduce that $S(U^v,V^{\prime})[\mathfrak{m}_{\overline{r}}]$ determines $\chi_{\rhobar_0}$ and hence $\widehat{\chi}_{\rhobar_0}$.

Now, we know that $\rhobar_0$ is determined by the Fontaine--Laffaille parameters $\{\mathrm{FL}_n^{i_0,j_0}(\rhobar_0)\in\mathbb{P}^1(\F)\mid 0\leq i_0<i_0+1<j_0\leq n-1\}$ and $\widehat{\chi}_{\rhobar_0}$, up to isomorphism. Our conclusion thus follows from Theorem \ref{theo: lgc}.
\end{proof}

\bibliographystyle{alpha}

\end{document}